\documentclass[12pt]{amsart}

\usepackage{multind}
\makeindex{A}
\makeindex{B}

\usepackage{float}
\usepackage{geometry}
\usepackage[all]{xy}
\usepackage{graphicx,youngtab}
\usepackage{amssymb, amsmath}
\usepackage{epstopdf}
\usepackage{mathtools}
\usepackage{xr}

\DeclareMathAlphabet{\mathpzc}{OT1}{pzc}{m}{it}

\newtheorem{thm}{Theorem}[section]

\renewcommand\thethmintro{\Alph{thmintro}}

\newtheorem{cor}[thm]{Corollary}

\newtheorem{lem}[thm]{Lemma}
\newtheorem{prop}[thm]{Proposition}
\newtheorem{fact}[thm]{Fact}
\newtheorem{prob}{Problem}

\newtheorem{example}[thm]{Example}
\theoremstyle{definition}
\newtheorem{defn}[thm]{Definition}
\theoremstyle{remark}
\newtheorem{rem}[thm]{Remark}
\newtheorem{claim}[thm]{Claim}
\numberwithin{figure}{section}%  Fig. 4.1
\numberwithin{table}{section}
\numberwithin{equation}{section}
\newcommand{\eps}{\varepsilon}
\newcommand{\To}{\longrightarrow}

\newcommand{\sgn}{\mathrm{sgn}}

\newcommand{\indpg}{\operatorname{ind}_{\mathfrak p}^{\mathfrak g}}
\newcommand{\indpgprime}{\operatorname{ind}_{\mathfrak p'}^{\mathfrak g'}}

\newcommand{\trans}{{}^t\!}
\newcommand{\restn}{\mathrm{Rest}_{x_n=0}}
\newcommand{\one}{1\hspace{-0.15cm}1}
\mathchardef\mhyphen="2D

           \newcommand{\eq}[1][r]
   {\ar@<-3pt>@{-}[#1]
    \ar@<-1pt>@{}[#1]|<{}="gauche"
    \ar@<+0pt>@{}[#1]|-{}="milieu"
    \ar@<+1pt>@{}[#1]|>{}="droite"
    \ar@/^2pt/@{-}"gauche";"milieu"
    \ar@/_2pt/@{-}"milieu";"droite"}
    
\newcommand{\Exterior}{\mathchoice{{\textstyle\bigwedge}}%
    {{\bigwedge}}%
    {{\textstyle\wedge}}%
    {{\scriptstyle\wedge}}}

\newcommand{\prn}{\Pi_{n-1}}
\newcommand{\iotan}{\iota_{\frac{\partial}{\partial x_n}}}
\newcommand{\Fdpi}[2]{\widehat{d\pi_{(#1, #2)^*}}}
\newcommand{\R}{\mathbb R}  %@@
\newcommand{\Z}{\mathbb Z}  %@@
\newcommand{\N}{\mathbb N}  %@@
\newcommand{\C}{\mathbb C}  %@@
\newcommand{\hooklongrightarrow}{\lhook\joinrel\longrightarrow}
\newcommand{\Kla}{K_{\ell,a}}

\newcommand{\rightsetse}[1]{%
\hidewidth\rotatebox[origin=c]{-45}{$\xrightarrow{\kern2em}$}
     \rlap{\raisebox{1ex}
     {$\kern-.8em\scriptstyle #1$}}\hidewidth}
\newcommand{\rightsetsw}[1]{%
\hidewidth\rotatebox[origin=c]{45}{$\xleftarrow{\kern2em}$}
     \rlap{\raisebox{.1ex}
     {$\kern-.8em\scriptstyle #1$}}\hidewidth}
\newcommand{\leftsetsw}[1]{%
\hidewidth
     \llap{\raisebox{1ex}
     {$\scriptstyle #1$\kern-.8em}}
    \rotatebox[origin=c]{45}{$\xleftarrow{\kern2em}$}\hidewidth}
\newcommand{\rightsetnw}[1]{%
\hidewidth\rotatebox[origin=c]{135}{$\xrightarrow{\kern2em}$}
     \rlap{\raisebox{1ex}
     {$\kern-.8em\scriptstyle #1$}}\hidewidth}
\newcommand{\rightsetd}[1]{%
\hidewidth\rotatebox[origin=c]{-90}{$\xrightarrow{\kern2em}$}
     \rlap{{$\scriptstyle #1$}}\hidewidth}
\usepackage[breaklinks]{hyperref}

\title[Conformal symmetry breaking for differential forms]
{Conformal symmetry breaking operators for differential forms on spheres}

\author{Toshiyuki Kobayashi}
\author{Toshihisa Kubo}
\author{Michael Pevzner}

\address{T. Kobayashi, Kavli IPMU (WPI)
and Graduate School of Mathematical Sciences, 
The University of Tokyo, 3-8-1 Komaba, Meguro, Tokyo, 153-8914 Japan}
\email{toshi@ms.u-tokyo.ac.jp}

\address{T. Kubo, Faculty of Economics, Ryukoku University,
67 Tsukamoto-cho, Fukakusa, Fushimi-ku, Kyoto 612-8577, Japan}
\email{toskubo@econ.ryukoku.ac.jp}

\address{M. Pevzner, Laboratoire de Math\'ematiques de Reims,
Universit\'e de Reims-Champagne-Ardenne,
FR 3399 CNRS, F-51687, Reims, France}
\email{pevzner@univ-reims.fr}

\begin{document}

\begin{abstract}
We give a complete classification of 
 conformally covariant differential operators between the spaces of $i$-forms on the sphere $S^n$ and $j$-forms on the totally geodesic hypersphere $S^{n-1}$. 
Moreover, we find explicit formul{\ae} for 
these new matrix-valued operators in the flat coordinates
in terms of basic operators in differential geometry and classical
orthogonal polynomials. 
We also establish matrix-valued factorization identities
among  all possible combinations of conformally covariant differential operators. 
The main machinery of the proof is the ``F-method" based on the
``algebraic Fourier transform of Verma modules" 
(Kobayashi--Pevzner [Selecta Math. 2016])
and its extension to matrix-valued case developed here.
A short summary of the main results was announced in 
[C.\ R.\ Acad.\ Sci.\ Paris, 2016].

\vskip7pt

\end{abstract}

\maketitle
\tableofcontents

%%%%%%%%%%%%%%%%%%%%%%%%%%%%%%%%%%%%%%%%%%%%%%%
\renewcommand\thethmintro{\Alph{thmintro}}
\section{Introduction}\label{sec:intro}

Let $(X,g)$ be a pseudo-Riemannian manifold. Suppose that a Lie group $G$ acts conformally on $X$. This means that there exists a positive-valued function $\Omega\in C^\infty(G\times X)$ (\index{B}{conformal factor}\emph{conformal factor}) such that
\index{A}{11zOmega@$\Omega(h,x)$, conformal factor|textbf}
\begin{equation*}
L_h^*g_{h\cdot x}=\Omega(h,x)^2g_x\quad\mathrm{for\, all}\,h\in G, x\in X,
\end{equation*}
where 
we write $L_h\colon X\to X, x\mapsto h\cdot x$ for the action of $G$ on $X$. 
When $X$ is orientable, we define a locally constant function
\index{A}{or@$\mathpzc{or}$|textbf}
$\mathpzc{or}\colon G\times X\To\{\pm1\}$ 
by $\mathpzc{or}(h)(x)=1$ if $(L_h)_{*x}\colon T_xX \To T_{L_hx}X$ 
is orientation-preserving, and $=-1$ if it is orientation-reversing.

Since $\Omega$ satisfies a cocycle condition, we can
form a family of representations 
$\varpi^{(i)}_{u,\delta}$ of $G$ with parameters 
$u\in\C$ and $\delta\in\Z/2\Z$
 on the space 
\index{A}{E10@$\mathcal E^i(X)$}
$\mathcal E^i(X)$ of $i$-forms on $X$ $(0\leq i \leq \dim X)$ 
 defined by
\index{A}{1pi@$\varpi^{(i)}_{u,\delta}$, conformal representation on $i$-forms|textbf}
\begin{equation}\label{eqn:varpi}
\varpi^{(i)}_{u,\delta}(h)\alpha:=\mathpzc{or}(h)^\delta\Omega(h^{-1},\cdot)^u 
L_{h^{-1}}^*\alpha,\quad (h\in G).
\end{equation}
The representation $\varpi^{(i)}_{u,\delta}$ of the conformal group $G$ on $\mathcal E^i(X)$
will be simply denoted by
\index{A}{E1i@$\mathcal E^i(X)_{u,\delta}$, 
conformal representation on $i$-forms on $X$|textbf}
$\mathcal E^i(X)_{u,\delta}$, and referred to as 
\index{B}{conformal representation on $i$-forms|textbf}
conformal representations on $i$-forms.

Suppose that $Y$ is an orientable submanifold such that $g$ is nondegenerate on the tangent space $T_yY$ for all $y\in Y$
(this holds automatically if $g$ is positive definite). 
Then $Y$ is endowed with a pseudo-Riemannian structure $g\vert_{ Y}$, 
and we can define in a similar way a family of representations $\varpi^{(j)}_{v,\eps}$ on $\mathcal E^j(Y)$ ($v\in\C,\eps\in\Z/2\Z, 0\leq j\leq\dim Y$) of the group
\begin{equation*}
G':=\left\{h\in G\,:\,h\cdot Y=Y\right\}
\end{equation*}
which acts conformally on $(Y,g{\vert}_{ Y})$.
\vskip 7pt

The object of our study is 
differential operators $\mathcal D^{i\to j}\colon \mathcal E^i(X)\To\mathcal E^j(Y)$ 
that intertwine the two representations 
$\varpi^{(i)}_{u,\delta}\vert_{G'}$ and $\varpi^{(j)}_{v,\eps}$ of $G'$.
Here $\varpi^{(i)}_{u,\delta}\vert_{G'}$ stands for the restriction of the $G$-representation
$\varpi^{(i)}_{u,\delta}$ to the subgroup $G'$. We say that such $\mathcal D^{i\to j}$ is a 
\index{B}{differential symmetry breaking operator|textbf}
\emph{differential symmetry breaking operator} 
and denote by
\index{A}{DiffG'@$\mathrm{Diff}_{G'}(\mathcal E^i(X)_{u,\delta},\mathcal E^j(Y)_{v,\eps})$|textbf}
$\mathrm{Diff}_{G'}(\mathcal E^i(X)_{u,\delta},\mathcal E^j(Y)_{v,\eps})$
\index{A}{E1i@$\mathcal E^i(X)_{u,\delta}$, 
conformal representation on $i$-forms on $X$}
the space of  differential symmetry breaking operators.
 
We address the following problems:
\vskip 0.1in

\begin{prob}\label{prob:1}
Find a necessary and sufficient 
condition on 6-tuple $(i,j,u,v,\delta,\eps)$ such that there exist nontrivial
differential symmetry breaking operators. More precisely, determine the dimension
of $\mathrm{Diff}_{G'}\left(\mathcal{E}^i(X)_{u,\delta}, \mathcal{E}^j(Y)_{v,\eps}\right)$.
\end{prob}

\begin{prob}\label{prob:2}
Construct explicitly a basis of 
$\mathrm{Diff}_{G'}\left(\mathcal{E}^i(X)_{u,\delta}, \mathcal{E}^j(Y)_{v,\eps}\right)$.
\end{prob}

In the case where $X=Y$, $G=G'$, and $i = j = 0$,
a classical prototype of such operators is 
a second order differential operator called the \index{B}{Yamabe operator}
Yamabe operator
\begin{equation*}
\Delta+\frac{n-2}{4(n-1)}\kappa\in \mathrm{Diff}_{G}(\mathcal E^0(X)_{\frac n2-1,\delta},\mathcal E^0(X)_{\frac n2+1,\delta}),
\end{equation*}
where $n$ is the dimension of the manifold $X$, 
$\Delta$ is the Laplace--Beltrami operator, and $\kappa$ is the
scalar curvature, see \cite{KO1}, for instance.
Conformally equivariant differential operators of higher order are also known:
the Paneitz operator (fourth order) \cite{P08}, 
which appears in four dimensional supergravity \cite{FT},
or more generally, the so-called 
\index{B}{GJMS operator}
GJMS operators (\cite{GJMS}) are such operators.
Analogous differential operators on  forms  
($i=j$ case) were studied 
by Branson \cite{Branson}. 
The exterior derivative $d$ and the codifferential $d^*$ also give 
examples of conformally covariant operators on forms, namely,
$j=i+1$ and $i-1$, respectively, with appropriate choice of $(u,v,\delta, \eps)$.
Maxwell's equations in four dimension can be expressed in terms of conformally 
covariant operators on forms.

Let us consider the more general case where $Y\neq X$ and $G' \neq G$. 
An obvious example of symmetry breaking operators is the restriction operator
$\mathrm{Rest}_Y$ which belongs to 
$\mathrm{Diff}_{G'}\left(\mathcal{E}^i(X)_{u,\delta}, \mathcal{E}^i(Y)_{v,\eps}\right)$
if $u=v$ and $\delta \equiv \eps \equiv 0  \; \mathrm{mod}\; 2$. 
Another elementary example is 
\index{A}{RiotaX@$\mathrm{Rest}_Y\circ\iota_{N_Y(X)}$}
$\mathrm{Rest}_Y\circ \iota_{N_Y(X)} \in 
\mathrm{Diff}_{G'}\left(\mathcal{E}^i(X)_{u,\delta},
\mathcal{E}^{i-1}(Y)_{v,\eps}\right)$
with $v = u+1$ and $\delta \equiv \eps \equiv 1 \; \mathrm{mod} \; 2$
where 
\index{A}{1iotaYX@$\iota_{N_Y(X)}$|textbf}
$\iota_{N_Y(X)}$ denotes
\index{B}{interior multiplication}
the interior multiplication by the normal vector field to $Y$ in $X$ 
when $Y$ is of codimension one in $X$
(see Proposition \ref{prop:1531110}).

In the model space
$(X,Y)=(S^n,S^{n-1})$, the pair $(G,G')$ of conformal groups amounts
to $(O(n+1,1),O(n,1))$ modulo center (see Lemma \ref{lem:20160319}), and
Problems \ref{prob:1} and \ref{prob:2} have been recently solved for $i=j=0$ by Juhl \cite{Juhl}. See also \cite{K14} and \cite{KOSS15}  
for different approaches, \emph{i.e.}, by the residue calculus
and the F-method, respectively. 
The classification of nonlocal symmetry breaking operators for $i=j=0$
has been also accomplished recently in \cite{KS13}.
On the other hand, the case $n=2$ with 
$(i,j)=(1,0)$ was studied in \cite{KKP15} with emphasis 
on the relation to the Rankin--Cohen brackets
\cite{C75, EZ85, Z94}.

\smallskip

This work gives a complete solution to Problems \ref{prob:1} and \ref{prob:2}
 for all $i$ and $j$ in the model space $(X,Y) = (S^n, S^{n-1})$: 
 we classify all differential 
 symmetry breaking operators from $i$-forms on $S^n$ to $j$-forms 
 on $S^{n-1}$ for all $i$ and $j$. We also find closed formul\ae{} for 
 these new operators in all the cases.

The key machinery of the proof is the F-method which has been 
recently introduced in \cite{K13} by the first author. 
See also \cite{K14, KOSS15, KP1}
for detailed account and some applications. The idea of the F-method is based on
the ``algebraic Fourier transform of Verma modules".
We shall develop an extension of the method
to the matrix-valued case in Chapter \ref{sec:Method}.

\vskip 0.1in

Let us state our main results. Here is a complete solution to Problem \ref{prob:1}
for the model space $(X,Y)=(S^n,S^{n-1})$ $(n\geq 3)$.

\begin{thm}\label{thm:1}
Let $n\geq3$. Suppose $0\leq i\leq n$, $0\leq j\leq n-1$, $u,v\in\C$, $\delta,\varepsilon\in\Z/2\Z$. Then the following three conditions 
on 6-tuple $(i,j,u,v,\delta,\eps)$ are equivalent:
\index{A}{Ei@$\mathcal{E}^i(S^n)_{u,\delta}$}
\begin{enumerate}
\item[(i)] $\mathrm{Diff}_{O(n,1)}(\mathcal E^i(S^n)_{u,\delta}, \mathcal E^j(S^{n-1})_{v,\eps})\neq\{0\}$,
\item[(ii)] $\dim_\C \mathrm{Diff}_{O(n,1)}(\mathcal E^i(S^n)_{u,\delta}, \mathcal E^j(S^{n-1})_{v,\eps})=1$,
\item[(iii)] One of the following twelve conditions holds.
\smallskip
\begin{enumerate}

\item[]\emph{Case (I).}
$j=i-2$, $2\leq i\leq n-1$, $(u,v)=(n-2i,n-2i+3)$, 
$\delta\equiv \eps \equiv 1\;\mathrm{mod}\;2$.
\vskip 0.05in

\item[]\emph{Case (I$'$).}
$(i,j) = (n,n-2)$, $u \in -n-\N$,
$v = 3-n$, $\delta \equiv \eps 
\equiv u+n+1\;\mathrm{mod}\;2$.
\vskip 0.05in

\item[]\emph{Case (II).}
$j=i-1$, $1\leq i\leq n$, $v-u\in\N_+$, 
$\delta\equiv \eps\equiv v-u\;\mathrm{mod}\;2$.
\vskip 0.05in

\item[]\emph{Case (III).}
$j=i$, $0\leq i\leq n-1$, $v-u\in\N$, 
$\delta\equiv \eps\equiv v-u\;\mathrm{mod}\;2$.
\vskip 0.05in

\item[]\emph{Case (IV).}
$j=i+1$, $1\leq i\leq n-2$, $(u,v)=(0,0)$, 
$\delta \equiv \eps \equiv 0\;\mathrm{mod}\;2$.
\vskip 0.05in

\item[]\emph{Case (IV$'$).}
$(i,j) = (0,1)$, $u \in - \N$, $v=0$, $\delta \equiv \eps \equiv u\;\mathrm{mod}\;2$.
\vskip 0.05in

\item[]\emph{Case ($*$I).}
$j=n-i+1$, $2 \leq i \leq n-1$, $u = n-2i$, $v=0$,
$\delta \equiv 1$, $\eps \equiv 0 \;\mathrm{mod}\;2$.
\vskip 0.05in

\item[]\emph{Case ($*$I$'$).}
$(i,j)=(n,1)$, $u \in -n - \N$,
$v=0$, $\delta \equiv \eps + 1 \equiv u+n+1\;\mathrm{mod}\;2$.
\vskip 0.05in

\item[]\emph{Case ($*$II).}
$j = n-i$, $1\leq i \leq n$, $v-u + n-2i \in \N$,
$\delta \equiv \eps + 1\equiv v-u+n+1 \;\mathrm{mod}\;2$.
\vskip 0.05in

\item[]\emph{Case ($*$III).}
$j=n-i-1$, $0\leq i \leq n-1$, $v-u+n-2i-1 \in \N$,
$\delta \equiv \eps + 1\equiv v-u+n+1 \;\mathrm{mod}\;2$.
\vskip 0.05in

\item[]\emph{Case ($*$IV).}
$j=n-i-2$, $1\leq i \leq n-2$, $(u,v) = (0, 2i-n+3)$,
$\delta \equiv 0$, $\eps  \equiv 1 \;\mathrm{mod}\;2$.
\vskip 0.05in

\item[]\emph{Case ($*$IV$'$).}
$(i,j) = (0,n-2)$, $u\in-\N$, $v=3-n$,
$\delta \equiv \eps + 1 \equiv u\;\mathrm{mod}\;2$.

\end{enumerate}
\end{enumerate}
\end{thm}

We shall give a proof of Theorem \ref{thm:1} in Section \ref{subsec:pfA}.

There are dualities in the twelve cases in Theorem \ref{thm:1}.
To be precise,
we set 
\begin{equation*}
\tilde{i}:=n-i, \; \tilde{j}:=n-j-1, \; \tilde{u}:=u+2i-n, \;
\tilde{v}:=v+2j-n+1, \; \tilde{\delta}\equiv \delta+1, \;
\tilde{\eps}\equiv \eps+1 \; \mathrm{mod}\;2.
\end{equation*}
Then it follows from the 
\index{B}{duality theorem for symmetry breaking operators
(conformal geometry)}
Hodge duality for symmetry breaking
operators (Theorem \ref{thm:XYduality}, see also Section \ref{subsec:pbAB}) that
$(i,j,u,v,\delta, \eps) \mapsto (\tilde{i}, \tilde{j}, \tilde{u}, \tilde{v}, \tilde{\delta}, \tilde{\eps})$
gives rise to the duality of parameters
\begin{equation*}
\mathrm{(I)} \iff \mathrm{(IV)}, \quad
\mathrm{(I}'\mathrm{)} \iff \mathrm{(IV}'\mathrm{)},\quad
\mathrm{(II)} \iff \mathrm{(III)},
\end{equation*}
and $(i,j,u,v,\delta, \eps) \mapsto (i,\tilde{j}, u, \tilde{v},\delta, \tilde{\eps})$ gives
rise to another duality of parameters
\begin{equation*}
\mathrm{(P)} \iff \mathrm{(*P)} \qquad
\text{for $P$ = I, I$'$, II, III, IV, IV$'$.} 
\end{equation*}
Differential symmetry breaking operators 
for the latter half, \emph{i.e.}, Cases ($*$I)--($*$IV$'$),
are given as the composition of the 
\index{B}{Hodge star operator}
Hodge star operator $*_{\R^{n-1}}$ and the 
corresponding symmetry breaking operators 
for the first half, \emph{i.e.}, Cases (I)--(IV$'$).

The equivalence (i) $\Leftrightarrow$ (ii) in Theorem \ref{thm:1} asserts
that differential symmetry breaking operators, if exist, 
are unique up to scalar multiplication for 
all the parameters $(i,j,u,v,\delta,\eps)$ if $n\geq 3$. 
This should be in contrast to the $n=2$ case, 
where the multiplicity jumps at countably many places to two
(\emph{cf.}\ \cite[Sect.\ 9]{KP2}).

The standard sphere $S^n$ is a conformal compactification of the 
flat Riemannian manifold $\R^n$.
Using the 
\index{B}{stereographic projection}
stereographic projection 
\index{A}{p@$p$, stereographic projection}
$p: S^n \To \R^n \cup \{\infty\}$, 
we give closed formul\ae{} of differential symmetry breaking operators 
in flat coordinates in Cases (I), (I$'$), (II), (III), (IV), and (IV$'$),
see Theorems \ref{thm:2ii-2}, \ref{thm:2}, \ref{thm:2ii} and \ref{thm:2ii+1}, respectively.
Change of coordinates in symmetry breaking operators from $\R^n$ to the conformal
compactification $S^n$
 is given by the 
 \index{B}{twisted pull-back}
 ``twisted pull-back" of the stereographic projection
in Section \ref{subsec:RtoS}. In order to
explain the explicit formul{\ae} of the symmetry breaking operators,
in the flat coordinates, we fix some notations for basic operators.
\vskip7pt

Suppose that a manifold $X$ 
is endowed with a pseudo-Riemannian structure $g$ of signature $(p,q)$
and an orientation.
Then, 
the metric tensor
$g$ induces a volume form $\mathrm{vol}_X$, and a pseudo-Riemannian structure on the cotangent
bundle $T^\vee X$, or more generally on the exterior power bundles $\Exterior^i(T^\vee X)$. The 
\index{B}{codifferential}
codifferential 
\index{A}{d*@$d^*$, codifferential|textbf}
$d^*\colon \mathcal E^i(X)\To\mathcal E^{i-1}(X)$ is the formal adjoint of 
the differential (exterior derivative) $d$ in the sense that
\begin{equation*}
\int_X g_x(\alpha, d\beta)\mathrm{vol}_X(x)=\int_X g_x(d^*\alpha,\beta)\mathrm{vol}_X(x)
\end{equation*}
for all $\alpha\in\mathcal E^i(X)$ and $\beta\in\mathcal E_c^{i-1}(X)$.
\index{B}{interior multiplication}
Interior multiplication $\iota_Z$ of an $i$-form $\omega$ by a vector field $Z$ is defined by
\begin{equation*}
\left(
\iota_Z\omega\right)(Z_1,\cdots,Z_{i-1}):=\omega(Z,Z_1,\cdots,Z_{i-1}).
\end{equation*}

For $\ell \in \N$ and $\mu \in \C$, we denote by 
\index{A}{C1ell1@$\widetilde C_\ell^\mu(t)$,
renormalized Gegenbauer polynomial}
$\widetilde{C}^\mu_\ell(t)$
\index{B}{Gegenbauer polynomial}
the Gegenbauer polynomial
which is renormalized in a way that $\widetilde{C}^{\mu}_\ell \not \equiv 0$
for any $\mu \in \C$
(see \eqref{eqn:Gegen2} in Appendix).
Then 
\index{A}{Iell@$I_\ell$, $\ell$-inflated polynomial}
\begin{equation*}
(I_\ell\widetilde{C}^\mu_\ell)(x,y): = 
x^{\frac{\ell}{2}}\widetilde{C}^\mu_\ell\left(\frac{y}{\sqrt{x}}\right)
\end{equation*}
is a polynomial of two variables $x$ and $y$. We replace formally $x$ by 
\index{A}{11D2DeltaR@$\Delta_{\R^{n-1}}$|textbf}
$-\Delta_{\R^{n-1}}=-\sum_{j=1}^{n-1}\frac{\partial^2}{\partial x_j^2}$ 
and $y$ by $\frac{\partial}{\partial x_n}$, and define
a family of scalar-valued differential operators on $\R^n$ of order $\ell$
\index{A}{Dell@$\mathcal D_\ell^\mu$|textbf}
\begin{equation}\label{eqn:Dl}
\mathcal D_\ell^\mu:=(I_\ell\widetilde{C}_\ell^\mu)
\left( -\Delta_{\R^{n-1}},{\frac{\partial}{\partial x_n}}\right).
\end{equation}
For instance, $\mathcal D^\mu_0=1, 
\mathcal D^\mu_1=2\frac{\partial}{\partial x_n}, 
\mathcal D^\mu_2=\Delta_{\R^{n-1}}+2(\mu+1)
\frac{\partial^2}{\partial x_n^2}$, 
$\mathcal D^\mu_3=2\Delta_{\R^{n-1}}\frac{\partial}{\partial x_n}
+\frac43(\mu+2)\frac{\partial^3}{\partial x_n^3}$,
etc. We regard $\mathcal{D}^\mu_\ell \equiv 0$ for negative integer $\ell$.

For $\mu\in\C$ and $a\in\N$, we set
\index{A}{1Cgamma@$\gamma(\mu,a)$|textbf}
\begin{equation}\label{eqn:gamma}
\gamma(\mu,a):=\frac{\Gamma\left(\mu+1+\left[\frac a2\right]\right)}{\Gamma\left(\mu+\left[\frac {a+1}2\right]\right)}=
\left\{
\begin{matrix*}[l]
1 & \mathrm{if}\quad a\;\mathrm{is\;odd},\\
\mu+\frac a2& \mathrm{if}\quad a\;\mathrm{is\;even}.
\end{matrix*}
\right.
\end{equation}

For $1\leq i\leq n$, we introduce a family of linear maps
\index{A}{Dii1@$\mathcal D_{u,a}^{i\to i-1}$|textbf}
$\mathcal D_{u,a}^{i\to i-1}\colon \mathcal E^i(\R^n)\to\mathcal E^{i-1}(\R^{n-1})$ 
with parameters $u \in \C$ and $a \in \N$
by
\begin{alignat}{2}
\mathcal D_{u,a}^{i\to i-1}&:=
&& \mathrm{Rest}_{x_n=0}
\circ\left(-\mathcal D_{a-2}^{\mu+1}d_{\mathbb{R}^n}d^*_{\mathbb{R}^n}
\iota_{\frac{\partial}{\partial x_n}}
-\gamma(\mu,a) \mathcal D_{a-1}^{\mu+1}d^*_{\mathbb{R}^n}+\frac12(u+2i-n)
\mathcal D_{a}^{\mu}\iota_{\frac{\partial}{\partial x_n}}
\right) \label{eqn:Dii1}\\
&=&&
\restn\circ\left(-\mathcal D_{a-2}^{\mu+1}d^*_{\mathbb{R}^n}\iotan d_{\mathbb{R}^n}
+\frac12(u+2i-n+a)\mathcal D_a^\mu\iotan\right) \label{eqn:Di-B}\\
&&&
-\gamma(\mu-\frac12,a)d_{\R^{n-1}}^*\circ\restn\circ\mathcal D_{a-1}^\mu,\nonumber
\end{alignat}
where $\mu:= u+ i -\frac{1}{2}(n-1)$
and 
\index{A}{1iotan@$\iotan$, interior multiplication|textbf}
$\iotan$ stands for the 
\index{B}{interior multiplication}
interior multiplication by the vector field 
$\frac{\partial}{\partial x_n}$.
Then, $\mathcal D_{u,a}^{i\to i-1}$  is
a matrix-valued homogeneous differential operator of order $a$.
See Definition \ref{def:diff} for the precise meaning of
``differential operators between two manifolds". The proof of the second
equality \eqref{eqn:Di-B} will be given in Proposition \ref{prop:152751}. 

\begin{example}
Here are some few examples of the operators $\mathcal{D}^{i\to i-1}_{u,a}$
for $i=1,n$ or $a=0,1$, and $2$:
\begin{eqnarray*}
\mathcal D_{u,a}^{1\to 0}&=&\restn\circ\left(-\gamma(u-\frac{n-3}2,a)\mathcal D_{a-1}^{u-\frac{n-5}2}d^*_{\mathbb{R}^n}+\frac12(u+2-n)\mathcal D_a^{u-\frac{n-3}2}\iotan\right),\\
\mathcal D_{u,a}^{n\to n-1}&=&\frac12(u+n+a)\restn\circ\mathcal D_a^{u+\frac{n+1}2}\iotan,\\
\mathcal D_{u,0}^{i\to i-1}&=&\frac12 (u+2i-n)\mathrm{Rest}_{x_n=0}\circ\iota_{\frac{\partial}{\partial x_n}},\\
\mathcal D_{u,1}^{i\to i-1}&=&\mathrm{Rest}_{x_n=0}\circ
\left(-d^*_{\mathbb{R}^n}
+(u+2i-n)\frac{\partial}{\partial x_n}\iota_{\frac{\partial}{\partial x_n}}\right),\\
\mathcal D_{u,2}^{i\to i-1}&=&\mathrm{Rest}_{x_n=0}\circ D,
\end{eqnarray*}
where $
D=\left(
-d_{\R^n}d_{\R^n}^*+\frac12(u+2i-n)\left(\Delta_{\R^{n-1}}+(n+2i+5)\frac{\partial^2}{\partial x_n^2}\right)\right)\iota_{\frac{\partial}{\partial x_n}}
-2\gamma\frac{\partial}{\partial x_n} d^*_{\mathbb{R}^n}.
$
\end{example}

For $0 \leq i \leq n-1$, we introduce another family of linear maps
$\mathcal D_{u,a}^{i\to i}\colon \mathcal E^i(\R^n)\To\mathcal E^i(\R^{n-1})$ 
with parameters $u\in\C$ and $a\in\N$ by
\index{A}{Dii2@$\mathcal D_{u,a}^{i\to i}$|textbf}
\begin{alignat}{2}
\mathcal D_{u,a}^{i\to i}&:=&&\restn\circ\left(\mathcal D_{a-2}^{\mu+1} d_{\R^n}
d_{\R^n}^*-\gamma(\mu-\frac12,a)\mathcal D_{a-1}^\mu d_{\R^n}\iotan
+\frac12(u+a)\mathcal D_a^\mu\right)\label{eqn:Dii}\\
&=&&-d_{\R^{n-1}}^*d_{\R^{n-1}}\circ\restn\circ\mathcal D_{a-2}^{\mu+1}
+\restn\circ\left(\gamma(\mu,a)\mathcal D_{a-1}^{\mu+1}\iotan d_{\R^{n}}+\frac u2\mathcal D_a^\mu\right),\label{eqn:DiB}
\end{alignat}
where $\mu=u+i-\frac{n-1}2$ as before. Then $\mathcal D_{u,a}^{i\to i}$ is a matrix-valued homogeneous differential operator of order $a$.
The second equality \eqref{eqn:DiB} and
an alternative definition of $\mathcal{D}^{i\to i}_{u,a}$
by means of the Hodge star operators 
\index{A}{Dii2@$\mathcal D_{u,a}^{i\to i}$|textbf}
\begin{equation}\label{eqn:Diistar}
\mathcal D_{u,a}^{i\to i}:=(-1)^{n-1}*{_{\R^{n-1}}}\circ\mathcal D_{u-n+2i,a}^{n-i\to n- i-1}\circ
\left(*{_{\R^{n}}}\right)^{-1}
\end{equation}
will be proved in Proposition \ref{prop:dualCi}.

\begin{example}\label{ex:152555}
Here are some few examples of
the operators 
$\mathcal{D}^{i \to i}_{u,a}$
for $i=0, n-1$ or $a=0,1$, and $2$.
\begin{alignat*}{2}
&\mathcal D_{u,a}^{0\to 0}&&= \frac{u+a}{2} \restn \circ \mathcal{D}^{u - \frac{n-1}{2}}_a.\\
&\mathcal D_{u,a}^{n-1 \to n-1} &&=
-d^*_{\R^{n-1}}d_{\R^{n-1}} \restn \circ \mathcal{D}^{\mu+1}_{a-2}
+\frac{u}{2} \restn \circ \mathcal{D}^{\mu}_a.\\
&\mathcal D_{u,0}^{i\to i}&&=\frac u2 \mathrm{Rest}_{x_n=0}.\\
&\mathcal D_{u,1}^{i\to i}&&= \mathrm{Rest}_{x_n=0}\circ
\left(-d_{\mathbb{R}^n}
\iota_{\frac{\partial}{\partial x_n}}+(u+1){\frac{\partial}{\partial x_n}}\right).\\
&\mathcal D_{u,2}^{i\to i}&&= \mathrm{Rest}_{x_n=0}\circ\left(d_{\R^n}d_{\R^n}^*+\left(\frac n2+1\right)\Delta_{\R^{n}}+\left(u+i-\frac{n}2+1\right)
\left((n+2){\frac{\partial^2}{\partial x_n^2}}
+2\frac{\partial}{\partial x_n}d_{\mathbb{R}^n}\iota_{\frac{\partial}{\partial x_n}}\right)\right).
%%Kobayashi, 3/24/2016
%&=& -d^*_{\R^{n-1}}d_{\R^{n-1}}\circ\restn+\restn\circ\left((u+i-\frac{n-3}2)
%(2\frac{\partial }{\partial x_n}\iotan d_{\mathbb{R}^n}+u \frac{\partial^2}{\partial x_n^2})+\frac %u2\Delta_{\R^{n-1}}\right).
%%
\end{alignat*}

\end{example}

These differential operators are generically nonzero, however, 
they may vanish in specific cases. To be precise,
we prove in Section \ref{subsec:prop12}:

\begin{prop}\label{prop:Dnonzero}
Suppose $u \in \C$ and $a \in \N$.
\begin{enumerate}
\item Let $1\leq i\leq n$.
Then $\mathcal D_{u,a}^{i\to i-1}$ vanishes if and only if $(u,a)=(n-2i,0)$ or $(u,i)=(-n-a,n)$.

\item Let $0\leq i\leq n-1$.
Then $\mathcal D_{u,a}^{i\to i}$ vanishes if and only if $(u,a)=(0,0)$ or $(u,i)=(-a,0)$.
\end{enumerate}
\end{prop}
In order to obtain nonzero operators for all the parameters $(i,a,u)$,
we renormalize $\mathcal D_{u,a}^{i\to i-1}$ and $\mathcal D_{u,a}^{i\to i}$, 
respectively, by
\index{A}{Dii4@$\widetilde {\mathcal D}_{u,a}^{i\to i-1}$|textbf}
\begin{eqnarray}
\widetilde {\mathcal D}_{u,a}^{i\to i-1}&:=&\left\{
\begin{matrix*}[l]
\restn\circ\iotan &\mathrm{if}& a=0,\\
\restn\circ\mathcal D_a^{u+\frac{n+1}2}\circ\iotan&\mathrm{if}& i=n,\\
\mathcal D_{u,a}^{i\to i-1}&\mathrm{otherwise.}&
\end{matrix*}
\right. \label{eqn:reDii1}
\\
\index{A}{Dii5@$\widetilde {\mathcal D}_{u,a}^{i\to i }$|textbf}
\widetilde {\mathcal D}_{u,a}^{i\to i}&:=&\left\{
\begin{matrix*}[l]
\restn &\qquad&\mathrm{if}& a=0,\\
\restn\circ\mathcal D_a^{u-\frac{n-1}2}&\qquad&\mathrm{if}& i=0,\\
\mathcal D_{u,a}^{i\to i}&\qquad&\mathrm{otherwise.}&
\end{matrix*}
\right.\label{eqn:reDii}
\end{eqnarray}
Clearly, these operators are well-defined because the formul\ae{} on the right-hand sides coincide in the overlapping cases such as $a=0$ and $i=n$.

\medskip
We are now ready to give a solution to Problem \ref{prob:2} 
when $j=i-1$ and $i$.

\begin{thm}[$j=i-1$]\label{thm:2} 
Let $1\leq i\leq n$. Suppose $(u,v)\in\C^2$ and $(\delta,\eps)\in\left(\Z/2\Z\right)^2$
satisfy $v-u\in\N_+$ and $\delta\equiv \eps\equiv v-u\;\mathrm{mod}\;2$. We set
\begin{eqnarray*}
a&:=&v-u-1\in\N.
\end{eqnarray*}
\begin{enumerate}
\item The differential operator $\widetilde{\mathcal D}_{u,a}^{i\to i-1}$
extends to the conformal compactification $S^n$ of $\R^n$, and
induces a nontrivial $O(n,1)$-homomorphism
$\mathcal E^i(S^n)_{u,\delta}\To \mathcal E^{i-1}(S^{n-1})_{v,\eps}$, to be denoted
by the same letter $\widetilde{\mathcal D}_{u,a}^{i\to i-1}$.

\item Any $O(n,1)$-equivariant differential operator 
from $\mathcal E^i(S^n)_{u,\delta}$ to $\mathcal E^{i-1}(S^{n-1})_{v,\eps}$ is 
proportional to $\widetilde{\mathcal D}_{u,a}^{i\to i-1}$.
\end{enumerate}
\end{thm}

\begin{thm}[$j=i$]\label{thm:2ii}
Let $0\leq i\leq n-1$. Suppose $(u,v)\in\C^2$ and $(\delta,\eps)\in\left(\Z/2\Z\right)^2$
satisfy $v-u\in\N$ and $\delta\equiv\eps\equiv v-u\;\mathrm{mod}\;2$. 
We set
\begin{equation*}
a:=v-u\in\N.
\end{equation*}
\begin{enumerate}
\item The differential operator $\widetilde{\mathcal D}_{u,a}^{i\to i}$ 
extends to $S^n$, and 
induces a nontrivial 
$O(n,1)$-homo\-morphism
$\mathcal E^i(S^n)_{u,\delta}\To \mathcal E^{i}(S^{n-1})_{v,\eps}$, to be denoted
by the same letter $\widetilde{\mathcal D}_{u,a}^{i\to i}$.

\item Any $O(n,1)$-equivariant differential operator 
from $\mathcal E^i(S^n)_{u,\delta}$ to $\mathcal E^{i}(S^{n-1})_{v,\eps}$ is 
proportional to
$\widetilde{\mathcal D}_{u,a}^{i\to i}$.
\end{enumerate}
\end{thm}

In contrast to the above cases where $j = i-1$ or $i$,
we prove that differential symmetry breaking operators 
of higher order are rare for $j \notin \{i-1, i\}$.
Let us describe all of them.
For $j = i+1$, a family of differential operators
\index{A}{Dii6@$\widetilde{\mathcal{D}}^{i \to i+1}_{u,a}$|textbf}
$\widetilde{\mathcal{D}}^{i \to i+1}_{u,a}\colon  
\mathcal{E}^i(\R^n) \To \mathcal{E}^{i+1}(\R^{n-1})$
are defined by
\index{A}{Dii6@$\widetilde{\mathcal{D}}^{i \to i+1}_{u,a}$|textbf}
\begin{equation}\label{eqn:Ditua}
\widetilde{\mathcal{D}}^{i \to i+1}_{u,a}:=\mathrm{Rest}_{x_n=0}
\circ \mathcal{D}^{u-\frac{n-1}{2}}_{-u} \circ d_{\R^n}
\end{equation}
but only when $a=1-u$ with additional constraints
$u=0$ $(1\leq i \leq n-2)$ in Case (IV)
in Theorem \ref{thm:1}; $u\in -\N$ $(i=0)$ in Case (IV$'$).
We note $\widetilde{\mathcal{D}}^{i\to i+1}_{0,1} = 
\restn \circ d_{\R^n}$ and 
$\widetilde{\mathcal{D}}^{0 \to 1}_{1-a,a} = 
d \circ \widetilde{\mathcal{D}}^{0\to 0}_{1-a, a-1}$
(Theorem \ref{thm:160422} (6)).

For $j = i-2$, a family of differential operators
$\widetilde{\mathcal{D}}^{i \to i-2}_{u,a}\colon  
\mathcal{E}^i(\R^n)\To \mathcal{E}^{i-2}(\R^{n-1})$
are defined by
\index{A}{Dii3@$\widetilde{\mathcal{D}}^{i\to i-2}_{u,a}$|textbf}
\begin{equation}\label{eqn:160462}
\widetilde{\mathcal{D}}^{i \to i-2}_{u,a}:=
\mathrm{Rest}_{x_n=0} \circ \mathcal{D}^{u+\frac{n+1}{2}}_{-u+n-2i}
\circ \iota_{\frac{\partial}{\partial x_n}} \circ d^*_{\R^n},
\end{equation}
but only when $a = 1 + n - 2i-u$ with additional constraints 
$u=n-2i$ $(2 \leq i \leq n-1)$ in Case (I),
$u \in \{-n, -n-1, -n-2, \cdots\}$ $(i=n)$ in Case (I$'$).
We note 
$\widetilde{\mathcal{D}}^{i\to i-2}_{n-2i,1}=\restn \circ \iotan \circ d^*_{\R^n}$
and $\widetilde{\mathcal{D}}^{n\to n-2}_{1-n-a,a} = 
-d^*\circ \widetilde{\mathcal{D}}^{n\to n-1}_{1-n-a,a-1}$
(see Theorem \ref{thm:160422} (8)).

Then the solution to Problem \ref{prob:2} in the remaining cases,
\emph{i.e.}, $j \in \{i+1, i-2\}$ is stated as follows:

\begin{thm}[$j=i+1$]\label{thm:2ii+1}
Let $0\leq i \leq n-2$. Suppose $(i,i+1, u,v, \delta, \eps)$
belongs to Case \emph{(IV)} or \emph{(IV$'$)} in Theorem \ref{thm:1}. 
In particular, $\delta \equiv \eps\; \mathrm{mod} \; 2$, $u \in -\N$ and $v=0$. 
We set
\begin{equation*}
a := v-u+1 = 1-u \in \N_+.
\end{equation*}
\begin{enumerate}
\item The differential operator $\widetilde{\mathcal{D}}^{i\to i+1}_{u,a}$ 
extends to the conformal compactification $S^n$, and 
induces a nontrivial $O(n,1)$-homomorphism 
$\mathcal{E}^i(S^n)_{u,\delta} \To \mathcal{E}^{i+1}(S^{n-1})_{0,\delta}$,
to be denoted by the same letter $\widetilde{\mathcal{D}}^{i\to i+1}_{u,a}$.

\item\emph{Case (IV):}
Suppose $1 \leq i \leq n-2$. Then any $O(n,1)$-equivariant differential 
operator $\mathcal{E}^{i}(S^n)_{0,0} \To \mathcal{E}^{i+1}(S^{n-1})_{0,0}$
is proportional to $\widetilde{\mathcal{D}}^{i\to i+1}_{0,1}=
\mathrm{Rest}_{S^{n-1}}\circ d_{S^{n}}$.

\item\emph{Case (IV$'$):}
Suppose $i=0$. Then any $O(n,1)$-equivariant differential operator 
$\mathcal{E}^0(S^n)_{u,\delta} \To \mathcal{E}^{1}(S^{n-1})_{0,\delta}$
$(u\in-\N, \delta \equiv u\; \mathrm{mod}\; 2)$ is 
proportional to
$\widetilde{\mathcal{D}}^{0 \to 1}_{u,1-u}$.
\end{enumerate}
\end{thm}

\begin{thm}[$j=i-2$]\label{thm:2ii-2}
Let $2\leq i \leq n$. Suppose $(i,i-2, u,v, \delta, \eps)$
belongs to \emph{Case (I)} or \emph{(I$'$)} in Theorem \ref{thm:1}. 
In particular, $\delta \equiv \eps\; \mathrm{mod} \; 2$, $u \in -n-\N$ and $v=n-2i+3$. 
We set
\begin{equation*}
a := v-u-2 = n-2i+1-u \in \N_+.
\end{equation*}
\begin{enumerate}

\item The differential operator $\widetilde{\mathcal{D}}^{i\to i-2}_{u,a}$
extends to $S^n$, and induces
a nontrivial $O(n,1)$-homomorphism 
$\mathcal{E}^i(S^n)_{u,\delta} \To \mathcal{E}^{i-2}(S^{n-1})_{n-2i+3,\delta}$,  
to be denoted by the same letter $\widetilde{\mathcal{D}}^{i\to i-2}_{u,a}$.

\item \emph{Case (I):} Suppose $2 \leq i \leq n-1$. Then any $O(n,1)$-equivariant differential 
operator $\mathcal{E}^{i}(S^n)_{n-2i,1} \To 
\mathcal{E}^{i-2}(S^{n-1})_{n-2i+3,1}$ is proportional to
$\widetilde{\mathcal{D}}^{i\to i-2}_{n-2i,1}=
\mathrm{Rest}_{S^{n-1}} \circ \iota_{N_{S^{n-1}}(S^n)} \circ d^*_{S^n}$.

\item \emph{Case (I$'$):} Suppose $i=n$. Then any $O(n,1)$-equivariant differential operator 
$\mathcal{E}^n(S^n)_{u,\delta} \To \mathcal{E}^{n-2}(S^{n-1})_{3-n,\delta}$ 
$(u\in-n-\N, \delta \equiv u+n+1\; \mathrm{mod}\; 2)$ is proportional to 
$\widetilde{\mathcal{D}}^{n \to n-2}_{u,1-u-n}$.
\end{enumerate}
\end{thm}

Thus Problems \ref{prob:1} and \ref{prob:2} have been settled for the pair
$(X,Y) = (S^n, S^{n-1})$. 

Finally, we discuss 
\index{B}{matrix-valued functional identities}
matrix-valued functional identities (\emph{factorization theorems})
arising from compositions of conformally equivariant operators. 
They are formulated as follows.
Suppose that
$T_X\colon 
\mathcal E^{i'}(X)\to\mathcal E^{i}(X)$ or 
$T_Y\colon \mathcal E^j(Y)\to\mathcal E^{j'}(Y)$ is a conformally equivariant operator for forms. Then
the composition $T_Y\circ D$ or $D\circ T_X$ of a
symmetry breaking operator $D=D_{X\to Y}\colon \mathcal E^i(X)\to\mathcal E^j(Y)$ with $T_X$ or $T_Y$
is again a symmetry breaking operator:
\begin{equation*}
\xymatrix{
\mathcal E^i(X)_{u,\delta}\ar[r]^{D_{X\to Y}}
\ar@{-->}[dr]&\mathcal E^j(Y)_{v,\eps}\ar[d]_{T_Y}\\
\mathcal E^{i'}(X)_{u',\delta'}\ar[u]^{T_X}\ar@{-->}[ur]&\mathcal E^{j'}(Y)_{v',\eps'}
}
\end{equation*}
\vskip5pt
In the setting where $X=S^n$ (or $Y=S^{n-1}$, respectively), 
conformally covariant differential operators 
$T_X\colon \mathcal E^{i'}(X)\to\mathcal E^i(X)$ 
(or $T_Y\colon \mathcal E^j(Y)\to\mathcal E^{j'}(Y)$, respectively)
are classified in Theorem \ref{thm:GGC}.
This case (\emph{i.e.}\ $X=X$ or $Y=Y$)
is much easier than the case $Y \varsubsetneqq X$ treated in
Theorem \ref{thm:1} for symmetry breaking operators.
For the proof, we again use the F-method in a self-contained manner,
although classical results of algebraic representation theory 
(\emph{e.g.}\ \cite{BC86}) could be used to simplify the proof.
Thus we see that $T_X$ (or $T_Y$, respectively)
is proportional to $d, d^*$,  
\index{B}{Branson's operator}
Branson's operators 
\index{A}{T2ell@$\mathcal T_{2\ell}^{(i)}$, Branson's operator}
$\mathcal T_{2\ell}^{(i)}$
(or 
\index{A}{T2ell'@$\mathcal {T'}_{2\ell}^{(j)}$}
$\mathcal {T'}_{2\ell}^{(j)}$, respectively) of order $2\ell$ (see \eqref{eqn:T2li}),
or the composition of these operators with the Hodge star operator.
On the other hand, the general multiplicity-freeness theorem (see Theorem \ref{thm:1}) guarantees that such compositions must be proportional to
the operators that we classified in Theorems \ref{thm:2}, \ref{thm:2ii}, \ref{thm:2ii+1} and \ref{thm:2ii-2}.
 
In Chapter \ref{sec:fi}, 
we give a complete list of 
\index{B}{factorization identity}
factorization identities 
with explicit proportionality constants for all possible cases.
We illustrate the new factorization identities by taking
$T_X$ or $T_Y$ to be Branson's operators
$\mathcal{T}^{(i)}_{2\ell}$ or $\mathcal{T}'^{(j)}_{2\ell}$ as follows.
For $\ell \in \N_+$ and $a\in \N$, 
we define a positive number $K_{\ell, a}$ by
\index{A}{Kell@$K_{\ell, a}$|textbf}
\begin{equation}\label{eqn:Kla}
\Kla:=\prod_{k=1}^\ell
\left(\left[\frac a2\right]+k\right).
\end{equation}

\begin{thm}[See Theorem \ref{thm:factor1}]\label{thm:intro-factor1} 
Suppose $0\leq i\leq n, a\in\N$ and $\ell\in\N_+$.
Then
\begin{eqnarray*}
&(1)&\mathcal {D}_{\frac n2-i+\ell,a}^{i\to i-1}\circ 
\mathcal {T}_{2\ell}^{(i)}=-\left(\frac n2-i-\ell\right)\Kla
\mathcal {D}_{\frac n2-i-\ell,a+2\ell}^{i\to i-1}\quad\mathrm{if}\; i\neq 0.\\
&(2)&\mathcal {D}_{\frac n2-i+\ell,a}^{i\to i}
\circ 
\mathcal {T}_{2\ell}^{(i)}=-\left(\frac n2-i+\ell\right)
\Kla\mathcal {D}_{\frac n2-i-\ell,a+2\ell}^{i\to i}\quad\mathrm{if}\; i\neq n.
\end{eqnarray*}
\end{thm}

\begin{thm}[See Theorem \ref{thm:factor2}]\label{thm:intro-factor2}
Suppose $0\leq i\leq n$, $a\in\N$ and $\ell\in\N_+$. 
We set $u:=\frac{n-1}2-i-\ell-a$. Then
\begin{alignat*}{3}
&(1)\quad\mathcal {T'}_{2\ell}^{(i-1)}\circ 
\mathcal {D}_{u,a}^{i\to i-1}&&=
-\left(\frac{n+1}2-i+\ell\right)\Kla\mathcal {D}_{u,a+2\ell}^{i\to i-1}
&&\quad\mathrm{if}\,i\neq0.\\
&(2)\quad  \mathcal{T'}_{2\ell}^{(i)}\circ\mathcal {D}_{u,a}^{i\to i}&&=
-\left(\frac{n-1}2-i-\ell\right)\Kla
\mathcal {D}_{u,a+2\ell}^{i\to i}
&&\quad\mathrm{if}\, i\neq n.
\end{alignat*}
\end{thm}

The scalar case ($i=0$) in Theorem \ref{thm:intro-factor1} (2) and \ref{thm:intro-factor2} (2)
was studied in \cite{Juhl, KOSS15}, and was extended to all the symmetry 
breaking operators (including nonlocal ones) in \cite{KS13}.
The other matrix-valued factorization identities are given in 
Theorems \ref{thm:152673} and \ref{thm:152557}, see also 
Theorems \ref{thm:factor1'}, \ref{thm:factor2'}, and  \ref{thm:160422} 
for the factorization identities of renormalized symmetry breaking operators.
We also analyze when the proportionality constant vanishes.

\vskip 0.1in

Finally, let us mention analogous results for the connected groups,
other real forms in pseudo-Riemannian geometry, and branching 
problems for Verma modules.
Throughout the paper,
we study Problems \ref{prob:1} and \ref{prob:2} in full detail for the whole group
of conformal diffeomorphisms of $S^n$ that preserves $S^{n-1}$, which is 
a disconnected group. Then results for the connected group $SO_0(n,1)$,
or equivalently, for conformal vector fields on $S^n$ along the 
submanifold $S^{n-1}$, can be extracted from our main results for the 
disconnected group $O(n,1)$, see Theorem \ref{thm:conn}.

Branching problems for (generalized) Verma modules for
$(\mathfrak{g},\mathfrak{g}') = (\mathfrak{o}(n+2,\C), \mathfrak{o}(n+1,\C))$ are 
the algebraic counterpart of Problems \ref{prob:1} and \ref{prob:2} 
for $(X,Y) = (S^n, S^{n-1})$ by a general duality theorem \cite{KOSS15,KP1}
that gives a one-to-one correspondence between differential symmetry breaking
operators and $\mathfrak{g}'$-homomorphisms for the restriction
of Verma modules of $\mathfrak{g}$. Branching laws for Verma modules
are discussed in Section \ref{subsec:branchV}. 

Our results can be also extended to the non-Riemannian
setting $S^{p,q} \supset S^{p-1,q}$ for the pair 
$(G,G')=(O(p+1,q), O(p,q))$ of conformal groups,
for which the $(i,j) = (0,0)$ case was studied in \cite{KOSS15}.

The main results were announced in \cite{KKP16}.

%%%%%%%%%%%%%%%%%%%%%%%%%%%%%%%%%%%%%%%%%%%%%%%
\vskip 0.2in

\noindent
Notation: $\mathbb{N} := \{0, 1, 2, \ldots\}$,  $\mathbb{N}_+ := \{1, 2, \ldots\}$.

\vskip 0.2in

\emph{Acknowledgements}:
The first author was partially supported by Institut des Hautes \'Etudes Scientifiques, France and Grant-in-Aid for Scientific Research (A) (25247006), Japan Society for the Promotion of
Science. All three authors were partially supported by CNRS Grant PICS n$^\mathrm{o}$ 7270.

\newpage
%%%%%%%%%%%%%%%%%%%%%%%%%%%%%%%%%%%%%%%%%%%%%%%
The relation between chapters is illustrated by the following figures.
Here, ``$\Rightarrow$" means a strong relation (\emph{e.g.}\ logical dependency),
and ``$\rightarrow$" means a weak relation (\emph{e.g.}\ setup or definition).

\vskip 0.2in

$\bullet$ Classification of differential symmetry breaking operators

\medskip

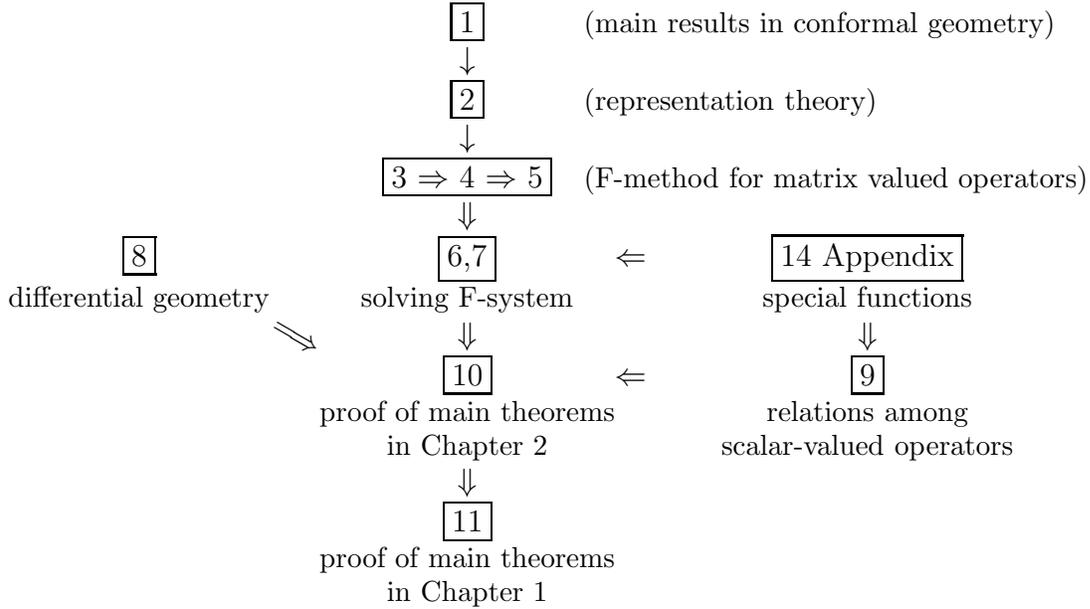
\begin{center}
\begin{tabular}{@{}c@{}c@{}c@{}c@{}c@{}}
&& \fbox{\ref{sec:intro}} & \multicolumn{2}{@{\kern-1em}l@{}}{\small (main results in conformal geometry)}\\
&& $\downarrow$ && \\
&& \fbox{\ref{sec:ps}} & \multicolumn{2}{@{\kern-1em}l@{}}{\small (representation theory)}\\
&& $\downarrow$ && \\
&& \fbox{\ref{sec:Method} $\Rightarrow$ \ref{sec:FON} $\Rightarrow$ \ref{sec:4}} & \multicolumn{2}{@{\kern-1em}l@{}}{\small (F-method for matrix valued operators)}\\
&& $\Downarrow$ && \\
\fbox{\ref{sec:6}} && \fbox{\ref{sec:7},\ref{sec:codiff}} 
& $\Leftarrow$ & \fbox{\ref{sec:appendix} Appendix}\\
{\small differential geometry} && {\small solving F-system} && {\small special functions}\\
& \rotatebox[origin=c]{330}{$\Longrightarrow$} & $\Downarrow$ && $\Downarrow$ \\
&& \fbox{\ref{sec:5}} & $\Leftarrow$ & \fbox{\ref{sec:formulaD}} \\
&& {\small \begin{tabular}{@{}c@{}}proof of main theorems \\in Chapter \ref{sec:ps}\end{tabular}} && {\small \begin{tabular}{@{}c@{}}relations among \\scalar-valued operators\end{tabular}} \\
&& $\Downarrow$ && \\
&& \fbox{\ref{sec:solAB}}&&\\
&& {\small \begin{tabular}{@{}c@{}}proof of main theorems \\in Chapter \ref{sec:intro}\end{tabular}}
\end{tabular}
\end{center}

\vskip 0.2in

$\bullet$ A baby case ($G=G'$) in Chapter \ref{sec:intertwiner} could be read independently:

\medskip

\hskip 1.85in
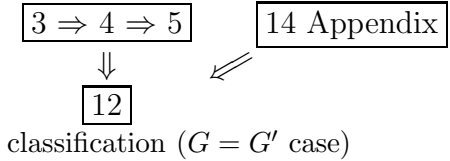
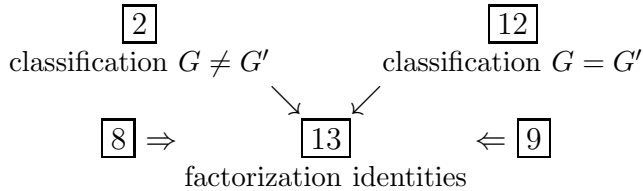
\begin{tabular}{@{}c@{}c@{}c@{}}
\fbox{\ref{sec:Method} $\Rightarrow$ \ref{sec:FON} $\Rightarrow$ \ref{sec:4}} && 
\fbox{\ref{sec:appendix} Appendix} \\
$\Downarrow$ & \rotatebox[origin=c]{30}{$\Longleftarrow$} &\\
\fbox{\ref{sec:intertwiner}} \\
{\small classification ($G=G'$ case)} \kern-5em
\end{tabular}

\vskip 0.2in

$\bullet$ Factorization identities

\medskip

\hskip 0.73in
\begin{tabular}{@{}c@{}c@{}c@{}c@{}c@{}c@{}c@{}}
\fbox{\ref{sec:ps}} &&&& \fbox{\ref{sec:intertwiner}}  \\
{\small classification $G\neq G'$} &&&& {\small classification $G=G'$}  \\
& $\searrow$ && $\swarrow$  \\
\fbox{\ref{sec:6}} ~ $\Rightarrow$ && \fbox{\ref{sec:fi}} && $\Leftarrow$ ~ \fbox{\ref{sec:formulaD}} \\
&& \kern-4em{\small factorization identities}\kern-4em &&
\end{tabular}

\newpage
%%%%%%%%%%%%%%%%%%%%%%%%%%%%%%%%%%%%%%%%%%%%%%%
\section{Symmetry breaking operators and principal series representations of $G=O(n+1,1)$}\label{sec:ps}

The conformal compactification $S^n$ of $\R^n$ may be thought of as the 
real flag variety of the indefinite orthogonal group $G = O(n+1,1)$,
and the twisted action $\varpi^{(i)}_{u,\delta}$ of $G$ on $\mathcal{E}^i(S^n)$
is a special case of the principal series representations of $G$. In this chapter,
we reformulate the solutions to Problems \ref{prob:1} and \ref{prob:2}
for $(X,Y) = (S^n, S^{n-1})$ given in Theorem \ref{thm:1} and
Theorems \ref{thm:2}-\ref{thm:2ii-2}, respectively, in 
Introduction
in terms of symmetry breaking operators for principal series representations
when $(G, G') = (O(n+1,1), O(n,1))$ in Theorems \ref{thm:psdual} and \ref{thm:1A},
respectively.

Some important properties (duality theorem of symmetry breaking operators,
reducible places) of the principal series representations of $G=O(n+1,1)$ are also
discussed in this chapter.

%%%%%%%%%%%%%%%%%%%%%%%%%%%%%%%%%%%%%%%%%%%%%%%
\subsection{Principal series representations of $G=O(n+1,1)$}\label{subsec:ps}

We set up notations for the group $O(n+1,1)$ and its parabolically induced representations.
Let 
\index{A}{Qn1@$Q_{n+1,1}(x)$, quadratic form of signature $(n+1,1)$|textbf}
$Q_{n+1,1}$ be the standard quadratic form of signature $(n+1,1)$ on $\R^{n+2}$ defined by
\begin{equation*}
Q_{n+1,1}(x):=x_0^2+x_1^2+\cdots+x_n^2-x_{n+1}^2\quad
\mathrm{for}\, x=(x_0,x_1,\cdots,x_{n+1})\in\R^{n+2},
\end{equation*}
and we realize the Lorentz  group 
\index{A}{On11@$O(n+1,1)$|textbf}
$O(n+1,1)$ as
\begin{equation*}
G:=O(n+1,1)=\{g\in GL(n+2,\R): Q_{n+1,1}(gx)=Q_{n+1,1}(x)\,\mathrm{for\,all}\, x\in\R^{n+2}\}.
\end{equation*}
Let $E_{pq}$ ($0\leq p,q\leq n+1$) be the matrix unit in $M(n+2, \R)$. 
We define the following elements of the Lie algebra $\mathfrak g=\mathfrak o(n+1,1)$:
\index{A}{Xpq@$X_{pq}$, basis of $\mathfrak{o}(n)$|textbf}
\index{A}{H1zero@$H_0$, generator of $\mathfrak{a}$|textbf}
\index{A}{C+@$C_\ell^+$ ($= 2N_\ell^+$), basis of $\mathfrak{n}_+(\R)$|textbf}
\index{A}{C-@$C_\ell^-$ ($= N_\ell^-$), basis of $\mathfrak{n}_-(\R)$|textbf}
\index{A}{Nell+@$N_\ell^+$ ($=\frac{1}{2}C^+_\ell$), basis of $\mathfrak{n}_+(\R)$|textbf}
\index{A}{Nell-@$N_\ell^-$ ($=C^-_\ell$), basis of $\mathfrak{n}_-(\R)$|textbf}
\begin{alignat}{2}
X_{pq}\quad&:=\quad-E_{pq}+E_{qp}&&\quad (1\leq p<q\leq n), \label{eqn:xpq}\\
H_0\quad&:=\quad E_{0,n+1}+E_{n+1,0},&&\nonumber\\
C_\ell^+\quad&:=\quad E_{\ell,0}-E_{\ell,n+1}-E_{0,\ell}-E_{n+1,\ell}
&&\quad (1\leq \ell\leq n),\nonumber\\
C_\ell^-\quad&:=\quad E_{\ell,0}+E_{\ell,n+1}-E_{0,\ell}+E_{n+1,\ell}
&&\quad (1\leq \ell\leq n),\nonumber\\
N_\ell^+ \quad&:=\quad \frac12 C_\ell^+ \;\; \text{and} \; \;
N_\ell^-:= C_\ell^- &&\quad (1\leq \ell\leq n). \label{eqn:Npm1}
\end{alignat}
Then $\{N_{\ell}^+\}_{\ell=1}^n$, $\{N_{\ell}^-\}_{\ell=1}^n$, and 
$\{X_{pq}\}_{1\leq p<q\leq n} \cup \{H_0\}$ form bases of
the Lie algebras
\index{A}{N1+1(R)@$\mathfrak{n}_+(\R)$, Lie algebra of $N_+$}
$\mathfrak n_+(\R):=\mathrm{Ker}(\mathrm{ad}(H_0)-\mathrm{id})$, 
$\mathfrak n_-(\R):=\mathrm{Ker}(\mathrm{ad}(H_0)+\mathrm{id})$, 
and 
$\mathfrak{m}(\R) + \mathfrak{a}(\R) 
=\mathfrak o(n)+\mathfrak{o}(1,1)
=\mathrm{Ker}(\mathrm{ad}(H_0))$,
 respectively. We note that the normalization of 
$N_\ell^+$ and 
$N_\ell^-$ in \eqref{eqn:Npm1} is not symmetric. A simple computation shows
\begin{equation}\label{eqn:Npm}
[N_k^+,N_\ell^-]=X_{k\ell}-\delta_{k\ell}H_0.
\end{equation}
We define the isotropic cone (\index{B}{light cone}\emph{light cone}) by
\index{A}{11OXi@$\Xi$, light cone}
\begin{equation*}
\Xi:=\{x\in\R^{n+2}\setminus\{0\}:Q_{n+1,1}(x)=0\},
\end{equation*}
which is clearly invariant under the dilation of the multiplicative group 
$\R^\times = \R\setminus \{0\}$.
Then the projection
\begin{equation*}
\Xi \To S^n, \qquad x \mapsto \frac{1}{x_{n+1}} \trans(x_0,\ldots, x_n)
\end{equation*}
induces a bijection $\Xi/\mathbb{R}^\times \stackrel{\sim}{\to} S^n$.
The group $G$ acts linearly on the isotropic cone $\Xi$, 
and conformally on $\Xi/\R^\times\simeq S^n$, endowed 
with the standard Riemannian metric. We set
\index{A}{1ks0@$\xi^{\pm}$ $(\in \Xi)$|textbf}
\begin{eqnarray*}
\xi^{\pm}&:=&\trans(\pm1,0,\cdots,0,1)\in\Xi.
\end{eqnarray*}
\index{A}{P1P@$P$, parabolic subgroup of $O(n+1, 1)$|textbf}
Let $P$ be the isotropy subgroup of 
\index{A}{1ks0b@$[\xi^\pm]$ $(\in \Xi/\R^\times = S^n)$}
$[\xi^{+}]\in \Xi/\mathbb{R}^\times$. 
Then $P$ is a parabolic subgroup with Levi decomposition $P=MAN_+$
of the disconnected group $G=O(n+1,1)$, 
where
\index{A}{A@$A$, split torus $(\simeq \R)$|textbf}
$A:= \exp(\mathbb{R}H_0)$,
\index{A}{N2+@$N_+$, unipotent subgroup of $O(n+1,1)$|textbf}
$N_+:=\exp (\mathfrak{n}_+(\mathbb{R}))$ and
\begin{eqnarray*}
\index{A}{M@$M$ ($=O(n)\times O(1)$)|textbf}M&:=&\left\{
\begin{pmatrix}
b& &\\ & B&\\ &&b
\end{pmatrix}
: B\in O(n), b\in O(1)\right\}\simeq O(n)\times O(1).
\end{eqnarray*}

For $x=(x_1, \ldots, x_n) \in \R^n$, we set
\index{A}{Qn@$Q_n(x)$|textbf}
\begin{equation*}
Q_n(x)\equiv Q_{n,0}(x):=\sum^n_{\ell=1}x_\ell^2.
\end{equation*}
Let 
\index{A}{N2-@$N_-$|textbf}
$N_-:=\exp(\mathfrak{n}_-(\R))$.
We define a diffeomorphism $n_-\colon \mathbb{R}^n \stackrel{\sim}{\to} N_-$
by
\index{A}{Nell-@$N_\ell^-$ ($=C^-_\ell$), basis of $\mathfrak{n}_-(\R)$}
\begin{equation*}
n_-(x):=\exp\left(\sum_{\ell = 1}^n x_\ell N_\ell^- \right)=
I_{n+2}+
\begin{pmatrix}
-\frac{1}{2}Q_n(x) & -\trans x & -\frac{1}{2}Q_n(x)\\
x & 0 & x\\
\frac{1}{2}Q_n(x) & \trans x & \frac{1}{2}Q_n(x)
\end{pmatrix},
\end{equation*}
which gives the coordinates on the \index{B}{open Bruhat cell}open Bruhat cell 
$N_-\cdot o \subset G/P \simeq \Xi/\mathbb{R}^\times \simeq S^n$:

\index{A}{1iota@$\iota$, conformal compactification}
\begin{equation}\label{eqn:RnSn}
\iota\colon 
\R^n \To S^n, \quad 
x = \trans (x_1, \ldots, x_n) \mapsto 
\frac{1}{1+Q_n(x)} \trans (1-Q_n(x), 2x_1, \ldots, 2x_n),
\end{equation}
because $n_-(x)\xi^+ =
\begin{pmatrix}
1\\
0\\
1
\end{pmatrix}
+
\begin{pmatrix}
-Q_n(x)\\
2x\\
Q_n(x)
\end{pmatrix}$.
We note that  the immersion $\iota$ is nothing 
but the inverse of the 
\index{B}{stereographic projection|textbf}
stereographic projection:
\index{A}{p@$p$, stereographic projection|textbf}
\begin{equation}\label{eqn:stereo}
p\colon S^n \setminus\{[\xi^-]\} \To \R^n, \quad
\omega=\trans (\omega_0, \ldots, \omega_n) \mapsto
\frac{1}{1+\omega_0}\trans (\omega_1,\ldots, \omega_n),
\end{equation}
where 
\index{A}{1ks1@$[\xi^-]$, north pole in $S^n$|textbf}
we recall 
$[\xi^-]=\trans (-1, 0,\ldots, 0)
\in \Xi / \R^\times \simeq S^n$.
For $\lambda \in \mathbb{C}$, we define a one-dimensional representation
\index{A}{C2lambda@$\mathbb{C}_\lambda$, one-dimensional representation of $A$|textbf}
$\mathbb{C}_\lambda$ of $A$ normalized by
\begin{equation}\label{eqn:Clmde}
A \To \mathbb{C}^\times, \qquad a=e^{tH_0} \mapsto a^\lambda:=e^{\lambda t}.
\end{equation}
Given an irreducible finite-dimensional representation $(\sigma, V)$ 
of $M \simeq O(n) \times O(1)$ and $\lambda \in \mathbb{C}$, 
we extend the outer tensor product representation 
\index{A}{1sigma-lambda@$\sigma_\lambda:=
\sigma \boxtimes \C_\lambda$|textbf}
$\sigma_\lambda:=\sigma \boxtimes \C_\lambda$ of the direct product 
group $MA$ to the parabolic subgroup $P=MAN_+$ by letting $N_+$
act trivially. Then we form
an (unnormalized)
principal series representation 
\index{A}{IndPG@$\mathrm{Ind}_P^G(\sigma_\lambda)$|textbf}
$\mathrm{Ind}^G_P(\sigma_\lambda)\equiv
\mathrm{Ind}^G_P(\sigma \boxtimes \mathbb{C}_\lambda)$ of $G$
on the space $(C^\infty(G)\otimes V)^P \simeq C^\infty(G,V)^P$
given by
\begin{equation*}
\{f \in C^\infty(G,V) : f(gman) = \sigma(m)^{-1} a^{-\lambda} f(g)
\quad \text{for all $m\in M$, $a\in A$, $g\in G$}\}.
\end{equation*}
Its flat picture
\index{B}{flat picture|textbf}
($N$-picture)\index{B}{Npicture@$N$-picture|textbf}
is defined on $C^\infty(\mathbb{R}^n) \otimes V$ via
the restriction to the open Bruhat cell:
\begin{equation}\label{eqn:Npic}
C^\infty(G/P, \mathcal{V}) \simeq (C^\infty(G) \otimes V)^P 
\to C^\infty(\mathbb{R}^n) \otimes V,
\quad
f \mapsto (x\mapsto F(x):=f(n_-(x))).
\end{equation}

We denote by $\Exterior^i(\C^n)$ the representation of $O(n)$ on the $i$-th exterior power of the standard representation. Then, $\Exterior^i(\C^n)$ ($0\leq i\leq n)$ are pairwise inequivalent, irreducible representations of $O(n)$, and $\Exterior^n(\C^n)$
is isomorphic to the one-dimensional representation $\det\colon  O(n)\to\C^\times$, 
$B\mapsto \det B$. 

For $\alpha \in \Z/2\Z$ and $\lambda \in \C$, we denote by
\index{A}{1sigma-lambda-alpha@$\sigma^{(i)}_{\lambda, \alpha}$, 
representation of $P$ on $\Exterior^i(\C^n)$|textbf}
$\sigma^{(i)}_{\lambda,\alpha}$ the outer tensor product representation
$\Exterior^i(\C^n) \boxtimes (-1)^\alpha \boxtimes \C_\lambda$ of the Levi subgroup
$L=MA\simeq O(n)\times O(1)\times \R$
given by
\begin{equation*}
\text{$(B, b, a) \mapsto b^\alpha a^\lambda \Exterior^iB \in 
\mathrm{GL}_\mathbb{C} \left(\Exterior^i(\mathbb{C}^n) \right)$
for $B \in O(n)$, \; $b\in \{\pm 1 \} \simeq O(1)$,\;  $a \in A$.}
\end{equation*}
We extend $\sigma^{(i)}_{\lambda,\alpha}$ to $P$ by letting 
$N_+$ act trivially. We denote by 
\index{A}{Iilambda@$I(i,\lambda)_\alpha$, principal series of $O(n+1,1)$|textbf}
$I(i,\lambda)_\alpha$ the 
\index{B}{principal series representation|textbf}
principal series representation
$\mathrm{Ind}^G_P\left(\sigma^{(i)}_{\lambda,\alpha}\right)$ of $G$.
By a little abuse of notation, we shall also write $I(i,\lambda)_k$ for $k \in \Z$
instead of $I(i,\lambda)_{k\; \mathrm{mod}\;2}$.

As the composition of \eqref{eqn:Npic} with the natural identification
\begin{equation*}
\eta\colon C^\infty(\R^n) \otimes \Exterior^i(\C^n) \stackrel{\sim}{\to} \mathcal{E}^i(\R^n) 
\quad \text{for $0\leq i\leq n$,}
\end{equation*}
the flat picture of the principal series representation $I(i,\lambda)_\alpha$ is 
realized in $\mathcal{E}^i(\R^n)$:
\index{A}{1iotaI@$\iota^{(i)}_\lambda$, map to flat picture|textbf}
\begin{equation}\label{eqn:Iiflat}
\iota^{(i)}_\lambda: I(i,\lambda)_\alpha  \hooklongrightarrow
\mathcal{E}^i(\R^n), \quad f \mapsto F,
\end{equation}
where $F(x) = \eta\left(f(n_-(x))\right)$.

\begin{rem}\label{rem:Icenter}
The central element $-I_{n+2}$ of $G$ acts on $I(i,\lambda)_\alpha$ 
as scalar multiplication by $(-1)^{i+\alpha}$.
We shall see in Remark \ref{rem:identific} that $I(i,\lambda)_\alpha$
appears as a representation of the conformal group $\mathrm{Conf}(S^n)$
only when $i+\alpha \equiv 0 \; \mathrm{mod} \;2$.
\end{rem}

We note that
\index{A}{GO@$G=O(n+1,1)$}
$G=O(n+1,1)$ has four connected components. 
Let $G_0$ denote the identity component of $G$.
Then we have $G/G_0\simeq \Z/2\Z\times\Z/2\Z$. 
Accordingly, there are four one-dimensional representations of $G$, 
\index{A}{1x@$\chi_{\pm\pm}$, one-dimensional representation of $O(n+1,1)$|textbf}
\begin{equation}\label{eqn:chiab}
\chi_{ab}\colon  G\To\{\pm1\}
\end{equation}
for $a,b\in\{\pm\}\equiv\{\pm1\}$ such that
\begin{equation*}
\chi_{ab}\left(\mathrm{diag}(-1,1,\cdots,1)\right)=a,\quad 
\chi_{ab}\left(\mathrm{diag}(1,\cdots,1,-1)\right)=b.
\end{equation*}
We note that $\chi_{--} = \mathrm{det}$. Then the restriction of 
\index{A}{1x--@$\chi_{--}$|textbf}
$\chi_{--}$ to $M\simeq O(n) \times O(1)$ is given by the outer tensor product:
\begin{equation}\label{eqn:chiM}
\chi_{--}\vert_{M}\simeq \det\boxtimes \one.
\end{equation}

In view of the isomorphism of 
$O(n)$-modules:
\begin{equation}\label{eqn:idet}
\Exterior^i(\C^n)\otimes\det\simeq\Exterior^{n-i}(\C^n),
\end{equation}
we get a $P$-isomorphism (with trivial $N_+$-action):
\begin{equation*}
\sigma^{(i)}_{\lambda,\alpha} \otimes \chi_{--}\vert_{P}\simeq
\sigma^{(n-i)}_{\lambda,\alpha}.
\end{equation*}
\noindent
Therefore, we have a natural isomorphism as $G$-modules:
\begin{align*}
I(i,\lambda)_\alpha\otimes \chi_{--}
&\simeq 
\mathrm{Ind}^G_P\left(
\sigma^{(i)}_{\lambda,\alpha}
\otimes \chi_{--}\vert_{P}
\right)\\
& \simeq 
\mathrm{Ind}^G_P
\left(
\sigma^{(n-i)}_{\lambda,\alpha}
\right)\\
&\simeq
I(n-i,\lambda)_\alpha.
\end{align*}

\noindent
Thus we have proved:

\begin{lem}\label{lem:psdual}
Let $0\leq i \leq n$, $\lambda \in \C$ and $\alpha \in \Z/2\Z$.
Then there is a natural $G$-isomorphism:
\begin{equation*}
I(i,\lambda)_\alpha \otimes \chi_{--}\simeq
I(n-i,\lambda)_\alpha.
\end{equation*}
\end{lem}

%%%%%%%%%%%%%%%%%%%%%%%%%%%%%%%%%%%%%%%%%%%%%%
\subsection{Conformal view on principal series representations of $O(n+1,1)$}

Since the group $G=O(n+1,1)$ is a double cover of the 
\index{B}{conformal group}
conformal group
of $S^n$ $(n\geq 2)$, and since 
\index{A}{P1P@$P$, parabolic subgroup of $O(n+1, 1)$}
$S^n \simeq G/P$, 
we may compare the two families of 
representations of $G=O(n+1,1)$: the family of conformal representations 
$\varpi^{(i)}_{u,\delta}$ and the principal series representations 
\index{A}{Iilambda@$I(i,\lambda)_\alpha$, principal series of $O(n+1,1)$}
$I(i,\lambda)_\alpha$. 
The correspondence is classically known for the connected component $G_0$ 
of $G$ (see \cite{KO1}) for instance). For disconnected groups $G$, we have the following:

\begin{prop}\label{prop:identific}
Let $G = O(n+1,1)$ with $n\geq 2$ and 
$0\leq i\leq n$, $u\in\C$. For  $\delta\in\Z/2\Z$, 
we have the following isomorphism of $G$-modules:
\index{A}{1pi@$\varpi^{(i)}_{u,\delta}$, conformal representation on $i$-forms}
\begin{equation*}
\varpi_{u,\delta}^{(i)}\simeq\left\{
\begin{matrix*}[l]
I(i,u+i)_i &\mathrm{if}&\delta=0;\\
I(n-i,u+i)_{n-i}&\mathrm{if}&\delta=1.
\end{matrix*}
\right.
\end{equation*}
Equivalently, for $\lambda \in \C$,
we have the following $G$-isomorphisms:
\begin{equation}\label{eqn:Iww}
I(i,\lambda)_i \simeq \varpi^{(i)}_{\lambda-i,0} \simeq \varpi^{(n-i)}_{\lambda-n+i,1}.
\end{equation}
\end{prop}

\begin{rem}\label{rem:identific}
Proposition \ref{prop:identific} implies that 
principal series representations $I(\ell, \lambda)_\alpha$
with $\alpha \equiv \ell \; \mathrm{mod}\;2$ 
are sufficient for the description of 
\index{B}{conformal representation}
conformal representations $\varpi^{(i)}_{u,\delta}$
on differential forms on $S^n$.
\end{rem}

\begin{proof}[Proof of Proposition \ref{prop:identific}]
We shall show a $G$-isomorphism:
\begin{equation}\label{eqn:confInd}
(\varpi_{u,\delta}^{(i)},\mathcal E^i(S^n))\simeq
\mathrm{Ind}_P^G\left(\left(\Exterior^i(\C^n)\otimes\det{}^\delta\right)
\boxtimes (-1)^{i+n\delta}
\boxtimes \C_{u+i}\right).
\end{equation}

Since the cotangent bundle of $X=G/P$
can be seen
as a $G$-homogeneous bundle 
$G\times_P\mathfrak n_+(\R)$, we have an isomorphism of $G$-modules
$$
\mathcal E^i(S^n)\simeq C^\infty(X,G\times_P\Exterior^i\mathfrak n_+).
$$
In our setting, 
\index{A}{H1zero@$H_0$, generator of $\mathfrak{a}$}
$\mathrm{ad}(H_0)$ acts on $\mathfrak{n}_{+}$ as the scalar multiplication
by one, and therefore,
the $P$-action on the exterior power $\Exterior^i\mathfrak n_+$
is given by the outer tensor product
$\Exterior^i(\C^n)\boxtimes(-1)^i\boxtimes\C_i$ of 
\index{A}{M@$M$ ($=O(n)\times O(1)$)}
\index{A}{A@$A$, split torus $(\simeq \R)$}
$MA\simeq O(n)\times O(1)\times\R$
with trivial $N_+$-action.
Thus we get the isomorphism \eqref{eqn:confInd} in the case 
where $u=0$ and $\delta=0$.
On the other hand, the 
\index{B}{orientation bundle}
orientation bundle $\mathpzc{or}(X)$ is associated to the 
one-dimensional representation of 
\index{A}{L1L@$L=MA$, Levi part of $P$}
\index{A}{N2+@$N_+$, unipotent subgroup of $O(n+1,1)$}
$P=LN_+\equiv MAN_+$ 
given by
$$
P\to P/N_+\simeq MA\To\{\pm1\},\quad (B,b,e^{tH_0})\mapsto b^n\det B,
$$
we also get \eqref{eqn:confInd} in the $u=0$ and $\delta=1$ case. Finally, observe that
the parameter $u$ in the definition of
the conformal representation $\varpi_{u,\delta}^{(i)}$ in \eqref{eqn:varpi}
is normalized in a way that 
the action on volume densities corresponds to the case $u = \dim X$
(with $i=0$ and $\delta=0$).
In our setting where $X=G/P \simeq S^n$, 
this coincides with
\index{A}{1rho@$\rho$}
\index{A}{C2rho@$\C_{2\rho}$}
$n = \mathrm{Trace}(
\mathrm{ad}(H_0)\colon  \mathfrak{n}_+(\R) \To \mathfrak{n}_+(\R))=2\rho$
via the normalization \eqref{eqn:Clmde} that we have adopted for
the principal series representations.
Hence, \eqref{eqn:confInd} is verified for all $u \in \C$ by interpolation. 
By \eqref{eqn:idet}, Proposition \ref{prop:identific} follows.
\end{proof}

%%%%%%%%%%%%%%%%%%%%%%%%%%%%%%%%%%%%%%%%%%%%%%%
\subsection{Representation theoretic properties of 
$(\varpi_{u,\delta}^{(i)},\mathcal E^i(S^n))$}\label{subsec:Aq}

Via the isomorphism in Proposition \ref{prop:identific}, 
we can apply the general theory of representations of real reductive groups 
to our representations
$(\varpi_{u,\delta}^{(i)},\mathcal E^i(S^n))$
of the conformal group.
Although the large majority of the literature in the representation theory 
of real reductive groups $G$
is limited to reductive groups of the Harish-Chandra class, our group $G=O(n+1,1)$ is 
disconnected and the adjoint group $\mathrm{Ad}(G)$ is not contained 
in the group $\mathrm{Int}(\mathfrak{g})$ of inner automorphisms of 
the complexified Lie algebra $\mathfrak{g} = \mathfrak{o}(n+2,\C)$ if 
$n$ is even. This does not cause any serious difficulties in the argument below,
but we shall be careful in preparing notation for the disconnected group $G$.

Let $Z_G(\mathfrak g)$ be the ring of $\mathrm{Ad}(G)$-invariant elements in the enveloping algebra $U(\mathfrak g)$ of the complexified Lie algebra $\mathfrak g\simeq\mathfrak o(n+2,\C)$. We note that 
\index{A}{ZG2@$Z_G(\mathfrak g)$|textbf}
$Z_G(\mathfrak g)$ is a subalgebra of the center 
\index{A}{ZG1@$Z(\mathfrak g)$, center of $U(\mathfrak{g})$|textbf}
$Z(\mathfrak g)$ of $U(\mathfrak g)$;
it coincides with $Z(\mathfrak{g})$ if $n$ is odd, 
and is of index two in $Z(\mathfrak g)$ if $n$ is even.

By taking the standard basis of a Cartan subalgebra $\mathfrak j$ of 
$\mathfrak g = \mathfrak{o}(n+2, \C)$,
we identify $\mathfrak j$ with $\C^{\left[\frac n2\right]+1}$.
The finite reflection group 
$W=\mathfrak S_{\left[\frac n2\right]+1}\ltimes\left(\Z/2\Z\right)^{\left[\frac n2\right]+1}$ 
acts naturally on $\mathfrak{j}$ and $\mathfrak{j}^\vee \simeq 
\C^{\left[\frac{n}{2}\right]+1}$.
We note that $W$ coincides with the Weyl group of the root system of type 
$B_{\left[\frac{n}{2}\right]+1}$ if $n$ is odd, and contains that 
of type $D_{\left[\frac{n}{2}\right]+1}$ as a subgroup of index two if $n$ is even.
Then the Harish-Chandra isomorphism for the disconnected group $G=O(n+1,1)$
asserts a $\C$-algebra isomorphism
between $Z_G(\mathfrak g)$ and the ring 
$S\left(\C^{\left[\frac n2\right]+1}\right)^W$ of $W$-invariants of the symmetric algebra
$S(\mathfrak{j})$.
In turn, we have a bijection
(\emph{Harish-Chandra's parametrization of infinitesimal characters})
\begin{equation}\label{eqn:ZGdual}
\mathrm{Hom}_{\C\text{-}\mathrm{algebra}}\left(Z_G(\mathfrak g),\C\right)
\simeq
\C^{\left[\frac n2\right]+1}/W.
\end{equation}
We normalize the Harish-Chandra isomorphism in a way that the $Z_G(\mathfrak g)$-infinitesimal character of the trivial one-dimensional representation $\one$ of $G$ is given by
\index{A}{1rhoG@$\rho_G$}
\begin{equation}\label{eqn:rhoG}
\rho_G:=\left(\frac n2,\frac n2-1,\cdots,\frac n2-\left[\frac n2\right]\right)\in
\C^{\left[\frac n2\right]+1}/W.
\end{equation}

\begin{prop}\label{prop:1521104}
The $Z_G(\mathfrak g)$-\index{B}{infinitesimal character}infinitesimal character 
of the representation 
$\varpi_{u,\delta}^{(i)}$ of $G$ on the space $\mathcal E^i(S^n)$ of $i$-forms 
is given by
\begin{alignat*}{2}
&\bigg( u+i-\frac n2,\underbrace{\frac n2,\frac n2-1,\cdots,\frac n2-i+1}_i,
\widehat{\frac n2-i},\underbrace{\frac n2-i-1,\cdots,\frac n2-\left[\frac n2\right]}_{\left[\frac n2\right]-i}\bigg)\;&&\mathrm{if}\; 0\leq i\leq\left[\frac n2\right],\\
&\bigg( u+i-\frac n2,\underbrace{\frac n2,\frac n2-1,\cdots,-\frac n2+i+1}_{n-i},
\widehat{-\frac n2+i},\underbrace{-\frac n2+i-1,\cdots,\frac n2-\left[\frac n2\right]}_{
i-\left[\frac{n+1}{2}\right]}\bigg)\;&&\mathrm{if}\;\left[\frac{n+1}{2}\right]\leq i\leq n,
\end{alignat*}
in the Harish-Chandra parametrization.
\end{prop}
In particular, $\varpi_{u,\delta}^{(i)}$ has the same infinitesimal character $\rho_G$
with the trivial representation if $u=0$ for all $0\leq i\leq n$ and $\delta\in\Z/2\Z$.

By the Frobenius reciprocity, every principal series representation 
\index{A}{1kmu-flat@$\mu^\flat\equiv \mu^\flat(i)$, small $K$-type}
$I(i,\lambda)_\alpha$ contains
\begin{equation}\label{eqn:smallK}
\mu^\flat\equiv \mu^\flat(i):=\Exterior^i(\C^{n+1})\boxtimes (-1)^\alpha\quad\mathrm{and}
\quad \mu^\#\equiv \mu^\#(i):=\Exterior^{i+1}(\C^{n+1})\boxtimes (-1)^{\alpha+1}
\end{equation}
as $K$-types. We are particularly interested in the $\lambda=i$ case, for which
\index{A}{1kmu-sharp@$\mu^\#\equiv \mu^\#(i)$, small $K$-type}
$I(i,\lambda)_\alpha$ is reducible (except for $n=2i$) and has 
$Z_G(\mathfrak g)$-infinitesimal character $\rho_G$.

We denote by 
\index{A}{Ibarflat@$\overline I(i)_\alpha^\flat$, irreducible subquotient|textbf}
$\overline I(i)_\alpha^\flat$ and 
\index{A}{Ibarsharp@$\overline I(i)_\alpha^\#$, irreducible subquotient|textbf}
$\overline I(i)_\alpha^\#$ the (unique) irreducible
subquotients of $I(i,i)_\alpha$ containing the $K$-types $\mu^\flat$ and $\mu^\#$,
respectively. Then we have $G$-isomorphisms
\begin{equation}\label{eqn:flatsharp}
\overline I(i)_\alpha^\#\simeq \overline I(i+1)_{\alpha+1}^\flat
\quad\mathrm{for}\; 0\leq i\leq n\;\mathrm{and}\;\alpha\in\Z/2\Z.
\end{equation}
For $n$ even, the unitary axis of $I\left(\frac n2,\lambda\right)_\alpha$ is given by
$\lambda=\frac n2+\sqrt{-1}\R$, and $I\left(\frac n2,\frac n2\right)_\alpha$ is irreducible for both $\alpha\equiv 0$ and 1 in $\Z/2\Z$. In particular, we have
\begin{equation}\label{eqn:nhalf}
\overline I\left(
\scalebox{1.3}{$\frac{n}{2}$}
\right)_\alpha^\flat
=\overline I\left(
\scalebox{1.3}{$\frac{n}{2}$}
\right)_\alpha^\#
\quad\mathrm{for}\;\alpha\in\Z/2\Z.
\end{equation}
For $0\leq \ell\leq n+1$ and $\delta\in\Z/2\Z$, we set
$$
\index{A}{11Pielldelta@$\Pi_{\ell,\delta}$, irreducible unitary representation
of $O(n+1,1)$|textbf}
\Pi_{\ell,\delta}:=\left\{
\begin{matrix*}[l]
\overline I\left(\ell\right)_{\ell+\delta}^\flat & (0\leq\ell\leq n),\\
\overline I\left(\ell-1\right)_{\ell+\delta+1}^\# & (1\leq\ell\leq n+1).\\
\end{matrix*}
\right.
$$
In view of \eqref{eqn:flatsharp} and \eqref{eqn:nhalf}, $\Pi_{\ell,\delta}$ is well-defined and
\begin{equation}\label{eqn:Gtemp}
\Pi_{\frac n2,\delta}\simeq\Pi_{\frac n2+1,\delta},
\end{equation}
when $n$ is even.

\begin{thm}\label{thm:160148}
Let $G=O(n+1,1)$ $(n\geq1)$.
\begin{itemize}
\item[1)] Irreducible representations of $G$ with $Z_G(\mathfrak g)$-infinitesimal
character $\rho_G$ are classified as 
$$
\left\{\Pi_{\ell,\delta}\,:\,0\leq\ell\leq n+1,\,\delta\in\Z/2\Z\right\}
$$
with the equivalence relation \eqref{eqn:Gtemp} when $n$ is even.
\item[2)] There are four one-dimensional representations of $G$, and they are given by
$$
\left\{\Pi_{0,\delta},\Pi_{n+1,\delta}\,:\,\delta\in\Z/2\Z\right\}
(=\{\chi_{ab}:a,b \in \Z/2\Z\}).
$$ 
\index{A}{1x@$\chi_{\pm\pm}$, one-dimensional representation of $O(n+1,1)$}
\item[3)] For $n$ odd, $\Pi_{\frac{n+1}2,\delta}$ 
\emph{($\delta\in\Z/2\Z$)} are discrete series representations of $G$. For $n$ even, $\Pi_{\frac{n}2,\delta}\left(\simeq \Pi_{\frac{n}2+1,\delta}\right)$ 
\emph{($\delta\in\Z/2\Z$)} are tempered representations of $G$.
\item[4)] Every $\Pi_{\ell,\delta}$ \emph{($0\leq\ell\leq n+1,\,\delta\in\Z/2\Z$)} is unitarizable.
\item[5)] Irreducible  and unitarizable $(\mathfrak g, K)$-modules with nonzero 
\index{B}{gKcoh@$(\mathfrak g, K)$-cohomology}
$(\mathfrak g, K)$-cohomologies are exactly given as the set of the underlying 
$(\mathfrak g, K)$-modules of $\Pi_{\ell,\delta}$ $(0\leq\ell\leq n+1,\,\delta\in\Z/2\Z)$ up to the equivalence \eqref{eqn:Gtemp} when $n$ is even.
\item[6)]
For $0\leq i\leq n$ with $n\neq 2i$, we have a nonsplitting exact sequence of $G$-modules
$$
0\To\Pi_{i,0}\To\varpi^{(i)}_{0,0}\To\Pi_{i+1,0}\To 0.
$$
For $n=2i$, we have a $G$-isomorphism:
$$
\varpi^{(i)}_{0,0}\simeq\Pi_{\frac n2,0}.
$$
Furthermore, the de Rham complex
$$
\mathcal E^0(S^n)\stackrel{d}{\To}
\mathcal E^1(S^n)\stackrel{d}{\To}
\mathcal E^2(S^n)\stackrel{d}{\To}
\cdots\stackrel{d}{\To}
 \mathcal E^n(S^n)\stackrel{d}{\To}\{0\}
$$
yields a family of intertwining operators for
$(\varpi^{(i)}_{0,0},\mathcal E^i(S^n))$, and 
\end{itemize}
\begin{eqnarray*}
\mathrm{Ker}(d\colon \mathcal E^i(S^n)\To\mathcal E^{i+1}(S^n))&=&
\left\{
\begin{matrix*}[l]
\Pi_{i,0}&(0\leq i\leq n-1),\\
\mathcal E^i(S^n)& (i=n),
\end{matrix*}
\right.\\
\mathrm{Image}(d\colon \mathcal E^{i-1}(S^n)\To\mathcal E^{i}(S^n))&=&
\left\{
\begin{matrix*}[l]
\{0\}&&( i=0),\\
\Pi_{i,0}&& (1\leq i\leq n),
\end{matrix*}
\right.
\end{eqnarray*}
giving rise to
$$
H^i_{\mathrm{de\, Rham}}(S^n;\C)\simeq
\left\{
\begin{matrix*}[l]
\Pi_{0,0} & (i=0),\\
\{0\}&(1\leq i\leq n-1),\\
\Pi_{n,0}& (i=n).
\end{matrix*}
\right.
$$
as $G$-modules.
\end{thm}

%%%%%%%%%%%%%%%%%%%%%%%%%%%%%%%%%%%%%%%%%%%%%%%
\subsection{Differential symmetry breaking operators for principal series}
${ }$

This section gives a group theoretic reformulation of the main results
stated in Introduction
(see Theorem \ref{thm:1} and Theorems \ref{thm:2}-\ref{thm:2ii-2})
via the isomorphism in Proposition \ref{prop:identific}.

Let us realize $G' = O(n,1)$ in $G$ as the stabilizer of the point
$\trans (0,\ldots, 0, 1, 0) \in \R^{n+2}$. Then $G'$ leaves 
$\Xi \cap \{x_n = 0\}$ invariant, and acts conformally on the 
totally geodesic hypersphere 
$S^{n-1} = \{ (y_0, \dots, y_n) \in S^n  : y_n = 0\}
\simeq (\Xi \cap \{x_n = 0 \})/\R^\times$.
The isotropy subgroup of 
\index{A}{1ks0b@$[\xi^\pm]$ $(\in \Xi/\R^\times = S^n)$}
$[\xi^+] \in S^{n-1}$ is a parabolic subgroup 
$P' = P \cap G'$, which has a Langlands decomposition 
\index{A}{P1P'@$P'$, parabolic subgroup of $O(n,1)$|textbf}
$P' = M' A N'_+$  
\index{A}{N2+'@$N'_+$, unipotent subgroup of $O(n,1)$|textbf}
with 
\index{A}{M'@$M'$ ($=O(n-1)\times O(1)$)|textbf}
$M'=M\cap G' \simeq O(n-1) \times O(1)$
and 
\index{A}{A@$A$, split torus $(\simeq \R)$}
$A$ being the same split abelian subgroup as in $P$.
The Lie algebra $\mathfrak{n}_+'(\R)$ of $N_+'$ is given by 
\begin{equation*}
\index{A}{N1+2(R)@$\mathfrak{n}_+'(\R)$, Lie algebra of $N_+'$|textbf}
\mathfrak{n}_+'(\R) = \sum_{k=1}^{n-1}\R N_k^+.
\end{equation*}
\index{A}{Nell+@$N_\ell^+$ ($=\frac{1}{2}C^+_\ell$), basis of $\mathfrak{n}_+(\R)$}

Given a representation $(\sigma,V)$ of $M \simeq O(n) \times O(1)$ 
and $\lambda \in \mathbb{C}$, we defined in Section \ref{subsec:ps} 
the principal series representation 
$\mathrm{Ind}^G_P(\sigma_\lambda) 
\equiv \mathrm{Ind}^{G}_{P}(\sigma \boxtimes \mathbb{C}_\lambda)$ of $G = O(n+1,1)$. Similarly,
for a given representation
 $(\tau, W)$ of $M' \simeq O(n-1) \times O(1)$ 
and $\nu \in \mathbb{C}$, we define the principal series representation 
$\mathrm{Ind}^{G'}_{P'}(\tau_\nu) \equiv 
\mathrm{Ind}^{G'}_{P'}(\tau \boxtimes \mathbb{C}_\nu)$ of $G' = O(n,1)$,
and consider its $N$-picture on $C^\infty(\mathbb{R}^{n-1}) \otimes W$.
Then differential symmetry breaking operators from 
$\mathrm{Ind}^G_P(\sigma_\lambda)$ to 
$\mathrm{Ind}^{G'}_{P'}(\tau_\nu)$ are given as 
differential operators $C^\infty(\R^n) \otimes V \to C^\infty(\R^{n-1}) \otimes W$,
namely,
$\mathrm{Hom}_\C(V,W)$-valued differential operators from $\R^n$ to $\R^{n-1}$
in the $N$-picture.
\vskip 0.1in

As in the case of $G=O(n+1,1)$, 
\index{A}{1tau-nu-beta@$\tau^{(j)}_{\nu,\beta}$, representation of $P'$ on $\Exterior^j(\C^{n-1})$|textbf}
$\tau^{(j)}_{\nu,\beta}$ ($0\leq j \leq n-1$, $\nu \in \C$, $\beta \in \Z/2\Z$)
denotes the representation of $P'=M'AN_+'$ such that 
$M'A\simeq O(n-1) \times O(1) \times A$ acts as the outer 
tensor product representation on 
$\Exterior^j(\C^{n-1})\boxtimes (-1)^\beta \boxtimes \C_\nu$
and $N_+'$ acts trivially. Then we define the principal series representation
of $G'=O(n,1)$ by
\index{A}{Jjnu@$J(j,\nu)_\beta$, principal series of $O(n,1)$|textbf}
$J(j,\nu)_\beta:=\mathrm{Ind}^{G'}_{P'}\left(\tau^{(j)}_{\nu,\beta}\right)$.
First we prove a 
\index{B}{duality theorem for symmetry breaking operators
(principal series)|textbf}
duality theorem for symmetry breaking operators:

\begin{thm}[duality theorem]\label{thm:psdual}
Let $0\leq i \leq n$, $0\leq j \leq n-1$, $\lambda, \nu \in \C$
and $\alpha,\beta\in \Z/2\Z$. Then
\index{A}{Iilambda@$I(i,\lambda)_\alpha$, principal series of $O(n+1,1)$}
\begin{equation}\label{eqn:psdual}
\mathrm{Diff}_{G'}(I(i,\lambda)_\alpha, J(j,\nu)_\beta)
\simeq \mathrm{Diff}_{G'}(I(n-i,\lambda)_\alpha, J(n-1-j, \nu)_\beta).
\end{equation}
\end{thm}

\begin{proof}
Applying Lemma \ref{lem:psdual} to $G$ and $G'$, we have 
natural $G$-and $G'$-isomorphisms:
\index{A}{1x--@$\chi_{--}$}
\begin{align*}
I(i,\lambda)_\alpha \otimes \chi_{--} 
&\simeq I(n-i, \lambda)_\alpha,\\
J(j,\nu)_\beta \otimes \chi_{--}\vert_{G'}
&\simeq J(n-1-j, \nu)_\beta.
\end{align*}
Therefore we have the following natural bijections:
\begin{align*}
\mathrm{Hom}_{G'}(I(i,\lambda)_\alpha, J(j,\nu)_\beta)
&\simeq \mathrm{Hom}_{G'}(I(i,\lambda)_\alpha \otimes \chi_{--},
J(j,\nu)_\beta \otimes\chi_{--}\vert_{G'})\\
&\simeq \mathrm{Hom}_{G'}(I(n-i,\lambda)_\alpha, J(n-1-j,\nu)_\beta).
\end{align*}

The above isomorphisms preserve
differential operators for the geometric
realizations of principal series on the Fr\'echet spaces of smooth sections
of equivariant vector bundles over real flag varieties. Thus we have shown
the isomorphism \eqref{eqn:psdual}.
\end{proof}

In order to avoid possible confusion with 
the parameter for the conformal representation 
$(\varpi^{(i)}_{u,\delta}, \mathcal{E}^{i}(S^n))$,
it is convenient to introduce another notation for 
the differential symmetry breaking operators 
between principal series representations in the $N$-picture.
The notation below follows from \cite{KS13} which treats both 
local (\emph{i.e.}, differential) and nonlocal symmetry breaking operators.

For $\lambda, \nu \in \C$ with $\nu-\lambda \in \N$, 
we define (scalar-valued) differential operators 
\index{A}{Clambdanu@$\widetilde{\C}_{\lambda,\nu}
\left(=\restn \circ \mathcal{D}^{\lambda-\frac{n-1}{2}}_{\nu-\lambda}\right)$,
Juhl's operator|textbf}
$\widetilde{\C}_{\lambda,\nu}\colon  C^\infty(\R^n) \To C^\infty(\R^{n-1})$ by
\begin{align}
\widetilde{\C}_{\lambda,\nu}
&:= \mathrm{Rest}_{x_n=0} \circ \left(I_{\nu-\lambda}
\widetilde{C}^{\lambda-\frac{n-1}{2}}_{\nu-\lambda} \right)
\left(-\Delta_{\R^{n-1}},\frac{\partial}{\partial x_n}\right)\label{eqn:Cln}\\
&=\mathrm{Rest}_{x_n=0}\circ\mathcal{D}^{\lambda-\frac{n-1}{2}}_{\nu-\lambda},\nonumber
\end{align}
where 
\index{A}{Iell@$I_\ell$, $\ell$-inflated polynomial}\index{A}{C1ell1@$\widetilde C_\ell^\mu(t)$, 
renormalized Gegenbauer polynomial}
$(I_{\ell}
\widetilde{C}^{\mu}_{\ell} )(x,y)=x^{\frac\ell2}\widetilde{C}^{\mu}_{\ell} \left(\frac{y}{\sqrt x}\right)$
is a polynomial of two variables associated with the renormalized Gegenbauer polynomial (see \eqref{eqn:Gegen2}) and the corresponding differential operator 
\index{A}{Dell@$\mathcal D_\ell^\mu$}
$\mathcal{D}^{\mu}_\ell$ is given
by \eqref{eqn:Dl}.

For example, we have
\begin{alignat*}{2}
&\widetilde{\C}_{\lambda+1, \nu-1}\;
&=\restn \circ \mathcal{D}^{\lambda-\frac{n-3}{2}}_{\nu-\lambda-2},\\
&\widetilde{\C}_{\lambda+1,\nu}
&=\restn \circ \mathcal{D}^{\lambda-\frac{n-3}{2}}_{\nu-\lambda-1},\\
&\widetilde{\C}_{\lambda,\nu-1}
&=\restn\circ \mathcal{D}^{\lambda-\frac{n-1}{2}}_{\nu-\lambda-1}.
\end{alignat*}

Next, for $\lambda,\nu \in\C$ with $\nu -\lambda\in \N$,
we define (matrix-valued) differential operators 
\begin{equation*}
\C^{i,j}_{\lambda,\nu}\colon  \mathcal{E}^i(\R^n) \To \mathcal{E}^j(\R^{n-1})
\end{equation*}
as follows: they are essentially the same with the operators $\mathcal{D}^{i\to j}_{u,a}$
or $\widetilde{\mathcal{D}}^{i \to j}_{u,a}$, respectively, introduced in Chapter \ref{sec:intro}.
To be precise, in all the cases below, the parameters for 
$\C^{i,j}_{\lambda,\nu}$ or $\widetilde{\C}^{i,j}_{\lambda,\nu}$
are taken as
\index{A}{Dii0@$\mathcal D_{u,a}^{i\to j}\; (=\mathbb{C}^{i,j}_{u+i,u+i+a})
\colon \mathcal{E}^i(\R^n)\To \mathcal{E}^j(\R^{n-1})$,
(unnormalized) differential symmetry breaking operator}
\begin{equation}\label{eqn:CDdict}
\C^{i,j}_{\lambda,\nu}=\mathcal{D}^{i \to j}_{u,a}, \qquad 
\index{A}{Dii20@$\widetilde{\mathcal{D}}_{u,a}^{i\to j}\; (=
\widetilde{\C}^{i,j}_{u+i,u+i+a})\colon \mathcal{E}^i(\R^n)\To \mathcal{E}^j(\R^{n-1})$,
normalized differential symmetry breaking operator}
\widetilde{\C}^{i,j}_{\lambda,\nu} = \widetilde{\mathcal{D}}^{i \to j}_{u,a}
\end{equation}
with $a = \nu -\lambda$ and $u=\lambda - i$.
\index{A}{Clambdanu0@$\C^{i,j}_{\lambda,\nu}\; 
(=\mathcal{D}^{i\to j}_{\lambda-i, \nu-\lambda})
\colon \mathcal{E}^i(\R^n) \To \mathcal{E}^j(\R^{n-1})$,
(unnormalized) differential symmetry breaking operator|textbf}
\index{A}{Clambdanu20@$\widetilde{\C}^{i,j}_{\lambda,\nu}\;
(=\widetilde{\mathcal{D}}^{i \to j}_{\lambda-i,\nu-j})$,
normalized differential symmetry breaking operator|textbf}
The parameters of the operators 
$\C^{i,j}_{\lambda,\nu}$ or $\widetilde{\C}^{i,j}_{\lambda,\nu}$
indicate that these operators
induce symmetry breaking operators
 from $I(i,\lambda)_\alpha$ to
$I(i,\lambda)_\beta$ when $\nu-\lambda \equiv \beta -\alpha \; \mathrm{mod}\;2$,
whereas the parameters of 
$\mathcal{D}^{i\to j}_{u,a}$ or $\widetilde{\mathcal{D}}^{i \to j}_{u,a}$ indicate that $a\in \N$ is the order of the differential operators
and $u \in \C$ is 
normalized in a way that $u =0$ gives the untwisted case.

For $j=i$ with $0\leq i \leq n-1$,
we recall from \eqref{eqn:Dii} and \eqref{eqn:DiB} the formul\ae{} 
of $\mathcal{D}^{i\to i}_{u,a}$, and set
\index{A}{Dii2@$\mathcal D_{u,a}^{i\to i}$}
\index{A}{Clambdanu2@$\C^{i,i}_{\lambda,\nu}$|textbf}
\begin{align}
\C^{i,i}_{\lambda,\nu}
&:=\mathcal{D}^{i\to i }_{\lambda-i, \nu-\lambda} \nonumber\\
&=\widetilde{\C}_{\lambda+1, \nu-1} d_{\R^n} d^*_{\R^n}
-\gamma(\lambda-\frac{n}{2}, \nu-\lambda) \widetilde{\C}_{\lambda,\nu-1}
d_{\R^n}\iota_{\frac{\partial}{\partial x_n}} 
+ \frac{1}{2}(\nu-i)\widetilde{\C}_{\lambda,\nu},\label{eqn:Cijln}\\
&=-d^*_{\R^{n-1}}d_{\R^{n-1}}\widetilde{\C}_{\lambda+1,\nu-1} + 
\gamma(\lambda-\frac{n-1}{2},\nu-\lambda)\widetilde{\C}_{\lambda+1,\nu}
\iotan d_{\R^n} + \frac{\lambda-i}{2} \widetilde{\C}_{\lambda,\nu}, \label{eqn:Cii2}
\end{align}
in the flat coordinates. The equalities
\begin{equation*}
\eqref{eqn:Cijln} = \eqref{eqn:Cii2} = 
(-1)^{n-1}*_{\R^{n-1}}\circ \C^{n-i,n-i-1}_{\lambda,\nu} \circ 
(*_{\R^n})^{-1}
\end{equation*}
will be proved in Proposition \ref{prop:dualCi}.

For $j = i-1$ with $1 \leq i \leq n$, 
we recall from \eqref{eqn:Dii1} and \eqref{eqn:Di-B} the formul\ae{}
of $\mathcal{D}^{i\to i-1}_{u,a}$, and set
\index{A}{Dii1@$\mathcal D_{u,a}^{i\to i-1}$}
\index{A}{Clambdanu1@$\C^{i,i-1}_{\lambda,\nu}$|textbf}
\begin{align}
\C^{i,i-1}_{\lambda,\nu} 
&:= 
\mathcal{D}^{i \to i-1}_{\lambda-i, \nu-\lambda} \nonumber\\
&=- \widetilde{\mathcal{\C}}_{\lambda+1, \nu-1} d_{\R^n}d^*_{\R^n}
\iota_{\frac{\partial}{\partial x_n}} -\gamma(\lambda-\frac{n-1}{2}, \nu-\lambda)
\widetilde{\C}_{\lambda+1, \nu}d^*_{\R^n} 
+\frac{1}{2}(\lambda+i-n)\widetilde{\C}_{\lambda,\nu} \iota_{\frac{\partial}{\partial x_n}}
\label{eqn:Cijln2} \\
&=-\widetilde{\C}_{\lambda+1, \nu-1}d^*_{\R^n}\iotan d_{\R^n} + \frac{1}{2}(\nu-n+i)
\widetilde{\C}_{\lambda,\nu}\iotan - \gamma(\lambda-\frac{n}{2}, \nu-\lambda)d^*_{\R^{n-1}}
\widetilde{\C}_{\lambda,\nu-1}. \label{eqn:Cii-2}
\end{align}
Then Proposition \ref{prop:Dnonzero} means that
\begin{align}
\C^{i,i}_{\lambda, \nu} &=0 \;
\text{if and only if $\lambda = \nu =i$ or $\nu = i = 0$},\label{eqn:Cizero}\\
\C^{i,i-1}_{\lambda,\nu}&=0 \;
\text{if and only if $\lambda=\nu=n-i$ or $\nu = n-i = 0$}.\label{eqn:Ci-zero}
\end{align}
We note that 
$\C^{0,0}_{\lambda,\nu} = \frac{1}{2}\nu \widetilde{\C}_{\lambda,\nu}$.
As in \eqref{eqn:reDii1} and \eqref{eqn:reDii}, we renormalize these operators by
\index{A}{Clambdanu7@$\widetilde{\C}^{i,i}_{\lambda,\nu}$|textbf}
\index{A}{Clambdanu6@$\widetilde{\C}^{i,i-1}_{\lambda,\nu}$|textbf}
\begin{equation}\label{eqn:reCij}
\widetilde{\C}^{i,i}_{\lambda,\nu}:=
\begin{cases}
\mathrm{Rest}_{x_n=0} & \text{if $\lambda = \nu$},\\
\widetilde{\C}_{\lambda,\nu} & \text{if $i=0$},\\
\C^{i,i}_{\lambda,\nu} & \text{otherwise},
\end{cases}
\quad
\text{and}
\quad
\widetilde{\C}^{i,i-1}_{\lambda,\nu}:=
\begin{cases}
\mathrm{Rest}_{x_n=0} \circ \iota_{\frac{\partial}{\partial x_n}} & \text{if $\lambda = \nu$},\\
\widetilde{\C}_{\lambda,\nu} \circ \iota_{\frac{\partial}{\partial x_n}} & \text{if $i=n$},\\
\C^{i,i-1}_{\lambda,\nu} & \text{otherwise}.
\end{cases}
\end{equation}

\noindent
Then $\widetilde{\C}^{i,i}_{\lambda,\nu}$ $(0\leq i \leq n-1)$ and 
$\widetilde{\C}^{i,i-1}_{\lambda,\nu}$ $(1 \leq i \leq n)$
are nonzero differential operators of order $\nu-\lambda$ for any $\lambda, \nu \in\C$
with $\nu-\lambda \in \N$. 

The differential symmetry breaking operator
$\index{A}{Clambdanu8@$\widetilde{\C}^{i,i+1}_{\lambda,\nu}$|textbf}
\widetilde{\C}^{i,i+1}_{\lambda,\nu}$ is defined by
\index{A}{Dii6@$\widetilde{\mathcal{D}}^{i \to i+1}_{u,a}$}
\begin{equation}\label{eqn:reCi+}
\widetilde{\C}^{i,i+1}_{\lambda,\nu}
:=\widetilde{\mathcal{D}}^{i\to i+1}_{\lambda-i,\nu-\lambda}
=\restn \circ 
\left(I_{i-\lambda}\widetilde{C}^{\lambda-i-\frac{n-1}{2}}_{i-\lambda}\right)
\left(-\Delta_{\R^{n-1}}, \frac{\partial}{\partial x_n}\right)d_{\R^n},
\end{equation}
but only when $(\lambda,\nu) = (i,i+1)$ for $1\leq i \leq n-2$ or $\lambda \in -\N$,
$\nu=1$
 for $i=0$. Explicitly, these operators take the following form:
\index{A}{Clambdanu80@$\widetilde{\C}^{i,i+1}_{i,i+1}$|textbf}
\index{A}{Clambdanu9@$\widetilde{\C}^{0,1}_{\lambda,1}$|textbf}
\begin{alignat*}{3}
&\widetilde{\C}^{i,i+1}_{i,i+1} 
&&= \mathrm{Rest}_{x_n=0} \circ d_{\R^n}
&&\text{for $1 \leq i \leq n-2$,}\\
&\widetilde{\C}^{0,1}_{\lambda,1}
&&=\mathrm{Rest}_{x_n=0} \circ 
\left(I_{-\lambda}\widetilde{C}^{\lambda-\frac{n-1}{2}}_{-\lambda}\right)
\left(-\Delta_{\R^{n-1}}, \frac{\partial}{\partial x_n}\right) d_{\R^n}
&&\\
&
&&=d_{\R^{n-1}} \circ \widetilde{\C}_{\lambda,0}
&&\text{for $\lambda \in -\N$.}
\end{alignat*}

\noindent
Similarly, we define
\begin{equation}\label{eqn:reCi--}
\widetilde{\C}^{i,i-2}_{\lambda,\nu} 
:= \widetilde{\mathcal{D}}^{i\to i-2}_{\lambda-i,\nu-\lambda}
=\restn \circ
\left(I_{n-i-\lambda}\widetilde{C}^{\lambda-i+\frac{n+1}{2}}_{n-i-\lambda}\right)
\left(-\Delta_{\R^{n-1}}, \frac{\partial}{\partial x_n} \right)\circ 
\iotan \circ d_{\R^n},
\end{equation}
but only when $(\lambda,\nu) = (n-i, n-i+1)$ $(2\leq i \leq n-1)$
or $\lambda \in -\N$, $\nu=1$ $(i=n)$. Explicitly, these operators
take the following form:
\index{A}{Clambdanu4@$\widetilde{\C}^{i,i-2}_{n-i,n-i+1}$|textbf}
\index{A}{Clambdanu5@$\widetilde{\C}^{n,n-2}_{\lambda,1}$|textbf}
\begin{alignat*}{3}
&\widetilde{\C}^{i,i-2}_{n-i,n-i+1}
&&=\mathrm{Rest}_{x_n=0} 
\circ \iota_{\frac{\partial}{\partial x_n}} d^*_{\R^n} 
&&\text{for $2 \leq i \leq n-1$,}\\
&\widetilde{\C}^{n,n-2}_{\lambda,1}
&&=\mathrm{Rest}_{x_n=0} \circ 
\left(I_{-\lambda}\widetilde{C}^{\lambda-\frac{n-1}{2}}_{-\lambda}\right)
\left(-\Delta_{\R^{n-1}},\frac{\partial}{\partial x_n}\right)
\iota_{\frac{\partial}{\partial x_n}} d^*_{\R^n}
&&\\
&
&&=-d^*_{\R^{n-1}}\circ \widetilde{\C}^{n,n-1}_{\lambda,0}
&&\text{for $\lambda \in -\N$.}
\end{alignat*}

To see the second equality, we use some elementary commutation relations 
which will be given in Lemma \ref{lem:152457} (2) 
and Lemma \ref{lem:152456} (2) among others.

We are ready to give a classification of symmetry breaking operators from 
the principal series representation $I(i,\lambda)_\alpha$ of $G=O(n+1,1)$
to the principal series representation $J(j,\nu)_\beta$ of the subgroup 
$G' = O(n,1)$.

\begin{thm}\label{thm:1A}
Let $n\geq 3$. Suppose $0\leq i\leq n$, $0\leq j\leq n-1$, $\lambda,\nu\in\C$,
and $\alpha,\beta\in\Z/2\Z$. 
The following three conditions on 6-tuple $(i,j,\lambda,\nu,\alpha,\beta)$ 
are equivalent:
\begin{enumerate}
\item[(i)]
\index{A}{Jjnu@$J(j,\nu)_\beta$, principal series of $O(n,1)$}
$\mathrm{Diff}_{O(n,1)}(
I(i,\lambda)_\alpha,
J(j,\nu)_\beta)\neq\{0\}.$
\item[(ii)]
$\dim\mathrm{Diff}_{O(n,1)}(
I(i,\lambda)_\alpha,
J(j,\nu)_\beta)=1.$
\item[(iii)]
The 6-tuple belongs to one of the following six cases:
\vskip 0.05in
\begin{enumerate}
\item[] \emph{Case 1}. $j=i-2$,
$2\leq i \leq n-1$, $(\lambda,\nu) = (n-i,n-i+1)$,
$\beta  \equiv \alpha + 1\; \mathrm{mod}\, 2$.
\vskip 0.05in

\item[] \emph{Case 1$'$}. $(i,j) = (n, n-2)$,
$-\lambda \in \N$, $\nu=1$, $\beta \equiv \alpha + \lambda+1\; \mathrm{mod}\, 2$.
\vskip 0.05in

\item[] \emph{Case 2}. $j = i-1$, 
$1 \leq i \leq n$, $\nu-\lambda \in \N$, 
$\beta -\alpha \equiv \nu - \lambda \; \mathrm{mod}\, 2$.
\vskip 0.05in

\item[] \emph{Case 3}. $j = i$, 
$0\leq i \leq n-1$, $\nu - \lambda \in \N$, 
$\beta-\alpha \equiv \nu -\lambda \; \mathrm{mod}\, 2$.
\vskip 0.05in

\item[] \emph{Case 4}. $j=i+1$, 
$1 \leq i \leq n-2$, 
$(\lambda,\nu) = (i,i+1)$, 
$\beta \equiv \alpha + 1 \; \mathrm{mod}\, 2$.
\vskip 0.05in

\item[] \emph{Case 4$'$}. $(i,j) = (0,1)$, 
$-\lambda \in \N$, $\nu = 1$, $\beta \equiv \alpha + \lambda+1 \; \mathrm{mod}\, 2$.
\end{enumerate}
\end{enumerate}
\end{thm}

\begin{thm}\label{thm:Cps}
Retain the setting and notations as in Theorem \ref{thm:1A}.
Then the following differential operators from
$\mathcal{E}^i(\R^n)$ to $\mathcal{E}^j(\R^{n-1})$
in the flat picture extend to a nonzero $O(n,1)$-homomorphism
from $I(i,\lambda)_\alpha$ to $J(j,\nu)_\beta$:
\begin{enumerate}

\item[] \emph{Cases 1 and 1$'$}. 
$\widetilde{\C}^{i,i-2}_{n-i,n-i+1}$  $(2\leq i \leq n-1)$, \;
$\widetilde{\C}^{n,n-2}_{\lambda,1}$;               
\vskip 0.05in

\item[] \emph{Case 2}. 
\index{A}{Clambdanu6@$\widetilde{\C}^{i,i-1}_{\lambda,\nu}$}
\hskip 0.565in  $\widetilde{\C}^{i,i-1}_{\lambda,\nu}$ $(1\leq i \leq n)$;
\vskip 0.05in

\item[] \emph{Case 3}. 
\index{A}{Clambdanu7@$\widetilde{\C}^{i,i}_{\lambda,\nu}$}
\hskip 0.565in $\widetilde{\C}^{i,i}_{\lambda,\nu}$ $(0\leq i \leq n-1)$;
\vskip 0.05in

\item[] \emph{Cases 4 and 4$'$}. 
\index{A}{Clambdanu80@$\widetilde{\C}^{i,i+1}_{i,i+1}$}
$\widetilde{\C}^{i,i+1}_{i,i+1}$  $(1 \leq i \leq n-2)$, \;
$\widetilde{\C}^{0,1}_{\lambda,1}$.
\end{enumerate}
Conversely, any differential symmetry breaking operator
from $I(i,\lambda)_\alpha$ to $J(j,\nu)_\beta$
in Theorem \ref{thm:1A}
is proportional to one of these operators.
\end{thm}

The proof of Theorem \ref{thm:1A}
is reduced to solving the F-system, which we carry out
in Chapter \ref{sec:7} for Case 2,
Chapter \ref{sec:codiff} (Theorem \ref{thm:Ai+}) for Cases 4 and 4$'$.
The remaining cases (\emph{i.e.}\ Cases 3, 1 and 1$'$)
in Theorem \ref{thm:1A} follows from Cases 2, 4, and 4$'$, respectively,
\index{B}{duality theorem for symmetry breaking operators (principal series)}
by the duality theorem (Theorem \ref{thm:psdual}).
In summary, Cases 1 and 1$'$, 
Case 2, Case 3, and Cases 4 and 4$'$ in 
Theorem \ref{thm:1A} are stated and proved in 
Theorem \ref{thm:Ai--}, \ref{thm:Ai-}, \ref{thm:Ai}, and \ref{thm:Ai+},
respectively. The proof of Theorem \ref{thm:Cps} will be completed in Chapter \ref{sec:5}.

In Chapter \ref{sec:solAB}, we shall see that 
Theorem \ref{thm:1} is derived from 
Theorem \ref{thm:1A} via the isomorphism in Proposition \ref{prop:identific}.
Theorems \ref{thm:1}, \ref{thm:2}, \ref{thm:2ii}, and \ref{thm:2ii-2}
are obtained from Theorem \ref{thm:Cps} (see Section \ref{subsec:pfB}).

%%%%%%%%%%%%%%%%%%%%%%%%%%%%%%%%%%%%%%%%%%%%%%%
\subsection{Symmetry breaking operators for connected group $SO_0(n,1)$}
\label{subsec:conn}

So far we have dealt with the \emph{disconnected} group $G'=O(n,1)$
in studying symmetry breaking operators. 
Results for the \emph{connected} group $G_0'=SO_0(n,1)$
(or equivalently, for conformal vector fields on $S^n$ along the submanifold $S^{n-1}$)
can be deduced from those in the disconnected case.

In this section we explain a trick for the reduction to the connected case.
Let $G_0=SO_0(n+1,1)$ be the identity component of $G=O(n+1,1)$,
and $G_0'=SO_0(n,1)$ be that of $G'=O(n,1)$. 

The connected group $G_0$ acts transitively on $G/P$, and we have a
natural isomorphism
\begin{equation*}
G_0/P_0 \stackrel{\sim}{\To} G/P \; (\simeq S^n),
\end{equation*}
where $P_0:=P\cap G_0 = M_0AN_+$
is a parabolic subgroup of $G_0$.
Then both $P_0$ and $M_0 \simeq SO(n)$ are connected.
For $0\leq i \leq n$ and $\lambda \in \C$, we write
$I(i,\lambda)$ for the (unnormalized) induced representation
$\mathrm{Ind}^{G_0}_{P_0}\left(\Exterior^i(\C^n) \boxtimes \C_\lambda\right)$
of $G_0$. (We note that $\Exterior^i(\C^n)$ is reducible
as an $M_0$-module if and only if $n=2i$, but we do not enter 
this point here.)
We recall from Section \ref{subsec:ps} that 
$I(i,\lambda)_\alpha$ ($\alpha \in \Z/2\Z)$ is a principal series
representation of $G$, which we may realize in the space
$\mathcal{E}^i(S^n)$ of $i$-forms on $S^n$.
The restriction to $G_0$ is independent of $\alpha \in \Z/2\Z$, and
we have isomorphisms as $G_0$-modules:
\begin{equation}\label{eqn:Iinf}
I(i,\lambda)_\alpha\vert_{G_0} \simeq I(i,\lambda) \simeq I(n-i,\lambda).
\end{equation}
Analogous notation will be applied to 
the subgroup
$G_0'=SO_0(n,1)$. 
In particular, $J(j,\nu)$ ($0\leq j \leq n-1$, $\nu\in\C$) denotes
the (unnormalized) induced representation
$\mathrm{Ind}^{G_0'}_{P_0'}\left(\Exterior^j(\C^{n-1})\boxtimes \C_\nu\right)$,
and we have isomorphisms as $G_0'$-modules:
\begin{equation}\label{eqn:Jinf}
J(j,\nu)_\beta\vert_{G_0'} \simeq
J(j,\nu) \simeq J(n-1,\nu),
\end{equation}
defined on $\mathcal{E}^j(S^{n-1})$.

In what follows, we set
\begin{equation*}
\tilde{i}:=n-i, \quad \tilde{j}:=n-1-j.
\end{equation*}

We are ready to state the results on differential symmetry breaking operators
for the \emph{connected} subgroup $G_0' = SO_0(n,1)$:

\begin{thm}\label{thm:conn}
Suppose $0\leq i \leq n$, $0\leq j \leq n-1$,
and $\lambda,\nu \in\C$.

\begin{enumerate}

\item
There are natural bijections: 
\begin{align*}
\mathrm{Diff}_{SO_0(n,1)}\left(I(i,\lambda), J(j,\nu) \right)
&\simeq
\mathrm{Diff}_{SO_0(n,1)}\left(I(\tilde{i},\lambda), J(j,\nu) \right)\\
&\simeq
\mathrm{Diff}_{SO_0(n,1)}\left(I(i,\lambda), J(\tilde{j},\nu) \right)\\
&\simeq
\mathrm{Diff}_{SO_0(n,1)}\left(I(\tilde{i},\lambda), J(\tilde{j},\nu) \right).
\end{align*}

\item The above space is nonzero only when
$\nu-\lambda \in \N$. Assume now $\nu-\lambda \in \N$.
We fix $\alpha \in \Z/2\Z$ and set $\beta:=\alpha + \nu - \lambda \; \mathrm{mod} \; 2$.
Then we have 
\begin{align*}
\mathrm{Diff}_{SO_0(n,1)}(I(i,\lambda), J(j,\nu))
&\simeq 
\mathrm{Diff}_{O(n,1)}(I(i,\lambda)_\alpha, J(j,\nu)_\beta) \oplus
\mathrm{Diff}_{O(n,1)}(I(\tilde{i},\lambda)_\alpha, J(j,\nu)_\beta)\\
&\simeq
\mathrm{Diff}_{O(n,1)}(I(i,\lambda)_\alpha, J(j,\nu)_\beta) \oplus
\mathrm{Diff}_{O(n,1)}(I(i,\lambda)_\alpha, J(\tilde{j},\nu)_\beta).
\end{align*}
\end{enumerate}
\end{thm}

The second statement shows that the classification 
and construction of differential
symmetry breaking operators for the connected group 
$G'_0=SO_0(n,1)$ are deduced from the one for the disconnected
case that we have given in Theorems \ref{thm:1A} and \ref{thm:Cps}.

\begin{proof}
The first statement follows directly from \eqref{eqn:Iinf} and 
\eqref{eqn:Jinf}. To see the second statement, we set 
\begin{alignat*}{3}
S&:=\{0,1,\ldots, n\}\times \C \times \Z/2\Z,
\quad
&&I(s):=I(i,\lambda)_\alpha
\quad
&&\text{for $s=(i,\lambda,\alpha) \in \Z/2\Z$},\\
T&:=\{0,1,\ldots, n-1\} \times \C \times \Z/2\Z,
\quad
&&J(t):=J(j,\nu)_\beta
\quad
&&\text{for $t = (j,\nu,\beta) \in \Z/2\Z$}.
\end{alignat*}

We recall from \eqref{eqn:chiab} that the quotient
groups are given by
\begin{equation*}
G'/G_0' \simeq G/G_0 \simeq \Z/2\Z \times \Z/2\Z,
\end{equation*}
and the set of their one-dimensional representations is given by
\index{A}{1x@$\chi_{\pm\pm}$, one-dimensional representation of $O(n+1,1)$}
\begin{equation*}
\left(G'/G_0'\right)^{\widehat{ }}\;
\simeq 
\left(G/G_0\right)^{\widehat{ }}\;
=\{\chi_{ab} : a, b \in \{\pm\}\}.
\end{equation*} 
By abuse of notation, we shall use the same letters
$\chi_{\pm \pm}$ to denote one-dimensional representations
of $G, G', G/G_0$, and $G'/G_0'$.

Let $s\in S$ and $t\in T$. Since $G'$ normalizes $G_0'$, 
the quotient group $G'/G_0'$ acts naturally on
\begin{equation*}
V(s,t):=\mathrm{Diff}_{G_0'}(I(s), J(t)),
\end{equation*}
by $D \mapsto J(t)(g) \circ D \circ I(s)(g^{-1})$,
and we have an irreducible decomposition:
\begin{equation*}
V(s,t) \simeq \bigoplus_{\chi \in (G'/G_0')^{\widehat{ }}} V(s,t)_{\chi}
\end{equation*}
where $V(s,t)_\chi$ denotes the $\chi$-component of $V(s,t)$.
We note that 
\begin{align*}
V(s,t)_\chi \simeq
\mathrm{Hom}_{G'}(I(s), J(t))
\left(=\mathrm{Hom}_{O(n,1)}(I(i,\lambda)_\alpha, J(j,\nu)_\beta) \right)
\end{align*}
if $\chi = \chi_{+ +}$ (trivial representation).

We let the character group $(G/G_0)^{\widehat{ }}\;$ act on $S$ by
the following formula:
\begin{alignat*}{2}
&\chi_{+ +} \cdot (i,\lambda, \alpha) :=(i,\lambda, \alpha),
\quad 
&&\chi_{+ -} \cdot (i,\lambda, \alpha) :=(i,\lambda, \alpha+1),\\
&\chi_{-+} \cdot (i,\lambda, \alpha) := (\tilde{i}, \lambda, \alpha+1),
\quad
&&\chi_{--} \cdot (i,\lambda,\alpha):=(\tilde{i}, \lambda,\alpha).
\end{alignat*}
Then as in Lemma \ref{lem:psdual}, we have a $G$-isomorphism
\begin{equation*}
I(s)\otimes \chi \simeq I(\chi\cdot s)
\quad
\text{for any $\chi \in (G/G_0)^{\widehat{ }}\;$ and $s \in S$.}
\end{equation*}
Therefore, we have natural isomorphisms as $G'/G_0'$-modules:
\begin{align*}
\chi^{-1}\otimes \mathrm{Diff}_{G_0'}(I(s), J(t))
\simeq 
\mathrm{Diff}_{G_0'}(I(s) \otimes \chi, J(t))
\simeq
\mathrm{Diff}_{G_0'}(I(\chi \cdot s), J(t)).
\end{align*}

Taking the $\chi_{++}$-component of the both sides, 
we get an isomorphism
\begin{equation*}
V(s,t)_{\chi} \simeq \mathrm{Diff}_{G'}(I(\chi\cdot s), J(t)).
\end{equation*}
Thus we have proved a $(G'/G_0')$-isomorphism:
\begin{equation*}
V(s,t) \simeq \bigoplus_{\chi \in (G'/G_0')^{\widehat{ } }}
\mathrm{Diff}_{G'}(I(\chi \cdot s), J(t)).
\end{equation*}

There are four summands in the right-hand side, 
however, two of them vanish by the parity condition.
In fact, if we take $\beta$ as in the statement of the theorem, 
then the two summands for $\chi_{-+}$ and $\chi_{+-}$ vanish,
as we shall
see in Proposition \ref{prop:153091} (1). 
Since $V(s,t) = \mathrm{Hom}_{G_0'}(I(i,\lambda), J(j,\nu))$,
the first equality in (2) has been proved. Likewise, we let 
the character group
$(G'/G_0')^{\widehat{ }}\;$ act on the set $T$ in a similar manner,
as we did for $S$. Then we get a $G'$-isomorphism:
\begin{equation*}
J(t)\otimes \chi \simeq J(\chi \cdot t)\quad 
\text{for any $\chi \in (G'/G_0')^{\widehat{ }}\;$ and $t\in T$.}
\end{equation*}
This leads us to the second equality.
\end{proof}

%%%%%%%%%%%%%%%%%%%%%%%%%%%%%%%%%%%%%%%%%%%%%%%
\subsection{Branching problems for Verma modules}
\label{subsec:branchV}

In this section, we discuss briefly branching problems for generalized 
Verma modules for the pair 
\begin{equation*}
(\mathfrak{g},\mathfrak{g}') = (\mathfrak{o}(n+2,\C),\mathfrak{o}(n+1,\C)),
\end{equation*}
see \cite{K12} for the general problem. In \cite[Thm.\ A]{KP1}
and \cite{KOSS15}, we established
a duality theorem that gives a one-to-one
correspondence between differential symmetry breaking operators
and $\mathfrak{g}'$-homomorphisms for the restriction of Verma modules
of $\mathfrak{g}$ in the general setting, 
see Fact \ref{fact:Fmethod}. Thus Theorem \ref{thm:1A}
for differential symmetry breaking operators leads us to the classification of 
$\mathfrak{g}'$-homomorphisms in certain branching problems of generalized 
Verma modules of $\mathfrak{g}$,
and Theorem \ref{thm:Cps} constructs the corresponding ``singular vectors".

For a $\mathfrak{p}$-module $F$ with trivial action of 
the nilpotent radical $\mathfrak{n}_+$,
we define a $\mathfrak{g}$-module 
\index{B}{generalized Verma module}
(\emph{generalized Verma module})
by
\begin{equation*}
\indpg(F):=U(\mathfrak{g})\otimes_{U(\mathfrak{p})}F.
\end{equation*}
If $F$ is a $P$-module, then the $\mathfrak{g}$-module 
$\indpg(F)$ carries a $P$-module structure, and we may regard
$\indpg(F)$ as a $(\mathfrak{g},P)$-module.

We recall that $\sigma^{(i)}_{\lambda,\alpha}$ is a $P$-module
whose restriction to $MA\simeq O(n) \times O(1) \times \R$
is given by $\Exterior^i(\C^n) \boxtimes (-1)^\alpha \boxtimes \C_\lambda$
for $0\leq i \leq n$, $\lambda \in \C$, $\alpha \in \Z/2\Z$.
We set 
\begin{align*}
M(i,\lambda)_\alpha
&:=
\indpg\left(\sigma^{(i)}_{\lambda,\alpha}\right)
=\indpg\left(\Exterior^i(\C^n) \boxtimes (-1)^\alpha \boxtimes \C_\lambda \right),\\
M(i,\lambda)
&:=
\indpg\left(\Exterior^i(\C^n) \boxtimes \C_\lambda\right).
\end{align*}
Then $M(i,\lambda)_\alpha$ is a $(\mathfrak{g}, P)$-module,
and $M(i,\lambda)$ is a $\mathfrak{g}$-module. 
The underlying $\mathfrak{g}$-module structure of $M(i,\lambda)_\alpha$
does not depend on $\alpha \in \Z/2\Z$, and 
we have the following isomorphisms as $\mathfrak{g}$-modules:
\begin{equation*}
M(i,\lambda)_\alpha\vert_{\mathfrak{g}}
\simeq
M(i,\lambda) 
\simeq
M(n-i,\lambda).
\end{equation*}

Similarly, for $0\leq j \leq n-1$, $\nu \in \C$, $\beta \in \Z/2\Z$, we set 
\begin{align*}
M'(j,\nu)_\beta 
&:= \indpgprime\left(\tau^{(j)}_{\nu,\beta}\right)
=\indpgprime\left(\Exterior^j(\C^{n-1})\boxtimes (-1)^\beta \boxtimes \C_\nu\right),\\
M'(j,\nu)
&:=
\indpgprime\left(\Exterior^j(\C^{n-1}) \boxtimes \C_\nu\right).
\end{align*}
Then $M'(j,\nu)_\beta$ is a $(\mathfrak{g}',P')$-module and
$M'(j,\nu)$ is a $\mathfrak{g}'$-module. We have
the following isomorphisms as $\mathfrak{g}'$-modules.
\begin{equation*}
M'(j,\nu)_\beta\vert_{\mathfrak{g}'}
\simeq
M'(j,\nu)
\simeq
M'(n-1-j,\nu).
\end{equation*}

As a part of branching problems, we wish to understand how
the $\mathfrak{g}$-module $M(i,\lambda)$ behaves when
restricted to the subalgebra $\mathfrak{g}'$, or how the 
$(\mathfrak{g},P)$-module $M(i,P)_\alpha$ behaves as a 
$(\mathfrak{g}',P')$-module. As a dual to Theorem \ref{thm:conn}
(see Fact \ref{fact:Fmethod}), we obtain:

\begin{thm}\label{thm:branchV}
Suppose $0\leq i \leq n$, $0\leq j\leq n-1$, and $\lambda,\nu \in \C$.
\begin{enumerate}
\item $\mathrm{Hom}_{\mathfrak{g}'}\left(
M'(\tilde{j},-\nu),M(\tilde{i},-\lambda)\vert_{\mathfrak{g}'}
\right)\neq \{0\}$ only if $\nu - \lambda \in \N$.
\item Assume $\nu-\lambda \in \N$. We fix $\alpha \in \Z/2\Z$
and set $\beta:=\alpha+\nu-\lambda\; \mathrm{mod} \; 2$. Then
we have
\begin{align*}
&\mathrm{Hom}_{\mathfrak{g}'}
\left(M'(\tilde{j}, -\nu), M(\tilde{i}, -\lambda) \right)\\
&\simeq 
\mathrm{Hom}_{\mathfrak{g}',P'}
\left(M'(\tilde{j}, -\nu)_\beta, M(\tilde{i}, -\lambda)_\alpha \right)
\bigoplus
\mathrm{Hom}_{\mathfrak{g}',P'}
\left(M'(\tilde{j}, -\nu)_\beta, M(i, -\lambda)_\alpha \right)\\
&\simeq
\mathrm{Hom}_{\mathfrak{g}',P'}
\left(M'(\tilde{j}, -\nu)_\beta, M(\tilde{i}, -\lambda)_\alpha \right)
\bigoplus
\mathrm{Hom}_{\mathfrak{g}',P'}
\left(M'(j, -\nu)_\beta, M(\tilde{i}, -\lambda)_\alpha \right).
\end{align*}
\end{enumerate}
\end{thm}

The summands in the right-hand sides in (2) of Theorem \ref{thm:branchV}
are classified as follows.

\begin{prop}
Let $n\geq 3$. Suppose $0\leq i\leq n$, $0\leq j\leq n-1$, $\lambda,\nu\in\C$,
and $\alpha,\beta\in\Z/2\Z$. 
Then the following three conditions on 6-tuple $(i,j,\lambda,\nu,\alpha,\beta)$ 
are equivalent:
\begin{enumerate}
\item[(i)] $\mathrm{Hom}_{\mathfrak{g}', P'}\left(
M'(\tilde{j},-\nu)_\beta, M(\tilde{i}, -\lambda)_\alpha\right) \neq \{0\}$.
\item[(ii)] $\dim\mathrm{Hom}_{\mathfrak{g}', P'}\left(
M'(\tilde{j},-\nu)_\beta, M(\tilde{i}, -\lambda)_\alpha\right) =1$.
\item[(iii)] The 6-tuple $(i,j,\lambda, \nu,\alpha, \beta)$ belongs to one 
of the six cases in Theorem \ref{thm:1A} (iii).
\end{enumerate}
\end{prop}

The left-hand side of the isomorphisms in Theorem \ref{thm:branchV} 
is isomorphic to
\begin{equation*}
\mathrm{Hom}_{\mathfrak{p}'}\left(
\Exterior^{n-1-j}(\C^{n-1}) \otimes \C_{-\nu},
\indpg\left(\Exterior^{n-i}(\C^n) \otimes \C_{-\lambda}\right)\right)
\end{equation*}
and vectors in the image of $\mathfrak{p}'$-homomorphisms
are sometimes referred to as 
\index{B}{singular vector|textbf}
singular vectors.
Fact \ref{fact:Fmethod} in the next chapter asserts that
one could get one from another among the following:
\vskip 0.05in

\begin{itemize}
\item (explicit construction of) singular vectors;
\vskip 0.02in
\item (explicit construction of ) symmetry breaking operators 
(Theorem \ref{thm:Cps});
\vskip 0.02in
\item (explicit construction of) polynomial solutions to the F-system
(Theorems \ref{thm:Fi-} and \ref{thm:Fiiplus}).
\end{itemize}

\newpage
%%%%%%%%%%%%%%%%%%%%%%%%%%%%%%%%%%%%%%%%%%%%%%%
\section{F-method for matrix-valued differential operators}\label{sec:Method}

In this chapter we recall from 
\cite{K12, K13, KOSS15, KP1}  a
method based on the Fourier transform (\emph{F-method}) to find explicit formul\ae{} 
of differential symmetry breaking operators. 
For our purpose we need to develop the F-method for matrix-valued operators.
A new ingredient is a canonical decomposition of 
the algebraic Fourier transform of the vector-valued principal series representations
into a ``scalar part" involving differential operators of higher order and into a 
``vector part" of first order. This is formulated and proved in Section \ref{subsec:3general}.

%%%%%%%%%%%%%%%%%%%%%%%%%%%%%%%%%%%%%%%%%%%%%%%
\subsection{Algebraic Fourier transform}\label{subsec:2alg}
Let $E$ be a vector space over $\C$. 
The 
\index{B}{Weyl algebra|textbf}Weyl algebra
\index{A}{DE@$\mathcal D(E)$, Weyl algebra|textbf}
$\mathcal D(E)$
is the ring of holomorphic differential operators on $E$
with polynomial coefficients. 

\begin{defn}\label{def:hat}
We define the 
\index{B}{algebraic Fourier transform (Weyl algebra)|textbf}
\emph{algebraic Fourier transform} 
as an algebra isomorphism of two Weyl algebras on $E$ and its dual space $E^\vee$:
\index{A}{That@$\widehat T$, algebraic Fourier transform|textbf}
\begin{equation*}
\mathcal D(E)\to\mathcal D(E^\vee), \qquad T\mapsto 
\widehat T,
\end{equation*}
induced by
\begin{equation}\label{eqn:FTdiff}
\widehat{\frac\partial{\partial z_\ell}}:=-\zeta_\ell,\quad
\widehat z_\ell:=\frac\partial{\partial\zeta_\ell},\quad 1\leq \ell\leq n,
\end{equation}
where $n=\dim_\C E, (z_1,\ldots,z_n)$ are coordinates on $E$ and $(\zeta_1,\ldots,\zeta_n)$ are the dual coordinates on $E^\vee$.
\end{defn}

Any linear transformation $A\in GL(E)$ gives rise to bijections
\index{A}{Asharp@$A_{\#}$|textbf}
\begin{alignat*}{2}
A_{\#}\colon &\; \operatorname{Pol}(E)\To\operatorname{Pol}(E), &&\quad F\mapsto F(A^{-1}\cdot),\\
A_*\colon &\; \mathcal D(E)\To\mathcal D(E),&&\quad T\mapsto A_{\#}\circ T\circ A_{\#}^{-1}.
\end{alignat*}
We write $\trans A\in GL(E^\vee)$ for the dual map. Then the following identity holds
\cite[Lem.\ 3.3]{KP1}:
\begin{equation}\label{eqn:AT}
\widehat{A_*T}=\left(\trans A^{-1}\right)_*\widehat T\qquad\mathrm{for}\,\mathrm{all}\quad T\in\mathcal D(E).
\end{equation}

%%%%%%%%%%%%%%%%%%%%%%%%%%%%%%%%%%%%%%%%%%%%%%%
\subsection{Differential operators between two manifolds}
${}$
\vskip5pt

We need a generalized notion of differential operators, not only for functions on the 
\emph{same} manifolds but also for functions on two \emph{different} manifolds with 
a morphism.

Let $\mathcal{V} \to X$ be a vector bundle over a smooth manifold $X$. We write
$C^\infty(X,\mathcal{V})$ for the space of smooth sections, endowed with the 
Fr\'echet topology of uniform convergence of sections and their derivatives of 
finite order on compact sets. Let $\mathcal{W} \to Y$ be another vector bundle.
Suppose a smooth map $p\colon  Y \to X$ is given.

\begin{defn}(\cite[Def.~2.1]{KP1})\label{def:diff}
We say a continuous operator 
$T\colon C^\infty(X,\mathcal{V}) \to C^\infty(Y,\mathcal{W})$ is a 
\index{B}{differential operator between two manifolds|textbf}
\emph{differential operator} if $T$ satisfies
\begin{equation*}
p(\mathrm{Supp}Tf) \subset \mathrm{Supp}f \quad \text{for all $f \in C^\infty(X,\mathcal{V})$.}
\end{equation*}
We write $\mathrm{Diff} (\mathcal V_X,\mathcal W_Y)$ for the space 
of differential operators from $C^\infty(X,\mathcal{V})$ to $C^\infty(Y,\mathcal{W})$.
\end{defn}

If $i\colon Y \to X$ is an immersion, then every 
$T \in \mathrm{Diff}(\mathcal{V}_X, \mathcal{W}_Y)$ is locally of the form
\begin{equation*}
T=\sum_{\alpha \in \N^k}\sum_{\beta \in \N^m}g_{\alpha, \beta}(y)
\frac{\partial^{|\alpha| + |\beta|}}{\partial y^\alpha \partial z^\beta}
\quad \text{(finite sum)},
\end{equation*}
where $(y_1, \ldots, y_k, z_1, \ldots, z_m)$ are local coordinates on $X$ such that
$Y$ is given by $z_1 = \cdots = z_m =0$ and $g_{\alpha,\beta}(y)$ are 
$\mathrm{Hom}(V,W)$-valued smooth functions on $Y$.

%%%%%%%%%%%%%%%%%%%%%%%%%%%%%%%%%%%%%%%%%%%%%%%
\subsection{F-method for principal series representations}\label{subsec:2Fmethod}
${}$
\vskip5pt

Let $G$ be a real reductive Lie group, and 
\index{A}{M@$M$ ($=O(n)\times O(1)$)}
\index{A}{A@$A$, split torus $(\simeq \R)$}
\index{A}{N2+@$N_+$, unipotent subgroup of $O(n+1,1)$}
$P=MAN_+$
a Langlands decomposition of a parabolic subgroup $P$
of $G$.
Their Lie algebras will be denoted by $\mathfrak g(\R)$, 
$\mathfrak p(\R)=\mathfrak m(\R)+\mathfrak a(\R)+ \mathfrak n_+(\R)$, 
and the complexified Lie algebras by 
$\mathfrak g$, 
$\mathfrak p=\mathfrak m+\mathfrak a+ \mathfrak n_+$, respectively.
\vskip 0.1in

Given $\lambda \in \mathfrak{a}^* \simeq \mathrm{Hom}_{\mathbb{R}}(\mathfrak{a}(\mathbb{R}), \mathbb{C})$,
we define one-dimensional representation $\mathbb{C}_\lambda$ of $A$
by $a \mapsto a^\lambda:=e^{\langle \lambda, \log a \rangle}$.
By letting $MN_+$ act trivially, we also regard $\mathbb{C}_\lambda$ as 
a representation of $P$.
Given a representation $(\sigma, V)$ of $M$ and $\lambda \in \mathfrak{a}^*$,
we write 
\index{A}{1sigma-lambda@$\sigma_\lambda:=
\sigma \boxtimes \C_\lambda$|textbf}
$\sigma_\lambda \equiv \sigma \boxtimes \C_\lambda$
for
the representation of $MA$ on $V$ defined by $ma \mapsto a^\lambda \sigma(m)$.
The same letter will be used for the representation of $P$ which is obtained by 
letting $N_+$ act trivially.
We define $\mathcal V \equiv \mathcal V_X
=G\times_P V$ as a $G$-equivariant vector bundle over
the real flag variety $X=G/P$ associated to $\sigma_\lambda$.
The (unnormalized) principal series representation 
\index{A}{1pisigma@$\pi_{(\sigma,\lambda)}$, principal series|textbf}
$\pi_{(\sigma,\lambda)}=
\index{A}{IndPG@$\mathrm{Ind}_P^G(\sigma_\lambda)$|textbf}
\mathrm{Ind}_P^G(\sigma_\lambda)$
is defined on the Fr\'echet space $C^\infty(X,\mathcal V)$ of smooth sections of the 
vector bundle $\mathcal{V} \to X$.

Let $\mathfrak{g}(\R) = 
\mathfrak{n}_-(\R) + \mathfrak{m}(\R) + \mathfrak{a}(\R) + \mathfrak{n}_+(\R)$
be the Gelfand--Naimark decomposition.
The vector bundle $\mathcal{V} \to X$ is trivialized when restricted to the open Bruhat cell
\begin{equation*}
\mathfrak n_-(\R)\simeq N_-\hookrightarrow G/P=X,
\end{equation*}
and we may regard $C^\infty(X,\mathcal V)$ as a subspace of
$C^\infty(\mathfrak n_-(\R))\otimes V$ via the restriction.
This model is called the 
\index{B}{Npicture@$N$-picture|textbf}
\emph{$N$-picture} 
\index{B}{flat picture}
or \emph{flat picture}
of the principal series representation and the case
of the Lorentz group $G=O(n+1,1)$ was discussed in detail in Chapter \ref{sec:ps}.
The infinitesimal representation of the Lie algebra on $C^\infty(\mathfrak{n}_-(\R))\otimes V$
will be denoted by $d\pi_{(\sigma, \lambda)}$. 

\index{A}{1rho@$\rho$|textbf}
Let $2\rho \in \mathfrak{a}^*$ 
be the homomorphism on $\mathfrak{a}(\mathbb{R})$ defined by 
$Z \mapsto \mathrm{Trace}(\mathrm{ad}(Z)\colon \mathfrak{n}_+(\mathbb{R}) \to \mathfrak{n}_+(\mathbb{R}))$.
As a representation of $P$, 
\index{A}{C2rho@$\C_{2\rho}$|textbf}
$\C_{2\rho}$ is given by
 $p\mapsto \chi_{2\rho}(p):= \vert 
\mathrm{det}(\mathrm{Ad}(p)\colon 
\mathfrak n_+(\mathbb{R})\to\mathfrak n_+(\mathbb{R}))\vert$.
We also define a one-dimensional representation 
$\index{A}{sgn@$\sgn$|textbf}\sgn$ of $P$ by
$p\mapsto\sgn\circ\det(\mathrm{Ad}(p)\colon \mathfrak n_+(\R)\to\mathfrak n_+(\R))$.
Observe that the 
\index{B}{density bundle}
density bundle $\Omega_X$ and the orientation bundle of $X=G/P$ are then
given as the homogeneous line bundles $G\times_P\C_{2\rho}$ and $G\times_P\sgn$, respectively.
Since $MN_+$ acts trivially on $\C_{2\rho}$, we shall
sometimes regard $\C_{2\rho}$ as a representation of $A\simeq P/MN_+$. 
Write $V^\vee = \mathrm{Hom}_\C(V,\C)$
and let $(\sigma^\vee, V^\vee)$ denote the contragredient representation of  the finite-dimensional representation $(\sigma, V)$ of $M$. 
For $\lambda \in \mathfrak{a}^*$, we define two representations of $P$ by 
$(\sigma_\lambda)^\vee :=(\sigma^\vee)_{-\lambda}
= \sigma^\vee \boxtimes \C_{-\lambda}$
and
\index{A}{1sigma-lambda*@$\sigma^*_\lambda:=\sigma^\vee \boxtimes \mathbb{C}_{2\rho-\lambda}$|textbf}
$\sigma^*_\lambda:=\sigma^\vee \boxtimes \mathbb{C}_{2\rho-\lambda}$
with trivial action of $N_+$ as before,
and form a representation 
\index{A}{1pisigma*@$\pi_{(\sigma,\lambda)^*}$|textbf}
\begin{equation*}
\pi_{(\sigma,\lambda)^*}=
\operatorname{Ind}_P^G(\sigma^*_\lambda)
\end{equation*}
of $G$ on $C^\infty(X,\mathcal V^*)$ where $\mathcal V^*=G\times_P V^\vee$ 
is the dualizing bundle associated to the
representation $\sigma^*_\lambda$ of $P$. 
The integration over $X$ gives rise to a natural $G$-invariant nondegenerate bilinear form
$$
\operatorname{Ind}_P^G(\sigma_\lambda)\times
\operatorname{Ind}_P^G(\sigma_\lambda^*)
\To\C.
$$
The infinitesimal representation of $\pi_{(\sigma,\lambda)^*}$
in the $N$-picture is given by a Lie algebra homomorphism
\index{A}{dpi2sigma*@$d\pi_{(\sigma,\lambda)^*}$|textbf}
\begin{equation*}
d\pi_{(\sigma,\lambda)^*}
\colon \mathfrak g\To\mathcal D(\mathfrak n_-)\otimes\mathrm{End}(V^\vee).
\end{equation*}
Applying the algebraic Fourier transform of the Weyl algebra (see Definition \ref{def:hat}), 
\index{A}{DE@$\mathcal D(E)$, Weyl algebra}
we get a Lie algebra homomorphism
$$
\index{A}{dpi3wsigma@$\widehat{d\pi_{(\sigma,\lambda)^*}}$, algebraic Fourier transform of principal series|textbf}
\widehat{d\pi_{(\sigma,\lambda)^*}}\colon \mathfrak g\To\mathcal D(\mathfrak n_+)\otimes\mathrm{End}(V^\vee),
$$
where we have identified $\mathfrak{n}_-^\vee$ with $\mathfrak{n}_+$ by 
an $\mathrm{Ad}(G)$-invariant, nondegenerate
symmetric bilinear form on $\mathfrak{g}$.

We define a $\mathfrak g$-module 
\index{B}{generalized Verma module|textbf}
(\emph{generalized Verma module}) by
\index{A}{indpg(Vvee)@$\indpg(V^\vee)$, generalized Verma module|textbf}
$$
\indpg(V^\vee):=U(\mathfrak g)\otimes_{U(\mathfrak p)} V^\vee,
$$
where $V^\vee$ is regarded as a $\mathfrak p$-module through 
$d\sigma^\vee\otimes(-\lambda)$ with trivial $\mathfrak n_+$-action.
We let $P$ act on $V^\vee$ by $(\sigma_\lambda)^\vee$.
Then the $\mathfrak g$-module $\indpg(V^\vee)$ carries a $P$-module structure,
so that we may regard $\indpg(V^\vee)$ as a $(\mathfrak g, P)$-module.
 This observation will be useful when $G$ is a real reductive Lie group
because the parabolic subgroup $P$ may be disconnected.
We recall from \cite[(3.23)]{KP1} that the algebraic Fourier transform
of the generalized Verma module is a $(\mathfrak{g}, P)$-isomorphism
\begin{equation*}
F_c\colon  \mathrm{ind}_\mathfrak{p}^\mathfrak{g} (V^\vee) 
\stackrel{\sim}{\To} \mathrm{Pol}(\mathfrak{n}_+) \otimes V^\vee,
\end{equation*}
where $\mathrm{Pol}(\mathfrak{n}_+)\otimes V^\vee$ is regarded as a 
$(\mathfrak{g}, P)$-module via $\widehat{d\pi_{(\sigma, \lambda)^*}}$.

Let $G'$ be a real reductive subgroup of $G$, and $P'$ a parabolic subgroup of $G'$.
Given a finite-dimensional representation $W$ of $P'$, we define two homogeneous
vector bundles:
\begin{eqnarray*}
&&\mathcal W_Y:=G'\times_{P'}W\To Y:= G'/P',\\
&&\mathcal W_Z:=G\times_{P'} W\To Z:=G/P'.
\end{eqnarray*}

Similarly to the representation $(\sigma_\lambda,V)$ of 
$P=MAN_+$, we shall consider a representation 
\index{A}{1tau-nu@$\tau_\nu \equiv \tau \boxtimes \C_\nu$|textbf}
$(\tau_\nu, W)$ of 
$P' = M'A'N_+'$ which extends the outer tensor product representation
$\tau \boxtimes \C_\nu$ for $\nu \in (\mathfrak{a}')^*$ by letting $N_+'$ act 
trivially.

We note that the base space $Z$ is not compact in general, 
whereas $Y$ is a real flag variety and thus compact. 
If $P'\subset P$, then there are natural maps:
$$
Y\To Z\relbar\joinrel\twoheadrightarrow X.
$$
We denote by $\mathrm{Diff}_{G'} (\mathcal V_X,\mathcal W_Y)$ the space 
of $G'$-equivariant 
operators from $C^\infty(X,\mathcal{V}_X)$ to $C^\infty(Y,\mathcal{W}_Y)$
which are differential operators
with respect to the above $G'$-equivariant map $Y\to X$ in the sense of Definition \ref{def:diff}.
The space $\mathrm{Diff}_{G} (\mathcal V_X,\mathcal W_Z)$ is defined in a similar way.

The map taking symbols of differential operators
on $\R^n$, to be denoted by $\mathrm{Symb}$, induces an isomorphism below
when restricted to differential operators with constant coefficients, 
\index{A}{Symb|textbf}
\begin{equation}\label{eqn:symb}
\mathrm{Symb}:
\index{A}{Diffconst@$\mathrm{Diff}^{\mathrm{const}}$}
\mathrm{Diff}^{\mathrm{const}}
\left(C^\infty(\R^n) \otimes V, C^\infty(\R^n)\otimes W\right)
\stackrel{\sim}{\To} 
\mathrm{Pol}[\zeta_1, \ldots, \zeta_n] \otimes \mathrm{Hom}_{\C}\left(V, W\right)
\end{equation}
such that 
\begin{equation*}
e^{-\langle z, \zeta \rangle}D\left(e^{\langle z, \zeta \rangle} \otimes v \right)
=   \mathrm{Symb}(D)(v)
\in \mathrm{Pol}[\zeta_1,\cdots,\zeta_n] \otimes W
\end{equation*}
for all $v \in V$.
We summarize the \index{B}{F-method|textbf}F-method in this setting from 
\cite[Thm.\ 2.9, Rem.\ 2.18, Thm.\ 4.1, Cor.\ 4.3]{KP1}:

\begin{fact}\label{fact:Fmethod}
Let $G\supset G'$ be a pair of real reductive Lie groups, 
and 
\index{A}{P1P@$P$, parabolic subgroup of $O(n+1, 1)$}
$P\supset P'$ 
\index{A}{P1P'@$P'$, parabolic subgroup of $O(n,1)$}
a pair of parabolic subgroups with compatible
Levi decompositions 
\index{A}{L1L@$L=MA$, Levi part of $P$}
$P=LN_+\supset P'=L'N'_+$ 
\index{A}{L1L'@$L'=M'A$, Levi part of $P'$}
such that $L\supset L'$ and $N_+\supset N'_+$.
Let  $(\sigma_\lambda,V)$ and $(\tau_\nu,W)$ be finite-dimensional representations 
of $P$ and $P'$ with trivial actions of $N_+$ and $N_+'$, respectively.
\begin{enumerate}
\item \emph{(duality)} There is a natural isomorphism:
\begin{equation*}
D_{X\to Y}\colon \mathrm{Hom}_{\mathfrak g',P'}(\indpgprime(W^\vee),\indpg(V^\vee))\stackrel{\sim}{\to}\mathrm{Diff}_{G'}(\mathcal V_X,\mathcal W_Y).
\end{equation*}
\item \emph{(extension)} The restriction ${\mathcal{W}_Z}\vert_Y\simeq \mathcal W_Y$
induces the bijection
\begin{equation*}
\mathrm{Rest}_{Y}\colon \mathrm{Diff}_{G} (\mathcal V_X,\mathcal W_Z)\stackrel{\sim}{\To}
\mathrm{Diff}_{G'} (\mathcal V_X,\mathcal W_Y).
\end{equation*}
\item \emph{(F-method)}
\index{A}{N1+@$\mathfrak{n}_+$, complex nilpotent Lie algebra}
For $\psi\in\left(\mathrm{Pol}(\mathfrak n_+)\otimes V^\vee\right)\otimes W\simeq
\mathrm{Hom}_{\C}(V,W\otimes\mathrm{Pol}(\mathfrak n_+))$, we consider
a system of partial differential equations \index{B}{F-system|textbf}\emph{(F-system)}
\index{A}{dpi3wsigma@$\widehat{d\pi_{(\sigma,\lambda)^*}}$, algebraic Fourier transform of principal series}
\index{A}{N1+2(R)@$\mathfrak{n}_+'(\R)$, Lie algebra of $N_+'$}
\begin{equation}\label{eqn:Fmethod2}
(\widehat{d\pi_{(\sigma,\lambda)^*}}(C)\otimes\operatorname{id}_W) \psi
=0\quad\mathrm{for\,\, all}\, C\in\mathfrak n_+',
\end{equation}
and set 
 \begin{equation}\label{eqn:24} 
 \index{A}{Sol1@$Sol(\mathfrak n_+;\sigma_\lambda,\tau_\nu)$|textbf}
Sol(\mathfrak n_+;\sigma_\lambda,\tau_\nu):=
\left\{\psi\in\mathrm{Hom}_{L'}(V,W\otimes\mathrm{Pol}(\mathfrak n_+)):
\psi\;\mathrm{solves}\;\eqref{eqn:Fmethod2}\right\}.
 \end{equation}
Then there
 is a natural isomorphism
\index{A}{1sigma-lambda@$\sigma_\lambda:=
\sigma \boxtimes \C_\lambda$}
\begin{eqnarray}\label{eqn:thm41}
\mathrm{Hom}_{\mathfrak g',P'}(\indpgprime(W^\vee),\indpg(V^\vee))\stackrel{\sim}{\To}
{Sol}(\mathfrak n_+;\sigma_\lambda, \tau_\nu).
 \end{eqnarray}

 \item
Assume that the nilradical\/ $\mathfrak n_+$ is abelian.
Then, the system \eqref{eqn:Fmethod2} is of second order, and the following diagram of six isomorphisms commutes:
$$
\xymatrix@R-=1pc@C+=1.3cm{
{}
& {Sol}(\mathfrak n_+;\sigma_\lambda,\tau_\nu)
\ar[rddd]^{\;\; \mathrm{Rest}_{Y}\; \circ \; \mathrm{Symb}^{-1}}
&{} \\
&&&\\
{}
&  \operatorname{Diff}_{G}(\mathcal V_X, \mathcal W_Z)
\ar[rd]_{\!\!\mathrm{Rest}_{Y}}
\ar[uu]  
&{} \\
\operatorname{Hom}_{\mathfrak g',P'}(\indpgprime(W^\vee),\indpg(V^\vee)) 
 \ar[ru]_{\quad \; D_{X\to Z}} \ar[rr]_{D_{X\to Y}}\ar[ruuu]^{F_c\otimes\mathrm{id}} 
 & {}
 & 
 \operatorname{Diff}_{G'}(\mathcal V_X, \mathcal W_Y)
}
$$  \end{enumerate}
\end{fact}

Fact \ref{fact:Fmethod} (3)
implies that, once we find such a polynomial solution $\psi$ to the F-system, 
we obtain  a $P'$-submodule $W^\vee$ in $\indpg(V^\vee)$ (sometimes referred to as 
\index{B}{singular vector}
\emph{singular vectors}) by
$(F_c\otimes\mathrm{id})^{-1}(\psi)$,
where we have used the canonical isomorphism
$\mathrm{Hom}_{P'}(W^\vee, \indpg(V^\vee))
\simeq \mathrm{Hom}_{\mathfrak{g}', P'}(\indpgprime(W^\vee), \indpg(V^\vee))$
when we apply the algebraic Fourier transform $F_c$ of a generalized Verma module.
Simultaneously, we obtain a differential symmetry breaking operator by
$\mathrm{Rest}_{Y}\circ
\index{A}{Symb}
\mathrm{Symb}^{-1}(\psi)$ in the 
\index{B}{Npicture@$N$-picture}
flat picture ($N$-picture), 
when $\mathfrak{n}_+$ is abelian.

The following useful lemma guarantees that the F-system \eqref{eqn:Fmethod2}
can be verified by a single nonzero element $C \in \mathfrak{n}_+'$ when
$L'$ acts irreducibly on $\mathfrak{n}_+'$, equivalently, when $\mathfrak{n}'_+$ 
is abelian.

\begin{lem}\label{lem:C0enough}
Suppose $\mathfrak n_+'$ is abelian. Then the following two conditions
on $\psi\in \mathrm{Hom}_{L'}(V,\mathrm{Pol}(\mathfrak n_+)\otimes W)$  are equivalent.
\begin{enumerate}
\item[(i)] For every $C\in \mathfrak n_+'$, $\left(\widehat{d\pi_{(\sigma,\lambda)^*}}(C)\otimes\mathrm{id}_W \right)\psi=0$.
\item[(ii)] For some nonzero $C_0\in \mathfrak n_+'$,
$\left(\widehat{d\pi_{(\sigma,\lambda)^*}}(C_0)\otimes\mathrm{id}_W\right)\psi=0$.
\end{enumerate}
\end{lem}
\begin{proof}
The implication (i)$\Rightarrow$(ii) is obvious.
We shall prove (ii)$\Rightarrow$(i). We set
$$
\mu:=\sigma_\lambda^*.
$$
Suppose $\psi\in \mathrm{Hom}_{L'}(V,\mathrm{Pol}(\mathfrak n_+)\otimes W)\simeq
\left(V^\vee\otimes\mathrm{Pol}(\mathfrak n_+)\otimes W\right)^{L'}$. This means that
$$
\chi_{2\rho}(\ell)\mu(\ell^{-1})\otimes\mathrm{Ad}_{\#}(\ell^{-1})\otimes\tau_\nu(\ell^{-1})\psi=\psi\qquad\mathrm{for\,\, all}\quad\ell\in L'.
$$
If $\psi$ satisfies (ii), then we have
\begin{equation}\label{eqn:ens20151222}
\left(\widehat{d\pi_\mu}(C_0)\otimes\mathrm{id}_W\right)\left(\mu(\ell^{-1})\mathrm{Ad}_{\#}(\ell^{-1})\right)\psi=0.
\end{equation}
We let the group $L$ act on $V^\vee\otimes\mathrm{Pol}(\mathfrak n_-)$ by $\pi_\mu(\ell)=\mu(\ell)\mathrm{Ad}_{\#}(\ell)$. Then we have
$$
d\pi_\mu(\mathrm{Ad}(\ell)C_0)=\pi_\mu(\ell)d\pi_\mu(C_0)\pi_\mu(\ell^{-1})
\qquad\mathrm{for\,\, all}\quad\ell\in L'(\subset L).
$$
Applying \eqref{eqn:AT} to the case where $E=\mathfrak n_-$, $A=\mathrm{Ad}(\ell)$, and
$T= \widehat{d\pi_\mu}(C_0)$, we have
$$
\widehat{d\pi_\mu}(\mathrm{Ad}(\ell)C_0)=\mu(\ell)\mathrm{Ad}(\ell)_\#
\widehat{d\pi_\mu}(C_0)\mathrm{Ad}(\ell^{-1})_\#\mu(\ell^{-1}),
$$
where we identify the action $\trans \mathrm{Ad}(\ell)^{-1}$ on 
$\mathfrak n_-^\vee$ with the one of $\mathrm{Ad}(\ell)$ on $\mathfrak n_+$ .
Then
$$
\left(\widehat{d\pi_\mu}(\mathrm{Ad}(\ell)C_0)\otimes\mathrm{id}_W\right)\psi=0,
$$
by \eqref{eqn:ens20151222}. 
Since $\mathfrak{n}_+'$ is abelian, the Levi subgroup $L'$ acts irreducibly on 
the nilradical 
$\mathfrak n_+'$ of the parabolic subalgebra $\mathfrak{p}_+' = \mathfrak{l}' 
+\mathfrak{n}_+'$, and therefore $\mathrm{Ad}(\ell)C_0$ $(\ell \in L')$ spans 
$\mathfrak{n}_+'$. Hence (ii)$\Rightarrow$(i) is proved.
\end{proof}

%%%%%%%%%%%%%%%%%%%%%%%%%%%%%%%%%%%%%%%%%%%%%%%
\subsection{Matrix-valued differential operators in the F-method}\label{subsec:3general}

This section provides a structural result on the key operator $\widehat{d\pi_{(\sigma,\lambda)^*}}$
in the F-method for 
the principal series representation
$\mathrm{Ind}_P^G(\sigma_\lambda)$ when $P$ is a parabolic
subgroup with abelian unipotent radical.
We shall prove that $\widehat{d\pi_{(\sigma,\lambda)^*}}$ has a canonical decomposition into a sum
of the ``scalar part" (differential operator of second order) 
depending only on the continuous parameter $\lambda\in\mathfrak a^*$
and the ``vector part" (differential operator of
first order) depending only on $\sigma\in\widehat M$.

We retain the notation in Section \ref{subsec:2Fmethod}, and simply write
\index{A}{dpi1l@$\widehat{d\pi_{\lambda^*}}$|textbf}
\begin{equation*}
\widehat{d\pi_{\lambda^*}}\colon \mathfrak g\To\mathcal D(\mathfrak n_+),
\end{equation*}
for $\widehat{d\pi_{(\sigma,\lambda)^*}}$ when $(\sigma,V)$ is the trivial one-dimensional representation. We define the
\index{B}{vector part|textbf}
``vector part"
of $\widehat{d\pi_{(\sigma,\lambda)^*}}$ as a linear map
\index{A}{Asigma@$A_\sigma$, vector part of $\widehat{d\pi_{(\sigma,\lambda)^*}}$|textbf}
$A_\sigma\colon \mathfrak g\To\mathcal D(\mathfrak n_+)\otimes\mathrm{End}(V^\vee)$ 
characterized by the formula
\begin{equation}\label{eqn:Fsv}
\Fdpi{\sigma}{\lambda}(Y)=\widehat{d\pi_{\lambda^*}}(Y)\otimes \mathrm{id}_{V^\vee}+A_\sigma(Y)\quad \mathrm{for}\quad Y\in\mathfrak g.
\end{equation}
Let $\{N_\ell^-\}$ be a basis of $\mathfrak n_-(\R)$, and $(\zeta_1,\cdots,\zeta_n)$
be the corresponding coordinates on $\mathfrak n_-^\vee(\R)\simeq \mathfrak n_+ (\R)$.

\begin{prop}\label{prop:vecpart} 
Assume $\mathfrak n_+$ is abelian. Then, for any $Y\in\mathfrak n_+$,
$A_\sigma(Y)$ is a holomorphic vector field on $\mathfrak n_+$ with constant coefficients in $\mathrm{End}(V^\vee)$. An explicit formula is given as follows.
\begin{equation}\label{eqn:Asigma}
A_\sigma(Y)F=-\sum_{\ell=1}^n \frac{\partial}{\partial\zeta_\ell} F\circ d\sigma\left([Y,N_\ell^-]\Big\vert_{\mathfrak m}\right)
\quad\mathrm{for}\quad F\in \operatorname{Pol}(\mathfrak n_+)\otimes V^\vee.
\end{equation}
In particular, the ``vector part" $A_\sigma$ 
is independent of the continuous parameter $\lambda$. 
Moreover, $A_\sigma$ is zero if $d\sigma = 0$.
\end{prop}

\begin{proof}
Let $G_\C$ be a connected complex Lie group with Lie algebra $\mathfrak g=\mathfrak g(\R)\otimes_\R\C$, and $P_\C=L_\C\exp\mathfrak n_+$ the parabolic subgroup with Lie algebra $\mathfrak p=\mathfrak l+\mathfrak n_+$.
According to the Gelfand--Naimark decomposition $\mathfrak g=\mathfrak n_-+\mathfrak l+\mathfrak n_+$
of the Lie algebra $\mathfrak{g}$, 
we have a diffeomorphism
$$
\mathfrak n_-\times L_\C\times \mathfrak n_+\to G_\C,\quad
(Z,\ell, Y)\mapsto (\exp Z)\ell(\exp Y),
$$
into an open dense subset $G^{\mathrm{reg}}_\C$ of $G_\C$. 
Let
$$
p_\pm\colon  
G^{\mathrm{reg}}_\C\To\mathfrak n_\pm,\qquad p_o\colon G^{\mathrm{reg}}_\C\to L_\C,
$$
 be the projections characterized by the identity
$$
\exp(p_-(g))p_o(g)\exp(p_+(g))=g.
$$

Then the following maps $\alpha$ and $\beta$ are determined by the Gelfand--Naimark 
decomposition
$\mathfrak g=\mathfrak n_-+\mathfrak l+\mathfrak n_+$ and independent of the choice of 
the complex Lie group $G_\C$:
\begin{align}\label{eqn:alpha}
(\alpha,\beta)\colon \mathfrak g\times \mathfrak n_-&\to\mathfrak l\oplus\mathfrak n_-,\quad
 (Y,Z)\mapsto \left.\frac {d}{dt}\right\vert_{t=0}(p_o\left(e^{tY}e^Z\right),p_-\left(e^{tY}e^Z\right)).
\end{align}
According to the direct sum decomposition $\mathfrak l=\mathfrak m+\mathfrak a$, we write
$$
\alpha(Y,Z)=\alpha(Y,Z)\vert_{\mathfrak m}+\alpha(Y,Z)\vert_{\mathfrak a}.
$$
For a fixed element $Y\in\mathfrak g$, $\beta(Y,\cdot)$ induces a complex linear map 
$\mathfrak n_-\to\mathfrak n_-$, and thus we may regard
$\beta(Y,\cdot)$ as a holomorphic  vector field on $\mathfrak n_-$ via the identification of 
$\mathfrak n_-$ with the holomorphic tangent space
at each point:
$$
\mathfrak n_-\ni Z\mapsto \beta(Y,Z)\in\mathfrak n_-\simeq T_Z\mathfrak n_-.
$$

Suppose $f\in C^\infty(\mathfrak n_-(\R),V^\vee)$, 
$Y\in\mathfrak g(\R)$ and $Z\in\mathfrak n_-(\R)$.
Since $N_+$ acts trivially on $V$, the infinitesimal representation 
$d\pi_{(\sigma,\lambda)^*}$ is given by
\begin{eqnarray*}
d\pi_{(\sigma,\lambda)^*}(Y)f(Z)&=&d\sigma^\vee(\alpha(Y,Z)\vert_{\mathfrak m})f(Z)+
(d\lambda^*(\alpha(Y,Z)\vert_{\mathfrak a})f(Z)-\beta(Y,\cdot)f(Z)).
\end{eqnarray*}

In view of the decomposition,
we define $d\pi^{\mathrm{vect}}_{(\sigma, \lambda)^*}, 
d\pi^{\mathrm{scalar}}_{(\sigma, \lambda)^*} \in 
\mathrm{Hom}_{\C}(\mathfrak{g}, \mathcal{D}(\mathfrak{n}_-) 
\otimes \mathrm{End}(V^\vee))$ by
\begin{align*}
\index{A}{dpi5sigmavect@$d\pi^{\mathrm{vect}}_{(\sigma,\lambda)^*}$|textbf}
d\pi^{\mathrm{vect}}_{(\sigma,\lambda)^*}(Y)
&:=d\sigma^\vee\left(\alpha(Y,Z)|_{\mathfrak{m}}\right),\\
\index{A}{dpi4sigmascalar@$d\pi^{\mathrm{scalar}}_{(\sigma,\lambda)^*}$|textbf}
d\pi^{\mathrm{scalar}}_{(\sigma,\lambda)^*}(Y)
&:=d\lambda^*\left(\alpha(Y,Z)|_{\mathfrak{a}}\right)-\beta(Y,\cdot).
\end{align*}
Clearly, 
\index{A}{dpi2sigma*@$d\pi_{(\sigma,\lambda)^*}$}
$
d\pi_{(\sigma,\lambda)^*} = 
d\pi^{\mathrm{scalar}}_{(\sigma,\lambda)^*}
+
d\pi^{\mathrm{vect}}_{(\sigma,\lambda)^*}$.
We say $d\pi^{\mathrm{scalar}}_{(\sigma,\lambda)^*}$ is the 
\index{B}{scalar-part@scalar part of $d\pi_{(\sigma, \lambda)^*}$}
\emph{scalar part} 
of $d\pi_{(\sigma, \lambda)^*}$, and
\index{B}{vector-part@vector part of $d\pi_{(\sigma, \lambda)^*}$}
$d\pi^{\mathrm{vect}}_{(\sigma,\lambda)^*}$ is the
\emph{vector part}.
Then the scalar part $d\pi^{\mathrm{scalar}}_{(\sigma,\lambda)^*}$
does not depend on $(\sigma,V)$, 
and takes the form
\begin{equation*}
d\pi^{\mathrm{scalar}}_{(\sigma,\lambda)^*}(Y)=
d\pi_{\lambda^*}(Y)\otimes\mathrm{id}_{V^\vee}
\quad 
\text{for all $Y \in \mathfrak{g}$}. 
\end{equation*}
Let us compute their algebraic Fourier transforms. Obviously,
the algebraic Fourier transform of $d\pi^{\mathrm{scalar}}_{(\sigma,\lambda)^*}(Y)$ 
is $\widehat{d\pi_{\lambda^*}}(Y)\otimes\mathrm{id}_{V^\vee}$.

If $\mathfrak n_+$ is abelian, we have
$$
\alpha(Y,Z)=[Y,Z],
\quad
\beta(Y,Z)=\frac12[Z,[Z,Y]]
\qquad\mathrm{for}\, Y\in\mathfrak n_+\,\mathrm{and}\,Z\in\mathfrak n_-,
$$
see \cite[Lem.\ 3.8]{KP1}.

We write $Z=\sum_\ell z_\ell N_\ell^-$. Then for $Y\in\mathfrak n_+$, we have
\begin{eqnarray*}
\left(d\pi_{(\sigma,\lambda)^*}^{\mathrm{vect}}(Y)f\right)(Z)
&=&-f(Z)\circ d\sigma([Y,Z]\vert_{\mathfrak{m}})\\
&=&-\sum_\ell z_\ell f\circ d\sigma([Y,N_\ell^-]\vert_{\mathfrak{m}}).
\end{eqnarray*}

By \eqref{eqn:FTdiff} its algebraic Fourier transform is given by \eqref{eqn:Asigma}.
The remaining assertions of Proposition \ref{prop:vecpart} are clear.
\end{proof}

\newpage
%%%%%%%%%%%%%%%%%%%%%%%%%%%%%%%%%%%%%%%%%%%%%%%
\section{Matrix-valued F-method for $O(n+1,1)$}
\label{sec:FON}

\index{B}{matrix-valued F-method|textbf}
This chapter summarizes a strategy and technical details in applying the F-mehod
to find matrix-valued symmetry breaking operators in the setting 
where $(G,G') = (O(n+1,1), O(n,1))$.

%%%%%%%%%%%%%%%%%%%%%%%%%%%%%%%%%%%%%%%%%%%%%%%
\subsection{Strategy of matrix-valued F-method for 
$(G,G') = (O(n+1,1), O(n,1))$}\label{subsec:VWstrategy}

We retain the notation of Chapter \ref{sec:ps}. In particular,
$P=L\exp(\mathfrak{n}_+(\R))$ and $P'=L'\exp(\mathfrak{n}_+'(\R))$
are the minimal parabolic subgroups of $G=O(n+1,1)$ and 
$G'=O(n,1)$, respectively, such that $L \supset L'$ and 
\index{A}{N1+2(R)@$\mathfrak{n}_+'(\R)$, Lie algebra of $N_+'$}
$\mathfrak{n}_+(\R) \supset \mathfrak{n}_+'(\R)$. 
We recall 
$L=MA \simeq O(n) \times O(1) \times \R$ and 
$\mathfrak{n}_{\pm}(\R)$ 
is identified with $\R^n$ via the basis
$\{N_1^{\pm},\ldots, N_n^{\pm}\}$, see \eqref{eqn:Npm1}.
Let $(\zeta_1, \ldots, \zeta_n)$ be the coordinates of 
$\mathfrak{n}_+ (\simeq \mathfrak{n}_-^\vee)$.
Then the $L$-module $\mathrm{Pol}(\mathfrak{n}_+)$ is identified with 
the polynomial ring $\mathrm{Pol}[\zeta_1, \ldots, \zeta_n]$ 
on which the action of 
\index{A}{H1zero@$H_0$, generator of $\mathfrak{a}$}
$L = MA \ni \left( (B,b), e^{tH_0}\right)$ is given by
\begin{equation}\label{eqn:LactPol}
f(\zeta) \mapsto f(b^{-1}e^{-t}B^{-1}\zeta)  \quad
\text{for $\zeta = \;^t(\zeta_1, \ldots, \zeta_n)$}.
\end{equation}

\index{A}{M'@$M'$ ($=O(n-1)\times O(1)$)}
The subgroup $L' = M'A\simeq O(n-1) \times O(1) \times \R$ 
stabilizes the last variable $\zeta_n$,
and acts irreducibly on $\mathfrak{n}_+' \simeq \C^{n-1}$. 
Then we may apply Lemma \ref{lem:C0enough}
by choosing $C_0 = N_1^+$. With this notation,
the $F$-method (Fact \ref{fact:Fmethod}) implies the following:

\begin{prop}\label{prop:Fmethod2}
Let $(G,G') = (O(n+1, 1), O(n,1))$, 
\index{A}{1sigma-lambda@$\sigma_\lambda:=
\sigma \boxtimes \C_\lambda$}
$\sigma_\lambda = \sigma \boxtimes \C_{\lambda}$ 
be a finite-dimensional representation of $P$ on $V$
that factors the quotient group
$P/N_{+} \simeq L=MA$,
and 
\index{A}{1tau-nu@$\tau_\nu \equiv \tau \boxtimes \C_\nu$}
$\tau_\nu =\tau \boxtimes \C_{\nu}$
be that of $P'$ that factors $P'/N_+'\simeq L'=M'A$ on $W$. 
The flat pictures of the principal series representations 
$\mathrm{Ind}_P^G(\sigma_\lambda)$ of $G$ and 
$\mathrm{Ind}_{P'}^{G'}(\tau_\nu)$ of $G'$ are defined 
in $C^\infty(\R^n) \otimes V$ and $C^\infty(\R^{n-1}) \otimes W$,
respectively, as in \eqref{eqn:Npic}.
Then we have the following.

\begin{enumerate}
\item 
 \index{A}{Sol1@$Sol(\mathfrak n_+;\sigma_\lambda,\tau_\nu)$}
$Sol(\mathfrak{n}_+; \sigma_{\lambda}, \tau_{\nu})$
(see \eqref{eqn:24}) is given by
\begin{align}\label{eqn:hatN1}
&
Sol(\mathfrak n_+;\sigma_\lambda,\tau_\nu) \nonumber \\
&=\left \{\psi \in \mathrm{Hom}_{L'}(V, W \otimes \mathrm{Pol}[\zeta_1, \ldots, \zeta_n])
:\left(\widehat{d\pi_{(\sigma, \lambda)^*}}(N_1^+) \otimes \mathrm{id}_W\right)
\psi = 0 \right\}.
\end{align}

\item Suppose $\psi \in Sol(\mathfrak{n}_+; \sigma_\lambda, \tau_\nu)$.
Let $D$ be the $\mathrm{Hom}_{\C}(V,W)$-valued differential operator
on $\mathbb{R}^n$ with constant coefficients such that 
\index{A}{Symb}
$\mathrm{Symb}(D) = \psi$.
Then the differential operator
\begin{equation*}
\mathrm{Rest}_{x_n=0} \circ D\colon C^\infty(\mathbb{R}^n)\otimes V 
\to C^\infty(\mathbb{R}^{n-1})\otimes W
\end{equation*}
extends to a symmetry breaking operator from 
$\mathrm{Ind}^G_P(\sigma_\lambda)$ 
to 
$\mathrm{Ind}^{G'}_{P'}(\tau_\nu)$.

\item Conversely, any $G'$-equivariant differential operator 
from $\mathrm{Ind}^G_P(\sigma_\lambda)$
to $\mathrm{Ind}^{G'}_{P'}(\tau_\nu)$
is obtained in this manner.
\end{enumerate}
\end{prop}

The rest of this chapter is devoted to an explicit characterization of the main ingredients of Proposition
\ref{prop:Fmethod2}. Namely, the space 
$\mathrm{Hom}_{L'}(V,W\otimes \mathrm{Pol}[\zeta_1, \ldots, \zeta_n])$ is described
in Section \ref{subsec:VWPol}, the scalar and vector parts of the operator
$\widehat{d\pi_{(\sigma, \lambda)^*}}(N_1^+)$ are given in Section \ref{subsec:3ortho} and the
matrix components of \eqref{eqn:hatN1} are studied in Section \ref{subsec:MIJF}.

Harmonic polynomials play a key role in the first two steps of this characterization.

%%%%%%%%%%%%%%%%%%%%%%%%%%%%%%%%%%%%%%%%%%%%%%%
\subsection{Harmonic polynomials} \label{subsec:sph}
${ }$

We review a classical fact on 
\index{B}{harmonic polynomials|textbf}
harmonic polynomials.
Let $N \in \N_+$ (later, we take $N$ to be $n-1$ or $n$).
For $k\in\N$, we denote by $\mathrm{Pol}^k[\zeta_1,\cdots,\zeta_N]$ the space of homogeneous polynomials of degree $k$. 
The space $\mathcal H^k(\C^N)$ of harmonic polynomials of degree $k$ is defined by
\index{A}{HkCN@$\mathcal H^k(\C^N)$, harmonic polynomials|textbf}
\begin{equation*}
\mathcal H^k(\C^N)
:=\left\{h\in\mathrm{Pol}^k[\zeta_1,\cdots,\zeta_N] : 
\Delta_{\C^N}h=0\right\},
\end{equation*}
where 
\index{A}{11D3DeltaCn@$\Delta_{\C^n}$, holomorphic Laplacian|textbf}
$\Delta_{\C^N}:=
\frac{\partial^2}{\partial \zeta_1^2} + \cdots + \frac{\partial^2}{\partial \zeta_N^2}$
denotes the holomorphic Laplacian on $\C^N$.
Then $\mathcal H^k(\C^N)\neq \{0\}$
for all $k \in \N$ if $N\geq 2$ and $\mathcal{H}^k(\C^1) \neq \{0\}$ for $k\in \{0,1\}$.

The orthogonal group $O(N)$ acts irreducibly on 
$\mathcal{H}^k(\C^N)$ for all $k \in \N$ unless it is zero.
We set
\begin{equation*}
\mathcal H(\C^N):=\bigoplus_{k=0}^\infty \mathcal H^k(\C^N).
\end{equation*}
Then we have a natural 
decomposition of the space of polynomials into spherical harmonics and
$O(N)$-invariant polynomials for any $N \in \N_+$:
\begin{equation}\label{eqn:HPP}
\mathrm{Pol}[\zeta_1^2+\cdots+\zeta_N^2]\otimes\mathcal H(\C^N)\stackrel{\sim}{\to}
\mathrm{Pol}[\zeta_1,\cdots,\zeta_N].
\end{equation}

%%%%%%%%%%%%%%%%%%%%%%%%%%%%%%%%%%%%%%%%%%%%%%%
\subsection{Description of $\mathrm{Hom}_{L'}(V,W\otimes\mathrm{Pol}(\mathfrak n_+))$}
\label{subsec:VWPol}
As the first step of the matrix-valued F-method for the Lorentz group $G=O(n+1,1)$, 
we give a description of the space 
$\mathrm{Hom}_{L'}(V,W\otimes\mathrm{Pol}(\mathfrak n_+))$ 
by using harmonic polynomials.

We retain the notation of Section \ref{subsec:VWstrategy}, in particular,
$(\zeta_1,\cdots,\zeta_n)$ are the coordinates of $\mathfrak n_+$ 
such that $\mathfrak n_+'$ is characterized by $\zeta_n=0$.
For $b\in \Z$ and a polynomial $g(t)$ of one variable $t$, we define a multi-valued meromorphic function of $n$ variables $\zeta=(\zeta_1,\cdots,\zeta_n)$ by
\index{A}{Ta@$T_a$|textbf}
\begin{equation}\label{eqn:Ta}
(T_bg)(\zeta):=Q_{n-1}(\zeta')^{\frac b2} g\left(\frac{\zeta_n}{\sqrt{Q_{n-1}(\zeta')}}\right),
\end{equation}
where $\zeta' = (\zeta_1, \ldots, \zeta_{n-1})$ and 
\index{A}{Qn11@$Q_{n-1}(\zeta')$}
$Q_{n-1}(\zeta')=\zeta_1^2+\cdots+\zeta_{n-1}^2$.
Clearly, $(T_bg)(\zeta)\equiv0$ if and only if $g(t)\equiv 0$.
We observe that $(T_bg)(\zeta)$ is a homogeneous polynomial
of $(\zeta_1, \ldots, \zeta_n)$ of degree $b$ if $b \in \N$ and 
$g\in \operatorname{Pol}_b[t]_{\mathrm{even}}$, where we
 set
\index{A}{Polell[t]@$\mathrm{Pol}_\ell[t]_{\mathrm{even}}$|textbf}
\begin{eqnarray}
\label{eqn:gs}
\mathrm{Pol}_b[t]_{\mathrm{even}}&:=&\C\operatorname{-span}\left\langle t^{b-2j}:0\leq j\leq\left[\frac b2\right]
\right\rangle.
\end{eqnarray}

\noindent
Then we have the following bijection
\begin{equation}\label{eqn:Tabij}
T_b\colon \mathrm{Pol}_b[t]_{\mathrm{even}}\overset{\sim}{\to}\bigoplus_{2\ell+c=b}\mathrm{Pol}^\ell[\zeta_1^2+\dots+\zeta^2_{n-1}]\otimes\mathrm{Pol}^c[\zeta_n].
\end{equation}

\begin{lem}\label{lem:psi}
Suppose $n\geq 2$. 
Then for every $a\in\N$, there is a natural bijection:
\begin{eqnarray*}
\bigoplus_{k=0}^a\operatorname{Pol}_{a-k}[t]_{\mathrm{even}}
\otimes
\mathrm{Hom}_{O(n-1)}(V,W\otimes\mathcal H^k(\C^{n-1}))
\overset{\sim}{\to}\mathrm{Hom}_{O(n-1)}(V,W\otimes\mathrm{Pol}^a(\mathfrak n_+))
\end{eqnarray*}
induced by
\begin{equation*}
\sum_{k=0}^ag_k\otimes H^{(k)}\mapsto
\sum_{k=0}^a\left(T_{a-k}g_k\right)H^{(k)}.
\end{equation*}
\end{lem}

\begin{proof}
Combining the following two $O(n-1)$-isomorphisms
\begin{eqnarray*}
\mathrm{Pol}^a(\mathfrak{n}_+)\simeq
\bigoplus_{b+c=a}\mathrm{Pol}^b[\zeta_1, \ldots, \zeta_{n-1}]
\otimes\mathrm{Pol}^c[\zeta_n],
\end{eqnarray*}
and \eqref{eqn:HPP} with $N=n-1$, namely,
 \begin{equation*}
 \mathrm{Pol}[\zeta_1,\ldots, \zeta_{n-1}]\simeq\bigoplus_{k+2\ell=b}\mathcal H^k(\C^{n-1})\otimes \mathrm{Pol}^\ell[\zeta_1^2+\cdots+\zeta_{n-1}^2],
\end{equation*}
we have
 \begin{eqnarray*}
&&\mathrm{Hom}_{O(n-1)}(V,W\otimes\mathrm{Pol}^a(\mathfrak{n}_+))
\\&\simeq&
 \bigoplus_{k+2\ell+c=a}\mathrm{Hom}_{O(n-1)}\left(V,W\otimes
 \mathcal H^k(\C^{n-1})\right)\otimes\mathrm{Pol}^\ell[\zeta_1^2+\cdots+\zeta_{n-1}^2]\otimes\mathrm{Pol}^c[\zeta_n].
\end{eqnarray*}
Then the statement follows from the bijection \eqref{eqn:Tabij}.
\end{proof}

By the F-method (Proposition \ref{prop:Fmethod2}) combined with results on 
finite-dimensional representations,
we obtain a necessary condition for the existence of nonzero differential 
symmetry breaking operators in the general setting:

\begin{cor}\label{cor:153415}
Suppose $(\sigma,V)\in\widehat M, (\tau,W)\in\widehat{M'}$ and $\lambda,\nu\in\C$. 
Suppose $\sigma\vert_{O(1)}$ is a multiple of $\alpha \in \Z/2\Z \simeq \widehat{O(1)}$, 
and $\tau\vert_{O(1)}$ is a multiple of $\beta \in \Z/2\Z$, where
$O(1)$ denotes the second factor of $M\simeq O(n) \times O(1)$ 
(or $M' \simeq O(n-1)\times O(1)$).
Then
$$
\mathrm{Diff}_{G'}\left(\mathrm{Ind}_P^G(\sigma_\lambda),
\mathrm{Ind}_{P'}^{G'}(\tau_\nu)\right)\neq\{0\}
$$
only if the following three conditions hold:
\begin{eqnarray*}
&&\nu-\lambda\in\N,\\
&&\beta - \alpha \equiv \nu -\lambda \; \; \mathrm{mod} \; 2,\\
&&
\mathrm{Hom}_{O(n-1)}\left(V,W\otimes\mathcal H^k(\C^{n-1})\right)\neq\{0\}\quad
\text{for some}\quad 0\leq k\leq \nu-\lambda.
\end{eqnarray*}
In particular, if $(\sigma,\tau) \in \widehat{M} \times \widehat{M'}$ satisfies
$$
\mathrm{Hom}_{O(n-1)}\left(V,W\otimes\mathcal H(\C^{n-1})\right)=\{0\},
$$
then $\mathrm{Diff}_{G'}\left(\mathrm{Ind}_P^G(\sigma_\lambda),\mathrm{Ind}_{P'}^{G'}
(\tau_\nu)\right)=\{0\}$ for all $\lambda,\nu\in\C$.
\end{cor}

\begin{proof}
It follows from Proposition \ref{prop:Fmethod2} that 
$\mathrm{Diff}_{G'}(\mathrm{Ind}_P^G(\sigma_\lambda),
\mathrm{Ind}_{P'}^{G'}(\tau_\nu)) \neq \{0\}$ only if 
\begin{equation*}
\mathrm{Hom}_{L'}(V,W \otimes \mathrm{Pol}[\zeta_1, \ldots, \zeta_n]) \neq \{0\}.
\end{equation*}

First, we consider the action of the first factor $O(n-1)$ of 
$L' \simeq O(n-1) \times O(1) \times \R$. Then we find $a \in \N$ such that 
$\mathrm{Hom}_{O(n-1)}(V,W \otimes \mathrm{Pol}^a[\zeta_1, \ldots, \zeta_n])\neq \{0\}$,
and therefore 
$\mathrm{Hom}_{O(n-1)}(V,W \otimes \mathcal{H}^k(\C^{n-1}))\neq \{0\}$
for some $k$ $(0\leq k \leq a)$.

Second, we consider the actions of the second and third factors
of $L'$. Since 
\index{A}{H1zero@$H_0$, generator of $\mathfrak{a}$}
$e^{t H_0} \in A$ and $-1 \in O(1)$ act on  $\mathfrak{n}_+ \simeq \C^n$
as the scalars $e^t$ and $-1$, respectively,
\begin{equation*}
\mathrm{Hom}_{O(1)\times A}(V,W \otimes \mathrm{Pol}^a[\zeta_1, \ldots, \zeta_n])\neq \{0\}
\end{equation*}
if and only if 
\begin{equation*}
\nu=\lambda + a \quad \text{and} \quad \beta \equiv \alpha +a \; \mathrm{mod}\; 2.
\end{equation*}
Thus the corollary is proved.
\end{proof}

In Chapter \ref{sec:4}, 
we shall prove a necessary and sufficient condition 
that the space
$\mathrm{Hom}_{O(n-1)}(V,W\otimes\mathcal H^k(\C^{n-1}))$ 
does not vanish
when $V=\Exterior^i(\C^n)$ and $W=\Exterior^j(\C^{n-1})$, 
and find their explicit generators, see Proposition \ref{prop:153417}.

%%%%%%%%%%%%%%%%%%%%%%%%%%%%%%%%%%%%%%%%%%%%%%%
\subsection{Decomposition of the equation
$(\widehat{d\pi_{(\sigma,\lambda)^*}}(N_1^+) \otimes \mathrm{id}_W)\psi = 0$}
\label{subsec:3ortho}
${}$

In Lemma \ref{lem:psi}, we have given a 
description of $\mathrm{Hom}_{L'}(V,W\otimes \mathrm{Pol}^a(\mathfrak{n}_+))$
by using spherical harmonics.
The next step of the matrix-valued F-method in our setting is to write down
explicitly the 
\index{B}{F-system}
F-system \eqref{eqn:hatN1} according to 
the canonical decomposition \eqref{eqn:Fsv}
\index{A}{Asigma@$A_\sigma$, vector part of $\widehat{d\pi_{(\sigma,\lambda)^*}}$}
\begin{equation*}
\widehat{d\pi_{(\sigma,\lambda)^*}} \otimes \mathrm{id}_W
=\widehat{d\pi_{\lambda^*}} \otimes \mathrm{id}_{\mathrm{Hom}(V,W)}
+A_\sigma \otimes \mathrm{id}_W.
\end{equation*}
The main result (Proposition \ref{prop:scalar-vect}) of this section
asserts that the differential operator whose symbol is in \eqref{eqn:hatN1} is given by 
\begin{center}
Gegenbauer-type operators $+$ matrix-valued vector fields.
\end{center}
To be precise, we introduce the following differential operator of second order
\begin{equation}\label{eqn:Rla128}
\index{A}{Rell@$R_\ell^\lambda$, imaginary Gegenbauer differential operator|textbf}
R_\ell^\lambda:=-\frac{1}{2}\left((1+t^2)\frac{d^2}{dt^2} 
+ (1+2\lambda)t\frac{d}{dt}-\ell(\ell+2\lambda)\right)
\end{equation}
with parameters $\lambda\in\C$ and $\ell\in\N$. 
Note that $R^\lambda_\ell g(t) =0$ is the 
\index{B}{imaginary Gegenbauer differential equation|textbf}
``imaginary" Gegenbauer differential equation
(see Lemma \ref{lem:Gesol}).

\begin{prop}\label{prop:scalar-vect}
Let $G=O(n+1,1)$, $(\sigma, V) \in \widehat{M}$, $\lambda\in\C$,
 and $W$ be a vector space over $\C$.
Suppose $0 \leq k\leq a$ and 
$\psi = (T_{a-k}g_k)H^{(k)}$ with 
$g_k(t) \in \mathrm{Pol}_{a-k}[t]_{\mathrm{even}}$ and
$H^{(k)} \in \mathrm{Hom}_{\C}(V,W\otimes \mathcal{H}^k(\C^{n-1}))$. Then,
\begin{enumerate}
\item $(\widehat{d\pi_{\lambda^*}}(N_1^+)\otimes \mathrm{id}_W)\psi
=\displaystyle{\frac{\zeta_1}{Q_{n-1}(\zeta')}T_{a-k}
\left(R^{\lambda - \frac{n-1}{2}}_{a-k}g_k \right)H^{(k)}
+(\lambda+a-1)(T_{a-k}g_k)\frac{\partial H^{(k)}}{\partial \zeta_1}}$,
\smallskip

\item $(A_\sigma(N_1^+)\otimes \mathrm{id}_W)\psi 
= \displaystyle{\sum_{\ell =1 }^n\frac{\partial}{\partial \zeta_\ell}(T_{a-k}g_k)H^{(k)}\circ d\sigma(X_{\ell 1})}$.
\end{enumerate}
\end{prop}

The rest of this section is devoted to the proof of Proposition \ref{prop:scalar-vect}.
We note that the $L'$-intertwining property of the linear maps $H^{(k)}$ is not used 
in Proposition \ref{prop:scalar-vect}. 

We begin with an explicit formula of the canonical decomposition 
\eqref{eqn:Fsv} of $\widehat{d\pi_{(\sigma, \lambda)^*}}$. 
We define the 
\index{B}{Euler homogeneity operator|textbf}
Euler homogeneity operator on $\C^n$ by
\index{A}{Ezeta@$E_{\zeta}$, Euler homogeneity operator|textbf}
\begin{equation*}
E_{\zeta}: = \sum_{\ell=1}^n \zeta_\ell \frac{\partial}{\partial \zeta_\ell}.
\end{equation*}
Then we have:

\begin{lem}\label{lem:FmethodOn}
Let $G=O(n+1,1)$
and $\{N_1^+, \ldots, N_n^+\}$ be the basis of $\mathfrak{n}_+(\R)$,
see \eqref{eqn:Npm1}. Suppose
$(\sigma,V)\in\widehat{M}$ and $\lambda\in\C$. 
Then the decomposition
\eqref{eqn:Fsv} amounts to
\begin{eqnarray*}
\widehat{d\pi_{(\sigma,\lambda)^*}}
\left( N_m^+ \right)
=
\widehat{d\pi_{\lambda^*}}\left(N_m^+\right)\otimes\mathrm{id}_{V^\vee}+ A_\sigma(N_m^+)
\qquad\qquad (1\leq m\leq n),
\end{eqnarray*}
where
\begin{align}
\widehat{d\pi_{\lambda^*}}(N_m^+)&=
\lambda\frac{\partial}{\partial \zeta_m}
+E_\zeta\frac{\partial}{\partial \zeta_m}
-\frac12\zeta_m\Delta_{\C^n}, \label{eqn:dpiscalar}\\
A_\sigma(N_m^+)F&=
\sum_{\ell=1}^n\frac{\partial}{\partial \zeta_\ell}F\circ d\sigma(X_{\ell m})
\quad \text{$\mathrm{for}$ $F \in \mathrm{Pol}(\mathfrak{n}_+) \otimes V^\vee$}. 
\label{eqn:dpivector}
\end{align}
\end{lem}

\begin{proof}
The ``scalar part" is given in \cite[Lem.\ 6.5]{KP2}. 

According to Proposition \ref{prop:vecpart} and \eqref{eqn:Npm},
the vector part $A_\sigma(N_m^+)$ is given by
\begin{eqnarray*}
A_{\sigma}(N_m^+)F&=&-\sum_{\ell=1}^n \frac{\partial}{\partial \zeta_\ell}
F\circ d\sigma\left([N_m^+,N_\ell^-]\vert_{\mathfrak m}\right)\\
&=&-\sum_{\ell=1}^n \frac{\partial}{\partial \zeta_\ell}
F\circ d\sigma(X_{m\ell})\\
&=&\sum_{\ell=1}^n \frac{\partial}{\partial \zeta_\ell}
F\circ d\sigma (X_{\ell m}).
\end{eqnarray*}
\end{proof}

Thus, the second assertion on the vector part of
Proposition \ref{prop:scalar-vect} is proved.
In order to show the first assertion on the scalar part,
we give a useful formula for the action of the
second-order differential operator
 $\widehat{d\pi_{\lambda^*}}\left(N_m^+\right)$ 
on $\mathrm{Pol}[\zeta_1, \ldots, \zeta_n]$.

\begin{lem}\label{lem:Poincare}
If $f\in\mathrm{Pol}^{a-k}(\C^n)^{O(n-1,\C)}$ and $h\in\mathcal H^k(\C^{n-1})$, then
\index{A}{dpi1l@$\widehat{d\pi_{\lambda^*}}$}
\begin{equation*}
\widehat{d\pi_{\lambda^*}}(N_m^+)(f h)=
\widehat{d\pi_{\lambda^*}}(N_m^+)(f )h+
(\lambda+a-1)f\frac{\partial h}{\partial\zeta_m}\qquad\mathrm{for}\,\, 1\leq m\leq n-1.
\end{equation*}
\end{lem}
\begin{proof}
By \eqref{eqn:dpiscalar}, we have
\begin{align}\label{eqn:Poincare1}
\widehat{d\pi_{\lambda^*}}(N_m^+)(f h)=
\lambda \frac{\partial}{\partial \zeta_m} (f h) 
+E_\zeta\frac{\partial}{\partial \zeta_m}(f h)
-\frac12\zeta_m\Delta_{\C^n}(f h).
\end{align}
Observe that
$f \frac{\partial h}{\partial \zeta_m}$ is homogeneous of degree $a-1$,
and therefore $E_\zeta\left(f \frac{\partial h}{\partial \zeta_m}\right)
=(a-1)f \frac{\partial h}{\partial \zeta_m}$.
We also observe that 
$\Delta_{\C^n}(fh) = \left(\Delta_{\C^n}f\right)h + 
2\sum_{\ell =1 }^n \frac{\partial f}{\partial \zeta_\ell}\frac{\partial h}{\partial \zeta_\ell}$
because $\Delta_{\mathbb{C}^n} h =0$.
It then follows from  a direct computation that
\eqref{eqn:Poincare1} may be simplified to
\begin{align}\label{eqn:Poincare2}
\widehat{d\pi_{\lambda^*}}(N_m^+)(f h)
=\widehat{d\pi_{\lambda^*}}(N_m^+) (f)h
+(\lambda + a - 1) f \frac{\partial h}{\partial \zeta_m}
+\frac{\partial f}{\partial \zeta_m}E_\zeta(h)
-\sum_{r=1}^n\zeta_m\frac{\partial f}{\partial \zeta_r} \frac{\partial h}{\partial \zeta_r}.
\end{align}
Since the polynomial $f$ is ${O(n-1,\C)}$-invariant,
 it is annihilated by the generators 
\index{A}{Xpq@$X_{pq}$, basis of $\mathfrak{o}(n)$}
$X_{mr}$ 
of the Lie algebra $\mathfrak{o}(n-1)$
(see \eqref{eqn:xpq}), that is,
$\zeta_m \frac{\partial f}{\partial \zeta_r} = \zeta_r \frac{\partial f}{\partial \zeta_m}$
for all $1\leq r,m \leq n-1$. Therefore,
\begin{equation}\label{eqn:Poincare3}
\sum_{r=1}^n\zeta_m\frac{\partial f}{\partial \zeta_r} \frac{\partial h}{\partial \zeta_r}
=\sum_{r=1}^n\zeta_r\frac{\partial f}{\partial \zeta_m} \frac{\partial h}{\partial \zeta_r}
= \frac{\partial f}{\partial \zeta_m}E_\zeta(h).
\end{equation}
Now the proposed equality follows from \eqref{eqn:Poincare2} and \eqref{eqn:Poincare3}.
\end{proof}

Finally, we recall the following formula from \cite[Lem.\ 6.11]{KP2}:

\begin{lem}\label{lem:12301}
Suppose $\ell \in \N$ and $\lambda \in \C$.
For any $g \in \mathrm{Pol}_\ell[t]_{\mathrm{even}}$,
\begin{equation*}
\widehat{d\pi}_{\lambda^*}(N_m^+)(T_\ell g) 
=\frac{\zeta_m}{Q_{n-1}(\zeta')} T_\ell\left(R^{\lambda-\frac{n-1}{2}}_\ell g \right)
\quad \emph{for $1\leq m \leq n-1$.}
\end{equation*}
\end{lem}

We are ready to complete the proof of Proposition \ref{prop:scalar-vect}.

\begin{proof}[Proof of Proposition \ref{prop:scalar-vect}]
The first statement of Proposition \ref{prop:scalar-vect} follows from
\eqref{eqn:dpiscalar} and Lemmas \ref{lem:Poincare} and \ref{lem:12301}.
The second statement has been proved in Lemma \ref{lem:FmethodOn}.
Hence the proof of Proposition \ref{prop:scalar-vect} is completed.
\end{proof}

%%%%%%%%%%%%%%%%%%%%%%%%%%%%%%%%%%%%%%%%%%%%%%%
\subsection{Matrix components in the F-method}
\label{subsec:MIJF}
${ }$
For actual computations in later chapters, it is convenient to
rewrite Proposition \ref{prop:scalar-vect} by means of matrix coefficients.

Let $V$ be a vector space
with a basis $\{e_I\}_{I\in \mathcal{I}}$, 
$W$ with a basis $\{w_{J}\}_{J \in \mathcal{J}}$,
and $\{w_{J}^\vee\}_{J \in \mathcal{J}}$ denote the dual basis in $W^\vee$.
Given a linear map $T\colon V\To W$ we define its matrix coefficient by
\begin{equation*}
T_{IJ}:=\left\langle T(e_I),w_{J}^\vee\right\rangle,
\end{equation*}
where $\langle\,,\,\rangle$ denotes the canonical pairing between $W$ and $W^\vee$. 
Clearly, we have for $S\in W^\vee$
\begin{equation}\label{eqn:ST}
(S\circ T)(e_I)=\sum_{J \in \mathcal{J}}S(w_{J})T_{IJ}.
\end{equation}
Suppose that $(\sigma,V)$ is a finite-dimensional representation of 
$M\,(\simeq O(n)\times O(1))$. We introduce a holomorphic vector field on
$\mathfrak n_+$ by
\index{A}{AII'@$A_{II'}$, matrix component of $A_\sigma$|textbf}
\begin{equation}\label{eqn:AII}
A_{II'}\equiv A^\sigma_{II'}
:=\sum_{\ell=1}^n \left(d\sigma(X_{\ell1})_{II'}\right)\frac{\partial}{\partial \zeta_\ell},
\end{equation}
with respect to the basis $\{e_I\}_{I \in \mathcal{I}}$ of $V$ 
and the dual basis $\{e^\vee_{I'}\}_{I' \in \mathcal{I}}$ of $V^\vee$.
Then $\{A_{II'}\}$ is the matrix expression of
the vector part 
\index{A}{Asigma@$A_\sigma$, vector part of $\widehat{d\pi_{(\sigma,\lambda)^*}}$}
$A_\sigma(N_1^+)$ 
\index{A}{Nell+@$N_\ell^+$ ($=\frac{1}{2}C^+_\ell$), basis of $\mathfrak{n}_+(\R)$}
of $\widehat{d\pi_{(\sigma,\lambda)^*}}(N_1^+)$ in the following sense.

\begin{lem}\label{lem:vectorpart}
Recall that $A_\sigma\colon  \mathfrak{g} \To \mathcal{D}(\mathfrak{n}_+) \otimes
\mathrm{End}(V^\vee)$ is defined by \eqref{eqn:Fsv}.
Suppose $F(\zeta)=\sum_I F_I(\zeta)e_I^\vee\in\operatorname{Pol}(\mathfrak n_+)\otimes V^\vee$. Then
$A_\sigma(N_1^+)F$  is given by
\begin{equation*}
A_{\sigma}(N_1^+)F=\sum_{I\in \mathcal{I}}
\left(\sum_{I'\in \mathcal{I}} A_{II'}F_{I'}\right)e^\vee_I.
\end{equation*}
\end{lem}

\begin{proof}
By \eqref{eqn:ST},
$(e^\vee_{I'} \circ d\sigma(X_{\ell m}))(e_I) = d\sigma(X_{\ell m})_{II'}$.
By \eqref{eqn:Asigma} and \eqref{eqn:Npm},
we have
$$
(A_\sigma(N_1^+)F)(e_I)=-\sum_{I'\in \mathcal{I}}\sum_{\ell}\frac{\partial}{\partial \zeta_\ell}
F_{I'}(\zeta) d\sigma(X_{1\ell})_{II'}=\sum_{I'\in \mathcal{I}}A_{II'}F_{I'}.
$$
\end{proof}

Given $\psi\in \mathrm{Hom}_\C(V,W\otimes \mathrm{Pol}(\mathfrak n_+))$,
we write
\begin{eqnarray}
\psi&=&\sum_{I,J}\psi_{IJ}e_I^\vee \otimes w_J, \nonumber\\
\widehat{d\pi_{(\sigma,\lambda)^*}}(N_1^+)\psi&=&\sum_{I,J}M_{IJ}e_I^\vee \otimes w_J,
\nonumber
\end{eqnarray}
for some polynomials $\psi_{IJ}(\zeta)$, $M_{IJ}(\zeta)\in \mathrm{Pol}(\mathfrak n_+)$. 
Then the $(I,J)$-components 
\index{A}{MIJ@$M_{IJ}$, matrix component of 
$\widehat{d\pi_{(i,\lambda)^*}}(N_1^+)\psi$|textbf}
$M_{IJ}$ of $\widehat{d\pi_{(\sigma, \lambda)^*}}(N_1^+)\psi$
can be computed from $\{\psi_{IJ}\}$ by the
following formula.

\begin{prop}\label{prop:MIJ}
Let $\psi_{IJ}$ and $M_{IJ}$ be the matrix coefficients of $\psi$ and  of 
$\widehat{d\pi_{(\sigma,\lambda^*)}}(N_1^+)\psi$ 
with respect to the basis $\{e_I\}$ of $V$ and
$\{w_J\}$ of $W$.
Then we have 
$M_{IJ}=M_{IJ}^{\mathrm{scalar}}+M_{IJ}^{\mathrm{vect}}$
if we set
\begin{eqnarray*}
\index{A}{Mscalar@$M_{IJ}^{\mathrm{scalar}}$|textbf}
M_{IJ}^{\mathrm{scalar}}&=&\left(\lambda\frac{\partial}{\partial \zeta_1}+E_\zeta\frac{\partial}{\partial \zeta_1}-\frac12\zeta_1\Delta_{\C^n}\right)\psi_{IJ},\\
\index{A}{Mvect@$M_{IJ}^{\mathrm{vect}}$|textbf}
M_{IJ}^{\mathrm{vect}}&=&\sum_{I'}A_{II'}\psi_{I'J},
\end{eqnarray*}
where $A_{II'}$ is a vector field defined in \eqref{eqn:AII}. 
In particular, if $\psi$ is of the form
\begin{equation*}
\psi=(T_{a-k}g_k)H^{(k)}
\end{equation*}
with $g_k(t) \in \mathrm{Pol}_{a-k}[t]_{\mathrm{even}}$ and 
$H^{(k)} = \sum_{I,J} H^{(k)}_{IJ} e^\vee_I \otimes w_j 
\in \mathrm{Hom}(V,W\otimes \mathcal{H}^k(\C^{n-1}))$
for some $0 \leq k \leq a$, then
\begin{align*}
M^{\mathrm{scalar}}_{IJ} &=
\frac{\zeta_1}{Q_{n-1}(\zeta')} T_{a-k}\left(R^{\lambda-\frac{n-1}{2}}_{a-k}g_k\right)
H^{(k)}_{IJ} + (\lambda+a-1)(T_{a-k}g_k)\frac{\partial H^{(k)}_{IJ}}{\partial \zeta_1},\\
M^{\mathrm{vect}}_{IJ}
&=\sum_{I'}A_{II'}(T_{a-k}g_k)H^{(k)}_{I'J}\\
&=\sum_{I'}\sum_{\ell=1}^n d\sigma(X_{\ell1})_{II'} 
\frac{\partial}{\partial \zeta_\ell} \left((T_{a-k}g_k)H^{(k)}_{I'J}\right).
\end{align*}
\end{prop}

\begin{proof}
Immediate from Lemmas \ref{lem:FmethodOn} and \ref{lem:vectorpart}.
\end{proof}

In Chapters \ref{sec:7} and \ref{sec:codiff}, we 
shall address the matrix-valued differential equation \eqref{eqn:hatN1} 
in Proposition \ref{prop:Fmethod2} for $V= \Exterior^i(\C^n)$ and 
$W=\Exterior^j(\C^{n-1})$ by solving the system of ordinary differential
equations for $\{\psi_{IJ}\}$
\begin{equation*}
M_{IJ}^{\mathrm{scalar}} + M_{IJ}^{\mathrm{vect}} = 0
\end{equation*}
for all the indices $I$ and $J$ of the bases of $V$ and $W$, respectively.

\newpage
%%%%%%%%%%%%%%%%%%%%%%%%%%%%%%%%%%%%%%%%%%%%%%%
\section{Application of finite-dimensional representation theory}\label{sec:4}

In this chapter we prepare some results on finite-dimensional representations
that will be used in applying the general theory developed in 
Chapters \ref{sec:Method} and \ref{sec:FON}
to symmetry breaking operators for differential forms.

For this, we construct an explicit basis of
$\mathrm{Hom}_{O(n-1)}(V,W\otimes\mathrm{Pol}[\zeta_1,\ldots, \zeta_n])$ for
$V=\Exterior^i(\C^n)$ and  $W=\Exterior^{j}(\C^{n-1})$
(see Proposition \ref{prop:Tgh}). The key ingredient of the proof is to
determine $O(N)$-invariant elements in the triple tensor product 
$\Exterior^i(\C^N) \otimes  \Exterior^j(\C^N) \otimes \mathcal{H}^k(\C^N)$,
which is carried out in Section \ref{subsec:ExtHarm}, see Lemma \ref{lem:Altharmonic}
and Proposition \ref{prop:Fij}.

At the end of this chapter, we give a proof of the (easy) 
implication (i)$\Rightarrow$(iii) in Theorem \ref{thm:1A}.

%%%%%%%%%%%%%%%%%%%%%%%%%%%%%%%%%%%%%%%%%%%%%%%
\subsection{Signatures in index sets}\label{subsec:3sgn}
${}$

We fix some set theoretic notation.
Given a set $S$, let $|S|$ denote the cardinality of elements in $S$.
\index{A}{ST@$S\setminus T$}
We denote by $S\setminus T$ 
the relative complement of $T$ in $S$ 
for given two sets $S$ and $T$, that is,
$S\setminus T :=\{x \in S : x \notin T\}$.

For $k \in \{1, \ldots, N\}$, we set 
\index{A}{Ink@$\mathcal{I}_{n,k}$, index set|textbf}
\begin{equation}\label{eqn:Index}
\mathcal{I}_{N,k}:=\{R \subset \{1, \ldots, N\} \; : \; |R| = k \}.
\end{equation}
It is convenient to define $\mathcal{I}_{N,0}$ as $\mathcal{I}_{N,0} := \{\emptyset\}$.

\begin{defn}\label{def:sign}
For $I\subset\{1,\cdots,N\}$ and $p, q\in\N$, we set
\index{A}{sgnp@$\sgn(I;p)$|textbf}
\index{A}{sgnpq@$\sgn(I;p,q)$|textbf}
\begin{eqnarray*}
\sgn(I;p)&:=&(-1)^{|\{r\in I:\, r<p\}|},\\
\sgn(I;p,q)&:=&(-1)^{|\{r\in I :\, \min(p,q)<\,r\,<\max(p,q)\}|}.
\end{eqnarray*}
\end{defn}

Here are some basic formul\ae{} for $\sgn(I;p)$ and $\sgn(I;p,q)$.

\begin{lem}\label{lem:sgn}
For $I\subset\{1,\cdots,N\}$ and $ p, q\in\N$, we have
\begin{enumerate}
\item[$(1)$] $\sgn(I;p)=\sgn(I\cup\{p\};p);$
\vskip 0.1in 
\item[$(2)$] $\sgn(I;p,q)=\sgn(I\cup\{p\};p,q)=\sgn\{I\cup\{q\};p,q);$
\vskip 0.1in
\item[$(3)$] $\sgn(I;p)\sgn(I;q)\sgn(I;p,q)=
\left\{ 
\begin{matrix}
+1 & \mathrm{if}& \min(p,q)\not\in I,\\
-1 & \mathrm{if}& \min(p,q)\in I;
\end{matrix}
\right.$
\vskip 0.1in
\item[$(4)$] $\sgn(I\cup \{p\} ; q )\sgn(I;p) + \sgn(I\cup\{q\}; p)\sgn(I;q)=0
\; \emph{for $p, q \notin I$}.$
\vskip 0.1in
\end{enumerate}
\end{lem}

\begin{proof}
The proof is a straightforward computation.
\end{proof}

Note that, by Lemma \ref{lem:sgn} (2) and (3),
for $I \in \mathcal{I}_{N,i}$ with $N\in I$, the following identity holds:
\begin{equation}\label{eqn:sgn}
\sgn(I\setminus\{N\};p) = 
\begin{cases}
(-1)^{i-1}\sgn(I;p,N) & \text{if $p \notin I$},\\
(-1)^{i}\sgn(I;p,N) & \text{if $p\in I$}.
\end{cases}
\end{equation}

%%%%%%%%%%%%%%%%%%%%%%%%%%%%%%%%%%%%%%%%%%%%%%%
\subsection{Action of $O(N)$ on the exterior algebra $\Exterior^*(\C^N)$}
\label{subsec:3BI} 

Let $\{e_1,\cdots, e_N\}$ be the standard basis of $\C^N$. Given $I=\{i_1,\cdots,i_k\}\subset\{1,\cdots,N\}$ with $i_1<\cdots< i_k$, we form the standard
 basis $\{e_I\}$ of $\Exterior^k(\C^N)$ by setting 
\begin{equation*}
e_I:=e_{i_1}\wedge\cdots\wedge e_{i_k}.
\end{equation*}
The natural action $\sigma$
of $O(N)$ on $\C^N$ induces the exterior representation 
on $\Exterior^i(\C^N)$, to be denoted by the same letter $\sigma$. 
Let $X_{pq}=-E_{pq}+E_{qp} \in \mathfrak{o}(N)$
$(1 \leq p\neq q \leq N)$ as in \eqref{eqn:xpq}.
The matrix coefficient
\index{A}{Xpq@$X_{pq}$, basis of $\mathfrak{o}(n)$}
$d\sigma(X_{pq})_{II'}=\langle d\sigma(X_{pq})(e_I),e_{I'}^\vee\rangle$
of the infinitesimal representation $d\sigma$ is given by
\begin{equation}\label{eqn:XpqgeI}
d\sigma(X_{pq})_{II'}=
\left\{
\begin{matrix}
\sgn(I;p,q) &\mathrm{if}&I=J\cup\{p\},\, I'=J\cup\{q\},\\
-\sgn(I;p,q) &\mathrm{if}&I=J\cup\{q\},\, I'=J\cup\{p\},\\
0 & &\mathrm{otherwise},
\end{matrix}
\right.
\end{equation}
where $J:= I\cap I'$ in the first two cases.

%\begin{equation}\label{eqn:XpqgeI}
%d\sigma(X_{pq})e_I=\left\{
%\begin{matrix}
%\sgn(I;p,q)\,e_{I-\{p\}\cup\{q\}} & \mathrm{if}& p\in I\not\ni q,\\
%-\sgn(I;p,q)\,e_{I-\{q\}\cup\{p\}} & \mathrm{if}& q\in I\not\ni p,\\
%0& \mathrm{if}& p,q\in I,\\
%0& \mathrm{if}& p,q\not\in I.
%\end{matrix}
%\right.
%\end{equation}

%%% The formula is correct, but omitted. (Kobayashi, 3/6/2016) 

Recall from Section \ref{subsec:2Fmethod} that,
given a representation $(\sigma, V)$ of $M\simeq O(n) \times O(1)$ 
and $\lambda \in \mathfrak{a}^*$,
we denote by
$\widehat{d\pi_{(\sigma,\lambda)^*}}$ the algebraic Fourier 
transform of the Lie algebra homomorphism
$d\pi_{(\sigma,\lambda)^*} : \mathfrak{g} \To \mathcal{D}(\mathfrak{n}_-)
\otimes \mathrm{End}(V^\vee)$.
When $\sigma=\Exterior^i(\C^n)$ and $\sigma\vert_{O(n)}$ 
is the exterior representation,
we write simply 
\index{A}{dpi6ihat@$\widehat{d\pi_{(i,\lambda)^*}}$|textbf}
$\widehat{d\pi_{(i,\lambda)^*}}$ for $\widehat{d\pi_{(\sigma,\lambda)^*}}$,
as it is independent of the restriction of $\sigma$ to the second factor $O(1)$.
Then the matrix components $A_{II'}$ of the vector part of 
$\widehat{d\pi_{(i,\lambda)^*}}(N_1^+)$ (see Lemma \ref{lem:vectorpart})
takes the following form:

\begin{lem}\label{lem:AII}
Let $I,I'\in\mathcal I_{n,i}$. 
Then the $(I,I')$-component $A_{II'}$
of the vector part of $\widehat{d\pi_{(i,\lambda)^*}}(N_1^+)$
is given by
the following vector field
\index{A}{AII'@$A_{II'}$, matrix component of $A_\sigma$}
\begin{equation*}
A_{II'}=
\begin{cases}
\sgn (I;\ell)\frac\partial{\partial\zeta_\ell}&\mathrm{if}\,  
(I\setminus I') \coprod (I'\setminus I) = \{1, \ell\} \;\; (\ell \neq 1),\\
0&\mathrm{otherwise}.
\end{cases}
\end{equation*}
\end{lem}

\begin{proof}
We recall from \eqref{eqn:AII} that 
$A_{II'} = \sum_{\ell =1}^n d \sigma(X_{\ell 1})_{II'} \frac{\partial}{\partial x_{\ell}}.$
Hence the lemma is clear from \eqref{eqn:XpqgeI}.
\end{proof}

\begin{example}
With respect to the basis $\{dx_2\wedge dx_3, dx_1\wedge dx_3, dx_1\wedge dx_2\}$ 
of $\mathcal{E}^2(\R^3)$ as a $C^\infty(\R^3)$-module,
the vector part of  $\Fdpi{i}{\lambda}(N_1^+)$
with $i=2$ acts on 
$F=\sum_{1\leq k < \ell \leq 3} F_{k \ell} 
dx_{k} \wedge dx_{\ell}
\in\mathcal E^2(\R^3)\simeq C^\infty(\R^3)\otimes\C^3$ by
\begin{equation*}
\begin{pmatrix}
F_{23}\\F_{13}\\F_{12}
\end{pmatrix}
\mapsto
\begin{pmatrix}
0 & \frac{\partial}{\partial \zeta_2}&-\frac{\partial}{\partial \zeta_3}\\
-\frac{\partial}{\partial \zeta_2}&0&0\\
\frac{\partial}{\partial \zeta_3}&0&0
\end{pmatrix}
\begin{pmatrix}
F_{23}\\F_{13}\\F_{12}
\end{pmatrix}.
\end{equation*}
\end{example}

%%%%%%%%%%%%%%%%%%%%%%%%%%%%%%%%%%%%%%%%%%%%%%%
\subsection{Construction of intertwining operators}\label{subsec:4Altpol}

For  $V=\Exterior^i(\C^N)$ and $W=\Exterior^j(\C^N)$,
we shall construct building blocks of 
$O(N)$-equivariant bilinear maps
\index{A}{B(k)@$B^{(k)}$, bilinear map|textbf}
\begin{equation*}
B^{(k)}\colon \Exterior^i(\C^N)\times\Exterior^j(\C^N)\To\mathrm{Pol}[\zeta_1,\cdots,\zeta_N],
\end{equation*}
as follows.
Suppose $j=i$. For $I,I'\in\mathcal I_{N,i}$, we define $\C$-bilinear maps
$B^{(0)}$ and $B^{(2)}$ by giving the images of the basis elements:
\begin{align}
B^{(0)}(e_I,e_{I'})&:=\left\{
\begin{matrix}
1&\mathrm{if}\quad I=I',\\0&\mathrm{otherwise}.
\end{matrix}
\right.\\
B^{(2)}(e_I,e_{I'})&:=\left\{
\begin{matrix*}[l]
\sum\limits_{\ell\in I}\zeta_\ell^2&\mathrm{if}\quad I=I',\\
\sgn(J;p,q)\,\zeta_p\zeta_q&\mathrm{if}\quad I=J\cup \{p\}, I'=J\cup\{q\}, p\neq q,\\
0&\mathrm{if}\quad |(I\setminus I')|>1.
\end{matrix*}
\right.\label{eqn:F2}
\end{align}

Suppose $j=i-1$. For $I\in\mathcal I_{N,i}$ and $J\in\mathcal I_{N,i-1}$, 
we set
\begin{equation}\label{eqn:F1}
B^{(1)}(e_I,e_{J}):=\left\{
\begin{matrix*}[l]
\sgn(J;\ell)\,\zeta_\ell&\mathrm{if}\quad I=J\cup\{\ell\},\\
0&\mathrm{if}\quad J\not\subset I.
\end{matrix*}
\right.
\end{equation}

\begin{lem}\label{lem:Tri}
The bilinear maps $B^{(k)}$ $(k=0,1,2)$ are $O(N)$-equivariant, namely,
\begin{equation*}
B^{(k)}(gv,gw) (g\zeta) = B^{(k)}(v,w)(\zeta)
\end{equation*}
for all $g\in O(N)$, $v\in V$, $w \in W$, and $\zeta= (\zeta_1,\ldots, \zeta_N)$.
\end{lem}

We could prove Lemma \ref{lem:Tri}
directly by the formula \eqref{eqn:XpqgeI}, but we shall give an alternative and simpler proof
 in Section \ref{subsec:symb} by using the symbol map for $O(N)$-equivariant differential operators.

Since $\Exterior^i(\C^N)$ is self-dual as an $O(N)$-module (cf.\ \eqref{eqn:Hodgeself}),
the bilinear forms $B^{(k)}$ induce the following $O(N)$-equivariant linear maps
\index{A}{H1ijk@$H_{i\to j}^{(k)}$|textbf}
\index{A}{Pol[zeta]@$\mathrm{Pol}^k[\zeta_1,\cdots,\zeta_N]$}
\begin{equation*}
H_{i\to j}^{(k)}\colon \Exterior^i(\C^N)\to\Exterior^j(\C^N)\otimes
\mathrm{Pol}^k[\zeta_1,\cdots,\zeta_N]
\end{equation*}
given by
\begin{eqnarray}
H_{i\to i}^{(0)}(e_I)&:=&\sum_{I'\in\mathcal I_{N,i}} B^{(0)}(e_I,e_{I'})e_{I'}\;=e_I,\label{eqn:F0}\\
H_{i\to i-1}^{(1)}(e_I)&:=&\sum_{J\in\mathcal I_{N,i-1}} B^{(1)}(e_I,e_{J})e_{J}
=\sum_{\ell\in I}\sgn(I;\ell) e_{I\setminus\{\ell\}}\zeta_\ell,\label{eqn:F1m}\\
H_{i-1\to i}^{(1)}(e_J)&:=&\sum_{I\in\mathcal I_{N,i}} B^{(1)}(e_I,e_{J})e_{I}\;\;=\sum_{\ell\not\in J}\sgn(J;\ell) e_{J\cup\{\ell\}}\zeta_\ell,\label{eqn:F1p}\\
H_{i\to i}^{(2)}(e_I)&:=&\sum_{I'\in\mathcal I_{N,i}} B^{(2)}(e_I,e_{I'})e_{I'}
\label{eqn:F22}\\
&=&(\sum_{\ell\in I}\zeta^2_\ell)e_I+ \sum_{q\not\in I}
\sum_{p\in I}\sgn(I;p,q) e_{I\setminus\{p\}\cup\{q\}}\zeta_p\zeta_q.\nonumber
\end{eqnarray}

Then all the matrix coefficients of $H^{(k)}_{i \to j}$ 
are harmonic polynomials for the first three cases, but not
for $ H_{i\to i}^{(2)}$. In order to make the matrix coefficients to be harmonic polynomials,
we set
\index{A}{H1wijk@$\widetilde H_{i\to j}^{(k)}
\colon\Exterior^i(\C^N)\To \Exterior^j(\C^N)\otimes \mathcal{H}^k(\C^N)$|textbf}
\begin{eqnarray}
\widetilde H_{i\to j}^{(k)}&:=&H_{i\to j}^{(k)}\quad\mathrm{if}\; j-i=k=0\; \mathrm{or}\; \vert j-i\vert=k=1,\nonumber\\
\widetilde H_{i\to i}^{(2)}&:=&H_{i\to i}^{(2)}-\frac iN Q_N(\zeta) H_{i\to i}^{(0)}.
\label{eqn:H2tilde}
\end{eqnarray}
Then the matrix coefficients 
$\left(\widetilde H_{i\to i}^{(2)}\right)_{II'}:=\left\langle\widetilde H_{i\to i}^{(2)}(e_I),e_{I'}^\vee\right\rangle$ ($I,I'\in\mathcal I_{N,i}$) are given by
$$
\left(\widetilde H^{(2)}_{i\to i}\right)_{II'}=\left\{
\begin{array}{lll}
\widetilde{Q}_{I}(\zeta)
&\mathrm{if}& I=I',\\
\sgn(J;p,q)\zeta_p\zeta_q & \mathrm{if}& I=J\cup\{p\}, I'=J\cup\{q\}\,\mathrm{with}\;p\neq q,\\
0 &&\mathrm{otherwise},
\end{array}
\right.
$$
where we set
\index{A}{QwI@$\widetilde{Q}_I(\zeta)$|textbf}
\begin{equation}\label{eqn:QItilde}
\widetilde{Q}_{I}(\zeta):=
\sum_{\ell \in I}\zeta^2_\ell - \frac{i}{N}Q_N(\zeta)
\qquad \text{for $I\in \mathcal{I}_{N,i}$.}
\end{equation}
Thus $\left(\widetilde{H}_{i \to i}^{(2)}\right)_{II'}$ are harmonic polynomials
for all $I,I' \in \mathcal{I}_{N,i}$. Hence we have defined the linear maps
\begin{equation}\label{eqn:Htildekij}
\widetilde H_{i\to j}^{(k)}\colon \Exterior^i(\C^N)\To\Exterior^j(\C^N)\otimes\mathcal H^k(\C^N),
\end{equation}
which are obviously $O(N)$-equivariant in all the cases. 
In the next section, we shall prove that $\widetilde H_{i\to j}^{(k)}$ exhaust all such $O(N)$-linear maps up to scalars, see Proposition \ref{prop:Fij} below.

We need to be careful at the extremal places where the modified maps
$\widetilde{H}_{i\to j}^{(k)}$ may vanish: 
\begin{equation}\label{eqn:H2zero}
\widetilde{H}^{(2)}_{0\to 0} = \widetilde{H}^{(2)}_{N\to N} = 0.
\end{equation}

%%%%%%%%%%%%%%%%%%%%%%%%%%%%%%%%%%%%%%%%%%%%%%%
\subsection{Application of finite-dimensional representation theory}\label{subsec:ExtHarm}
In this section we prove that the linear maps $\widetilde H_{i\to j}^{(k)}$ 
introduced in \eqref{eqn:Htildekij}
exhaust all nonzero
$O(N)$-homomorphisms $\Exterior^i(\C^N)\To\Exterior^j(\C^N)\otimes\mathcal H^k(\C^N)$
up to scalar multiplication. The results will be used 
for actual calculations in solving the 
\index{B}{F-system}
F-system, which yield
all differential symmetry breaking operators $\mathcal E^i(S^n)_{u,\delta}\To\mathcal E^j(S^{n-1})_{v,\eps}$,
see Theorems \ref{thm:2}-\ref{thm:2ii-2}. To be more precise, we prove
the following.

\begin{lem}\label{lem:Altharmonic}
Let $N\geq 1$.
Then the following three conditions on $(i,j,k)$ 
with $0 \leq i, j \leq N$ and $k \in \N$ are equivalent.
\begin{eqnarray*}
&\mathrm{(i)}& \mathrm{Hom}_{O(N)}\left(\Exterior^{i}(\C^N),\Exterior^j(\C^{N})\otimes\mathcal{H}^k(\C^N)\right)\neq\{0\},\\
&\mathrm{(ii)}& \dim_\C\left(\mathrm{Hom}_{O(N)}\left(\Exterior^i(\C^N),\Exterior^j(\C^{N})\otimes\mathcal{H}^k(\C^N)\right)\right)=1,\\
&\mathrm{(iii)}& 
\emph{The triple}\; (i,j,k) \;\text{\emph{belongs to one of the following three cases}}:\\
&&\qquad\mathrm{(a)}\;\; i = j \; \mathrm{and} \; k=0.\\
&&\qquad\mathrm{(b)} \; \; i=j \in \{1, 2, \ldots, N-1\} \; \mathrm{and} \; k=2.\\
&&\qquad\mathrm{(c)} \; \; |i-j| = k = 1.
\end{eqnarray*}
\end{lem}

We observe that nonzero $O(N)$-homomorphisms $\widetilde H^{(k)}_{i\to j}$ 
were constructed 
in Section \ref{subsec:4Altpol} for all the triples $(i,j,k)$ appearing in (iii) of 
Lemma \ref{lem:Altharmonic}. Then
the multiplicity-free property ((ii) of Lemma \ref{lem:Altharmonic}) implies
the following proposition.

\begin{prop}\label{prop:Fij}
Suppose $(i,j,k)$ satisfies one of (therefore all of) the equivalent conditions in Lemma \ref{lem:Altharmonic}. Then, we have
\begin{eqnarray*}
\operatorname{Hom}_{O(N)}\left(\Exterior^i(\C^N),\Exterior^j(\C^N)\otimes\mathcal H^k(\C^N)\right)=\C \widetilde H^{(k)}_{i\to j}.
\end{eqnarray*}
\end{prop}

\begin{rem}
Since $\mathrm{Pol}(\C^N) \simeq \C[Q_N] \otimes \mathcal{H}(\C^N)$
as an $O(N)$-module (see \eqref{eqn:HPP}), any $O(N)$-homomorphism
from $\Exterior^i(\C^N)$ to $\Exterior^j(\C^N) \otimes \mathrm{Pol}(\C^N)$
can be written as a linear combination of $Q^\ell_N \widetilde{H}_{i \to j}^{(k)}$
($\ell \in \N$, $k \in \{0,1,2\}$).
\end{rem}

The rest of this section is devoted to the proof of Lemma \ref{lem:Altharmonic}.
For this, we observe that $\Exterior^i(\C^N)$ may be thought of as
a $U(N)$-module, whereas $\mathcal{H}^k(\C^N)$ is just an $O(N)$-module.
Then our strategy is to use
the branching laws with respect to a chain of subgroups
\begin{equation*}
U(N)\times U(N)\supset U(N)\supset O(N),
\end{equation*}
and the proof is divided into the
 following two steps. 

\noindent Step 1. Decompose $\Exterior^{i}(\C^N)\otimes\Exterior^j(\C^{N})$ into irreducible $U(N)$-modules, see Lemma \ref{lem:step1}.

\noindent Step 2. Consider the branching law $U(N)\downarrow O(N)$, and find the multiplicities of the $O(N)$-module
$\mathcal{H}^k(\C^N)$ occurring in the irreducible $U(N)$-summands
of the tensor product representation in Step 1, see Lemma \ref{lem:step2}.

\vskip 0.1in

We fix some notations. We set
\index{A}{0tExterior@$\Lambda^+(N)$}
\begin{equation*}
\Lambda^+(N):=\{ \lambda = (\lambda_1, \ldots, \lambda_N) \in \Z^N:
\lambda_1 \geq \lambda_2 \geq \cdots \geq \lambda_N \geq 0\}.
\end{equation*}
We write $F(U(N),\lambda)$ for the irreducible finite-dimensional representation 
of $U(N)$ (or equivalently, the irreducible polynomial representation of $GL(N,\C)$)
with highest weight $\lambda$. 
If $\lambda$ is of the form $(\underbrace{c_1,\cdots,c_1}_{m_1},
\underbrace{c_2,\cdots,c_2}_{m_2},\cdots,\underbrace{c_\ell,\cdots,c_\ell}_{m_\ell},0\cdots,0)
$, then we also write $\lambda=\left(c_1^{m_1},c_2^{m_2},\cdots,c_\ell^{m_\ell}\right)$ as usual. For instance $F(U(N),1^i)\simeq \Exterior^i(\C^N)$. As Step 1, we use the following lemma:

\begin{lem}\label{lem:step1}
We have the following isomorphisms of $U(N)$-modules.
\begin{eqnarray*}
\Exterior^{i}(\C^N)\otimes\Exterior^i(\C^{N})&\simeq &\bigoplus_{k=\max(0,2i-N)}^{i}
F(U(N),(2^k,1^{2i-2k})),\\
\Exterior^{i}(\C^N)\otimes\Exterior^{i-1}(\C^{N})&\simeq &\bigoplus_{k=\max(0,2i-N-1)}^{i-1}F(U(N),(2^k,1^{2i-2k-1})).
\end{eqnarray*}
\end{lem}
\begin{proof}
Both decompositions are given by the 
\index{B}{skew Pieri rule}
skew Pieri rule for the tensor product of
the exterior representations $\Exterior^i(\C^N)$.
\end{proof}

As Step 2, we consider how each $U(N)$-irreducible summand in Lemma \ref{lem:step1}
decomposes as an $O(N)$-module. This decomposition is not always multiplicity-free,
however, it turns out that the $O(N)$-irreducible module $\mathcal{H}^s(\C^N)$
($s\in \N$) occurs at most once. To be precise, we have the following.

\begin{lem}\label{lem:step2}
Let $N\geq 2$.
Suppose $s, k, \ell \in \N$ satisfy $k+\ell\leq N$. 
Then the following three conditions on $(s, k, \ell)$ are equivalent:
\begin{enumerate}
\item[(i)]
$\displaystyle{\mathrm{Hom}_{O(N)}\left(\mathcal H^s(\C^N),
F(U(N),(2^k,1^\ell))\Big\vert_{O(N)}\right)\neq \{0\}}$,\\
\item[(ii)]
$\displaystyle{\dim\mathrm{Hom}_{O(N)}\left(\mathcal H^s(\C^N),
F(U(N),(2^k,1^\ell))\Big\vert_{O(N)}\right)
=1}$,\\
\item[(iii)]
$(s,\ell) = (0,0)$ \emph{with} $0\leq k \leq N$, 
$(s,\ell) = (1,1)$ \emph{with} $0\leq k \leq N-1$, \emph{or} 
$(s,\ell) = (2,0)$ \emph{with} $1\leq k \leq N$.
\end{enumerate}
\end{lem}

For the proof of Lemma \ref{lem:step2}, we need some combinatorics
related to representations of $U(N)$ and $O(N)$.

We shall identify $\lambda \in \Lambda^+(N)$ with the corresponding
Young diagram. For $\lambda, \nu,\mu \in \Lambda^+(N)$, we denote by 
$c^\lambda_{\nu\mu} \in \N$ the Littlewood--Richardson coefficient, 
namely, the structure constant for the product in the $\C$-algebra of 
symmetric functions with respect to the basis of Schur functions
\begin{equation*}
s_\nu s_\mu = \sum_{\lambda} c^\lambda_{\nu \mu} s_\lambda.
\end{equation*}
We note that $c^\lambda_{\nu\mu} \neq 0$ only if 
$\nu \subset \lambda$ and $\mu \subset \lambda$,
namely, $\nu_j \leq \lambda_j$ and $\mu_j \leq \lambda_j$
for all $j$ $(1\leq j \leq N)$.
The Littlewood--Richardson coefficient $c^{\lambda}_{\nu \mu}$ has 
a combinatorial description in several ways such as
\begin{equation*}
c^\lambda_{\nu\mu}=
\left| \{
\text{tableau $T$ on skew diagram $\lambda \setminus \nu$:
$\mathrm{weight}(T) = \mu$, 
$\mathrm{word}(T)$ is a lattice permutation}
\}\right|,
\end{equation*}
where we recall:

\begin{itemize}

\item 
\index{A}{1jlambda1@$\lambda \setminus \nu$, skew diagram|textbf}
$\lambda \setminus \nu$ is the skew 
diagram obtained by removing all the boxes of $\nu$ from the diagram $\lambda$
with the same top-left corner;
\vskip 0.1in

\item a tableau $T$ is a filling  of the boxes of a skew diagram with positive integers,
weakly increasing in rows and strictly decreasing in columns;
\vskip 0.1in

\item $\mathrm{weight}(T)$ 
is a vector such that $i$-th component equals
the times
of occurrences of the positive integer $i$ in the tableau $T$;
\vskip 0.1in

\item $\mathrm{word}(T)$ is a sequence of positive integers in $T$
when we read from right to left in successive rows, starting with
the top row;
\vskip 0.1in

\item A sequence $a_1, \ldots, a_N$ of positive integers
is said to be a lattice permutation
if $|\{ 1 \leq k \leq r : a_k =i\}|$ is a weakly decreasing function of $i \in \N$
for every $r$ $(1\leq r \leq N)$.
\end{itemize}

We introduce the following map
\index{A}{1jlambda2@$\lambda/\nu$|textbf}
\begin{equation}\label{eqn:ZRL}
\Lambda^+(N)\times \Lambda^+(N)
\To \Z[\Lambda^+(N)], \quad
(\lambda,\nu) \mapsto 
\lambda/\nu := \bigoplus_{\mu \in \Lambda^+(N)} c^\lambda_{\nu\mu}\mu,
\end{equation}
where $\Z[S]$ denotes the free $\Z$-module generated by elements of a set $S$.

If the skew diagram $\lambda \setminus \nu$ 
is a Young diagram,
namely, if $\nu_j = \lambda_j$ $(1 \leq j \leq k)$ and $\nu_j =0$
$(k+1 \leq j \leq N)$ for some $k$,
then it is readily seen 
from the above combinatorial description that
\begin{equation}\label{eqn:cmf}
c^\lambda_{\nu \mu}=
\begin{cases}
1 & \text{if $\mu = \lambda \setminus \nu$},\\
0 & \text{if $\mu \neq \lambda \setminus \nu$}.
\end{cases}
\end{equation}
Thus, $\lambda/\nu = \lambda \setminus \nu$ if $\lambda \setminus \nu$ is a 
Young diagram.

We define two subsets of $\Lambda^+(N)$ by

\index{A}{0tExterior1@$\Lambda^+(N)_{\mathrm{even}}$|textbf}
\index{A}{0tExterior2@$\Lambda^+(N)_{\mathrm{BD}}$|textbf}
\begin{align*}
\Lambda^+(N)_{\mathrm{even}}&:=
\{ \lambda \in \Lambda^+(N) : 
\text{$\lambda_j \in 2\Z$ for $1\leq j \leq N$}\},\\ 
\Lambda^+(N)_{\mathrm{BD}}&:= \{\lambda \in \Lambda^+(N):
\lambda_1' + \lambda_2' \leq N\},
\end{align*}
where $\lambda_1':=\max\{i : \lambda_i \geq 1 \}$ and 
$\lambda_2': = \max\{i : \lambda_2 \geq 2\}$ 
for $\lambda =(\lambda_1, \ldots, \lambda_N) \in \Lambda^+(N)$.
Then $\lambda_1'$ is nothing but the maximal column length, denoted also by
\index{A}{L0llambda@$\ell(\lambda)$, column length|textbf}
$\ell(\lambda)$.

It is readily seen that $\Lambda^+(N)_{BD}$ consists of elements of 
the following two types:
\begin{itemize}
\item[]Type I: $(a_1,\cdots,a_k,\underbrace{0,\cdots,0}_{N-k})$,
\item[]Type II: $(a_1,\cdots,a_k,\underbrace{1,\cdots,1}_{N-2k},\underbrace{0,\cdots,0}_k)$,
\end{itemize}
 with $a_1\geq a_2\geq\cdots\geq a_k > 0$ and $0\leq k\leq\left[\frac N2\right]$.

Following Weyl (\cite[Chap.\ V, Sect.\ 7]{Weyl97}),
we parametrize the set $\widehat{O(N)}$ of equivalence classes
of irreducible representations of $O(N)$ as
\begin{equation} \label{eqn:Weyl}
\Lambda^+(N)_{\mathrm{BD}} \stackrel{\sim}{\To} \widehat{O(N)},
\quad
\lambda \mapsto [\lambda],
\end{equation}
where $[\lambda]$ is the $O(N)$-irreducible summand of $F(U(N),\lambda)$ which
contains the highest weight vector.

\begin{example}
$\mathcal{H}^s(\C^N)=[s]$ $(s \in \N)$, and 
$\Exterior^\ell (\C^N) = [1^\ell]$ $(0\leq \ell \leq N)$.
\end{example}

Moreover, Types I and II are related by the following $O(N)$-isomorphism:
\begin{equation}\label{eqn:type1to2}
[(a_1,\cdots,a_k,1,\cdots,1,0,\cdots,0)]=\det\otimes[(a_1,\cdots,a_k,0,\cdots,0)].
\end{equation}
The restriction of the $O(N)$-module $[(a_1, \ldots, a_k, 0, \ldots, 0)]$ 
to the subgroup $SO(N)$ is reducible if and only
if $N=2k$. In this case we have:
\begin{equation*}
[(a_1,\cdots,a_k,0,\cdots,0)]\vert_{SO(N)}=F(SO(N),(a_1,\cdots,a_{k}))\oplus
F(SO(N),(a_1,\cdots,a_{k-1},-a_{k})).
\end{equation*}

For $\lambda\notin \Lambda^+(N)_{\mathrm{BD}}$, we apply the
$O(N)$-modification rule which is a map
\begin{equation}\label{eqn:mod}
\Lambda^+(N) \setminus \Lambda^+(N)_{\mathrm{BD}}
\To \Z[\widehat{O(N)}],
\quad
\lambda \mapsto [\lambda]
\end{equation}
constructed as follows (see \cite[Sect.\ 3]{Ki71}, \cite{KT87}). 
If $\ell(\lambda) \geq \left[\frac{N}{2}\right]$ then we define $\tilde{\lambda}$ to be the 
removal of a continuous boundary hook of length $h:= 2\ell(\lambda)-N$ 
and row length $x$,
starting in the first column of the Young diagram associated to $\lambda$.
We set 
$[\lambda]:=0$ if $\tilde{\lambda}$ is not a Young diagram; 
$[\lambda]:=(-1)^x \det \otimes [\tilde{\lambda}]$ 
if $\tilde{\lambda} \in \Lambda^+(N)_{\mathrm{BD}}$. Otherwise, we repeat this 
procedure to $\tilde{\lambda} \in \Lambda^+(N)\setminus \Lambda^+(N)_{\mathrm{BD}}$.

We note that $\Lambda^+(N)_{\mathrm{BD}}$ contains elements $\lambda$
with $\ell(\lambda)\geq \left[\frac{N}{2}\right]$, namely, 
elements of Type II. The $O(N)$-modification rule also applies to these
elements, and yields the isomorphism \eqref{eqn:type1to2}.
In fact, suppose $\lambda \in \Lambda^+(N)_{\mathrm{BD}}$ is of 
Type II, say $\lambda= (a_1, \ldots, a_k, 1, \ldots, 1, 0, \ldots, 0)$.
In this case, $\ell(\lambda) = n-k$ $(\geq \left[\frac{n}{2}\right])$,
$h = 2(n-k)-n = n-2k$ and $x=0$, and 
thus $\tilde{\lambda}=(a_1, \ldots, a_k, 0,\ldots, 0)$.
Thus the $O(N)$-modification rule in this special case gives rise
to the isomorphism \eqref{eqn:type1to2}.

Combining \eqref{eqn:Weyl} 
and \eqref{eqn:mod}, we get a $\Z$-linear map
\index{A}{1j[lambda]@$[\lambda]$, $O(N)$-modification rule|textbf}
\begin{equation}\label{eqn:domON}
\Z[\Lambda^+(N)] \To \Z[\widehat{O(N)}],\quad
\lambda \mapsto [\lambda].
\end{equation}

For $\lambda \in \Lambda^+(N)$, the representation $F(U(N),\lambda)$ decomposes
as an $O(N)$-module in accordance with the 
\index{B}{modification rule for $O(n)$|textbf}
$O(N)$-modification rule applied to the 
universal character formula \cite{Ki71}, \cite{KT87}:
\begin{equation}\label{eqn:UtoO}
F(U(N),\lambda)\vert_{O(N)}\simeq 
\bigoplus_{\nu \in \bigwedge^+(N)_{\mathrm{even}}}[\lambda / \nu],
\end{equation}
where $\lambda/\nu$ is defined in \eqref{eqn:ZRL} as an element of $\Z[\Lambda^+(N)]$.
For the proof of Lemma \ref{lem:step2}, we use the following two claims.

\begin{claim}\label{claim:1608105}
Suppose $\lambda = (2^k,1^\ell) \in \Lambda^+(N)$
and 
$\nu \in \Lambda^+(N)_{\mathrm{even}}$ with $\nu \subset \lambda$.
Then $\nu$ is of the form $\nu = (2^{k-r})$ for some $0\leq r \leq k$ and 
$\lambda / \nu = (2^r, 1^k)$.
\end{claim}

\begin{proof}[Proof of Claim \ref{claim:1608105}]
The first assertion is clear because $\nu \subset (2^k, 1^\ell)$ 
and $\nu \in \Lambda^+(N)_{\mathrm{even}}$. Then, the skew diagram 
$\lambda \setminus \nu$ is actually a Young diagram
 $(2^r, 1^k)$, and therefore, 
the claim follows from \eqref{eqn:cmf}.
\end{proof}

We write $\mathrm{pr}_{\mathcal{H}} \colon  \Z[\widehat{O(N)}] 
\To \Z[\{\mathcal{H}^s(\C^N): s \in \N\}]$ for the canonical projection.

\begin{claim}\label{claim:1608106}
Suppose $\lambda = (2^r, 1^\ell) \in \Lambda^+(N)$. Then we have
\begin{equation*}
\mathrm{pr}_{\mathcal{H}} ([\lambda]) 
= 
\begin{cases}
\mathcal{H}^{2r+\ell}(\C^N) & \text{if $(r,\ell) = (0,0), (0,1)$, or $(1,0)$},\\
0 & \text{otherwise}.
\end{cases}
\end{equation*}
\end{claim}

\begin{proof}[Proof of Claim \ref{claim:1608106}]
The assertion is obvious from the bijection \eqref{eqn:Weyl} if 
$\lambda \in \Lambda^+(N)_{\mathrm{BD}}$.

What remains to prove is $\mathrm{pr}_{\mathcal{H}}([\lambda]) = 0$
for $\lambda \notin \Lambda^+(N)_{\mathrm{BD}}$. 
First of all, we see from the $O(N)$-modification rule \eqref{eqn:mod} that 
the $\mathcal{H}^s(\C^N)$-component of $[\lambda]$ is nonzero only if 
$s \in \{0, 1, 2\}$ corresponding to $\emptyset$,
$\yng(1)$ or $\yng(2)$. Further,
$s=|\lambda|-h$ and $h = 2(r+\ell)-N\; (>0)$. Hence $\ell = N-s$.

For $s=0$, we have $\ell = N$, and therefore, the only possible form of $\lambda$
is $\lambda = (1^N)$. Hence the corresponding $O(N)$-representation is 
$[\lambda]= \det \not \simeq \one$.

For $s=1$, we have $\ell = N-1$, and therefore, the only possible forms of 
$\lambda$ are either $(1^{N-1})$ with $N\geq 3$ or $(2^1, 1^{N-1})$.
Then $[\lambda] \simeq \det \otimes \mathcal{H}^1(\C^N)$ or $\{0\}$, respectively,
by the $O(N)$-modification rule \eqref{eqn:mod}. Thus $\mathrm{pr}_{\mathcal{H}} ([\lambda])=0$ in either case.

For $s=2$, we have $\ell = N-2$, and therefore, the only possible forms of $\lambda$
are either $\lambda = (2^1, 1^{N-2})$ with $N\geq 3$ or $(2^2, 1^{N-2})$.
Then $[\lambda] = \det \otimes \mathcal{H}^2(\C^N)$ or $-\det \otimes \mathcal{H}^2(\C^N)$,
respectively, again by the $O(N)$-modification rule 
\eqref{eqn:mod}. Hence, we have $\mathrm{pr}_{\mathcal{H}}([\lambda]) = 0$
in either case. Thus the claim is shown.
\end{proof}

We are ready to  complete the proof of Lemma \ref{lem:step2}.

\begin{proof}[Proof of Lemma \ref{lem:step2}]
By the branching law \eqref{eqn:UtoO} for the restriction $U(N) \downarrow O(N)$,
we have from Claim \ref{claim:1608105}
\begin{equation*}
F(U(N), (2^k,1^\ell))\vert_{O(N)} \simeq 
\bigoplus^k_{r=0} [(2^r, 1^\ell)].
\end{equation*}
Therefore
\begin{equation*}
\mathrm{pr}_{\mathcal{H}}(F(U(N),(2^k,1^\ell))\vert_{O(N)})
=
\begin{cases}
\mathcal{H}^0(\C^N) & (k,\ell) = (0,0),\\
\mathcal{H}^1(\C^N) & \ell =1 , k \geq 0,\\
\mathcal{H}^0(\C^N)\oplus \mathcal{H}^2(\C^N) & \ell = 0, k\geq 1,\\
0 & \text{otherwise}.
\end{cases}
\end{equation*}
Thus the lemma is proved.
\end{proof}

%%%%%%%%%%%%%%%%%%%%%%%%%%%%%%%%%%%%%%%%%%%%%%%
\subsection{Classification of $\mathrm{Hom}_{O(n-1)}\left(\Exterior^i(\C^n),\Exterior^j(\C^{n-1})\otimes\mathcal H^k(\C^{n-1})\right)$}\label{subsec:ijkbranch}
\index{A}{HkCN@$\mathcal H^k(\C^N)$, harmonic polynomials}

We recall that the group $ O(n-1)$ acts on $\mathfrak{n}_+ \simeq \C^n$ 
stabilizing the last coordinate $\zeta_n$, and thus acts also on 
\index{A}{N1+'@$\mathfrak{n}_+'$}
$\mathfrak{n}_+'=\mathfrak n_+\cap\{\zeta_n=0\}\simeq \C^{n-1}$, and thus
\index{A}{N1+@$\mathfrak{n}_+$, complex nilpotent Lie algebra}
we have an isomorphism
$\mathrm{Pol}(\mathfrak{n}_+)\simeq\mathrm{Pol}[\zeta_1,\ldots, \zeta_{n-1}]
\otimes \mathrm{Pol}[\zeta_n]$ as 
an $O(n-1)$-module.
In this section we determine explicitly $\mathrm{Hom}_{O(n-1)}\left(V\vert_{O(n-1)},
W\otimes\mathcal H(\C^{n-1})\right)$ 
for the $O(n)$-module $V=\Exterior^i(\C^n)$ and
the $O(n-1)$-module $W=\Exterior^j(\C^{n-1})$. The results will play a basic role in the classification of all differential symmetry
breaking operators
$\mathcal E^i(S^n)_{u,\delta}\To\mathcal E^j(S^{n-1})_{v,\eps}$.

The main result of this section is the following:
\begin{prop}\label{prop:153417}
Let $n \geq 2$. Suppose that $0 \leq i \leq n$, $0\leq j \leq n-1$, and $k \in \N$.
Then the following three conditions on $(i,j,k)$ are equivalent.
\begin{enumerate}
\item[(i)] 
$\displaystyle{\mathrm{Hom}_{O(n-1)}\left(\Exterior^i(\C^n),\Exterior^j(\C^{n-1})\otimes\mathcal H^k(\C^{n-1})\right)\neq \{0\}}$.\\

\item[(ii)]
$\displaystyle{\dim
\mathrm{Hom}_{O(n-1)}\left(\Exterior^i(\C^n),\Exterior^j(\C^{n-1})\otimes\mathcal H^k(\C^{n-1})\right)= 1}$.\\

\item[(iii)]
The triple $(i,j,k)$ belongs to one of the following cases:
\begin{enumerate}
\item[]Case 1: $j = i-2$ $(2\leq i \leq n)$, $k=1$,
\vskip 0.1in
\item[]Case 2: $j=i-1$
\begin{enumerate}
\item[] 2-a: $i=1$, $k=0,1$,
\item[] 2-b: $2\leq i \leq n-1$, $k=0,1,2$,
\item[] 2-c: $i=n$, $k=0$,
\end{enumerate}
\vskip 0.1in
\item[] Case 3: $j=i$:
\begin{enumerate}
\item[] 3-a: $i=0$, $k=0$,
\item[] 3-b: $1\leq i \leq n-2$, $k=0,1,2$,
\item[] 3-c: $i=n-1$, $k=0,1$,
\end{enumerate}
\vskip 0.1in
\item[]Case 4: $j=i+1$ $(0\leq i\leq n-2)$, $k=1$.
\end{enumerate}

\end{enumerate}

\end{prop}

Explicit generator $h^{(k)}_{i\to j}$ in 
$\mathrm{Hom}_{O(n-1)}\left(\Exterior^i(\C^n),\Exterior^j(\C^{n-1}) \otimes
\mathcal{H}^k(\C^{n-1})\right)$
will be given in \eqref{eqn:Hi--}--\eqref{eqn:Hi+} below. 
We begin with the following elementary lemma:

\begin{lem}\label{lem:On-1deco}
As an $O(n-1,\C)$-module, $V=\Exterior^i(\C^n)$ decomposes as
$$
\Exterior^i(\C^n)=\Exterior^i(\C^{n-1})\oplus \Exterior^{i-1}(\C^{n-1}).
$$
\end{lem}

The spaces $\Exterior^i(\C^{n-1})$ and $\Exterior^{i-1}(\C^{n-1})$ have bases
$\{e_I : I \in \mathcal{I}_{n-1, i}\}$ and $\{e_I : I \in \mathcal{I}_{n-1, i-1}\}$, respectively.
We normalize the first and the second projections by
\index{A}{prij@$\mathrm{pr}_{i\to j}\colon \Exterior^i(\C^n) \To \Exterior^j(\C^{n-1})$|textbf}
\begin{eqnarray}
\mathrm{pr}_{i \to i} (e_I) &:=&
\begin{cases}
e_I & \text{if $n \notin I$},\\
0 & \text{if $n \in I$},
\end{cases}
\label{eqn:prii}\\
\mathrm{pr}_{i \to i-1} (e_I) &:=&
\begin{cases}
0 & \text{if $n \notin I$},\\
(-1)^{i-1}e_{I\setminus\{n\}} & \text{if $n \in I$}.
\end{cases}
\label{eqn:prii-}
\end{eqnarray}

The signature of $\mathrm{pr}_{i\to i-1}$ is taken in a way that it fits with the 
\index{B}{interior multiplication}
interior multiplication 
\index{A}{1iotan@$\iotan$, interior multiplication}
$\iotan$ for differential forms (see \eqref{eqn:intn}).

\begin{proof}[Proof of Proposition \ref{prop:153417}]
By Lemma \ref{lem:On-1deco}, the proof reduces to Lemma \ref{lem:Altharmonic} with $N=n-1$. In fact, explicit generator 
\index{A}{H0hijk@$h^{(k)}_{i\to j}$|textbf}
$h^{(k)}_{i\to j}$ is given as follows:
\begin{alignat}{3}
&\text{Case $j=i-2$:} \;\;
&&h^{(1)}_{i \to i-2} :=\widetilde H_{i-1\to i-2}^{(1)}\circ \mathrm{pr}_{i\to i-1}. 
&&\label{eqn:Hi--}\\
&\text{Case $j=i-1$:} \;\;
&&h^{(k)}_{i \to i-1}
:=\widetilde H_{i-1\to i-1}^{(k)}\circ \mathrm{pr}_{i\to i-1}\; \text{($k=0,2$)}, \;\;
&&h^{(1)}_{i\to i-1}:=\widetilde H_{i\to i-1}^{(1)}\circ \mathrm{pr}_{i\to i}.
\label{eqn:Hi-}\\
&\text{Case $j=i$:}\;\;
&&h^{(k)}_{i\to i}:=\widetilde H_{i\to i}^{(k)}\circ \mathrm{pr}_{i\to i} \;\text{($k=0,2$)},\;\;
&&h^{(1)}_{i\to i}:=\widetilde H_{i-1\to i}^{(1)}\circ \mathrm{pr}_{i\to i-1}.
\label{eqn:Hi}\\
&\text{Case $j=i+1$:}\;\;
&&h^{(1)}_{i \to i+1}:=\widetilde H_{i\to i+1}^{(1)}\circ \mathrm{pr}_{i\to i}. 
&&\label{eqn:Hi+}
\end{alignat}

Here we have applied \eqref{eqn:Htildekij} to $N=n-1$ for 
$\widetilde{H}^{(k)}_{i\to j}$ in the above formula.

We see from \eqref{eqn:H2zero} that some of 
these operators vanish, namely,
\begin{equation}\label{eqn:hijkzero}
h^{(2)}_{1\to 0} =0, \quad
h^{(1)}_{n \to n-1}= h^{(2)}_{n \to n-1} = 0,\quad
h^{(1)}_{0\to 0} = h^{(2)}_{0\to 0} = 0,\quad
h^{(2)}_{n-1 \to n-1} = 0,
\end{equation}
and that $h^{(k)}_{i\to j} \neq 0$ as far as $(i,j,k)$ satisfies the condition (iii)
in Proposition \ref{prop:153417}. 
Hence we have shown Proposition \ref{prop:153417}.
\end{proof}

We shall use the basis 
\index{A}{H0ii-k@$h^{(k)}_{i\to i-1}$}
$h^{(k)}_{i \to i-1}$ in Chapter \ref{sec:7}, 
and $h^{(1)}_{i \to i+1}$ in Chapter \ref{sec:codiff}, respectively.
For later purpose, we give explicit formul\ae{} of 
$h^{(k)}_{i\to j}\left(e_I\right)$ for $I \in \mathcal{I}_{n,i}$
in Table \ref{table:Table160487}.
The proof is immediate from \eqref{eqn:F0}-\eqref{eqn:H2tilde}
and the definitions \eqref{eqn:prii}--\eqref{eqn:Hi+}.
Here 
\index{A}{QwI@$\widetilde{Q}_I(\zeta)$}
we recall from \eqref{eqn:QItilde} that 
\index{A}{Qn11@$Q_{n-1}(\zeta')$}
$\widetilde{Q}_J(\zeta') = 
\sum_{\ell \in J} \zeta_\ell^2 - \frac{i}{n-1}Q_{n-1}(\zeta')$ 
for $\zeta' =(\zeta_1, \ldots, \zeta_{n-1})$
and $J \in \mathcal{I}_{n-1,i}$.

\begin{table}[h]
\caption{Formul\ae{} of $h^{(k)}_{i \to j}\left(e_{I}\right)$ for $I \in \mathcal{I}_{n,i}$}
\begin{center}
\begin{tabular}{ccc|c|ccc}
&&&$n \notin I$ & $n \in I$&&\\
\cline{3-5} 
\rule{0pt}{3ex}
&&
$h^{(0)}_{i \to i-1}\left(e_{I}\right)$ & $0$ & $(-1)^{i-1}e_{I\setminus \{n\}}$
&&\\
&&
$h^{(0)}_{i \to i}\left(e_{I}\right)$ & $e_I$ & $0$
&&
\\[1pt]
\cline{3-5} 
\rule{0pt}{3ex}
&&
%%%%%%%%
$h^{(1)}_{i \to i-2}\left(e_{I}\right)$ & $0$ & 
$\displaystyle{-\sum_{\ell \in I\setminus\{n\}}\sgn(I;\ell,n)e_{I\setminus \{\ell, n\}}\zeta_\ell}$
&&\\
&&
$h^{(1)}_{i \to i-1}\left(e_{I}\right)$ & 
$\displaystyle{\sum_{\ell \in I}\sgn(I;\ell)e_{I\setminus \{\ell\}}\zeta_\ell}$ & $0$
&&\\
&&
$h^{(1)}_{i \to i}\left(e_{I}\right)$ & $0$ & 
$\displaystyle{\sum_{\ell \notin I}\sgn(I;\ell,n)
e_{I\setminus \{ n\}\cup \{\ell\}}\zeta_\ell}$
&&\\
&&
$h^{(1)}_{i \to i+1}\left(e_{I}\right)$ & 
$\displaystyle{\sum_{\substack{\ell \notin I\\ \ell \neq n}}\sgn(I;\ell)e_{I \cup\{\ell\}}\zeta_\ell}$
 & $0$
 &&\\
%%%%%%%%
\multicolumn{7}{c}{}\\
\multicolumn{7}{c}{
$\begin{aligned}
h^{(2)}_{i \to i-1}\left(e_I\right)
&=
\begin{cases}
0 & \text{if $n \notin I$},\\
\displaystyle{(-1)^{i-1}
\left(\widetilde{Q}_{I\setminus \{n\}}(\zeta') e_{I\setminus \{n\}}
+\sum_{p\in I\setminus\{n\}} \sum_{q \notin I}
\sgn\left(I;p,q\right)\zeta_p \zeta_q
e_{I \setminus\{p,n\}\cup \{q\}}\right)}
 & \text{if $n \in I$}.
\end{cases}\\
&&\\
h^{(2)}_{i \to i}\left(e_I\right)
&=
\begin{cases}
\widetilde{Q}_I\left(\zeta'\right)e_I + 
\displaystyle{\sum_{p \in I} \sum_{\substack{q \notin I\\ q \neq n}}\sgn(I;p,q)\zeta_p\zeta_q
e_{I\setminus \{p\} \cup \{q\}}} 
& \text{if $n \notin I$},\\
0 & \text{if $n \in I$}.
\end{cases}
\end{aligned}$
}
\end{tabular}
\end{center}
\label{table:Table160487}
\end{table}

%%%%%%%%%%%%%%%%%%%%%%%%%%%%%%%%%%%%%%%%%%%%%%%
\subsection{Descriptions of $\mathrm{Hom}_{O(n-1)}
(\Exterior^i(\C^n),\Exterior^j(\C^{n-1})\otimes\mathrm{Pol}[\zeta_1,\ldots, \zeta_n])$}
\label{subsec:4VWPol}${}$

It follows from Lemma \ref{lem:psi} and Proposition \ref{prop:153417} that 
the spaces $\mathrm{Hom}_{O(n-1)}(V,W\otimes \mathrm{Pol}(\mathfrak{n}_+))$
are determined explicitly for $(V,W) = (\Exterior^i(\C^n),\Exterior^j(\C^{n-1}))$
as follows:

\begin{prop}\label{prop:Altpol}
Let $n \geq 2$, $0\leq i \leq n$, and $0\leq j \leq n-1$. Then the
following two conditions on $(i,j)$ are equivalent:
\begin{enumerate}
\item[(i)] $\mathrm{Hom}_{O(n-1)}\left(\Exterior^i(\C^n),
\Exterior^j(\C^{n-1})\otimes\mathrm{Pol}[\zeta_1,\ldots, \zeta_n]\right)\neq\{0\}$.
\item[(ii)] $j\in\{i-2, i-1, i, i+1\}$. 
\end{enumerate}
\end{prop}

\begin{prop}\label{prop:Tgh} 
Let $n\geq 2$.
Suppose that 
$0\leq i \leq n$, $0\leq j \leq n-1$, and
$a \in \mathbb{N}$.
Then 
$\mathrm{Hom}_{O(n-1)}\left(\Exterior^i(\C^n),\Exterior^{j}(\C^{n-1})
\otimes\mathrm{Pol}^a[\zeta_1, \ldots, \zeta_n]\right)$ is equal to
\begin{alignat*}{2}
&\{(T_{a-1}g)h^{(1)}_{i\to i-2}\colon g \in\mathrm{Pol}_{a-1}[t]_{\mathrm{even}}\}
\; \; &&\mathrm{if} \;  j = i-2,\\
&\C \operatorname{-span}
\left\{(T_{a-k}g_k)h^{(k)}_{i \to i-1}\colon  g_k\in\mathrm{Pol}_{a-k}[t]_{\mathrm{even}}\right\}
\; \; &&\mathrm{if} \;  j = i-1,\\
&\C \operatorname{-span}
\left\{(T_{a-k}g_k)h^{(k)}_{i \to i}\colon  g_k\in\mathrm{Pol}_{a-k}[t]_{\mathrm{even}}\right\}
&&\mathrm{if} \;  j = i,\\
&\left\{(T_{a-1}g)h^{(1)}_{i \to i+1}\colon  g\in\mathrm{Pol}_{a-1}[t]_{\mathrm{even}}\right\}
\; \;&&\mathrm{if} \;  j = i+1,\\
&\{0\} 
\; \; &&\mathrm{otherwise}.
\end{alignat*}
Here we regard $\mathrm{Pol}_{-1}[t]_{\mathrm{even}} = \{0\}$.
We also regard
\begin{align}
h^{(k)}_{i \to i-1} &= 0 \quad \emph{when $(i,k)=(1,2), (n,1)$, or $(n,2)$},
\label{eqn:h-van}\\
h^{(k)}_{i\to i} &=0 \quad \emph{when $(i,k) = (0,1), (0,2)$, or $(n-1,2)$}.
\label{eqn:hivan}
\end{align}
\end{prop}

We note that when $j = i$, $i-1$, we have $k\leq\min\{2,a\}$.

%%%%%%%%%%%%%%%%%%%%%%%%%%%%%%%%%%%%%%%%%%%%%%%
\subsection{Proof of the implication (i)$\Rightarrow$(iii) in Theorem \ref{thm:1A}}
\label{subsec:4pfThm1}

In this section, we give a proof of the implication (i) $\Rightarrow$ (iii) 
in Theorem \ref{thm:1A}.

We recall that characters of $A$ are parametrized by $\C$ via the normalization 
\eqref{eqn:Clmde}.
For $0 \leq i \leq n$, $\alpha \in \Z/2\Z$, and $\lambda \in \C$,
we denote by 
\index{A}{1sigma-lambda-alpha@$\sigma^{(i)}_{\lambda, \alpha}$, representation of $P$ on $\Exterior^i(\C^n)$|textbf}
$\sigma^{(i)}_{\lambda, \alpha}$ the outer tensor product
representation $\Exterior^i(\C^n) \boxtimes (-1)^\alpha \boxtimes \C_\lambda$
of the Levi subgroup $L = MA \simeq O(n) \times O(1) \times \R$ on the $i$-th
exterior tensor space $\Exterior^i(\C^n)$.
Similarly, 
\index{A}{1tau-nu-beta@$\tau^{(j)}_{\nu,\beta}$, representation of $P'$ on $\Exterior^j(\C^{n-1})$|textbf}
$\tau^{(j)}_{\nu,\beta}$ ($0\leq j \leq n-1$, $\nu \in \C$, $\beta\in \Z/2\Z$)
stands for the outer tensor product representation 
$\Exterior^j(\C^{n-1}) \boxtimes (-1)^{\beta} \boxtimes \C_\nu$
of the Levi subgroup
$L' = M'A \simeq O(n-1) \times O(1) \times \R$.

\begin{lem}\label{lem:Lprime}
Suppose that $n\geq 2$.
Let $0\leq i\leq n$, $0\leq j\leq n-1$, $\lambda,\nu\in\C$,
$\alpha,\beta\in\Z/2\Z$ and $a\in\N$. Then the following two conditions 
on $(i, j, \lambda, \nu, \alpha, \beta, a)$
are equivalent:
\begin{enumerate}
\item[]$\mathrm{(i)}$ 
$\displaystyle{
\mathrm{Hom}_{L'}\left(
\sigma^{(i)}_{\lambda,\alpha}\vert_{L'}, \tau^{(j)}_{\nu,\beta}
\otimes\mathrm{Pol}^a(\mathfrak n_+)\right)\neq\{0\}}$.

\item[]$\mathrm{(ii)}$
$\displaystyle{
 j\in\{i-2,i-1,i,i+1\},\quad \nu-\lambda=a,\quad\mathrm{and}
\quad \beta-\alpha \equiv a\,\mathrm{mod}\,2}$.\\
\text{ }\quad Moreover, $a \geq 1$ when $j = i-2$ or  $i+1$.
\end{enumerate}
\end{lem}

\begin{proof}
First of all, we consider the actions of the second and third factors of $L'\simeq O(n-1)\times O(1)\times\R$. Since $e^{tH_0}\in A$ and $-1\in O(1)$ act on $\mathfrak n_+\simeq\C^n$ as the scalars
$e^t$ and $-1$, respectively,
\begin{equation*}
\mathrm{Hom}_{O(1)\times A}\left(
\sigma^{(i)}_{\lambda,\alpha}, \tau^{(j)}_{\nu,\beta}
\otimes\mathrm{Pol}^a(\mathfrak n_+)\right)\neq\{0\}
\end{equation*}
if and only if 
\begin{equation*}
\nu=\lambda+a\quad\mathrm{and}\quad \beta\equiv\alpha+a\,\mathrm{mod}\,2.
\end{equation*}
Then the proof of the lemma reduces to Proposition \ref{prop:Tgh} for the action of the first
factor $O(n-1)$ of $L'$.
\end{proof}

We recall
from Section \ref{subsec:ps} that 
\index{A}{Iilambda@$I(i,\lambda)_\alpha$, principal series of $O(n+1,1)$}
$I(i,\lambda)_\alpha$ and 
\index{A}{Jjnu@$J(j,\nu)_\beta$, principal series of $O(n,1)$}
$J(j,\nu)_\beta$ are (unnormalized) principal series representations
of $G$ and $G'$, respectively. 
By the F-method summarized in Fact \ref{fact:Fmethod},
we have a natural bijection
\index{A}{Sol2@$Sol(\mathfrak n_+;\sigma^{(i)}_{\lambda,\alpha},
\tau^{(j)}_{\nu,\beta})$}
\begin{equation}\label{eqn:153091}
\mathrm{Diff}_{G'}(I(i,\lambda)_\alpha, J(j,\nu)_\beta)
\simeq 
Sol(\mathfrak{n}_+; \sigma^{(i)}_{\lambda,\alpha}, \tau^{(j)}_{\nu,\beta}),
\end{equation}
where we recall from \eqref{eqn:24}
\begin{equation*}
Sol(\mathfrak{n}_+; \sigma^{(i)}_{\lambda,\alpha} , 
\tau^{(j)}_{\nu,\beta}) = 
\{\psi \in \mathrm{Hom}_{L'}(\sigma^{(i)}_{\lambda,\alpha} \vert_{L'},
\tau^{(j)}_{\nu,\beta} \otimes \mathrm{Pol}(\mathfrak{n}_+)):
\widehat{d\pi_{(i,\lambda)^*}}(C)\psi=0\; \text{for all $C \in \mathfrak{n}_+'$}\}.
\end{equation*}

\begin{prop}\label{prop:153091}
Suppose $0\leq i\leq n$, $0\leq j\leq n-1$, $\lambda,\nu\in\C$, and
$\alpha,\beta\in\Z/2\Z$. Then,
\begin{enumerate}
\item $\mathrm{Diff}_{G'}(I(i,\lambda)_\alpha, J(j,\nu)_\beta)\neq\{0\}$ only if
\begin{equation}\label{eqn:Fnonzero}
j\in\{i-2,i-1,i,i+1\},\,\nu-\lambda\in\N,\,\mathrm{and}\, \beta-\alpha\equiv\nu-\lambda\,
\mathrm{mod}\,2.
\end{equation}

\item Suppose \eqref{eqn:Fnonzero} is satisfied. 
Then,
\index{A}{dpi6ihat@$\widehat{d\pi_{(i,\lambda)^*}}$}
\begin{equation*}
Sol(\mathfrak{n}_+; \sigma^{(i)}_{\lambda,\alpha}, \tau^{(j)}_{\nu,\beta})=
\left\{\psi\in 
\mathrm{Hom}_{O(n-1)}\left(\Exterior^i(\C^n),
\Exterior^j(\C^{n-1})\otimes\mathrm{Pol}^{\nu-\lambda}(\mathfrak n_+)\right) : 
\text{$\widehat{d\pi_{(i,\lambda)^*}}(N_1^+)
\psi=0$}\right\}.
\nonumber
\end{equation*}
\end{enumerate}
\end{prop}

\begin{proof}
(1) The first assertion is a direct consequence of \eqref{eqn:153091} and 
Lemma \ref{lem:Lprime}.

(2) Suppose \eqref{eqn:Fnonzero} is fulfilled. Then, it follows from the proof of Lemma \ref{lem:Lprime} that
\begin{equation*}
\mathrm{Hom}_{L'}\left(
\sigma^{(i)}_{\lambda,\alpha}\vert_{L'}, \tau^{(j)}_{\nu,\beta}
\otimes\mathrm{Pol}(\mathfrak n_+)\right)
\simeq
\mathrm{Hom}_{O(n-1)}\left(\Exterior^i(\C^n),
\Exterior^j(\C^{n-1})\otimes\mathrm{Pol}^{\nu-\lambda}(\mathfrak n_+)\right).
\end{equation*}
Hence the second statement follows.
\end{proof}

Owing to Proposition \ref{prop:Tgh} and Proposition \ref{prop:153091} (2), 
the 
\index{B}{F-system}
F-system for $Sol(\mathfrak{n}_+; \sigma^{(i)}_{\lambda,\alpha},
\tau^{(j)}_{\nu,\beta})$ boils down to a system of ordinary differential equations 
of $g_j(t)$ $(j = 0, 1, 2)$.
We shall find explicitly the polynomials $g_j(t)$, and determine 
$Sol\left(\mathfrak{n}_+; \sigma^{(i)}_{\lambda,\alpha}, \tau^{(j)}_{\nu,\beta} \right)$
for $j = i -1$ 
in Chapter \ref{sec:7}, 
and for $j=i+1$ in Chapter \ref{sec:codiff}.

%%%%%%%%%%%%%%%%%%%%%%%%%%%%%%%%%%%%%%%%%%%%%%%

\newpage
%%%%%%%%%%%%%%%%%%%%%%%%%%%%%%%%%%%%%%%%%%%%%%%
\section{F-system for symmetry breaking operators ($j=i-1, i$ case)}\label{sec:7}

As we discussed in the previous chapter,
the 
\index{B}{F-method}
F-method (see Fact \ref{fact:Fmethod}) establishes
a natural bijection \eqref{eqn:153091} between the space
$\mathrm{Diff}_{G'}\left(I(i,\lambda)_\alpha,  J(j, \nu)_\beta \right)$
of differential symmetry breaking operators and the space
\index{A}{1tau-nu-beta@$\tau^{(j)}_{\nu,\beta}$, representation of $P'$ on $\Exterior^j(\C^{n-1})$}
$Sol\left(\mathfrak{n}_+; \sigma^{(i)}_{\lambda,\alpha}, \tau^{(j)}_{\nu,\beta} \right)$
of $\mathrm{Hom}_{\C}\left(\Exterior^i(\C^n), \Exterior^j(\C^{n-1})\right)$-valued
polynomial solutions to the F-system.

In this chapter, we study the 
\index{B}{F-system}
F-system in detail for $j=i-1$. The case $j=i+1$ will be investigated
in the next chapter. Via the duality theorem (see Theorem \ref{thm:psdual}),
the cases $j=i, i-2$ are understood as the dual
to the cases $j=i-1, i+1$, respectively. 
The results of this chapter ($j=i-1$ case)
are summarized as follows.
We recall from \eqref{eqn:Hi-} that 
\index{A}{H0ii-k@$h^{(k)}_{i\to i-1}$}
$h^{(k)}_{i \to i-1} \colon \Exterior^i(\C^n) \To \Exterior^{i-1}(\C^{n-1}) \otimes 
\mathcal{H}^k(\C^{n-1})$ are $O(n-1)$-homomorphisms for $k=0, 1$, and $2$.
Let
\index{A}{C1ell1@$\widetilde C_\ell^\mu(t)$, renormalized Gegenbauer polynomial}
$\widetilde{C}^\mu_{\ell}(t)$
be the 
\index{B}{renormalized Gegenbauer polynomial}
renormalized Gegenbauer polynomial (see \eqref{eqn:Gegen2}).

\begin{thm}\label{thm:Fi-}
Let $n \geq 3$.
Suppose $1\leq i \leq n$, $\lambda,\nu \in \C$ and $\alpha, \beta \in\Z/2\Z$.
Let $\sigma^{(i)}_{\lambda,\alpha}$, $\tau^{(i-1)}_{\nu,\beta}$
be the 
\index{A}{1sigma-lambda-alpha@$\sigma^{(i)}_{\lambda, \alpha}$, representation of $P$ on $\Exterior^i(\C^n)$}
outer tensor product representations of 
\index{A}{L1L@$L=MA$, Levi part of $P$}
$L=MA\simeq
O(n)\times O(1) \times \R$, 
\index{A}{L1L'@$L'=M'A$, Levi part of $P'$}
$L'=M'A\simeq O(n-1) \times O(1) \times \R$ on 
$\Exterior^i(\C^n) \boxtimes
(-1)^\alpha \boxtimes \C_\lambda$,
$\Exterior^{i-1}(\C^{n-1}) \boxtimes (-1)^\beta \boxtimes \C_\nu$,
respectively.
Then
\begin{equation}\label{eqn:Soli-}
Sol\left(\mathfrak{n}_+; \sigma^{(i)}_{\lambda,\alpha}, \tau^{(i-1)}_{\nu,\beta}\right)
=
\begin{cases}
\C 
\displaystyle{\sum_{k=0}^2\left(T_{\nu-\lambda-k}g_k \right) h^{(k)}_{i \to i-1}}
& \text{\emph{if $\nu - \lambda \in \N$ and 
$\beta-\alpha \equiv \nu-\lambda \; \mathrm{mod} \; 2$},}\\
\{0\} & \text{\emph{otherwise}}.
\end{cases}
\end{equation}
\index{A}{Sol2@$Sol(\mathfrak n_+;\sigma^{(i)}_{\lambda,\alpha},
\tau^{(j)}_{\nu,\beta})$}

From now, we assume $a:=\nu-\lambda \in \N$ and 
$\beta-\alpha \equiv a \; \mathrm{mod}\; 2$. We consider 
the following polynomials:  
\begin{eqnarray}
&&e^{-\frac{\pi\sqrt{-1}}{2}}B\,t\,\widetilde C_{a-1}^{\lambda-\frac{n-3}2}\left(e^{\frac{\pi\sqrt{-1}}{2}}t\right)+
C\,\widetilde C_{a-2}^{\lambda-\frac{n-3}2}\left(e^{\frac{\pi\sqrt{-1}}{2}}t\right),
\label{eqn:g0}\\
&&e^{-\frac{\pi\sqrt{-1}}{2}}A\widetilde C_{a-1}^{\lambda-\frac{n-3}2}
\left(e^{\frac{\pi\sqrt{-1}}{2}}t\right),\label{eqn:g1}\\
&&\widetilde C_{a-2}^{\lambda-\frac{n-3}2}\left(e^{\frac{\pi\sqrt{-1}}{2}}t\right),
\label{eqn:g2}
\end{eqnarray}
where
\begin{equation*}
A=\gamma(\lambda-\frac{n-1}{2},a), \quad
B=A\left(1+\frac{\lambda-n+i}{a}\right), \quad
C=\frac{\lambda-n+i}{a}+\frac{i-1}{n-1}
\end{equation*}
with $\gamma\left(\mu,a\right) =1$ $(a: \emph{odd})$;
$=\mu+\frac{a}{2}$ $(a: \emph{even})$\;
\emph{(see \eqref{eqn:gamma})}.
Then the polynomials $g_k(t)$ $(k=0, 1, 2)$ are
given as follows.

\begin{alignat*}{3}
&\emph{(1)}
\quad
&&\text{$i=1$, $a\geq 1:$}
\quad
&&\text{$(g_0(t), g_1(t), g_2(t)) = (\eqref{eqn:g0}, \eqref{eqn:g1}, 0)$};\\
&\emph{(2)} 
\quad
&&\text{$2\leq i \leq n-1$, $a \geq 1:$}
\quad
&&\text{$(g_0(t), g_1(t), g_2(t)) = (\eqref{eqn:g0}, \eqref{eqn:g1},\eqref{eqn:g2})$};\\
&\emph{(3)}
\quad
&&\text{$i=n$, $a \geq 1:$}
\quad
&&\text{$(g_0(t), g_1(t), g_2(t)) = \left(\widetilde{C}^{\lambda-\frac{n-3}{2}}_a
\left(e^{\frac{\pi\sqrt{-1}}{2}}t\right), 0, 0 \right)$};\\
&\emph{(4)}
\quad
&&\text{$1\leq i \leq n$, $a=0:$}
\quad
&&\text{$(g_0(t), g_1(t), g_2(t)) = \left(1, 0, 0 \right)$}.
\end{alignat*}

\end{thm}

\begin{rem}
The exceptional cases (3) and (4) in Theorem \ref{thm:Fi-} are closely related to
the vanishing conditions of the family of the symmetry breaking operators 
\index{A}{Dii1@$\mathcal D_{u,a}^{i\to i-1}$}
$\mathcal{D}^{i\to i-1}_{u,a}$
given in Proposition \ref{prop:Dnonzero}.
As we introduced the renormalized
operator $\widetilde{\mathcal{D}}^{i \to i-1}_{u,a}$ in \eqref{eqn:reDii1},
we separated (3) and (4) from the others.
The relationship will  be clarified in Section \ref{subsec:pfThm} where
we explain how the triple $(g_0, g_1, g_2)$ 
of polynomials gives rise to the 
\index{B}{renormalized differential symmetry  breaking operator}
differential symmetry breaking operator 
\index{A}{Dii4@$\widetilde {\mathcal D}_{u,a}^{i\to i-1}$}
$\widetilde{\mathcal{D}}^{i\to i-1}_{u,a}$ $(=\widetilde{\C}^{i,i-1}_{\lambda,\nu})$
from $\mathcal{E}^i(\R^n)$ to $\mathcal{E}^{i-1}(\R^{n-1})$ as stated in Theorem \ref{thm:Cps}.
\end{rem}

The proof of Theorem \ref{thm:Fi-} is divided into two parts:

\begin{itemize}

\item to reduce a system of partial differential equations (F-system)
to a system of ordinary differential equations on $g_k(t)$ $(k=0,1,2)$ 
(see Theorem \ref{thm:7}).

\item to find explicit polynomial solutions $\{g_k(t)\}$ to the latter system
(see Theorem \ref{thm:gis}).

\end{itemize}

In the next section
we first complete the proof of Theorem \ref{thm:1A} for $j = i-1,i$
by admitting Theorem \ref{thm:Fi-}.

%%%%%%%%%%%%%%%%%%%%%%%%%%%%%%%%%%%%%%%%%%%%%%%
\subsection{Proof of Theorem \ref{thm:1A} for $j = i-1,i$}\label{subsec:pfThm1A}

In this section, we prove that Theorem \ref{thm:Fi-} determines the 
dimension of the space of differential symmetry breaking operators
from principal series representations
$I(i,\lambda)_\alpha$ of $G=O(n+1,1)$ 
to $J(j,\nu)_\beta$ of $G'=O(n,1)$ when
$j=i-1, i$. The following two theorems correspond to Theorem \ref{thm:1A} 
in the cases $j=i-1$ and $i$, respectively.

\begin{thm}[$j=i-1$ case]\label{thm:Ai-}
Let $n \geq 3$.
Suppose $1\leq i\leq n$, $\lambda,\nu \in\C$, and $\alpha, \beta \in \Z/2\Z$.
Then the following three conditions on $(i,\lambda,\nu,\alpha, \beta)$ are equivalent:
\begin{enumerate}
\item[(i)] $\mathrm{Diff}_{G'}\left(I(i,\lambda)_\alpha, J(i-1, \nu)_\beta \right)\neq \{0\}$.
\item[(ii)] $\dim \mathrm{Diff}_{G'}\left(I(i,\lambda)_\alpha, J(i-1, \nu)_\beta \right) =1$. 
\item[(iii)] $\nu -\lambda \in \N$, $\alpha - \beta \equiv \nu -\lambda \; \mathrm{mod}\;2$.
\end{enumerate}
\end{thm}

\begin{thm}[$j=i$ case]\label{thm:Ai}
Let $n \geq 3$.
Suppose $0\leq i \leq n-1$, $\lambda, \nu \in\C$ and $\alpha, \beta \in \Z/2\Z$.
Then the following three conditions on $(i,\lambda,\nu,\alpha, \beta)$ are equivalent:
\begin{enumerate}
\item[(i)] $\mathrm{Diff}_{G'}\left(I(i,\lambda)_\alpha, J(i, \nu)_\beta \right)\neq \{0\}$.
\item[(ii)] $\dim \mathrm{Diff}_{G'}\left(I(i,\lambda)_\alpha, J(i, \nu)_\beta \right) =1$. 
\item[(iii)] $\nu -\lambda \in \N$, $\alpha - \beta \equiv \nu -\lambda \; \mathrm{mod}\;2$.
\end{enumerate}
\end{thm}

\begin{proof}[Proof of Theorem \ref{thm:Ai-}]
By the general theory of the F-method (see \eqref{eqn:153091}),
we have the vector space isomorphism \eqref{eqn:153091}.
Thus the equivalence will follow if we show that the solutions
in Theorem \ref{thm:Fi-} are nonzero when the condition (iii) is satisfied.

For this, we observe that 
we renormalized the Gegenbauer polynomial in a way that
\index{A}{C1ell1@$\widetilde C_\ell^\mu(t)$,
renormalized Gegenbauer polynomial}
$\widetilde{C}^\alpha_{\ell}(t)$ is nonzero for all $\alpha \in \C$ and $\ell \in\N$
(see Section \ref{subsec:Apx1}). We also know that 
$h^{(k)}_{i \to i-1} \neq 0$ except for the cases 
$(i,k)=(1,2), (n,1)$, or $(n,2)$ (see \eqref{eqn:h-van}).
Moreover, for each $i$ $(1\leq i \leq n)$, these nonzero maps
$h^{(k)}_{i \to i-1}$ ($k=0,1,2$) are  linearly independent by 
Proposition \ref{prop:153417}. 
Hence the solutions constructed in Theorem \ref{thm:Fi-} are nonzero
by the decomposition \eqref{eqn:HPP}.
Now the desired statement is proved. 
\end{proof}

For the proof of Theorem \ref{thm:Ai}, we use 
\index{B}{duality theorem for symmetry breaking operators (principal series)}
the duality theorem for symmetry  breaking operators
(Theorem \ref{thm:psdual}) instead of solving the F-system.

\begin{proof}[Proof of Theorem \ref{thm:Ai}]
It follows from Theorem \ref{thm:psdual} that we have a natural bijection
\begin{equation*}
\mathrm{Diff}_{G'}\left(I(i,\lambda)_\alpha, J(i,\nu)_\beta\right)\simeq
\mathrm{Diff}_{G'}\left(I\left(\tilde{i}, \lambda\right)_\alpha, 
J\left(\tilde{i}-1, \nu\right)_\beta\right),
\end{equation*}
where $\tilde{i}:=n-i$. Then it is easy to see that $(\tilde{i}, \lambda,\nu, \alpha, \beta)$
satisfies the condition (iii) in Theorem \ref{thm:Ai-} if and only if 
$(i,\lambda,\nu,\alpha,\beta)$ satisfies (iii) in Theorem \ref{thm:Ai}. 
Hence Theorem \ref{thm:Ai} is deduced from Theorem \ref{thm:Ai-}.
\end{proof}

%%%%%%%%%%%%%%%%%%%%%%%%%%%%%%%%%%%%%%%%%%%%%%%
\subsection{Reduction theorem}\label{subsec:71}

We begin by stating the main theorem of the rest of this chapter.
Recall from Section \ref{subsec:3ortho} that,
for $\mu\in\C$ and $\ell\in\N$, $R_\ell^\mu$ denotes the following
differential operator
\begin{equation*}
\index{A}{Rell@$R_\ell^\lambda$, 
imaginary Gegenbauer differential operator}
R_\ell^\mu=-\frac{1}{2}\left((1+t^2)\frac{d^2}{dt^2} 
+ (1+2\mu)t\frac{d}{dt}-\ell(\ell+2\mu)\right),
\end{equation*}
equivalently,
\begin{equation}\label{eqn:Rla1}
R_\ell^\mu=-\frac1{2t^2}\left((1+t^2)\vartheta_t^2-(1-2\mu t^2)\vartheta_t
-\ell(\ell+2\mu)t^2\right)
\end{equation}
with $\vartheta_t:=t\frac{d}{dt}$. 
For polynomials $g_j(t)$ $(j=0,1,2)$ of one variable $t$,
we then define other polynomials 
\index{A}{L1Ljg@$L_j(g_0, g_1, g_2)$|textbf}
$L_r(g_0,g_1, g_2)(t)$ 
of the same variable $t$ ($r=1,2,\ldots, 7$) as follows:
\begin{align}
&L_1(g_0, g_1, g_2) :=R_{a-2}^{\lambda-\frac{n-3}2}g_2, \label{eqn:Fe}\\
&L_2(g_0, g_1, g_2) :=R_{a-1}^{\lambda-\frac{n-3}2}g_1,\label{eqn:Fd}\\
&L_3(g_0, g_1, g_2) :=(a-1-\vartheta_t)g_1-\frac{dg_2}{dt},\label{eqn:Fstar}\\
&L_4(g_0, g_1, g_2) :=
(\vartheta_t+2\lambda+a-n+1)g_2-\frac{dg_1}{dt},\label{eqn:Fa}\\
&L_5(g_0, g_1, g_2) :=
\frac{dg_0}{dt}+\frac{n-i}{n-1}\frac{dg_2}{dt}-(a+\lambda-n+i)g_1,\label{eqn:Fb}\\
&L_6(g_0, g_1, g_2) :=
(\lambda-n+i+a-\frac{n-i}{n-1}(a-\vartheta_t))g_2-(a-\vartheta_t)g_0,\label{eqn:Fac}\\
&L_7(g_0, g_1, g_2) :=
R_a^{\lambda-\frac{n-1}2}g_0+\frac{n-i}{n-1}\frac{dg_1}{dt}.\label{eqn:Ff}
\end{align}
For later convenience we also set
$L_{8}(g_0,g_1, g_2):=L_6(g_0,g_1, g_2)-tL_5(g_0,g_1, g_2)$, 
$L_{9}(g_0,g_1,g_2):=\frac{n-i}{n-1}L_3(g_0,g_1,g_2)+L_5(g_0,g_1,g_2)$,
namely,
\begin{align}
L_{8}(g_0,g_1, g_2) 
&= \left(\lambda-n+i+\frac{a(i-1)}{n-1}\right)g_2(t)+(a+\lambda-n+i)tg_1(t)-ag_0(t),
\label{eqn:Fabc}\\
L_9(g_0,g_1,g_2)
&=\frac{dg_0}{dt} + \left(\frac{i-1}{n-1}(\vartheta_t - n- a+2) 
- (\vartheta_t + \lambda-n+2)\right) g_1. \nonumber
\end{align}
Note that $L_1(g_0,g_1,g_2), \ldots, L_4(g_0,g_1,g_2)$ are independent of $g_0$.
Likewise, $L_2(g_0,g_1,g_2)$, $L_7(g_0,g_1,g_2)$,
and $L_9(g_0,g_1,g_2)$ are independent of $g_2$.

By Proposition \ref{prop:Tgh} , any element
$\psi \in \mathrm{Hom}_{O(n-1)}\left(\Exterior^i(\C^n),\Exterior^{i-1}(\C^{n-1})\otimes\mathrm{Pol}^a[\zeta_1,\ldots, \zeta_n]\right)$ 
is of the form 
\index{A}{H0ii-k@$h^{(k)}_{i\to i-1}$}
\begin{equation*}
\psi = 
\begin{cases}
\sum_{k=0}^1(T_{a-k}g_k)h^{(k)}_{1\to 0} & (i=1),\\
\sum_{k=0}^2(T_{a-k}g_k)h^{(k)}_{i \to i-1} & (2 \leq i \leq n-1),\\
(T_ag_0)h^{(0)}_{n\to n-1} & (i=n),
\end{cases}
\end{equation*}
for some polynomials $g_k(t) \in \mathrm{Pol}_{a-k}[t]_{\mathrm{even}}$
$(k=0,1,2)$,
where $T_{a-k}g_k\in \mathrm{Pol}^{a-k}[\zeta_1, \ldots, \zeta_n]$ 
are given as in \eqref{eqn:Ta}.
For $i=1$ or $n$, we may also write as 
$\psi = \sum_{k=0}^2(T_{a-k}g_k)h^{(k)}_{i \to i-1}$
with $g_2=0$ for $i=1$ and $g_1 = g_2=0$ for $i=n$.
In what follows, we understand
\begin{equation}\label{eqn:gjvan}
g_1 = g_2 = 0 \quad \text{for $a=0$}; \quad
g_2 = 0 \quad \text{for $a=1$}; \quad
g_2=0 \quad \text{for $i=1$}; \quad
g_1 = g_2=0 \quad \text{for $i=n$}.
\end{equation}

Theorem \ref{thm:Fi-} can be separated into 
Theorem \ref{thm:7} (finding equations) 
and Theorem \ref{thm:gis} (finding solutions) below.

\begin{thm}\label{thm:7}
Let $n \geq 3$ and $1\leq i\leq n$.
Then, for $\psi = \sum_{k=0}^2(T_{a-k}g_k)h^{(k)}_{i \to i-1}$,
the following hold.
\begin{enumerate}
\item Suppose $i=1$. The following two conditions on $g_0$, $g_1$ are 
equivalent:
\begin{enumerate}
\item[(i)] $\psi$ satisfies $\widehat{d\pi_{(1,\lambda)^*}}(C)\psi = 0$ 
for all $C \in \mathfrak{n}_+'$.
\item[(ii)] $L_r(g_0,g_1,0) = 0$ for $r=2, 7, 9$.
\end{enumerate}
\item Suppose $2\leq i \leq n-1$.
The following two conditions on $g_0, g_1, g_2$
are equivalent:
\begin{enumerate}
\item[(i)] $\psi$ satisfies 
$\Fdpi{i}{\lambda}(C)\psi=0$ for all $C \in \mathfrak{n}'_+$.
\item[(ii)] $L_r(g_0,g_1, g_2) = 0$ for all $r =1, \ldots, 7$.
\end{enumerate}
\item Suppose $i=n$. 
The following two conditions on $g_0$ are equivalent:
\begin{enumerate}
\item[(i)] $\psi$ satisfies 
$\Fdpi{n}{\lambda}(C)\psi=0$ for all $C \in \mathfrak{n}'_+$.
\item[(ii)] $L_7(g_0,0, 0) = 0$.
\end{enumerate}
\end{enumerate}
\end{thm}

\begin{rem}\label{rem:nL7}
For $i=n$, the equation $L_7(g_0,0,0)=0$ amounts to the 
\index{B}{imaginary Gegenbauer differential equation}
imaginary Gegenbauer differential equation $R_a^{\lambda-\frac{n-1}{2}}g_0=0$.
\end{rem}

\begin{thm}\label{thm:gis}
Let $n\geq 3$ and $1\leq i \leq n$.
Suppose $g_k(t)\in \mathrm{Pol}_{a-k}[t]_{\mathrm{even}}$ $(k=0,1,2)$
with the convention \eqref{eqn:gjvan}. 
Then, up to scalar multiple,
the solution $(g_0, g_1, g_2)$ of the F-system
$L_r(g_0,g_1,0)=0$ for $r=2,7,9$ when $i=1$;
$L_r(g_0, g_1, g_2) = 0$ for $r = 1, \ldots, 7$ 
when $2\leq i \leq n-1$; $L_r(g_0,0,0)=0$ for $r=7$ when $i=n$,
is given as follows:
\begin{alignat*}{3}
(1)&\quad && i=1, \; a\geq 1: &&\;(g_0,g_1,g_2) = (\eqref{eqn:g0}, \eqref{eqn:g1},0);\\
(2)&\quad &&2\leq i \leq n-1, \;a \geq 1: &&\;(g_0,g_1, g_2)
=(\eqref{eqn:g0}, \eqref{eqn:g1}, \eqref{eqn:g2});\\
(3)&\quad &&i=n, \; a\geq 1: &&\;(g_0,g_1, g_2) 
=\left(\widetilde C_{a}^{\lambda-\frac{n-1}2}\left(e^{\frac{\pi\sqrt{-1}}{2}}t\right),0,0 \right);\\
(4)&\quad &&1\leq i \leq n,\; a =0: &&\; (g_0, g_1, g_2) = (1, 0, 0).
\end{alignat*}
\end{thm}

\begin{rem}
The formula (3) for $a=0$ coincides with the formula (4) for $i=n$ because
$\widetilde{C}^\mu_0(t) = 1$.
\end{rem}

The proof of Theorem \ref{thm:gis} will be given in 
Section \ref{subsec:gis} by using some basic properties of the 
Gegenbauer polynomials that are summarized in Appendix.
Alternatively, the theorem could also be shown 
by solving directly 
the F-system $L_r(g_0,g_1,g_2)=0$ $(r=1 \ldots, 7)$
with the following remark.

\begin{rem}\label{prop:7to4} 
Let $a \geq 3$ and assume that $(g_0, g_1, g_2)$ satisfies
$L_r(g_0, g_1, g_2) = 0$ for $r = 1, 2, 3$.
Then the following two conditions on the triple $(g_0,g_1,g_2)$ are equivalent:
\begin{enumerate}
\item[(i)] $L_r(g_0, g_1, g_2) = 0$ for $r = 4,5,6,7$.
\item[(ii)] $L_8(g_0, g_1, g_2) = 0$.
\end{enumerate}
\end{rem}

The rest of this chapter is devoted to proving Theorem \ref{thm:7}.
Since the argument requires a number of lemmas and propositions,
we separate it into a few steps as follows.
Let $N^+_1$ be the element of the nilpotent Lie algebra 
$\mathfrak{n}'_+$ defined in \eqref{eqn:Npm1}. 
\vskip 0.1in

\begin{enumerate}
\item[Step 1.] Reduce the condition
\index{A}{dpi6ihat@$\widehat{d\pi_{(i,\lambda)^*}}$}
$\Fdpi{i}{\lambda}(C)\psi=0$ for all $C \in \mathfrak{n}'_+$ 
to 
\index{A}{Nell+@$N_\ell^+$ ($=\frac{1}{2}C^+_\ell$), basis of $\mathfrak{n}_+(\R)$}
$\Fdpi{i}{\lambda}(N^+_1)\psi=0$.
\item[Step 2.] Consider the equation 
$\Fdpi{i}{\lambda}(N^+_1)\psi=0$ in terms of matrix coefficients 
\index{A}{MIJ@$M_{IJ}$, matrix component of 
$\widehat{d\pi_{(i,\lambda)^*}}(N_1^+)\psi$}
$M_{IJ}$.
\item[Step 3.] Reduce the number of cases 
for the matrix coefficients $M_{IJ}$ to consider.
\item[Step 4.] Express the matrix coefficients $M_{IJ}$ 
in terms of $L_r(g_0,g_1,g_2)$ for $r= 1, \ldots, 7$.
\item[Step 5.]  Deduce $L_r(g_0,g_1, g_2) = 0$ for $r =1, \ldots, 7$ 
(resp. for $r=7$) from $M_{IJ} = 0$ for $2\leq i \leq n-1$ (resp. for $i=n$)
\end{enumerate}
\vskip 0.1in

Observe that Step 1 was done 
in Lemma \ref{lem:C0enough} in a more general setting
(see also Proposition \ref{prop:153091} (2)).
In the next sections we shall discuss Steps 2--5.

%%%%%%%%%%%%%%%%%%%%%%%%%%%%%%%%%%%%%%%%%%%%%%%
\subsection{Step 2: Matrix coefficients $M_{IJ}$ for $\Fdpi{i}{\lambda}(N^+_1)\psi$}
\label{subsec:MIJ}

In this section, 
along the strategy discussed in Section \ref{subsec:MIJF}, 
we  consider the differential equation (F-system)
$\Fdpi{i}{\lambda}(N^+_1)\psi=0$ in terms of matrix coefficients
$M_{IJ} = M_{IJ}^{\mathrm{scalar}}+ M_{IJ}^{\mathrm{vect}}$.
The scalar part $M_{IJ}^{\mathrm{scalar}}$ of $M_{IJ}$ is also computed.
Since the arguments work for any $n$ and $i$, we assume that 
$n \geq 1$ and $1\leq i \leq n$ in this section.

We begin with a quick review of Section \ref{subsec:MIJF}.
First, for $\ell \in \{1, \ldots, n\}$ and $m \in \{0, 1, \ldots, n\}$, 
we write 
\index{A}{Ink@$\mathcal{I}_{n,k}$, index set}
\begin{equation*}
\mathcal{I}_{\ell, m} = \{ R \subset \{1, \ldots, \ell\} : |R| = m\}
\end{equation*}
as in \eqref{eqn:Index}.
Here $\mathcal{I}_{\ell,0}$ is understood as $\mathcal{I}_{\ell,0} =\{\emptyset\}$.
Note that $\{e_I : I \in \mathcal{I}_{n,i}\}$ and 
$\{w_J : J \in \mathcal{I}_{n-1,i-1}\}$ are
the standard bases of 
$\Exterior^i(\C^n)$
and
$\Exterior^{i-1}(\C^{n-1})$, respectively.
For $\{e_I : I \in \mathcal{I}_{n,i}\}$ and 
$\{w_J : J \in \mathcal{I}_{n-1,i-1}\}$,
we then set
\begin{equation*}
M_{IJ}\equiv M_{IJ}(g_0,g_1,g_2)
:=\left\langle
\Fdpi{i}{\lambda}(N^+_1)\psi(\zeta) e_I,w_J^\vee\right\rangle.
\end{equation*}

\begin{lem}\label{lem:FMIJ}
The following two conditions on $(g_0,g_1,g_2)$ are equivalent:
\begin{enumerate}
\item[(i)] $\Fdpi{i}{\lambda}(N^+_1)\psi=0$. 
\item[(ii)] $M_{IJ}=0$ for all $I\in\mathcal I_{n,i}$ and $J\in\mathcal I_{n-1,i-1}$.
\end{enumerate}
\end{lem}

\begin{proof}
Clear.
\end{proof}

According to the decomposition \eqref{eqn:Fsv}
into the ``scalar part" and ``vector part"
\begin{equation*}
\Fdpi{i}{\lambda}(N^+_1)
=\widehat{d\pi_{\lambda^*}}(N^+_1)\otimes\mathrm{id}_{V^\vee}
+A_\sigma(N^+_1)
\end{equation*}
with $V = \Exterior^i(\mathbb{C}^n)$, we decompose $M_{IJ}$ as
\begin{equation*}
M_{IJ}=
\index{A}{Mscalar@$M_{IJ}^{\mathrm{scalar}}$}
\index{A}{Mvect@$M_{IJ}^{\mathrm{vect}}$}
M_{IJ}^{\mathrm{scalar}}+ M_{IJ}^{\mathrm{vect}}
\end{equation*}
(see Proposition \ref{prop:MIJ}).

For $I\in\mathcal I_{n,i}$ and $J\in\mathcal I_{n-1,i-1}$, 
we write $h^{(k)}_{IJ}$ for
the matrix coefficient
\index{A}{H0ii-k@$h^{(k)}_{i\to i-1}$}
$\left(h^{(k)}_{i \to i-1} \right)_{IJ} 
=\langle h^{(k)}_{i \to i-1}(e_I),e_J^\vee\rangle$
of $h^{(k)}_{i \to i-1}$.
It follows from Table \ref{table:Table160487} that we have
\begin{align}\label{eqn:hk0}
(-1)^{i-1}h^{(0)}_{IJ}&=
\left\{
\begin{matrix*}[l]
1 &\qquad \qquad \qquad\qquad\quad\quad   \mathrm{if}& J\subset I\ni n,\\
0 &\qquad \qquad \qquad\qquad\quad\quad   \mathrm{otherwise.}&
\end{matrix*}
\right.
\\
h^{(1)}_{IJ}&=
\left\{
\begin{matrix*}[l]
\sgn(I;\ell)\zeta_\ell & \qquad\qquad \quad\quad \mathrm{if}&J\subset I\not\ni n,\\
0 &\qquad \qquad\quad\quad  \mathrm{otherwise.}&
\end{matrix*}
\right.\label{eqn:hk1}\\
(-1)^{i-1}h^{(2)}_{IJ}&=
\left\{
\begin{matrix*}[l]
 \sum\limits_{\ell\in I\setminus\{n\}}\zeta_\ell^2-\frac{i-1}{n-1}Q_{n-1}(\zeta')
 &\mathrm{if}&J\subset I\ni n,\\
\sgn(I;p,q) \zeta_p\zeta_q
 & \mathrm{if}& \vert J\setminus I\vert=1\; \mathrm{and}\; I\ni n,\\
 0&\mathrm{otherwise.}&
\end{matrix*}
\right.
\label{eqn:hk2}
\end{align}
Here $Q_{n-1}(\zeta') = \sum_{m = 1}^{n-1}\zeta_m^2$, and
we write $I=J\cup\{\ell\}$ ($1\leq\ell\leq n-1$) if $J\subset I\not\ni n$, and
$I=K\cup\{p,n\}, J=K\cup\{q\}$ if $\vert J\setminus I\vert= 1$ and $I\ni n$.
By \eqref{eqn:hk0}-\eqref{eqn:hk2}, we observe:
\begin{align}
h^{(k)}_{IJ} &= 0 \quad 
\text{for $n \notin I$ (\emph{i.e.}\ $I \in \mathcal{I}_{n-1,i}$) for $k=0,2$}, \label{eqn:h02}\\
h^{(1)}_{IJ} &= 0 \quad 
\text{for $n \in I$ (\emph{i.e.}\ $I \in \mathcal{I}_{n,i} \setminus \mathcal{I}_{n-1,i}$}).
\label{eqn:h1}
\end{align}

\vskip 0.1in

By using $h^{(k)}_{IJ}$, we then have the following.

\begin{lem}\label{lem:MIJ} 
With 
$G_k :=T_{a-k}\left(R^{\lambda-\frac{n-1}{2}}_{a-k}g_k\right)$
for $k = 0, 1, 2$,
the scalar part $M_{IJ}^{\mathrm{scalar}}$
is given as follows.

\begin{align*}
M_{IJ}^{\mathrm{scalar}}&=
\begin{cases}
\displaystyle{
\frac{\zeta_1}{Q_{n-1}(\zeta')} G_0
h_{IJ}^{(0)}+
\frac{\zeta_1}{Q_{n-1}(\zeta')} G_2
h_{IJ}^{(2)}+
(\lambda+a-1)T_{a-2}g_2\frac{\partial h_{IJ}^{(2)}}{\partial\zeta_1}} &   (n\in I),\\
\displaystyle{
\frac{\zeta_1}{Q_{n-1}(\zeta')} G_1 h_{IJ}^{(1)}
+(\lambda+a-1)T_{a-1}g_1\frac{\partial h_{IJ}^{(1)}}{\partial\zeta_1}} & (n\not\in I).
\end{cases}
\end{align*}
\end{lem}

\begin{proof}
As $\psi = \sum_{k=0}^2(T_{a-k}g_k)h^{(k)}_{i\to i-1}$,
it follows from Proposition \ref{prop:scalar-vect} (1) that 
\begin{equation*}
M^{\mathrm{scalar}}_{IJ}=
\sum_{k=0}^2 
\left(
\displaystyle{\frac{\zeta_1}{Q_{n-1}(\zeta')}G_kh^{(k)}_{IJ}
+(\lambda+a-1)(T_{a-k}g_k)\frac{\partial h^{(k)}_{IJ}}{\partial \zeta_1}}\right).
\end{equation*}
Since $\frac{\partial h^{(0)}_{IJ}}{\partial \zeta_1}=0$ by \eqref{eqn:hk0},
the proposed identity holds from \eqref{eqn:h02} and \eqref{eqn:h1}.
\end{proof}

In Section \ref{subsec:Meqn4},
 by using \eqref{eqn:hk0}-\eqref{eqn:hk2} 
and Lemma \ref{lem:MIJ},
we shall give explicit formul\ae{} for 
$M_{IJ}=M_{IJ}^{\mathrm{scalar}}+ M_{IJ}^{\mathrm{vect}}$.

\vskip 0.2in

We conclude this section by showing the following 
lemma.

\begin{lem}\label{lem:MIJscal}
The following hold.
\begin{enumerate}
\item If $|J\setminus I|\geq 1$ and $n\not\in I$, then
$M_{IJ}^{\mathrm{scalar}}=0$.
\item If $|J\setminus I|\geq 2$ and $n\in I$, then
$M_{IJ}^{\mathrm{scalar}}=0$.
\end{enumerate}
Consequently, 
if $|J\setminus I|\geq2$, then $M_{IJ}^{\mathrm{scalar}}=0$.
\end{lem}

\begin{proof}
To show the first claim, as $n \notin I$, it suffices to show that $h^{(1)}_{I J} = 0$.
Since $J \not\subset I$, it follows that $h^{(1)}_{IJ} = 0$. The second claim
may be shown similarly. Indeed, if $|J\setminus I|\geq 2$, then
$h^{(k)}_{IJ} = 0$ for $k = 0,1,2$. Therefore the second claim also holds.
\end{proof}

The vector part $M_{IJ}^{\mathrm{scalar}}$ will be treated in the next section.

%%%%%%%%%%%%%%%%%%%%%%%%%%%%%%%%%%%%%%%%%%%%%%%
\subsection{Step 3: Case-reduction for $M^\mathrm{vect}_{IJ}$}
\label{subsec:7comb}

In view of Lemma \ref{lem:FMIJ}, we wish to solve $M_{IJ} = 0$
for all $I \in \mathcal{I}_{n,i}$ and $J\in \mathcal{I}_{n-1, i-1}$.
The aim of this section is to reduce the number of cases for $M_{IJ}$ 
to consider. We would like to emphasize that, 
consequently, no matter how large $n$ is, 
it is sufficient to consider at most eleven cases.
This is achieved in Proposition \ref{prop:11cases}.
As in Section \ref{subsec:MIJ},
throughout this section, we assume that $n \geq 1$ and $1\leq i \leq n$.

As Lemma \ref{lem:MIJscal} treats $M^\mathrm{scalar}_{IJ}$
for $M_{IJ}= M^\mathrm{scalar}_{IJ}+M^\mathrm{vect}_{IJ}$,
it suffices to consider $M^\mathrm{vect}_{IJ}$.

It follows from Proposition \ref{prop:MIJ} that,
for $\psi = \sum_{k=0}^2(T_{a-k}g_k)h^{(k)}_{i\to i-1}$,
we have
\begin{equation*}
M_{IJ}^{\mathrm{vect}}=\sum_{I' \in \mathcal{I}_{n, i}}A_{II'}\psi_{I'J}=
\sum_{k=0}^2  \sum_{I' \in \mathcal{I}_{n, i}} 
A_{II'}\left(T_{a-k}g_kh^{(k)}_{I'J}\right),
\end{equation*}
where 
\index{A}{AII'@$A_{II'}$, matrix component of $A_\sigma$} 
$A_{II'}$ is the vector field given in Lemma \ref{lem:AII},
namely,
\begin{equation}\label{eqn:AII2}
A_{II'}=
\begin{cases}
\sgn (I;\ell)\frac\partial{\partial\zeta_\ell}&\mathrm{if}\,  
(I\setminus I') \coprod (I'\setminus I) = \{1, \ell\} \;\; (\ell \neq 1),\\
0&\mathrm{otherwise}.
\end{cases}
\end{equation}
Then, in order to evaluate $M^{\mathrm{vect}}_{IJ}$, 
one needs to compute
$\sum_{I' \in \mathcal{I}_{n, i}}A_{II'}\psi_{I'J}$.
However, in fact, almost all the terms $A_{II'}\psi_{I'J}$ are zero.
We formulate it precisely by introducing the definition of $\mathrm{Supp}(I,J;k)$ as follows.

\begin{defn}\label{def:SuppIJk}
For $I\in\mathcal I_{n,i}$, $J\in\mathcal I_{n-1,i-1}$, and $k\in\{0,1,2\}$, 
define a subset
\index{A}{Supp@$\mathrm{Supp}(I,J;k)$|textbf}
$\mathrm{Supp}(I,J;k)$ of 
\index{A}{Ink@$\mathcal{I}_{n,k}$, index set}
$\mathcal I_{n,i}$ by
\begin{align*}
\mathrm{Supp}(I,J;k):=
\{I'\in\mathcal I_{n,i}: A_{II'}\neq0, \; \text{and} \; h_{I'J}^{(k)} \neq0\}.
\end{align*}
\end{defn}

It follow from \eqref{eqn:hk0}, \eqref{eqn:hk1}, and \eqref{eqn:hk2} that we have 
\begin{align*}
\mathrm{Supp}(I,J;k) 
\subset
\begin{cases}
\mathcal{I}_{n,i} \setminus \mathcal{I}_{n-1,i}&
\text{for $k=0,2$},\\
\mathcal{I}_{n-1,i} &
 \text{for $k=1$}.
\end{cases}
\end{align*}

By using $\mathrm{Supp}(I,J;k)$, 
$M^\mathrm{vect}_{IJ}$ may be given as follows:
\index{A}{Mvect@$M_{IJ}^{\mathrm{vect}}$}
\begin{equation}\label{eqn:Mvect}
M_{IJ}^{\mathrm{vect}}=
\sum_{k=0}^2 \left(\sum_{I'\in\mathrm{Supp}(I,J;k)} A_{II'} 
\left( T_{a-k}g_k h_{I'J}^{(k)} \right)\right).
\end{equation}

We now show that if $|J\setminus I|$ is large, then 
$\mathrm{Supp}(I,J;k) = \emptyset$ and thus $M_{IJ}^{\mathrm{vect}}=0$.
Together with the results in Lemma \ref{lem:MIJscal}, 
this allows us to focus on the cases when $|J \setminus I|$ is small.
In fact, it turns out that it suffices to consider only the cases when
$|J \setminus I| \in \{0,1\}$, see Lemma \ref{lem:MIJvect2} below.

\vskip 0.1in
We first show that if $|J\setminus I|\geq2$, then $M_{IJ}^{\mathrm{vect}}=0$.
We prove it in two steps, namely, Lemmas \ref{lem:MIJvect1} and \ref{lem:MIJvect2}.
The following ``triangle inequality" for arbitrary three sets $I$, $I'$, and $J$
is used in the proof for Lemma \ref{lem:MIJvect1}:
\begin{equation}\label{eqn:Set_ineq}
|J\setminus I| \leq |J\setminus I'| + |I'\setminus I|.
\end{equation}

\begin{lem}\label{lem:MIJvect1}
We have the following:
\begin{enumerate}
\item If $|J\setminus I|\geq 2$, then
$\mathrm{Supp}(I,J;k)=\emptyset$ for $k=0,1$.
\item If $|J\setminus I|\geq 3$, then
$\mathrm{Supp}(I,J;2)=\emptyset$.
\end{enumerate}
Consequently, 
if $|J\setminus I|\geq3$, then $M_{IJ}^{\mathrm{vect}}=0$.
\end{lem}

\begin{proof}
Observe that, for $k = 0,1,2$,
if $\mathrm{Supp}(I,J;k)\neq\emptyset$, then there exists $I'$ 
so that 
\index{A}{AII'@$A_{II'}$, matrix component of $A_\sigma$}
$A_{II'} \neq 0$; in particular, $|I'\setminus I| = 1$
by \eqref{eqn:AII2}.
On the other hand, 
if $I' \in \mathrm{Supp}(I,J;k)$, then
$h^{(k)}_{I'J} \neq 0$, and therefore
$|J\setminus I'| = 0$ for $k= 0,1$ and
$|J\setminus I'|  \leq 1$
for $k=2$ by \eqref{eqn:hk0}-\eqref{eqn:hk2}. 
We then get $|J\setminus I| \leq 2$ for $k=0,1$
and $|J \setminus I| \leq 1$ for $k=2$ by
\eqref{eqn:Set_ineq}.
\end{proof}

\begin{lem}\label{lem:MIJvect2}
If $|J\setminus I| = 2$, then $M_{IJ}^{\mathrm{vect}}=0$.
\end{lem}

\begin{proof}
Under the condition $|J\setminus I| = 2$,
we first observe $\mathrm{Supp}(I,J;k) = \emptyset$ for $k=0,1$ by 
Lemma \ref{lem:MIJvect1} (1). Further, for any $I' \in \mathrm{Supp}(I,J;2)$,
$I' \ni n$ and $|J \setminus I'| \leq 1$ by \eqref{eqn:hk2}. On the other hand,
$|J\setminus I'| \geq  |J \setminus I| - |I \setminus I'| = 2-1 = 1$ by \eqref{eqn:AII2}.
Hence $|J \setminus I'|=1$.

Assume $n \notin I$. Then $(I, I')$ must be of the form $I = K \cup \{1\}$
and $I' = K \cup \{n\}$ by \eqref{eqn:AII2}, which is impossible because 
$2=|J\setminus I| \leq |J\setminus K| = |J\setminus I'| =1$. Hence 
$\mathrm{Supp}(I,J;2) = \emptyset$ if $n \notin I$ and $|J\setminus I | =2$.

Assume now $n \in I$.
Then there are two cases, namely, 
$1 \in I$ and $1 \notin I$, and in each case
$\mathrm{Supp}(I,J;2)$ consists of two elements.
Indeed, for $K:=I \cap J \subset \mathcal{I}_{n-1,i-3}$,
we have the following.
\begin{enumerate}
\item $I = K\cup \{1, r, n\}$, $J=K \cup \{p,q\}$ for some $r$:
\begin{equation*}
\mathrm{Supp}(I,J;2)=\{K \cup \{p,r,n\}, K \cup \{q, r, n\}\}.
\end{equation*}
\item $I =K \cup \{p, q, n\}$, $J = K \cup \{1, r\}$ for some $r$:
\begin{equation*}
\mathrm{Supp}(I,J;2)=\{K \cup \{1, p,n\}, K \cup \{1, q, n\}\}.
\end{equation*}
\end{enumerate}
In either case, we have
$M_{IJ}^{\mathrm{vect}}=0$ if and only if 
$\zeta_r\left(\zeta_p \frac{\partial}{\partial \zeta_q} 
- \zeta_q \frac{\partial }{\partial \zeta_p}\right)(T_{a-2} g_2)=0$,
which clearly holds, as $T_{a-2} g_2$ is $O(n-1, \mathbb{C})$-invariant.
Now the assertion holds.
\end{proof}

\begin{rem}
The case that $|J\setminus I| = 2$ happens only when
$n \geq 5$ and $3 \leq i \leq n-2$.
\end{rem}

Now we obtain the following lemma.

\begin{lem}\label{lem:MIJvect}
If $|J\setminus I|\geq2$ then $M_{IJ}^{\mathrm{vect}}=0$.
\end{lem}

\begin{proof}
This follows from Lemmas \ref{lem:MIJvect1} and \ref{lem:MIJvect2}.
\end{proof}

\vskip 0.1in
By Lemmas \ref{lem:MIJscal} and \ref{lem:MIJvect},
it suffices to focus on $M_{I J}$ with $|J\setminus I|\leq 1$.
Observe that among the indices $1, 2,  \ldots, n$ for 
$\{1, 2, \ldots, n\}$,
``$1$" and ``$n$" play special roles
for $I$ and $J$, as ``$1$" comes from
our choice of $N^+_1$ for $\Fdpi{i}{\lambda}(N^+_1)$
and as ``$n$" makes the difference between
$M=O(n)\times O(1)$ and 
$M'=O(n-1) \times O(1)$.
On the other hand, all the pairs $(I, J) \in \mathcal{I}_{n,i} \times \mathcal{I}_{n-1, i-1}$
with $|J\setminus I|\leq 1$
are classified into $2^4(=16)$ cases according to 
whether each of the following conditions on $(I,J)$ holds or not:
$1\in J$, $1\in I, n\in I$, and $J \subset I$.
For simplicity we represent them by quadruples $[\pm,\pm,\pm,\pm]$ as follows.

\begin{defn}\label{def:pm}
We mean by quadruples 
$[\pm,\pm,\pm,\pm]$
the cases 
according to
whether each condition
$1\in J$, $1\in I, n\in I$, and $J \subset I$ holds.
\end{defn}

For instance, by $[-,+,-,+]$, we mean that $(I,J)$ satisfies
$1\not\in J,1\in I,n\not\in I$ and $J\subset I$.

Among $2^4(=16)$ cases  for $(I, J)$  with $|J\setminus I|\leq 1$,
we show that at most eleven cases need to be taken into account,
and thus Lemma \ref{lem:FMIJ} can be refined as follows.

\begin{prop}\label{prop:11cases}
Let $\psi = \sum_{k=0}^2 (T_{a-k}g_k)h^{(k)}_{i\to i-1}$.
The the following two conditions on $(g_0,g_1,g_2)$
are equivalent:
\begin{enumerate}
\item[(i)] $\widehat{d\pi_{(i,\lambda)^*}}(N_1^+)\psi=0$.
\item[(ii)] $M_{IJ}=0$ for any
$(I, J) \in \mathcal{I}_{n,i} \times \mathcal{I}_{n-1, i-1}$, 
subject to  the eleven cases in Table \ref{table:IJ}:
\end{enumerate}
\begin{table}[H]
\caption{$(I,J)$  for $1\leq i \leq n$}
\begin{center}
\begin{tabular}{ccccc} 
\hline
&$I$  & $J$ & $[\pm, \pm, \pm, \pm]$ & \\
\hline
\emph{(1)}&$J\cup \{n\}$ & $J$ & $++++$ & \\
\emph{(2)}&$K\cup \{1, n\}$ & $K\cup\{p\}$ & $-++-$ &\\
\emph{(3)}&$K \cup \{p, n\}$ & $K \cup \{1\}$ & $+-+-$ &  \\
\emph{(4)}&$J\cup\{n\}$& $J$ & $--++$ & \\
\emph{(5)}&$K\cup\{p, n\}$ & $K\cup \{q\}$ & $--+-$ & \\
\emph{(6)}&$K\cup\{p, q\}$ & $K\cup \{1\}$ & $+---$ & \\
\emph{(7)}&$J\cup\{p\}$ & $J$ & $---+$ & \\
\emph{(8)}&$J\cup\{p\}$ & $J$ & $++-+$ & \\
\emph{(9)}&$K\cup \{1, p\}$ & $K\cup \{q\}$ & $-+--$ & \\
\emph{(10)}&$J\cup \{1\}$ & $J$ & $-+-+$ & \\
\emph{(11)}&$K\cup \{1, p,n\}$ & $K\cup \{1, q\}$ &$+++-$ & \\
\hline
\end{tabular} \label{table:IJ}
\end{center}
\end{table}
\end{prop}

For later convenience, we have described in Table \ref{table:IJ}
the general form of $(I,J)$ for types (1)-(11), where 
each union is disjoint and $1, p, q, n$ are all distinct numbers.
By this description, we observe that some of these types do not occur
when $i$ or $n-i$ is very small. To be precise, we have:

\begin{rem}\label{rem:i1n-1n}
For $i=1, n-1$, or $n$, only the following cases occur:
\begin{itemize}
\item[(a)] $i=n$: $(1)$;
\item[(b)] $n=2$ and $i=1$: (4), (10);
\item[(c)] $n\geq 3$:
\begin{itemize}
\item[(c1)] $i=1$: $(4), (7), (10)$;
\item[(c2)] $i=n-1$: $(1), (2), (3), (4), (8), (10)$.
\end{itemize}
\end{itemize}
\end{rem}

Hence Proposition \ref{prop:11cases} includes the following degenerate cases.

\begin{prop}[$i=1$]\label{prop:3cases}
The following two conditions on $(g_0,g_1)$ are equivalent:
\begin{enumerate}
\item[(i)] $\widehat{d\pi_{(1,\lambda)^*}}(N_1^+)\psi = 0$.
\item[(ii)] $M_{IJ}=0$ for any pair $(I,J) \in \mathcal{I}_{n,1} \times \mathcal{I}_{n-1,0}$ 
that belongs to (4), (7), (10) in Table \ref{table:IJ}, namely,
for $(I,J) = (\{n\}, \emptyset)$, $(\{p\}, \emptyset)\; (1 \leq p \leq n-1)$,
$(\{1\}, \emptyset)$.
\end{enumerate}
\end{prop}

\begin{prop}[$i=n$]\label{prop:1case}
The following two conditions on $g_0$ are equivalent:
\begin{enumerate}
\item[(i)] $\widehat{d\pi_{(n,\lambda)^*}}(N_1^+)\psi=0$.
\item[(ii)] $M_{IJ}=0$ for any $(I,J) \in \mathcal{I}_{n,n} \times \mathcal{I}_{n-1, n-1}$
that belongs to (1) in Table \ref{table:IJ}, namely, $(I,J) = (\{1, \ldots, n\}, \{1,\ldots, n-1\})$.
\end{enumerate}
\end{prop}

The proof of Proposition \ref{prop:11cases} 
consists of several lemmas; nonetheless,
it is basically done in two steps. 
First we observe that the three cases $[+,-,+,+]$, $[+,-,-,+]$, and $[-,+,+,+]$
do not exist set-theoretically. 
We then show that, for the other two cases $[+, +, -, -]$ and $[-,-,-,-]$, we have $M_{IJ}=0$.

\begin{lem}\label{lem:three_cases}
Set-theoretically, the three cases 
$[+, -, +, +]$, $[+, -, -, +]$, and $[-, +, + , +]$ do not exist.
\end{lem}
\begin{proof}
If $ 1\in J \subset I$, then $1 \in I$. Therefore $[+, -, \pm, +]$ do not exist.
If $J \subset I$ and $n \in I$, then $J= I \setminus \{n\}$;
in particular, in the case, if $1\in I$, then $1\in J$.
Hence, $[-, + , +, +]$ does not exist.
\end{proof}

Next we aim to show that $M_{IJ} = 0$ for $(I,J)$ of type $[+,+,-,-]$ and $[-,-,-,-]$.
More generally, we observe that 
the set $\mathrm{Supp}(I,J;k)$ 
(see Definition \ref{def:SuppIJk})
is determined by the types of 
$[\pm,\pm,\pm,\pm]$, and actually, this is the reason that we introduced the notation
\index{A}{0pm@$[\pm, \pm, \pm, \pm]$}
$[\pm,\pm,\pm,\pm]$ here. The simplest case is $k=0$, where
we have 
\begin{equation*}
\mathrm{Supp}(I,J;0)=\left\{
\begin{matrix}
J\cup\{n\} &\mathrm{for}\quad [-,+,-,+]\,\mathrm{or}\,[+,-,+,-],\\
\emptyset&\mathrm{otherwise}.
\end{matrix}
\right.
\end{equation*}
By using this idea we consider Lemmas \ref{lem:87} and \ref{lem:90} below.

\begin{lem}\label{lem:87} 
The following hold.
\begin{enumerate}
\item If $(I,J)$ is of type $[+,+,\pm,\pm]$, then $\mathrm{Supp}(I,J;1)=\emptyset$.
\item If $(I,J)$ is of type $[+,+,-,-]$, then $\mathrm{Supp}(I,J;k)=\emptyset$ for $k=0,1,2$.
\end{enumerate}
\end{lem}

\begin{proof}
For the first statement suppose that $ 1\in I \cap J$. 
If $I' \in \mathcal{I}_{n,i}$ satisfies $A_{II'} \neq 0$, then
$1 \notin I'$ as $1 \in I$. Hence $J \not \subset I'$ since $1 \in J$.
Therefore $h^{(1)}_{I'J} = 0$.

To show the second statement, it suffices to show that $J \subset I$ 
if $1 \in J$, $1 \in I \not\ni n$, and $\mathrm{Supp}(I,J;k) \neq \emptyset$
for $k=0$ or $2$. 
For $k=0,2$, let $I' \in \mathrm{Supp}(I,J;k)$. 
By \eqref{eqn:AII2},
 $I' = I\setminus \{1\} \cup \{n\}$
because $1 \in I \not \ni n$ and $n \in I'$. Then 
$|J\setminus (I' \setminus\{n\})| = 1 + |J\setminus I|$
because 
\begin{equation*}
J \setminus (I'\setminus \{n\}) 
= J \setminus (I \setminus \{1\}) 
=  \{1\} \cup (J \setminus I).
\end{equation*}
Since $h^{(k)}_{I'J} \neq 0$ implies that
$|J \setminus (I' \setminus \{n\})| \leq 1$, this shows that $J \subset I$.
\end{proof}

\begin{lem}\label{lem:90}
The following hold.
\begin{enumerate}
\item If $(I,J)$ is of type $[-,-,-,\pm]$, then $\mathrm{Supp}(I,J;0)=\mathrm{Supp}(I,J;2)=\emptyset$.
\item If $(I,J)$ is of type $[-,-,-,-]$, then $\mathrm{Supp}(I,J;k)=\emptyset$ for $k=0,1,2$.
\end{enumerate}
\end{lem}

\begin{proof}
For the first statement observe that 
if $A_{II'} \neq 0$, then $I' \subset I \cup \{1\}$ because $1\notin I$.
Since $n \notin I$, we have $n\notin I'$, which shows that 
$\mathrm{Supp}(I,J;k) = \emptyset$ for $k = 0,2$ by \eqref{eqn:h02}.
To show the second statement, it suffices to show that
$J \subset I$ if $ 1 \notin J$, $1\notin I \not \ni n$, and
$\mathrm{Supp}(I,J;1) \neq \emptyset$. Let $I' \in \mathrm{Supp}(I,J;1)$.
Since $A_{II'} \neq 0$, we have $I' \subset I\cup \{1\}$ by \eqref{eqn:AII2} 
because $1 \notin I$.
Since $h^{(1)}_{I'I} \neq 0$, we have $J \subset I'$ by \eqref{eqn:hk1}. 
Thus $J\subset I \cup \{1\}$.
Therefore $J \subset I$ as $1 \notin J$.
\end{proof}

\begin{lem}\label{lem:++--}
For the cases $[+,+,-,-]$ and $[-,-,-,-]$, 
we have $M_{IJ} = 0$.
\end{lem}

\begin{proof}
By Lemma \ref{lem:MIJscal}, we have $M_{IJ}^{\mathrm{scalar}}=0$.
Moreover, it follows from Lemmas \ref{lem:87} and \ref{lem:90} that 
$M_{IJ}^{\mathrm{vect}}=0$. 
As $M_{IJ} = M_{IJ}^{\mathrm{scalar}} + M_{IJ}^{\mathrm{vect}}$,
this proves the lemma.
\end{proof}

\begin{proof}[Proof for Proposition \ref{prop:11cases}]
The assertion follows from
Lemmas \ref{lem:three_cases} and \ref{lem:++--}.
\end{proof}

%%%%%%%%%%%%%%%%%%%%%%%%%%%%%%%%%%%%%%%%%%%%%%%
\subsection{Step 4 - Part I: Formul\ae{} for saturated differential equations}\label{subsec:7Ta}

The goal of Step 4 is to express the matrix coefficients $M_{IJ}$ 
for $(I,J)$ in Table \ref{table:IJ}  in terms of $L_r(g_0,g_1,g_2)$ for $r= 1, \ldots, 7$.
In this short section we collect several useful formul\ae{}.
The actual expressions for $M_{IJ}$ are obtained in the next section.

Recall from \eqref{eqn:Ta} that
we have defined a multi-valued meromorphic function $T_ag(\zeta)$ of $n$ variables
$\zeta = (\zeta_1, \ldots, \zeta_n)$ by 
\index{A}{Ta@$T_a$}
\begin{equation*}
(T_ag)(\zeta) = 
Q_{n-1}(\zeta')^{\frac{a}{2}}
g\left(\frac{\zeta_n}{\sqrt{Q_{n-1}(\zeta')}}\right)
\end{equation*}
for $a \in \mathbb{N}$ and $g(t) \in \C[t]$,
where $Q_{n-1}(\zeta') = \zeta_1^2 + \cdots +\zeta_{n-1}^2$.
As in \cite[Sect.\ 3.2]{KP2}, we say that a differential operator
$D$ on $\mathbb{C}^n$ is \emph{$T$-saturated} if there exists
an operator $E$ on $\mathbb{C}[t]$ such that the following diagram
commutes:
\begin{equation*}
\xymatrix{
\mathbb{C}[t] \ar[d]_E \ar[r]^{T_a \phantom{mmm}} 
& \mathbb{C}(\zeta_1,\ldots, \zeta_n)  \ar[d]^D\\
\mathbb{C}[t]  \ar[r]^{T_a\phantom{mmm}} &\mathbb{C}(\zeta_1,\ldots, \zeta_n).
}
\end{equation*}
Such an operator $E$ is unique as far as it exists. 
We denote the operator $E$ by 
\index{A}{TaS@$T_a^\sharp$, $T$-saturated differential operator|textbf}
$T_a^\sharp D$. We allow $D$ to have meromorphic coefficients.
We note that 
\begin{equation*}
T_a^\sharp (D_1 \cdot D_2) = T_a^\sharp(D_1) \cdot T_a^\sharp(D_2)
\end{equation*}
whenever it makes sense. For more general definition of $T$-saturated 
differential operators see \cite[Sect.\ 3.2]{KP2}.

\begin{lem}\label{lem:1-6} 
Let $R^\mu_\ell$ be the differential operator defined in \eqref{eqn:Rla128}
and  
\index{A}{1theta-z@$\vartheta_z=z\frac{d}{dz}$}
$\vartheta_t =t \frac{d}{dt}$ be the Euler operator.
For $a\in\N$ and $g(t)\in\operatorname{Pol}_a[t]_{\mathrm{even}}$
\emph{(}see \eqref{eqn:gs}\emph{)}, the following hold:
\begin{alignat*}{3}
&(1)\quad&&(T_ag)(\zeta)=Q_{n-1}(\zeta')(T_{a-2}g)(\zeta),&&\\
&(2)\quad&&T_a(tg(t))(\zeta)=\zeta_n(T_{a-1}g)(\zeta),&&\\
&(3)\quad&&\frac\partial{\partial\zeta_m}(T_ag)(\zeta)
=\frac{\zeta_m}{Q_{n-1}(\zeta')}
T_a\left((a - \vartheta_t)g\right)(\zeta) &&\quad (1\leq m\leq n-1),\\
&(4)\quad&&\frac\partial{\partial\zeta_n}(T_ag)(\zeta)=T_{a-1}\left(\frac{dg}{dt}\right)(\zeta),\\
&(5)\quad&&T_\ell^\sharp\left(\frac{Q_{n-1}(\zeta')}{\zeta_m}\frac\partial{\partial\zeta_m}\right)=\ell-\vartheta_t &&\quad(1\leq m\leq n-1),\\
&(6)\quad&&T_\ell^\sharp\left(\frac{Q_{n-1}(\zeta')}{\zeta_m}
\widehat{d\pi_{\lambda^*}}(N^+_m)\right)=
R_\ell^{\lambda-\frac{n-1}2}&&\quad (1\leq m\leq n-1).
\end{alignat*}
\end{lem}

\begin{proof}
Formula (6) is a restatement of Lemma \ref{lem:12301}.
Formula (5) is shown in \cite[Lem.\ 6.10]{KP2}.
Also (1), \ldots, (4) can be verified in the same spirit.
\end{proof}

\begin{lem}\label{lem:Rla} 
For $\mu \in \C$ and $\ell \in \N$, we have
\begin{equation}\label{eqn:Rla}
R_{\ell}^{\mu+1}=R_\ell^\mu+(\ell-\vartheta_t).
\end{equation}
\end{lem}

\begin{proof}
This immediately follows from the definition of $R^{\mu}_\ell$.
\end{proof}

%%%%%%%%%%%%%%%%%%%%%%%%%%%%%%%%%%%%%%%%%%%%%%%
\subsection{Step 4 - Part II: Explicit formul\ae{} for $M_{IJ}$}\label{subsec:Meqn4}
In this section, by using the formul\ae{} in Lemma \ref{lem:1-6},
we express $M_{IJ}$ for $(I,J)$ in Table \ref{table:IJ} in terms of 
$L_r \equiv L_r(g_0, g_1, g_2)$ for $r = 1, \ldots, 7$.
For $J \in \mathcal{I}_{n,i-1}$, we write
\index{A}{Qw1I@$Q_I(\zeta)$, quadratic form for index set $I$}
$Q_J(\zeta')= \sum_{m \in J} \zeta_m^2$.

\begin{lem}\label{lem:Meqn}
Let $n \geq 3$ and $1\leq i \leq n$.
For each case of $(1), \ldots, (11)$ in Table \ref{table:IJ},
$M_{IJ}$ is given as follows:

\begin{enumerate}
\item $M_{IJ} = (-1)^{i-1}\zeta_1 \big( Q_J(\zeta')T_{a-4}(L_1) 
+ T_{a-2}\left(L_7 -\frac{i-1}{n-1}L_1+ \frac{n-i}{n-1}L_4\right) \big)$,
\item $M_{IJ} =(-1)^{i-1}\sgn(K;p)\zeta_p\big(\zeta_1^2T_{a-4}(L_1) 
+ T_{a-2}\left(L_6\right) \big)$,
\item $M_{IJ} =(-1)^{i-1}\sgn(K;p)\zeta_p \big(\zeta_1^2T_{a-4}(L_1) + T_{a-2}(L_4-L_6) \big)$,
\item $M_{IJ} =(-1)^{i-1}\zeta_1 \big( Q_J(\zeta')T_{a-4}(L_1) 
+ T_{a-2}\left(L_7 -\frac{i-1}{n-1}L_1- \frac{i-1}{n-1}L_4\right) \big)$,
\item $M_{IJ} =(-1)^{i-1}\sgn(K;p,q)\zeta_1\zeta_p\zeta_qT_{a-4}(L_1)$,
\item $M_{IJ} =0$,
\item $M_{IJ} =\sgn(J;p)\zeta_1\zeta_pT_{a-3}(L_2)$,
\item $M_{IJ} = \sgn(J;p)\zeta_1\zeta_pT_{a-3}(L_2 -L_3)$,
\item $M_{IJ} =\sgn(K;p,q)\zeta_p\zeta_q T_{a-3}(L_3)$,
\item $M_{IJ} =\zeta_1^2 T_{a-3}(L_2) + Q_J(\zeta')T_{a-3}(L_3)
-T_{a-1}(L_3 + L_5)$,
\item $M_{IJ} =(-1)^{i-1}\sgn(K;p,q)\zeta_1\zeta_p\zeta_qT_{a-4}(L_1)$.
\end{enumerate}
\end{lem}

\begin{rem}\label{rem:Meqn1}
Suppose $i=1$. Then, only (4), (7), and (10) occur
(see Remark \ref{rem:i1n-1n}) on one hand,
$g_2=0$ by \eqref{eqn:gjvan} on the other hand.
Therefore $M_{IJ}$ in Lemma \ref{lem:Meqn} amounts to
\begin{enumerate}
\item[(4)] $M_{IJ} =\zeta_1(T_{a-2}L_7)$,
\item[(5)] $M_{IJ}=\zeta_1\zeta_p(T_{a-3}L_2)$,
\item[(8)] $M_{IJ}=\zeta_1^2(T_{a-3}L_2)-T_{a-1}(L_3+L_5)$.
\end{enumerate}
\end{rem}

\begin{rem}\label{rem:nL7_2}
Suppose $i=n$. Then only (1) occurs 
(see Remark \ref{rem:i1n-1n}) on one hand, $g_1=g_2=0$
by \eqref{eqn:gjvan} on the other hand. Therefore,
$M_{IJ}$ in Lemma \ref{lem:Meqn} 
is given by
$M_{IJ} = (-1)^{n-1}\zeta_1T_{a-2}(L_7)$.
\end{rem}

\begin{proof}
We only demonstrate two cases explicitly, namely, Cases (3) and (6); 
the other nine cases can be shown similarly. 
We choose Case (6) as an easy example and Case (3) as the most complicated
example.
\vskip 0.1in

\noindent
Case (6): 
$I = K \cup \{p,q\}$, 
$J = K\cup \{1\} $.
\vskip 0.05in
We wish to show that 
$M_{IJ} =0$.
Since $n \notin I$, by Lemma \ref{lem:MIJ},
$M_{IJ}^{\mathrm{scalar}}$ is given by
\begin{equation*}
M_{IJ}^{\mathrm{scalar}}=
\frac{\zeta_1}{Q_{n-1}(\zeta')} T_{a-1}\left(R^{\lambda-\frac{n-1}{2}}_{a-1}g_1\right) h_{IJ}^{(1)}
+(\lambda+a-1)T_{a-1}g_1\frac{\partial h_{IJ}^{(1)}}{\partial\zeta_1}.
\end{equation*}
As $I \not\supset J$, we have $h^{(1)}_{I J} (\zeta) = 0$ by \eqref{eqn:hk1}.
Therefore, $M_{IJ}^{\mathrm{scalar}}=0$.
To evaluate $M_{IJ}^{\mathrm{vect}}$, observe that we have
\begin{align*} 
\mathrm{Supp}(I, J;0) &= \emptyset,\\
\mathrm{Supp}(I,J;1) &= \{ K \cup \{1, p\}, K\cup\{1, q\}\},\\
\mathrm{Supp}(I,J;2) &= \emptyset.
\end{align*}
It then follows from \eqref{eqn:Mvect}
and Lemma \ref{lem:1-6} (1) and (3) that
\begin{align*}
M_{IJ}^{\mathrm{vect}}
&=\sum_{k=0,2}\sum_{I'\in\mathrm{Supp}(I,J;k)} A_{II'} 
\left( T_ag_0 h_{I'J}^{(k)} \right) +
\sum_{I'\in\mathrm{Supp}(I,J;1)} A_{II'} \left( T_{a-1}g_1 h_{I'J}^{(1)} \right)\\
&=
A_{K \cup \{p, q\},K\cup \{1, p\}} 
\left(h^{(1)}_{K\cup\{1, p\}, K\cup \{1\}} (\zeta) T_{a-1}g_1\right)
+
A_{K \cup \{p, q\},K\cup \{1, q\}} 
\left(h^{(1)}_{K\cup\{1, q\}, K\cup \{1\}} (\zeta) T_{a-1}g_1\right)\\
&=- \big( \sgn(K\cup \{p\} ; q )\sgn(K;p) + \sgn(K\cup\{q\}; p)\sgn(K;q) \big)
\zeta_p\zeta_qT_{a-3}((a-1-\vartheta_t)g_1),
\end{align*}
which vanishes by Lemma \ref{lem:sgn} (4).
Hence we obtain 
$M_{IJ} = M_{IJ}^{\mathrm{scalar}} + M_{IJ}^{\mathrm{vect}} =0$.

\vskip 0.1in
\noindent
Case (3): 
$I = K \cup \{p,n\}$, 
$J = K\cup \{1\} $.
\vskip 0.05in

We wish to show that
\begin{align}\label{eqn:03081}
M_{IJ}
&=(-1)^{i-1}\sgn(K;p) \zeta_p
\bigg( \zeta_1^2 T_{a-4}\left(R^{\lambda-\frac{n-3}{2}}_{a-2}g_2\right)\\
&\quad +T_{a-2}\left((a-\vartheta_t)g_0 -\frac{dg_1}{dt} 
+ \big(\lambda+a-i+1-\frac{i-1}{n-1} (a-\vartheta_t)\big)g_2\right)\bigg).\nonumber
\end{align}
\noindent
As in Case (6), we evaluate $M_{IJ}^{\mathrm{scalar}}$ and $M_{IJ}^{\mathrm{vect}}$, separately. To begin with, we claim that
\begin{equation}\label{eqn:MS_Case3}
M_{IJ}^{\mathrm{scalar}}= (-1)^{i-1}\sgn(K;p)\zeta_p\left(
\zeta_1^2 T_{a-4}\left(R^{\lambda-\frac{n-1}{2}}_{a-2}g_2\right) 
+T_{a-2}((\lambda+a-1)g_2) \right).
\end{equation}
First observe that,
as $I \neq J$, we have $h^{(0)}_{IJ} (\zeta) = 0$.
Then, by Lemma \ref{lem:MIJ}, $M_{IJ}^{\mathrm{scalar}}$ is given by
\begin{align*}
M^{\mathrm{scalar}}_{IJ} 
&=
\frac{\zeta_1}{Q_{n-1}(\zeta')}T_{a-2}\left(R^{\lambda-\frac{n-1}{2}}_{a-2}g_2\right)
h_{IJ}^{(2)}+
(\lambda+a-1)T_{a-2}g_2\frac{\partial h_{IJ}^{(2)}}{\partial\zeta_1}\\
&=:(S1)+(S2),
\end{align*}
as $n \in I$.
%To evaluate $(S1)$ and $(S2)$, observe that, 
By (1) of Lemma \ref{lem:1-6}, we have
\begin{align*}
\frac{\zeta_1}{Q_{n-1}(\zeta')}T_{a-2}\left(R^{\lambda-\frac{n-1}{2}}_{a-2}g_2\right)
=\zeta_1 T_{a-4}\left(R^{\lambda-\frac{n-1}{2}}_{a-2}g_2\right).
\end{align*}
Moreover, $h^{(2)}_{I J}(\zeta)$ is given by
\begin{equation*}
h^{(2)}_{I J}(\zeta)
=h^{(2)}_{K\cup \{p,n\}, K \cup \{1\} }(\zeta)
=(-1)^{i-1}\sgn(K\cup \{p,n\};p) \zeta_1\zeta_p
=(-1)^{i-1}\sgn(K;p) \zeta_1\zeta_p.
\end{equation*}
Therefore, 
\begin{align*}
(S1) 
= (-1)^{i-1}\sgn(K;p) \zeta_p\zeta_1^2 T_{a-4}\left(R^{\lambda-\frac{n-1}{2}}_{a-2}g_2\right)
\; \text{and} \;
(S2) 
= (-1)^{i-1}\sgn(K;p) \zeta_pT_{a-2}((\lambda+a-1)g_2).
\end{align*}
Now \eqref{eqn:MS_Case3} follows from $(S1)$ and $(S2)$.
\vskip 0.1in
To evaluate $M_{IJ}^{\mathrm{vect}}$, observe that we have
\begin{align}
\mathrm{Supp}(I, J;0) &= \{K\cup \{1, n\}\}, \nonumber\\
\mathrm{Supp}(I,J;1) &= \{ K \cup \{1, p\}\},\nonumber\\
\mathrm{Supp}(I,J;2) &= \{K\cup \{1, n\}\} 
\cup \bigcup_{r\in K} \big\{(K\setminus \{r\}) \cup \{1, p, n\} \big\}. \label{eqn:Supp3}
\end{align}
Accordingly, we decompose $M_{IJ}^{\mathrm{vect}}$ as 
$M_{IJ}^{\mathrm{vect}} = (M0) + (M1) + (M2)$, where we set
\begin{align*}
(Mk)&:=\sum_{I'\in\mathrm{Supp}(I,J;k)} A_{II'} 
\left( T_{a-k}g_k h_{I'J}^{(k)} \right),\\
(M'k)&:=(-1)^{i-1}\sgn(K;p)\zeta_p^{-1}(Mk)
\end{align*}
for $k=0,1,2$.
We claim that 
\begin{align}
(M'0) &=T_{a-2}((a-\vartheta_t)g_0), \label{eqn:M0}\\
(M'1) &=-T_{a-2}\left(\frac{d g_1}{dt}\right),\label{eqn:M1}\\
(M'2) &=T_{a-4}((a-2-\vartheta_t)g_2)
-T_{a-2}\left((i-2)+\frac{i-1}{n-1}(a-\vartheta_t)g_2\right).\label{eqn:M2} 
\end{align}
Indeed, for $(M0)$, we have
\begin{align*}
(M0) 
&= \sum_{I'\in\mathrm{Supp}(I,J;0)} A_{II'} 
\left( T_ag_0 h_{I'J}^{(0)} \right)\\
&=A_{K \cup \{p, n\},K\cup \{1, n\}} 
\left( h^{(0)}_{K\cup\{1,n\}, K\cup\{1\}}(\zeta) T_a g_0\right)\\
&=(-1)^{i-1}\sgn(K;p)\frac{\partial}{\partial \zeta_p}(T_a g_0).
\end{align*}
Now \eqref{eqn:M0} follows from (1) and (3) of Lemma \ref{lem:1-6}.
\eqref{eqn:M1} can be shown similarly.

Then, we have from \eqref{eqn:Supp3} 
\begin{align*}
(M2)
&=\sum_{I'\in\mathrm{Supp}(I,J;2)} A_{II'} 
\left( T_{a-2}g_2 h_{I'J}^{(2)} \right) \\
&=A_{K\cup\{p,n\}, K\cup\{1, n\}}
\left(h^{(2)}_{K\cup \{1,n\}, K\cup\{1\}}T_{a-2}g_2\right)\\
&\quad \quad +\sum_{r\in K} A_{K\cup \{p,n\},(K\setminus\{r\} )\cup \{1, p,n\} }
\left(h^{(2)}_{(K\setminus\{r\})\cup\{1, p,n\}, K\cup \{1\}}(\zeta) T_{a-2}g_2 \right).
\end{align*}
By the formula of $h^{(2)}_{I'J}$ in Table  \ref{table:Table160487} and 
a computation of signature 
\begin{equation}\label{eqn:Kpr}
\sgn(K\cup\{p\};r)\sgn(K\cup \{p\} ; p, r) = -\sgn(K\cup\{p\}; p) =  -\sgn(K; p)
\end{equation}
from Lemma \ref{lem:sgn} (1) and (3),
we have
\begin{align}
(M'2)
&=\zeta_p^{-1}\frac{\partial}{\partial \zeta_p}
\left( 
\widetilde{Q}_{K\cup \{1\}}(\zeta')
T_{a-2}g_2\right) 
- \sum_{r\in K} \frac{\partial}{\partial \zeta_r} \left(\zeta_rT_{a-2}g_2\right),\label{eqn:M2_1}
\end{align}
where $\widetilde{Q}_{K\cup \{1\}}(\zeta') = Q_{K\cup \{1\}}(\zeta') - 
\frac{i-1}{n-1}Q_{n-1}(\zeta')$.
By applying the formul\ae{} in Lemma \ref{lem:1-6} accordingly,
\eqref{eqn:M2_1} is evaluated to

\begin{align*}
\eqref{eqn:M2_1}
&=
\zeta_p^{-1}\frac{\partial}{\partial \zeta_p}
\left( \widetilde{Q}_{K\cup \{1\}}(\zeta')
T_{a-2}g_2\right) 
-\sum_{r\in K} 
\frac{\partial}{\partial \zeta_r} \left(\zeta_rT_{a-2}g_2\right) \\
&=\zeta_p^{-1}\frac{\partial}{\partial \zeta_p}\left( 
Q_{K\cup \{1\}}(\zeta')
T_{a-2}g_2
-\frac{i-1}{n-1}T_{a}g_2\right)
-\sum_{r\in K} \big(T_{a-2}g_2 + \zeta_r^2T_{a-4}((a-2-\vartheta_t)g_2)\big)\\
&=Q_{K\cup \{1\}}(\zeta')  T_{a-4}((a-2-\vartheta_t)g_2
- \frac{i-1}{n-1} T_{a-2}((a-\vartheta_t) g_2)\\
&\quad -\left( (i-2) T_{a-2}g_2 + 
(Q_{K}(\zeta')T_{a-4}((a-2-\vartheta_t)g_2)\right)\\
&=\zeta_1^2T_{a-4}((a-2-\vartheta_t)g_2)
-T_{a-2}\left((i-2)+\frac{i-1}{n-1}(a-\vartheta_t)g_2\right).
\end{align*}

\vskip 0.1in
\noindent
Thus, \eqref{eqn:M2} holds.

Now, by using Lemma \ref{lem:Rla},
one obtains \eqref{eqn:03081} from 
\eqref{eqn:MS_Case3}, \eqref{eqn:M0}, \eqref{eqn:M1}, and \eqref{eqn:M2}
as $M_{IJ} = M_{IJ}^{\mathrm{scalar}} + (M0) + (M1) + (M2)$.
This completes the proof for Case (3).
\end{proof}

%%%%%%%%%%%%%%%%%%%%%%%%%%%%%%%%%%%%%%%%%%%%%%%
\subsection{Step 5: Deduction from $M_{IJ}=0$ to $L_r(g_0,g_1, g_2) = 0$}\label{subsec:Meqn5}
In this final step we deduce $L_r(g_0,g_1, g_2) = 0$ from $M_{IJ}=0$.
The following observation is useful.

\begin{lem}\label{lem:p1p2}
Let $p_1$, $p_2$ be $O(n-1,\mathbb{C})$-invariant polynomials in 
$\mathrm{Pol}(\mathbb{C}^n)$ 
and $R \subsetneq \{1, \ldots, n-1\}$.
If $(\sum_{r\in R}\zeta_r^2) p_1 + p_2 = 0$, then
$p_1 = p_2 = 0$.
\end{lem}

\begin{proof}
If $p_1 \neq 0$, then 
it follows from the hypothesis that $\sum_{r\in R}\zeta_r^2
= \frac{-p_2}{p_1}$ is $O(n-1,\mathbb{C})$-invariant.
However, since $R \subsetneq \{1, \ldots, n-1\}$, 
we have  $\sum_{r\in R}\zeta_r^2
\notin \mathrm{Pol}(\mathbb{C}^n)^{O(n-1,\mathbb{C})}$.
Hence, $p_1=0$ and, consequently, $p_2 = 0$.
\end{proof}

\begin{prop}\label{prop:MIJL}
Let $n\geq 3$ and $1 \leq i\leq n$.
Let $g_k \in \mathrm{Pol}_{a-k}[t]_{\mathrm{even}}$ $(k=0,1,2)$.
\begin{enumerate}
\item Suppose $i=1$. The following two conditions on $(g_0,g_1)$ are equivalent:
\begin{enumerate}
\item[(i)] $M_{IJ}=0$ for all $I \in \mathcal{I}_{n,i}$ and $J \in \mathcal{I}_{n-1,i-1}$.
\item[(ii)] $L_r(g_0,g_1,0)=0$ $(r=2,7,9)$.
\end{enumerate}

\item Suppose $2\leq i \leq n-1$. The following two conditions on 
$(g_0,g_1,g_2)$ are equivalent:
\begin{enumerate}
\item[(i)] $M_{IJ}=0$ for all $I\in\mathcal I_{n,i}$ and $J\in\mathcal I_{n-1,i-1}$. 
\item[(ii)] $L_r(g_0,g_1, g_2) = 0$ for all $r =1, \ldots, 7$.
\end{enumerate}
\item Suppose $i=n$. The following two conditions on 
$(g_0,g_1,g_2)$ are equivalent:
\begin{enumerate}
\item[(i)] $M_{IJ}=0$ for all $I\in\mathcal I_{n,i}$ and $J\in\mathcal I_{n-1,i-1}$. 
\item[(ii)] $L_7(g_0,g_1, g_2) = 0$.
\end{enumerate}
\end{enumerate}
\end{prop}

\begin{proof}
(1) Suppose $i=1$. Then the equivalence follows from Proposition \ref{prop:3cases}
and Remark \ref{rem:Meqn1}.

(2) Suppose $2 \leq i \leq n-2$.
By Proposition \ref{prop:11cases},
we can replace condition (i)  with the condition that 
$M_{IJ}=0$ for all $(I, J)$ in $(1), \ldots, (11)$ in Table \ref{table:IJ}.
By Lemma \ref{lem:Meqn}, the implication from (ii) to (i) is then clear.
The other implication also easily follows from Lemmas \ref{lem:Meqn} 
and \ref{lem:p1p2}, 
as $T_b(g(t))$ is $O(n-1,\mathbb{C})$-invariant for any $b \in \mathbb{N}$
(see \eqref{eqn:Tabij}).
For $i=n-1$, we can replace condition (i) with the condition that 
$M_{IJ}=0$ for the six cases $(1), \ldots, (4), (8), (10)$ in Table \ref{table:IJ},
as we saw in (c2) of Remark \ref{rem:i1n-1n}. If $M_{IJ}=0$ for the six cases,
we still get $L_r(g_0,g_1,g_2)=0$ for $r=1,2,\ldots, 7$ by Lemmas \ref{lem:Meqn}
and \ref{lem:p1p2}. Thus the implication (i)$\Rightarrow$(ii) is verified also for $i=n-1$.
The converse implication is clear.

(3) Suppose $i=n$.
The equivalence follows from 
Proposition \ref{prop:1case} and Remark \ref{rem:nL7_2}
with the same argument as above.
\end{proof}

Now we give a proof for Theorem \ref{thm:7}, as a summary of this section.

\begin{proof}[Proof for Theorem \ref{thm:7}]
The equivalence of the statements follow from
Lemma \ref{lem:C0enough},
Lemma \ref{lem:FMIJ}, and
Propositions \ref{prop:11cases} and \ref{prop:MIJL}.
\end{proof}

\newpage
%%%%%%%%%%%%%%%%%%%%%%%%%%%%%%%%%%%%%%%%%%%%%%%
\section{F-system for symmetry breaking operators ($j = i-2$, $i+1$ case)}\label{sec:codiff}

\index{A}{Jjnu@$J(j,\nu)_\beta$, principal series of $O(n,1)$}
In this chapter we solve the F-system for $j=i+1$, 
and give a complete classification of differential symmetry breaking operators 
which raise the degree of differential forms by one or decrease the degree by two,
\begin{align*}
I(i,\lambda)_\alpha &\To J(i+1, \nu)_\beta,\\
I(i,\lambda)_\alpha &\To J(i-2, \nu)_\beta,
\end{align*}
for $\lambda, \nu \in \C$ and $\alpha, \beta \in \Z/2\Z$.

In contrast to the case with $j=i-1, i$ that was treated in Chapter \ref{sec:7}, 
we see that there are not many differential symmetry breaking operators
for $j = i-2$ or $i+1$. Here are the main results of this chapter,
which are a part of Theorem \ref{thm:1} ($j=i-2, i+1$ case):

\begin{thm}\label{thm:Ai+}
Suppose $0\leq i \leq n-2$, $\lambda, \nu \in \C$, and $\alpha, \beta \in \Z/2\Z$.
Then the following three conditions on $(i,\lambda, \nu, \alpha, \beta)$ are equivalent:
\begin{enumerate}
\item[(i)] 
$\mathrm{Diff}_{G'}(I(i,\lambda)_\alpha, J(i+1,\nu)_\beta)\neq\{0\}$.
\item[(ii)]
$\mathrm{dim}_\C\mathrm{Diff}_{G'}(I(i,\lambda)_\alpha, J(i+1,\nu)_\beta)=1$.
\item[(iii)] 
$\lambda \in \{0, -1, -2, \cdots\}$, $\nu=1$, $\beta \equiv \alpha+\lambda+1\;\mathrm{mod}\; 2$
when $i=0$; \\
$\lambda=i$, $\nu = i+1$, $\beta \equiv \alpha+1\;\mathrm{mod}\;2$ when $i \geq 1$.
\end{enumerate}
\end{thm}

\begin{thm}\label{thm:Ai--}
Suppose $2\leq i \leq n$, $\lambda,\nu \in \C$, and $\alpha, \beta \in \Z/2\Z$.
Then the following three conditions on $(i,\lambda, \nu,\alpha, \beta)$ are equivalent:
\begin{enumerate}
\item[(i)] $\mathrm{Diff}_{G'}(I(i,\lambda)_\alpha, J(i-2, \nu)_\beta) \neq \{0\}$.
\item[(ii)] $\mathrm{dim}_\C\mathrm{Diff}_{G'}(I(i,\lambda)_\alpha, J(i-2, \nu)_\beta) =1$.
\item[(iii)] $\lambda \in \{0, -1, -2, \cdots\}, \nu=1, \beta \equiv \alpha+\lambda +1
\; \mathrm{mod} \; 2$
when $i=n$;\\
$(\lambda,\nu) = (n-i, n-i+1), \beta \equiv \alpha+1 \; \mathrm{mod} \; 2$ when $2\leq i \leq n-1$.
\end{enumerate}
\end{thm}

For the proof of Theorems \ref{thm:Ai+} and \ref{thm:Ai--},
we first observe that the latter is derived from the former.
In fact, the 
\index{B}{duality theorem for symmetry breaking operators
(principal series)}
duality theorem for symmetry breaking operators
(see Theorem \ref{thm:psdual})
implies that there is a natural bijection:
\begin{equation*}
\mathrm{Diff}_{G'}\left(I(i,\lambda)_\alpha, J(i-2,\nu)_\beta\right)
\simeq
\mathrm{Diff}_{G'}\left(I\left(\tilde{i}, \lambda\right)_\alpha,
J\left(\tilde{i}+1, \nu\right)_\beta\right)
\end{equation*}
where $\tilde{i}:=n-i$. 
Then it is easy to see that $(\tilde{i}, \lambda,\nu, \alpha, \beta)$
satisfies the condition (iii) in Theorem \ref{thm:Ai+} if and only if 
$(i,\lambda, \nu, \alpha, \beta)$ satisfies the condition (iii) in Theorem \ref{thm:Ai--},
whence we conclude that Theorem \ref{thm:Ai--} follows
from Theorem \ref{thm:Ai+} applied to the right-hand side.
The rest of this chapter is devoted to the proof of Theorem \ref{thm:Ai+}.

%%%%%%%%%%%%%%%%%%%%%%%%%%%%%%%%%%%%%%%%%%%%%%%
\subsection{Proof of Theorem \ref{thm:Ai+}}\label{sec:iiplsu}

We have seen in \eqref{eqn:153091} that the F-method gives
a natural isomorphism 
\begin{equation*}
\mathrm{Diff}_{G'}(I(i,\lambda)_\alpha, J(i+1, \nu)_\beta)
\simeq
Sol\left(\mathfrak{n}_+; \sigma^{(i)}_{\lambda,\alpha}, \tau^{(i+1)}_{\nu,\beta} \right),
\end{equation*}
where 
$Sol(\mathfrak{n}_+; \sigma^{(i)}_{\lambda,\alpha},\tau^{(i+1)}_{\nu,\beta})$
is the space of 
$\mathrm{Hom}_{\C}\left(\Exterior^i(\C^n), \Exterior^{i+1}(\C^{n-1})\right)$-valued
polynomial solutions on $\mathfrak{n}_+\simeq \C^n$ to the F-system 
associated to
the outer tensor product representations 
\index{A}{1sigma-lambda-alpha@$\sigma^{(i)}_{\lambda, \alpha}$, representation of $P$ on $\Exterior^i(\C^n)$}
$\sigma^{(i)}_{\lambda,\alpha}=\Exterior^i(\C^n) \boxtimes (-1)^\alpha \boxtimes \C_\lambda$ 
and
\index{A}{1tau-nu-beta@$\tau^{(j)}_{\nu,\beta}$, representation of $P'$ on $\Exterior^j(\C^{n-1})$}
$\tau^{(i+1)}_{\nu,\beta}=\Exterior^{i+1}(\C^{n-1}) \boxtimes (-1)^\beta \boxtimes \C_\nu$,
of $L$ and $L'$, respectively.
Then Theorem \ref{thm:Ai+} is deduced from the following explicit results.

\begin{thm}\label{thm:Fiiplus} 
Suppose $0\leq i \leq n-2$.
We recall from \eqref{eqn:Hi+} that 
$h^{(1)}_{i \to i+1} \colon \Exterior^i(\C^n) \To 
\Exterior^{i+1}(\C^{n-1}) \otimes \mathcal{H}^1(\C^{n-1})$
 is a nonzero $O(n-1)$-homomorphism. 
 Let $\lambda, \nu \in \C$, and $\alpha, \beta \in \Z/2\Z$.
 \index{A}{H0ii+1@$h^{(1)}_{i\to i+1}$}
 Then
\index{A}{Sol2@$Sol(\mathfrak n_+;\sigma^{(i)}_{\lambda,\alpha},
\tau^{(j)}_{\nu,\beta})$}
\begin{eqnarray*}
&&
Sol\left(\mathfrak{n}_+; \sigma^{(i)}_{\lambda,\alpha}, \tau^{(i+1)}_{\nu,\beta} \right)
\\&=&
\begin{cases}
\C\left(T_{-\lambda}
\widetilde C_{-\lambda}^{\lambda-\frac{n-1}2}\left(e^{\frac{\pi\sqrt{-1}}2}t\right)\right)
h^{(1)}_{i\to i+1}
& \mathrm{if} \;
\nu=1, -\lambda \in \N, \; \beta-\alpha\equiv 1-\lambda\,\mathrm{mod}\,2,\; i=0,\\
\C\cdot h^{(1)}_{i \to i+ 1}
& \mathrm{if}\; (\lambda,\nu) = (i,i+1),\; \beta\equiv\alpha+1\,\mathrm{mod}\,2, \;
1 \leq i \leq n-2,\\
\{0\}& \mathrm{otherwise}.
\end{cases}
\end{eqnarray*}
\end{thm}

In order to determine $Sol(\mathfrak{n}_+; \sigma^{(i)}_{\lambda,\alpha}, 
\tau^{(i+1)}_{\nu,\beta})$, we begin with a description of 
$\mathrm{Hom}_{L'}(\sigma^{(i)}_{\lambda,\alpha}, \tau^{(i+1)}_{\nu,\beta}
\otimes \mathrm{Pol}[\zeta_1, \ldots, \zeta_n]$).

\begin{lem}\label{lem:152202}
Suppose that $0 \leq i \leq n-2$. Then,
\begin{eqnarray*}
&&\mathrm{Hom}_{L'}\left(
\sigma^{(i)}_{\lambda,\alpha}, \tau^{(i+1)}_{\nu,\beta}\otimes
\mathrm{Pol}[\zeta_1,\cdots,\zeta_n]\right)\\
&\simeq&
\begin{cases}
\left\{(T_{\nu-\lambda-1}g)h^{(1)}_{i \to i+1}\colon g\in\mathrm{Pol}_{\nu-\lambda-1}[t]_{\mathrm{even}}\right\} & \mathrm{if} \; \nu-\lambda\in\N_+\;\mathrm{and}\;
\beta-\alpha\equiv\nu-\lambda\;\mathrm{mod}\;2,\\
\{0\} &  \mathrm{otherwise}.
\end{cases}
\end{eqnarray*}
\end{lem}

\begin{proof}
The statement follows from Proposition \ref{prop:Tgh} and Lemma \ref{lem:Lprime}.
\end{proof}

From now, assume $\nu-\lambda\in\N_+$ and 
$\beta-\alpha\equiv\nu-\lambda\;\mathrm{mod}\;2$. We set
\begin{equation*}
a:=\nu-\lambda.
\end{equation*}
Then it follows from Proposition \ref{prop:Fmethod2} and Lemma \ref{lem:152202} that
we have a bijection:
\begin{align*}
&\left\{g\in\mathrm{Pol}_{a-1}[t]_{\mathrm{even}} 
: \widehat{d\pi_{(i,\lambda)^*}}(N_1^+)(T_{a-1}g) h^{(1)}_{i \to i + 1}=0\right\}
\stackrel{\sim}{\to}
Sol\left(\mathfrak n_+;\sigma^{(i)}_{\lambda,\alpha}, \tau^{(i+1)}_{\nu,\beta}\right),
\end{align*}
by $g \mapsto \psi:=(T_{a-1}g) h^{(1)}_{i \to i+1}$.

Given $g \in \mathrm{Pol}_{a-1}[t]_{\mathrm{even}}$, we define
$\psi$ as above, and polynomials $M_{I \widetilde{I}}$ 
of $n$ variables $\zeta_1,\ldots, \zeta_n$
for $I \in \mathcal{I}_{n,i}$ and $\widetilde{I} \in \mathcal{I}_{n-1,i+1}$ by
\index{A}{MIJ@$M_{IJ}$, matrix component of 
$\widehat{d\pi_{(i,\lambda)^*}}(N_1^+)\psi$}
\begin{equation*}
M_{I\widetilde I}\equiv M_{I \widetilde{I}}(g)
:=\langle \widehat{d\pi_{(i,\lambda)^*}}(N_1^+)\psi(e_I),e_{\widetilde I}^\vee\rangle.
\end{equation*}
As in Section \ref{subsec:MIJF}, 
clearly $\widehat{d\pi_{(i,\lambda)^*}}(N_1^+)\psi = 0$ 
if and only if $M_{I \widetilde{I}} =0$ for all $I \in \mathcal{I}_{n,i}$ and 
$\widetilde{I} \in \mathcal{I}_{n-1,i+1}$.

Now the proof of Theorem \ref{thm:Fiiplus} is reduced to the following lemma:

\begin{lem}\label{lem:Fiiplus}
Suppose $a:=\nu-\lambda\in\N_+$, $\beta\equiv\alpha+1$ $\mathrm{mod}$ $2$, 
and $g(t)\in\mathrm{Pol}_{a-1}[t]_{\mathrm{even}}$ is a nonzero polynomial such that 
$M_{I \widetilde{I}}(g) =0$ for all $I \in \mathcal{I}_{n,i}$ 
and $\widetilde{I} \in \mathcal{I}_{n-1,i+1}$.

\begin{enumerate}
\item If $i=0$, then $\lambda=1-a, \nu=1$, and $g$ is proportional to 
$\widetilde C_{a-1}^{\lambda-\frac{n-1}2}\left(e^{\frac{\pi\sqrt{-1}}2}t\right).$
\item If $i\geq 1$, then  $\lambda=i$, $\nu=i+1$, $a=1$ and $g(t)$ is a constant.
\end{enumerate}
\end{lem}

In order to prove Lemma \ref{lem:Fiiplus}, 
we examine the matrix components $M_{I\widetilde{I}}$ by decomposing
\index{A}{Mscalar@$M_{IJ}^{\mathrm{scalar}}$}
\index{A}{Mvect@$M_{IJ}^{\mathrm{vect}}$}
\begin{equation*}
M_{I\tilde{I}} = M^{\mathrm{scalar}}_{I\widetilde{I}} 
+ M^{\mathrm{vect}}_{I\widetilde{I}}
\end{equation*}
as in Proposition \ref{prop:MIJ}, corresponding to the decomposition of 
$\widehat{d\pi_{(i,\lambda)^*}}(N_1^+)$ into the scalar and vector parts.
We use the following lemma.

\begin{lem}\label{lem:20160520}
For $I\in\mathcal I_{n,i}$ and $\widetilde I\in\mathcal I_{n-1,i+1}$, we set
\begin{align*}
\psi_{I\widetilde I}:=\langle\psi(e_I),e_{\widetilde I}^\vee\rangle
=(T_{a-1}g)\langle h^{(1)}_{i\to i +1}(e_I), e_{\widetilde{I}}^\vee\rangle.
\end{align*}

\begin{enumerate}
\item [(1)] We have
\begin{equation*}
\psi_{I\widetilde I}=
\left\{
\begin{matrix*}[l]
\sgn(I;p)(T_{a-1}g)\zeta_p & \mathrm{if} &\widetilde I=I\cup\{p\},\\
0 & \mathrm{if} &\widetilde I\not\supset I.
\end{matrix*}
\right.
\end{equation*}

\item[(2)] $M_{I\widetilde I}^{\mathrm{scalar}}=0$ if $\widetilde I\not\supset I$. 
If $\widetilde I=I\cup\{p\}$, then 
\begin{eqnarray*}
M_{I\widetilde I}^{\mathrm{scalar}}
&=&
\sgn(I;p)\left(\frac{\zeta_1\zeta_p}{Q_{n-1}(\zeta')}T_{a-1}(R_{a-1}^{\lambda-\frac{n-1}2}g)
+(\lambda+a-1)\delta_{p1}T_{a-1}g\right),
\end{eqnarray*}
where $\delta_{p1}$ is the Kronecker delta.

\item[(3)] The vector part $M^{\mathrm{vect}}_{I\tilde{I}}$ is given by
\begin{eqnarray*}
M_{I\widetilde I}^{\mathrm{vect}}&=&
\left\{
\begin{matrix*}[l]
\sum_{q\in I}\sgn(I;q)\frac{\partial}{\partial\zeta_q}\psi_{I\setminus\{q\}\cup\{1\},\widetilde I} & \mathrm{if}& 1\not\in I,\\
\sum_{q\not\in I}\sgn(I;q)\frac{\partial}{\partial\zeta_q}\psi_{I\setminus\{1\}\cup\{q\},\widetilde I} & \mathrm{if}& 1\in I.
\end{matrix*}
\right.
\end{eqnarray*}
\end{enumerate}
\end{lem}

\begin{proof}
The first statement is immediate from Table \ref{table:Table160487}
on the matrix coefficients of $h^{(k)}_{i\to j}$.
The second statement follows from Proposition \ref{prop:scalar-vect} (1),
and the third one from Lemma \ref{lem:AII} and Proposition \ref{prop:MIJ}.
\end{proof}

We are ready to prove Lemma \ref{lem:Fiiplus}.

\begin{proof}[Proof of Lemma \ref{lem:Fiiplus}.] 
(1) Suppose $i=0$. Then $M_{I\widetilde I}^{\mathrm{vect}}=0$ 
by Proposition \ref{prop:vecpart}. We note
that $I=\emptyset$.
Let $\widetilde I=\{p\}$ ($1\leq p\leq n-1$). 
By Lemma \ref{lem:20160520} (3),
\begin{eqnarray*}
M_{I\widetilde I}&=&M_{I\widetilde I}^{\mathrm{scalar}}=
\frac{\zeta_1\zeta_p}{Q_{n-1}(\zeta')}T_{a-1}(R_{a-1}^{\lambda-\frac{n-1}2}g)+\delta_{p1}
(\lambda+a-1)(T_{a-1}g).
\end{eqnarray*}

Hence $M_{I\widetilde I}=0$ for all $\widetilde I=\{p\}$ if and only if
$$
R_{a-1}^{\lambda-\frac{n-1}2}g=0\quad\mathrm{and}\quad \lambda+a-1=0.
$$
Thus the first assertion is obtained by 
Lemma \ref{lem:Gesol} about the polynomial solutions to
\index{B}{imaginary Gegenbauer differential equation}
imaginary Gegenbauer differential equation $R^{\lambda-\frac{n-1}{2}}_{a-1}g=0$.

(2) Let $i\geq1$.
Obviously $M_{I\widetilde{I}} = 0$ for all $I$ and $\widetilde{I}$ if $g$ is a constant function.
In order to prove the converse statement, we choose the following four cases.
\begin{enumerate}
\item[] Case 1. $I\subset \widetilde{I}$ and $1 \notin \widetilde{I}$.
\item[] Case 2. $I \subset \widetilde{I}$ and $1 \in I$.
\item[] Case 3. $I \subset \widetilde{I}$ and $1 \notin I$, $1 \in \widetilde{I}$.
\item[] Case 4. $|I \setminus \widetilde{I}| =1$, $1 \notin I$, and $1 \in \widetilde{I}$.
\end{enumerate}

First we treat Case 1. We may write
$\widetilde I=I\cup\{p\}$.
By Lemma \ref{lem:20160520}, we have
\begin{eqnarray*}
M_{I\widetilde I}^{\mathrm{scalar}}&=& \sgn(I;p)\frac{\zeta_1\zeta_p}{Q_{n-1}(\zeta')}T_{a-1}
(R_a^{\lambda-\frac{n-1}2}g),\\
M_{I\widetilde I}^{\mathrm{vect}}&=&0.
\end{eqnarray*}
Hence the condition $M_{I\widetilde I}=0$ implies
\begin{equation}\label{eqn:152255}
R_{a-1}^{\lambda-\frac{n-1}2}g=0.
\end{equation}
Second, we treat Case 2. We may write 
$\widetilde I=I\cup\{p\}$ with $p\neq 1$ and $1\in I$. By using
\eqref{eqn:152255}, we have $M_{I\widetilde I}^{\mathrm{scalar}}=0$, whereas
Lemma \ref{lem:20160520} (3) shows 
$$
M_{I\widetilde I}^{\mathrm{vect}}=\sgn(I;p)\zeta_1\frac{\partial}{\partial\zeta_p}(T_{a-1}g).
$$
Hence the condition $M_{I\widetilde I}=0$ implies
\begin{equation}\label{eqn:152259}
\frac{\partial}{\partial\zeta_p}(T_{a-1}g)=0.
\end{equation}
By Lemma \ref{lem:1-6} (3), \eqref{eqn:152259} yields an ordinary 
differential equation on $g(t)$:
\begin{equation}\label{eqn:152260}
(a-1-\vartheta_t)g(t)=0,
\end{equation}
where $\vartheta_t=t\frac{d}{dt}$.
Third, we treat Case 4 before Case 3. We may write 
$I=K\cup\{n\}$ and $\widetilde I=K\cup\{1,p\}$ with
$K\in\mathcal I_{n-1,i-1}$ and $p\in\{2,\cdots,n-1\}\setminus K$. Then,
again by Lemma \ref{lem:20160520},
$M_{I\widetilde I}^{\mathrm{scalar}}=0$ and
\begin{eqnarray*}
M_{I\widetilde I}^{\mathrm{vect}}&=& \sgn(I;n)\frac{\partial}{\partial\zeta_n}
\psi_{I\setminus\{n\}\cup\{1\},\widetilde I}\\
&=&-\sgn(K;p,n)\zeta_p\frac{\partial}{\partial\zeta_n}
(T_{a-1}g).
\end{eqnarray*}
Hence the condition $M_{I\widetilde I}=0$ implies
$\frac{\partial}{\partial\zeta_n}
(T_{a-1}g)=0$, and therefore we get
$$
T_{a-2}\left(\frac{dg}{dt}\right)=0
$$ 
by Lemma \ref{lem:1-6} (4). Hence $g(t)$ is a constant. In turn, $a=1$ by \eqref{eqn:152260}.

Finally, we consider Case 3, namely, $\widetilde I =I\cup\{1\}$. Then 
\begin{eqnarray*}
M_{I\widetilde I}^{\mathrm{scalar}}&=&(\lambda+a-1)T_{a-1}g,\\
M_{I\widetilde I}^{\mathrm{vect}}&=&-\sum_{q\in I}\frac{\partial}{\partial\zeta_q}(T_{a-1}g)\zeta_q=-i(T_{a-1}g),
\end{eqnarray*}
where we have used \eqref{eqn:152259} for $p\in\{2,\cdots,n-1\}$. Hence we get
$$
M_{I\widetilde I}=(\lambda+a-1-i)T_{a-1}g,
$$
and conclude $\lambda=i$.
Hence the proof of Lemma \ref{lem:Fiiplus} is completed.
\end{proof}

Thus we have proved Theorem \ref{thm:Fiiplus}, whence 
Theorem \ref{thm:Ai+}.

%%%%%%%%%%%%%%%%%%%%%%%%%%%%%%%%%%%%%%%%%%%%%%%

\newpage
%%%%%%%%%%%%%%%%%%%%%%%%%%%%%%%%%%%%%%%%%%%%%%
\section{Basic operators in differential geometry 
and conformal covariance}\label{sec:6}

In this chapter we collect some elementary properties 
of basic operators such as the Hodge star operators, 
the codifferential $d^*$, and the interior multiplication 
$\iota_{N_Y(X)}$
by the normal vector field for hypersurfaces $Y$ in pseudo-Riemannian manifolds $X$.
These operators are obviously invariant under isometries, but also satisfy
certain conformal covariance which we formulate in terms of the representations
\index{A}{1pi@$\varpi^{(i)}_{u,\delta}$, conformal representation on $i$-forms}
$\varpi_{u,\delta}^{(i)}$ ($u\in\C,\delta\in\Z/2\Z$), see \eqref{eqn:varpi},
of the conformal group on the space $\mathcal E^i(X)$ of $i$-forms.

The conformal covariance of the Hodge star
plays an important role in 
the classification of differential symmetry breaking operators 
as we have seen in Theorem \ref{thm:1} and shall see in Section \ref{subsec:Cdual}, 
whereas that of 
the other operators such as $d,d^*$ or $\iota_{N_Y(X)}$
is only a small part of the global conformal covariance of our symmetry breaking operators
$\mathcal D_{u,a}^{i\to j}$.

%%%%%%%%%%%%%%%%%%%%%%%%%%%%%%%%%%%%%%%%%%%%%%
\subsection{Twisted pull-back of differential forms by conformal transformations}
\label{subsec:conf}

Suppose $(X, g_X)$ and $(X', g_{X'})$ are pseudo-Riemannian manifolds of the 
same dimension $n$. A local diffeomorphism $\Phi\colon X \to X'$ is said to be 
\emph{conformal} if there exists a positive-valued function 
$\Omega\equiv \Omega_{\Phi}$ 
(\emph{conformal factor}) on $X$ such that 
\begin{equation*}
\Phi^*(g_{X',\Phi(x)}) = \Omega(x)^2g_{X,x} \quad \text{for all $x \in X$}.
\end{equation*}
We define a locally constant function 
\index{A}{or@$\mathpzc{or}$|textbf}
$\mathpzc{or}(\Phi)$ on $X$ by
\begin{equation}\label{eqn:orPhi}
\mathpzc{or}(\Phi)(x)
\equiv
\mathpzc{or}_X(\Phi)(x)
=
\begin{cases}
1 &\text{if $\Phi_{*x}\colon T_xX \To T_{\Phi(x)}X$ is orientation-preserving},\\
-1 &\text{if $\Phi_{*x}\colon T_xX \To T_{\Phi(x)}X$ is orientation-reversing}.
\end{cases}
\end{equation}
The \index{B}{twisted pull-back|textbf}
twisted pull-back 
\index{A}{11wPhi@$\Phi^*_{u,\delta}\equiv \left(\Phi^{(i)}_{u,\delta}\right)^*$|textbf}
$\Phi^*_{u,\delta}\equiv \left(\Phi^{(i)}_{u,\delta}\right)^*$
with parameters $u\in \C$ and $\delta \in \Z/2\Z$
on $i$-forms is defined by
\begin{equation}\label{eqn:twistpb}
\Phi^*_{u,\delta}\colon \mathcal{E}^i(X') \To \mathcal{E}^i(X),
\quad \alpha \mapsto \mathpzc{or}(\Phi)^\delta \Omega^u\Phi^*\alpha.
\end{equation}
If $X=X'$ and $G$ is the conformal group of $X$ acting by 
$x \mapsto L_hx$ $(h\in G)$, then the representation 
\index{A}{1pi@$\varpi^{(i)}_{u,\delta}$, conformal representation on $i$-forms}
$\varpi^{(i)}_{u,\delta}$ of $G$ on $\mathcal{E}^i(X)$ introduced in 
\eqref{eqn:varpi} is written as 
\begin{equation}\label{eqn:wLh}
\varpi^{(i)}_{u,\delta}(h) = \left((L_{h^{-1}})^{(i)}_{u,\delta} \right)^*.
\end{equation}

%%%%%%%%%%%%%%%%%%%%%%%%%%%%%%%%%%%%%%%%%%%%%%
\subsection{Hodge star operators under conformal transformations}\label{subsec:6Hodge}
${}$

We recall the standard notion of the Hodge star operator, and fix some notations. 
Given an oriented real vector space $V$ of dimension 
$n=p+q$ equipped with a nondegenerate symmetric bilinear
form $\langle\;, \;\rangle$ of signature $(p,q)$, 
we have canonical isomorphisms $V \simeq V^\vee$ and $\Exterior^nV \simeq \R$.
Then the natural perfect pairing
\begin{equation*}
\Exterior^iV\times\Exterior^{n-i}V\To\Exterior^nV\simeq\R
\quad (0\leq i \leq n)
\end{equation*}
gives rise to the $i$-th 
\index{B}{Hodge star operator|textbf}
\emph{Hodge star operator}
\index{A}{0star@$*$, Hodge star operator|textbf}
\begin{equation}\label{eqn:Hodgeself}
*\colon \Exterior^iV\to\left(\Exterior^{n-i}V\right)^\vee\simeq\Exterior^{n-i}V^\vee.
\end{equation}
Equivalently, for any $\omega, \eta \in \Exterior^iV$,
\begin{equation*}
\omega \wedge *\eta = \langle \omega, \eta\rangle_i \mathrm{vol},
\end{equation*}
where $\langle\;, \;\rangle_i$ denotes the nondegenerate symmetric bilinear form
on $\Exterior^iV$ induced by
\begin{equation*}
\langle u_1 \wedge \cdots \wedge u_i,
v_1 \wedge \cdots \wedge v_i \rangle_i
=\mathrm{det}\left(\langle u_k, v_\ell \rangle\right)
\end{equation*}
and $\mathrm{vol} \in \Exterior^nV$ is the oriented unit.

Suppose $\{e_1,\cdots,e_n\}$ is a basis of $V$
such that $\langle e_k, e_k \rangle = \pm 1$ $(1 \leq k \leq n)$
and $\langle e_k, e_\ell \rangle =0$ $(k\neq \ell)$.
If $e_1\wedge\cdots\wedge e_n$ defines the orientation
of $V$, then
\begin{equation}\label{eqn:stareI}
*e_I=(-1)^{\mathrm{neg}(I)}\eps_n(I)e_{I^c}\qquad\mathrm{for}\, I\in\mathcal I_{n,i},
\end{equation}
where we set 
\index{A}{Ic@$I^c$, complement of index set $I$|textbf}
$I^c:=\{1,2,\cdots,n\}\setminus I$ and
\index{A}{N1neg(I)@$\mathrm{neg}(I)$|textbf}
\index{A}{1Eepsilon@$\eps_n(I)$, signature of index set $I$|textbf}
\begin{align}
\mathrm{neg}(I)&:= |\{i \in I : \langle e_k, e_k \rangle = -1\}|,\label{eqn:negI}\\
\eps(I)\equiv\eps_n(I)&:=(-1)^{|\{(a,b)\in I\times I^c: a>b\}|}=\prod_{a\in I}\sgn(I^c;a).\label{eqn:enI}
\end{align}
The last equality of \eqref{eqn:enI}
follows readily from the definition of $\sgn(I;a)$
(see Definition \ref{def:sign}).
A special case of \eqref{eqn:stareI} shows $*1 = \mathrm{vol}$.
The signature $\eps_n\colon \mathcal I_{n,i}\To\{\pm1\}$ satisfies the following formul\ae.
\begin{alignat}{2}
\eps_n(I)\eps_n(I^c)\quad&=\quad(-1)^{i(n-i)},&&\label{eqn:enIc}\\
\eps_n(I)\eps_n(I\setminus\{\ell\})\quad&=
\quad(-1)^{i+\ell}\quad&&\mathrm{if}\,\ell\in I,\label{eqn:IIell}\\
\eps_n(J)\eps_{n-1}(J)\quad&=\quad1
\quad&&\mathrm{if}\, J\in\mathcal I_{n-1,i}\;
(\subset \mathcal I_{n,i}). \label{eqn:eJn}
\end{alignat}

From \eqref{eqn:stareI} and \eqref{eqn:enIc}, we have 
\begin{equation}\label{eqn:star2}
** = (-1)^{(n-i)i} (-1)^q\;\mathrm{id} \quad \text{on $\Exterior^i V$}.
\end{equation}
For an oriented pseudo-Riemannian manifold $(X,g)$ of dimension $n$, the Hodge star operator is a linear map
\begin{equation*}
*_X\equiv *\colon \mathcal E^i(X)\To\mathcal E^{n-i}(X)
\end{equation*}
induced from the bijection 
$*_{X,x}\colon \Exterior^iT^\vee_xX\To\Exterior^{n-i}T^\vee_x X$ 
for the cotangent space $T_x^\vee X$
at every $x\in X$. If $\omega, \eta \in \mathcal{E}^i(X)$ 
and if 
at least one of the supports of $\omega$ or $\eta$
is compact, we set
\begin{equation}\label{eqn:L2X}
(\omega, \eta) := \int_X \omega \wedge *\eta.
\end{equation}

We continue a review on basic notion and results.
The 
\index{B}{codifferential|textbf}
codifferential $d^*\colon \mathcal E^{i}(X)\To\mathcal E^{i-1}(X)$ is given by
\begin{equation}\label{eqn:dstardef}
d^*=(-1)^{i}*^{-1}d*=(-1)^{ni+n+1}(-1)^q*d*=(-1)^{n+i+1}*d*^{-1}
\end{equation}
if the signature of the pseudo-Riemannian metric is $(n-q,q)$.
The second and third identities follow from \eqref{eqn:star2}.
Then the codifferential $d^*$ is the formal adjoint of the exterior derivative
$d$ in the sense that
\begin{equation*}
(\omega, d^*\eta) = (d \omega, \eta) \quad
\text{for all $\omega \in \mathcal{E}^i_c(X)$ and
$\eta \in \mathcal{E}^{i+1}(X)$},
\end{equation*}
because $\int_X d(\omega \wedge *\eta) = 0$.

\begin{lem}\label{lem:dd}
The following identities hold:
\begin{equation*}
*dd^**^{-1}=d^*d,
\quad *d^*d*^{-1}=dd^*.
\end{equation*}
\end{lem}

\begin{proof}
Use \eqref{eqn:star2} and \eqref{eqn:dstardef}.
\end{proof}

\index{A}{11D1Delta@$\Delta= -(dd^* + d^*d)$, Hodge Laplacian|textbf}
\index{B}{Hodge Laplacian|textbf}
The Hodge Laplacian $\Delta$, also known as the Laplace--de Rham operator,
is a differential operator acting on differential forms is given by
\begin{equation}\label{eqn:Lapform}
\Delta = -(dd^* + d^*d).
\end{equation}

Obviously, 
the Hodge star operator commutes with isometries. 
More generally, the Hodge star operator has a conformal covariance, which is formulated
in terms of the twisted pull-back \eqref{eqn:twistpb} as follows.

\begin{lem}\label{lem:20160317}
Suppose that $(X, g_X)$ and $(X',g_{X'})$ are oriented pseudo-Riemannian
manifolds of the same dimension $n$ and that $\Phi\colon X \To X'$ is a conformal
map with conformal factor $\Omega \in C^\infty(X)$. Then,
for any $u \in \C$, $\eps \in \Z/2\Z$ and $0\leq i\leq n$, we have
\begin{equation*}
*_X \circ \left(\Phi^{(i)}_{u,\eps}\right)^* = 
\left(\Phi^{(n-i)}_{u-n+2i,\eps+1}\right)^* \circ *_{X'}
\quad \emph{on $\mathcal{E}^i(X')$.}
\end{equation*}
\end{lem}

\begin{proof}
By $\Phi^*g_{X', \Phi(x)}=\Omega(x)^2g_{X,x}$,
we have the following equality:
\begin{equation}\label{eqn:starconf}
*_{X,x}\circ \Phi^* = \mathpzc{or}(\Phi) \Omega(x)^{-n+2i}
\Phi^* \circ *_{X',\Phi(x)} \quad \text{on $\mathcal{E}^i(X')$.}
\end{equation}
Suppose $\omega \in \mathcal{E}^i(X')$. 
By the definition \eqref{eqn:twistpb} of $(\Phi_{u,\eps}^{(i)})^*$, we have
\begin{align*}
*_X \circ \left(\Phi_{u,\eps}^{(i)}\right)^*\omega
&=*_X\left(\mathpzc{or}(\Phi)^\eps \Omega^u \Phi^*\omega\right).
\end{align*}
By \eqref{eqn:starconf}, the right-hand side is equal to 
\begin{equation*}
\mathpzc{or}(\Phi)^\eps \Omega^u \mathpzc{or}(\Phi) \Omega^{-u+2i}
\Phi^*(*_{X'} \omega) 
=\left(\Phi^{(u-i)}_{u-n+2i, \eps+1}\right)^*(*_{X'}\omega).
\end{equation*}
Hence the lemma is proved.
\end{proof}

By Lemma \ref{lem:20160317} and \eqref{eqn:wLh},
the Hodge star operator can be considered as an intertwining operator 
of the representations $(\varpi_{u,\eps}^{(i)},\mathcal E^i(X))$ of the conformal
group of $X$:

\begin{prop}\label{prop:Hodgeconf}
Suppose that $G$ acts conformally on an oriented pseudo-Riemannian manifold $X$ of dimension $n$.
Let $u\in\C$ and $\eps\in\Z/2\Z$. Then the Hodge star operator
\begin{equation*}
*\colon \mathcal E^i(X)\To\mathcal E^{n-i}(X)
\end{equation*}
intertwines the two representations $\varpi_{u,\eps}^{(i)}$ 
and $\varpi_{u-n+2i,\eps+1}^{(n-i)}$ of $G$, i.e.\ 
\begin{equation*}
* \circ 
\varpi_{u,\eps}^{(i)}(h)=\varpi_{u-n+2i,\eps+1}^{(n-i)}(h)\circ *\quad \mathrm{for\, all}\,\,h\in G.
\end{equation*}
\end{prop}

\begin{example}
For $n\geq 2$ the conformal group of the standard Riemannian sphere $X=S^n$ is given by
 $\mathrm{Conf}(X)\simeq O(n+1,1)/\{\pm I_{n+2}\}$.
The Hodge star operator induces 
an isomorphism
\begin{equation*}
\mathcal E^i(S^n)_{\lambda-i,0}\stackrel{\sim}{\To}\mathcal E^{n-i}(S^n)_{\lambda-n+i,1}
\end{equation*}
as $\mathrm{Conf}(X)$-modules by Proposition \ref{prop:Hodgeconf},
which gives a geometric realization of the 
$G$-isomorphism between principal series representations
\begin{equation}\label{eqn:Iidual}
I(i,\lambda)_i\stackrel{\sim}{\To} I(n-i,\lambda)_i\otimes \chi_{--}
\end{equation}
(see Lemma \ref{lem:psdual}) via \eqref{eqn:Iww}. 
\end{example}

The exterior derivative $d$ commutes with any diffeomorphism.
By the conformal covariance for the Hodge star operator 
(Proposition \ref{prop:Hodgeconf}), we have one for the codifferential $d^*$:

\begin{lem}\label{lem:dstarconf}
Suppose that $X$ and $X'$ are oriented pseudo-Riemennian
manifolds of the same dimension $n$, and that $\Phi\colon X\To X'$
is a conformal map with conformal factor $\Omega \in C^\infty(X)$.
\begin{alignat*}{4}
&(1) \quad &&d_X \circ \left(\Phi^{(i)}_{0,\eps} \right)^* 
&&=\left(\Phi^{(i+1)}_{0,\eps}\right)^* \circ d_{X'} &&\emph{on $\mathcal{E}^i(X')$.}\\
&(2) \quad &&d^*_X \circ \left(\Phi^{(i)}_{n-2i,\eps}\right)^* 
&&= \left(\Phi^{(i-1)}_{n-2i+2,\eps}\right)^* \circ d^*_{X'} &&\emph{ on $\mathcal{E}^i(X')$.}
\end{alignat*}
\end{lem}

\begin{proof}
The first statement is obvious because the exterior derivative $d$ commutes
with any diffeomorphism. To see the second statement, we recall from
\eqref{eqn:dstardef} that $d^*=c * d * $ with $c :=(-1)^{ni+n+1}(-1)^q$
if the signature of the pseudo-Riemannian metric is $(n-q, q)$.
By Proposition \ref{prop:Hodgeconf} and the first statement of 
this lemma, we have
\begin{align*}
d^*_X\circ \left(\Phi^{(i)}_{n-2i,\eps}\right)^*
&=c *_X \circ d_X \circ *_X \circ \left(\Phi^{(i)}_{n-2i,\eps}\right)^*\\
&=c *_X \circ d_X \circ \left(\Phi^{(n-i)}_{0,\eps+1}\right)^*\circ *_{X'}\\
&=c *_X \circ \left(\Phi^{(n-i+1)}_{0,\eps+1}\right)^* \circ d_{X'} \circ *_{X'}\\
&=c\left(\Phi^{(i-1)}_{n-2i+2, \eps}\right)^* \circ *_{X'} \circ d_{X'}\circ *_{X'}\\
&=\left(\Phi^{(i-1)}_{n-2i+2,\eps}\right)^* \circ d^*_{X'}.
\end{align*}
Thus the lemma is proved.
\end{proof}

The following proposition is immediate from Lemma \ref{lem:dstarconf}.

\begin{prop}\label{prop:dstarconf}
Suppose $X$ is an oriented 
pseudo-Riemannian manifold of dimension $n$, and $G$ acts conformally
on $X$. 
\begin{enumerate}
\item The exterior derivative $d\colon  \mathcal{E}^i(X) \To \mathcal{E}^{i+1}(X)$
intertwines the two representations $\varpi^{(i)}_{0,\eps}$ and $\varpi^{(i+1)}_{0,\eps}$
of $G$ for $\eps \in \Z/2\Z$.

\item The codifferential
\index{A}{d*@$d^*$, codifferential}
$d^*\colon \mathcal E^{i+1}(X)\To\mathcal E^i(X)$
intertwines the two representations 
$\varpi^{(i+1)}_{n-2i-2,\eps}$ and 
$\varpi^{(i)}_{n-2i,\eps}$ of $G$ for $\eps\in\Z/2\Z$.
\end{enumerate}
\end{prop}

\begin{rem}
We shall prove in Section \ref{sec:intertwiner} that there does not exist any nonzero
conformally equivariant differential operator 
$\mathcal{E}^i(X)_{u,\delta} \To \mathcal{E}^{i+1}(X)_{v,\eps}$ or
$\mathcal E^{i+1}(X)_{u,\delta}\To \mathcal E^{i}(X)_{v,\eps}$ 
other than 
the differential $d$ or
the codifferential $d^*$ (up to scalar), respectively, 
when $X$ is the standard Riemannian sphere $S^n$.
\end{rem}

Applying the conformal covariance of 
the Hodge star operator,
we obtain a duality theorem for symmetry breaking operators in
conformal geometry:

\index{B}{duality theorem for symmetry breaking operators
(conformal geometry)|textbf}
\begin{thm}[duality theorem]\label{thm:XYduality}
Suppose $(X,g)$ is an $n$-dimensional oriented pseudo-Riemannian manifold, $Y$ is an 
$m$-dimensional submanifold such that $g\vert_Y$ is nondegenerate, and $G'$
is a group acting conformally on $X$  and leaving $Y$ invariant. Then for any $u,v\in\C,
\delta,\eps\in\Z/2\Z$, and $0\leq i\leq n,\;0\leq j\leq m$ there is a natural bijection
\index{A}{E1i@$\mathcal E^i(X)_{u,\delta}$, 
conformal representation on $i$-forms on $X$}
$$
\operatorname{Diff}_{G'}(\mathcal E^i(X)_{u,\delta},\mathcal E^j(Y)_{v,\eps})
\stackrel{\sim}{\To}
\operatorname{Diff}_{G'}(\mathcal E^{n-i}(X)_{u-n+2i,{\delta+1}},\mathcal E^{m-j}(Y)_{v-m+2j,\eps+1}).
$$
\end{thm}
\begin{proof}
Let $*_X$ and $*_Y$ be the Hodge star operators on $(X,g)$ and $(Y,g\vert_Y)$, respectively. Then the assertion of the theorem is deduced from 
Proposition \ref{prop:Hodgeconf}, summarized in 
the following diagram of $G'$-homomorphisms:
$$
\xymatrix{
\mathcal E^i(X)_{u,\delta}\ar[rr]^{\sim}_{*_X} \ar[d]
&& \mathcal E^{n-i}(X)_{u-n+2i,\delta+1}\ar[d]\\
\mathcal E^j(Y)_{v,\eps}\ar[rr]^{\sim}_{*_Y}& &\mathcal E^{m-j}(Y)_{v-m+2j,\eps+1}.
}
$$
\end{proof}

%%%%%%%%%%%%%%%%%%%%%%%%%%%%%%%%%%%%%%%%%%%%%%
\subsection{Normal derivatives under conformal transformations}

Suppose that $(X,g)$ is an oriented pseudo-Riemannian manifold of dimension $n$, 
and $Y$ an oriented submanifold of $X$ such that $g_{\vert Y}$ is nondegenerate.
Let
\index{A}{C1X@$\mathrm{Conf}(X)$|textbf}
$G=\mathrm{Conf}(X)\equiv\mathrm{Conf}(X,g)$
be the group of conformal diffeomorphisms of $(X,g)$, and \\
\index{A}{C1XY@$\mathrm{Conf}(X;Y)$|textbf}
\begin{equation*}
G' = \mathrm{Conf}(X;Y):=\{ h \in G : hY=Y\}.
\end{equation*} 
As in \eqref{eqn:orPhi}, we have group homomorphisms
\index{A}{or@$\mathpzc{or}$}
\begin{equation*}
\mathpzc{or}_X\colon  G \To \{\pm 1\}, \quad 
\mathpzc{or}_Y\colon  G' \To \{\pm 1\},
\end{equation*}
depending on whether or not the transformation preserves the orientation
of $X$, $Y$, respectively.

We begin with the conformal invariance of the restriction map 
$\mathrm{Rest}_Y$.

\begin{lem}\label{lem:1604116}
Let $X$ and $Y$ be oriented pseudo-Riemannian manifolds as above.
Then the restriction map 
\begin{equation*}
\mathrm{Rest}_Y\colon \mathcal{E}^i(X) \To \mathcal{E}^i(Y)
\end{equation*}
is a symmetry breaking operator from the representation
$\varpi^{(i)}_{u,\delta}\vert_{G'}$ of $G$ restricted to $G'$
to the representation $\varpi^{(i)}_{u,\eps}$ of $G'$ for all $u \in \C$
if $\delta \equiv \eps \equiv 0 \; \mathrm{mod}\; 2$.
\end{lem}

\begin{proof}
We consider the condition on $(u,v;\delta, \eps) \in \C^2 \times (\Z/2\Z)^2$
such that $\mathrm{Rest}_{Y}$ intertwines $\varpi^{(i)}_{u,\delta}\vert_{G'}$
and $\varpi^{(i)}_{v,\eps}$.
For $h \in G'$ and $\eta \in \mathcal{E}^i(X)$,
\begin{align*}
\varpi_{v,\eps}^{(i)}(h) \circ \mathrm{Rest}_Y\eta
&=\mathpzc{or}_Y(h)^\eps \Omega(h^{-1}, \mathrm{Rest}_Y\;\cdot)^v
(L_{h^{-1}})^*\mathrm{Rest}_Y\eta,\\
\mathrm{Rest}_Y \circ \varpi_{u,\delta}^{(i)}(h)\eta
&=\mathpzc{or}_X(h)^\delta \Omega(h^{-1}, \mathrm{Rest}_Y\;\cdot)^u
(L_{h^{-1}})^*\mathrm{Rest}_Y\eta,
\end{align*}
by the definition \eqref{eqn:varpi}. Hence the right-hand sides of the two
equalities coincide for any $h \in G'$ if $u = v$ and $\delta \equiv \eps \equiv 0 \; \mathrm{mod} \; 2$.
\end{proof}

Suppose now that $Y$ is of codimension one in $X$.
Then we can define the
\index{B}{normal vector field|textbf}
normal vector field $N_Y(X)$ on $Y$ such that

\index{A}{1iotaYX@$\iota_{N_Y(X)}$|textbf}
\begin{equation*}
\iota_{N_Y(X)}\mathrm{vol}_X=(-1)^{n-1}\mathrm{vol}_Y \quad \text{on $Y$},
\end{equation*}
where $\mathrm{vol}_X$ and $\mathrm{vol}_Y$ are the 
oriented volume forms of $X$ and $Y$, respectively.

\begin{example}
Let $(X,Y) = (\R^n ,\R^{n-1} \times \{0\})$.
With the standard orientation for $Y \subset X$, the normal vector field $N_Y(X)$
is given by
\begin{equation*}
N_Y(X)=\frac{\partial}{\partial x_n}\quad\mathrm{on}\; Y,
\end{equation*}
because $\iotan (dx_1 \wedge \cdots \wedge dx_n) = (-1)^{n-1}
dx_1\wedge \cdots \wedge dx_{n-1}$.
\end{example}

Similarly to the pair $X \supset Y$, suppose $Y'$ is an oriented hypersurface 
of a pseudo-Riemannian manifold $(X',g')$.
Let $\Phi\colon X\to X'$ be a conformal map such that $\Phi(Y)\subset Y'$. 
We write $\Omega\equiv \Omega_{\Phi}\in C^\infty(X)$ for the conformal factor, namely, 
$\Phi^*( g_{X'})=\Omega^2 g_{X}$. 
By a little abuse of notation, we define 
\index{A}{orYX@$\mathpzc{or}_{X/Y}$, relative orientation|textbf}
$\mathpzc{or}_{X/Y}(\Phi) \in \{\pm1\}$ by the identity
\begin{equation}\label{eqn:orXY}
\mathpzc{or}_{X}(\Phi) = \mathpzc{or}_{Y}(\Phi) \mathpzc{or}_{X/Y}(\Phi),
\end{equation}
where 
we recall $\mathpzc{or}_{X}(\Phi) \in \{\pm1\}$ from \eqref{eqn:orPhi},
and $\mathpzc{or}_Y(\Phi) \in \{\pm1\}$ is defined similarly for
$\Phi\vert_{Y} \colon Y \To Y'$.
Then, we have the following:

\begin{lem}\label{lem:1531111}
(1) For all $\omega\in\mathcal E^i(X')$, we have
\begin{equation*}
\iota_{N_{Y}(X)}(\Phi^*\omega)=
\mathpzc{or}_{X/Y}(\Phi)
\Omega\Phi^*\left(\iota_{N_{Y'}(X')}\omega\right)\quad\mathrm{on}\,\; Y.
\end{equation*}
(2) For any $u \in \C$, we have
\begin{equation*}
\left(\mathrm{Rest}_Y\circ \iota_{N_Y(X)} \right)
\circ \left(\Phi^{(i)}_{u,1}\right)^*
=
\left(\Phi^{(i-1)}_{u+1,1} \right)^*\circ 
\left(\mathrm{Rest}_{Y'}\circ \iota_{N_{Y'}(X')}\right)
\quad \text{on $\mathcal{E}^i(X')$.}
\end{equation*}
\end{lem}

\begin{proof}
(1) 
Take $p \in Y$ and local coordinates $(x'_1,\cdots, x'_n)$ on 
$X'$ near $p':=\Phi(p)$ such that $Y'$ is given locally by $x'_n=0$ 
and that $\left\{\frac{\partial}{\partial x'_1},\cdots,\frac{\partial}{\partial x'_n}\right\}$ forms 
an oriented orthonormal basis of $T_{p'}(X')$. 
We set $x_j:= x'_j\circ\Phi$. Then $(x_1,\cdots, x_n)$ are local coordinates near $p$
and the submanifold $Y$ is given locally by
$x_n=0$. Then,
\begin{equation*}
\left\{\Omega(p)\frac{\partial}{\partial x_1},\cdots,
\Omega(p)\frac{\partial}{\partial x_{n-2}},
\mathpzc{or}_Y(\Phi)\Omega(p)\frac{\partial}{\partial x_{n-1}},
\mathpzc{or}_{X/Y}(\Phi) \Omega(p)\frac{\partial}{\partial x_n}\right\}
\end{equation*}
is an oriented orthonormal basis of $T_{p}X$. We note that 
$\frac{\partial}{\partial x_n'}\vert_{p'}$ and 
$\mathpzc{or}_{X/Y}(\Phi)\Omega(p)\frac{\partial}{\partial x_n}\vert_{p}$ 
are the normal vectors to $Y'$ in $X'$ at $p'$, and to $Y$ in $X$ at $p$, respectively.

Let $\omega=fdx_I$ be an $i$-form near $p'$, where $I\in\mathcal I_{n,i}$. 
Then,
\begin{eqnarray*}
\iota_{N_{Y}(X)}(\Phi^*\omega)\vert_{p}&=&
\mathpzc{or}_{X/Y}(\Phi)f(p')\iota_{\Omega(p)\frac{\partial}{\partial x_n}\vert_{p}}dx_I\\
&=&\mathpzc{or}_{X/Y}(\Phi)f(p')\Omega(p)\iota_{\frac{\partial}{\partial x_n}\vert_{p}}dx_I,\\
\Omega\,\Phi^*\left(\iota_{N_{Y'}(X')}\omega\right)\vert_{p}&=& \Omega(p)\Phi^*
\left(\iota_{\frac{\partial}{\partial x_n'}\vert_{p]}}f(p')dx_I'\right)\\
&=&f(p')\Omega(p)\Phi^*\left(\iota_{\frac{\partial}{\partial x_n'}}dx_I'\right).
\end{eqnarray*}
Hence we have proved
$$
\iota_{N_{Y}(X)}(\Phi^*\omega)=
\mathpzc{or}_{X/Y}(\Phi)\Omega\Phi^*\left(\iota_{N_{Y'}(X')}\omega\right)
$$
for all $p'\in Y'$ and $\omega\in\mathcal E^i(X')$.

(2) We consider the condition on $(u,v; \delta, \eps) \in \C^2 \times (\Z/2\Z)^2$
such that 
\begin{equation*}
\left(\mathrm{Rest}_Y\circ \iota_{N_Y(X)}\right)
\circ \left(\Phi^{(i)}_{u,\delta}\right)^* = 
\left(\Phi^{(i-1)}_{v,\eps}\right)^* \circ
\left(\mathrm{Rest}_{Y'} \circ \iota_{N_{Y'}(X')}\right)
\quad \text{on $\mathcal{E}^i(X')$}.
\end{equation*}

Let $\eta \in \mathcal{E}^i(X')$. Then
\begin{align*}
\left(\mathrm{Rest}_{Y}\circ \iota_{N_Y(X)}\right)
\circ \left(\Phi^{(i)}_{u,\delta}\right)^*\eta
&=
\mathpzc{or}_X(\Phi)^\delta(\mathrm{Rest}_Y \circ \Omega)^u
\mathrm{Rest}_Y \circ \iota_{N_Y(X)}(\Phi^*\eta)\\
&=\mathpzc{or}_X(\Phi)^\delta(\mathrm{Rest}_Y \circ \Omega)^{u+1}
\mathpzc{or}_{X/Y}(\Phi) \mathrm{Rest}_Y \circ \Phi^*(\iota_{N_{Y'}(X')}\eta)
\end{align*}
by the first statement. On the other hand,
\begin{equation*}
\left(\Phi^{(i-1)}_{v,\eps}\right)^* \circ (\mathrm{Rest}_{Y'}\circ 
\iota_{N_{Y'}(X')})\eta
=\mathpzc{or}_Y(\Phi)^\eps (\mathrm{Rest}_Y\circ \Omega)^v
\mathrm{Rest}_Y \circ \Phi^* (\iota_{N_{Y'}(X')}\eta)
\end{equation*}
because the conformal factor of the map $\Phi\vert_{Y}\colon Y \To Y'$
is given by $\mathrm{Rest}_{Y}\circ \Omega$.
The right-hand sides are equal if
\begin{equation*}
u+1 = v, \quad \mathpzc{or}_X(h)^\delta\mathpzc{or}_{X/Y}(h) 
= \mathpzc{or}_Y(h)^\eps.
\end{equation*}
Hence the second statement follows from the definition \eqref{eqn:orXY} of 
$\mathpzc{or}_{X/Y}$.
\end{proof}

As an immediate consequence of Lemma \ref{lem:1531111} (2),
we obtain:

\begin{prop}\label{prop:1531110}
Let 
\index{A}{C1X@$\mathrm{Conf}(X)$}
$G=\mathrm{Conf}(X)$ 
\index{A}{C1XY@$\mathrm{Conf}(X;Y)$}
and $G'=\mathrm{Conf}(X;Y):=\{h\in G\,:\,hY=Y\}$. Then the interior multiplication
by a normal vector field
\index{A}{RiotaX@$\mathrm{Rest}_Y\circ\iota_{N_Y(X)}$}
\begin{equation*}
\mathrm{Rest}_Y\circ\iota_{N_Y(X)}\colon\mathcal E^i(X)\To\mathcal E^{i-1}(Y)
\end{equation*}
yields a symmetry breaking operator from  the representation 
${\varpi_{u,\delta}^{(i)}}$ of $G$ to the representation
$\varpi_{u+1,\eps}^{(i-1)}$ of the subgroup $G'$, for all $u\in\C$ if 
$\delta \equiv \eps \equiv 1 \; \mathrm{mod}\;2$.
\end{prop}

\begin{rem}\label{rem:20160329}
Alternatively, we can reduce the proof of Proposition \ref{prop:1531110}
to Lemma \ref{lem:1604116} by Theorem \ref{thm:XYduality} and by the following
identity:
\begin{equation*}
*_Y \circ \mathrm{Rest}_Y \circ \iota_{N_Y(X)} \circ (*_X)^{-1} 
= \kappa \, \mathrm{Rest}_Y
\end{equation*}
with $\kappa = \pm 1$ depending on the signature of
$g(N_Y(X), N_Y(X))$. See 
Lemma \ref{lem:starRest} below for the case $(X,Y)=(\R^n, \R^{n-1})$.
\end{rem}

%%%%%%%%%%%%%%%%%%%%%%%%%%%%%%%%%%%%%%%%%%%%%
\subsection{Basic operators on $\mathcal{E}^i(\R^n)$}\label{subsec:basic}

In this section, we assume that $Y =\R^{n-1}$ is the hyperplane given by $x_n = 0$
in the Euclidean space $X=\R^n$ equipped with the standard flat Riemannian structure,
and 
collect some basic formul\ae{} for operators $d, d^*, *, \iota_{N(Y)}$ and 
$\mathrm{Rest}_Y$ on differential forms $\mathcal{E}^i(\R^n)$ $(0\leq i\leq n)$.

By definition, the 
\index{B}{interior multiplication}
interior multiplication $\iotan$ is given by
\index{A}{1iotan@$\iotan$, interior multiplication|textbf}
\begin{equation}\label{eqn:intn}
\iotan(fdx_I)=
\left\{
\begin{matrix*}[l]
0 & \text{if $n\not\in I$},\\
(-1)^{i-1}f dx_{I\setminus\{n\}}& \text{if $n\in I$},
\end{matrix*}
\right.
\end{equation}
for $f \in C^\infty(\R^n)$ and 
\index{A}{Ink@$\mathcal{I}_{n,k}$, index set}
$I\in\mathcal I_{n,i}$.

By using the notation 
\index{A}{sgnp@$\sgn(I;p)$}
$\sgn(I;\ell)$ (see Definition \ref{def:sign}), 
the differential $d$ and its formal adjoint $d^*$ 
(\index{B}{codifferential}codifferential) are given by
\index{A}{d@$d$, differential}
\index{A}{d*@$d^*$, codifferential}
\begin{eqnarray}
d_{\R^n}(fdx_I)&=&\sum_{\ell\not\in I}\sgn(I;\ell)\frac{\partial f}{\partial x_\ell}dx_{I\cup\{\ell\}},\label{eqn:d}\\
d^*_{\R^n}(fdx_I)&=&-\sum_{\ell\in I}\sgn(I;\ell)\frac{\partial f}{\partial x_\ell}dx_{I\setminus\{\ell\}}. \label{eqn:dstar}
\end{eqnarray}

Combining \eqref{eqn:d} and \eqref{eqn:dstar} with Lemma \ref{lem:sgn} (3), we have
\begin{eqnarray}\label{eqn:ddstar}
d_{\R^n}d^*_{\R^n}(fdx_I)&=&-\sum_{p\in I} \frac{\partial^2f}{\partial x_p^2}dx_I-\sum_{\stackrel{p\in I}{q\not\in I}}
\sgn(I;p,q)\frac{\partial^2f}{\partial x_p\partial x_q}dx_{I\setminus\{p\}\cup\{q\}},\\
d^*_{\R^n}d_{\R^n}(fdx_I)&=&-\sum_{q\not\in I} \frac{\partial^2f}{\partial x_q^2}dx_I+\sum_{\stackrel{p\in I}{q\not\in I}}
\sgn(I;p,q)\frac{\partial^2f}{\partial x_p\partial x_q}dx_{I\setminus\{p\}\cup\{q\}}.\label{eqn:dstard}
\end{eqnarray}
The Laplacian $\Delta_{\R^n} = -\left(d_{\R^n}d^*_{\R^n}+d^*_{\R^n}d_{\R^n}\right)$
on $\mathcal{E}^i(\R^n)$
(\eqref{eqn:Lapform}) takes the form
\index{A}{11D2DeltaR@$\Delta_{\R^{n-1}}$|textbf}
\begin{equation*}
\Delta_{\R^n}\left(fdx_I\right)=\left(\sum_{j=1}^n\frac{\partial^2f}{\partial x_j^2}\right)dx_I.
\end{equation*}

We note that the ``scalar-valued" operators $\frac{\partial}{\partial x_n}$ and 
$\Delta_{\R^n} \in \mathrm{End}(\mathcal{E}^i(\R^n))$ 
commute with any of ``vector-valued" operators 
$*_{\R^n}$, $d_{\R^n}, d^*_{\R^n}$, and $\iotan$. Here are commutation relations
among vector-valued operators $\mathcal{E}^i(\R^n) \To \mathcal{E}^j(\R^n)$:

\begin{lem}\label{lem:152457}
We have the following identities on $\mathcal{E}^i(\R^n)$ $(0\leq i \leq n)$.
\begin{eqnarray*}
&(1)& d_{\R^n}\iota_{\frac{\partial}{\partial x_n} }+\iota_{\frac{\partial}{\partial x_n}}d_{\R^n}=\frac{\partial}{\partial x_n}.\\
&(2)&d_{\R^n}^*\iota_{\frac{\partial}{\partial x_n} }+\iota_{\frac{\partial}{\partial x_n}}d_{\R^n}^*=0.\\
&(3)&\iota_{\frac{\partial}{\partial x_n} }d_{\R^n}d_{\R^n}^*=d_{\R^n}d_{\R^n}^*\iota_{\frac{\partial}{\partial x_n} }+d_{\R^n}^*
\frac{\partial}{\partial x_n}.\\
&(4)&
\iota_{\frac{\partial}{\partial x_n} } 
d^*_{\mathbb{R}^n} d_{\mathbb{R}^n}
=d^*_{\mathbb{R}^n} d_{\mathbb{R}^n}
\iota_{\frac{\partial}{\partial x_n} }
-d^*_{\mathbb{R}^n}\frac{\partial}{\partial x_n}.
\end{eqnarray*}
\end{lem}
\begin{proof}
The first and second statements follow from \eqref{eqn:intn}, \eqref{eqn:d}, and \eqref{eqn:dstar}.
 The third and fourth statements are immediate from (1) and (2).
\end{proof}

Next we deal with differential operators from $i$-forms on $\R^n$ to $j$-forms 
on the hyperplane $\R^{n-1}$.
We collect commutation relations among $d_{\R^n}, d^*_{\R^n}$ 
and $\iotan$ together with the restriction map $\restn$.

\begin{lem}\label{lem:152456}
We have the following identities of operators from
$\mathcal{E}^i(\R^n)$ to $\mathcal{E}^j(\R^{n-1})$.
\begin{eqnarray*}
&(1)& d_{\R^{n-1}}\circ\mathrm{Rest}_{x_n=0}=\mathrm{Rest}_{x_n=0}\circ d_{\R^{n}}.\\
&(2)& d_{\R^{n-1}}^*\circ\mathrm{Rest}_{x_n=0}=\mathrm{Rest}_{x_n=0}\circ (d_{\R^{n}}^*
+\frac{\partial}{\partial x_n}\iotan).\\
&(3)& d_{\R^{n-1}}d_{\R^{n-1}}^*\circ\mathrm{Rest}_{x_n=0}=\mathrm{Rest}_{x_n=0}\circ \left(d_{\R^{n}}d_{\R^{n}}^*+\frac{\partial}{\partial x_n}d_{\R^{n}} \iota_{\frac{\partial}{\partial x_n} }\right).\\
&(4)& d_{\R^{n-1}}^*d_{\R^{n-1}}\circ\restn=\restn\circ\left(
\frac{\partial^2}{\partial x_n^2}+ d_{\R^{n}}^*d_{\R^{n}}-\frac{\partial}{\partial x_n}
d_{\R^n}\iotan\right).
\end{eqnarray*}
\end{lem}
\begin{proof}
(1) Clear. (2) Verified by \eqref{eqn:intn} and \eqref{eqn:dstar}. (3) Immediate from (1) and (2).

(4) Applying $d^*_{\R^{n-1}}$ to the identity (1), and using Lemma \ref{lem:152457} (1), 
we get the fourth statement.
\end{proof}

%%%%%%%%%%%%%%%%%%%%%%%%%%%%%%%%%%%%%%%%%%%%%%%
\subsection{Transformation rules involving the Hodge star operator and $\restn$.}\label{subset:transrules}

This section collects some useful formul\ae{} involving the Hodge star operator, in particular, those for $\R^n$ and its hyperplane $\R^{n-1}$, see Lemma \ref{lem:152305}.

We begin with basic formul\ae{} for the conjugation by the Hodge star operator in $\R^n$.

\begin{defn}\label{defn:Tsharp}
Given an operator $T\colon \mathcal E^{n-i}(\R^n)\To\mathcal E^{n-j}(\R^n)$, 
we define a linear operator $T^\sharp\colon \mathcal E^i(\R^n)\To\mathcal E^j(\R^n)$ by
\begin{equation*}
T^\sharp:=(-1)^{n-i}*_{\R^n}\circ T\circ(*_{\R^n})^{-1}.
\end{equation*}
\end{defn}

\begin{lem}\label{lem:Tsharp}
The correspondence $T \mapsto T^\sharp$ is given as in Table \ref{tab:Tsharp}.

\begin{table}[htbp]
\begin{center}
\caption{Correspondence for $T \mapsto T^\sharp$}
\begin{tabular}{c|c|c|c|c|c|c}
$T$ & $d_{\mathbb{R}^n}$ & $d_{\mathbb{R}^n}^*$ 
&$\iotan$&$\frac{\partial}{\partial x_n}\iotan+d_{\mathbb{R}^n}^*$
&$-d_{\mathbb{R}^n}^*\iotan d_{\mathbb{R}^n}$ & $d_{\mathbb{R}^n}d_{\mathbb{R}^n}^*\iotan$ \\
&&&&&&\\
\hline
&&&&&&\\
$T^\sharp$ & $-d_{\mathbb{R}^n}^*$& $d_{\mathbb{R}^n}$ 
& $-dx_n\wedge$ & 
$d_{\mathbb{R}^n}-\frac{\partial}{\partial x_n} dx_n\wedge$&
$-(dx_n\wedge)\circ d_{\mathbb{R}^n}d_{\mathbb{R}^n}^*$ &$-d_{\mathbb{R}^n}^*d_{\mathbb{R}^n}\circ\left( dx_n\wedge\right)$
\end{tabular}
\label{tab:Tsharp}
\end{center}
\end{table}
\end{lem}
\begin{proof}
The first two formul\ae{} follow from the definition of the Hodge star operator,
the codifferential and \eqref{eqn:star2},
and the last three follow from the first three. 
We thus only demonstrate the third one, namely,
\begin{equation}\label{eqn:iotasharp}
(-1)^{n-i}*_{\R^n}\iota_{\frac{\partial}{\partial x_n}} (*_{\R^n})^{-1} (f dx_{I})
=-dx_n \wedge f dx_{I}
\end{equation}
for $f \in C^\infty(\R^n)$ and $I \in \mathcal{I}_{n,i}$.
Obviously, both sides vanish if $ n \in I$. Suppose $n \notin I$.
Then,
\index{A}{Ic@$I^c$, complement of index set $I$}
\begin{equation*}
(-1)^{n-i}*_{\R^n}\iota_{\frac{\partial}{\partial x_n}} (*_{\R^n})^{-1} (fdx_I)
=-\eps(I^c)\eps(I^c\setminus \{n\}) fdx_{I \cup \{n\}}
\end{equation*}
by \eqref{eqn:stareI} and \eqref{eqn:intn},
which amounts to $(-1)^{i+1}fdx_{I\cup \{n\}}$ by \eqref{eqn:IIell}.
Hence \eqref{eqn:iotasharp} is proved.
\end{proof}

We introduce a linear operator
$\Pi_{n-1}\colon \mathcal{E}^i(\R^n) \to \mathcal{E}^i(\R^n)$ by
\index{A}{11prn@$\prn$, projection onto $\mathrm{Ker}(\iotan)$|textbf}
\begin{equation}\label{eqn:Pin1}
\prn:=\iota_{\frac{\partial}{\partial x_n}}\circ \left( dx_n \wedge \right).
\end{equation}
In the coordinates, for $f \in C^\infty(\R^n)$ and $I \in \mathcal{I}_{n,i}$, 
we have
\begin{equation}\label{eqn:Pindx}
\Pi_{n-1}(fdx_I)=\left\{
\begin{matrix*}[l]
fdx_I&\mathrm{if}& n\not\in I,\\
0&\mathrm{if}& n\in I.
\end{matrix*}
\right.
\end{equation}
Then we have 

\begin{lem}\label{lem:152560}
The following identities hold on $\mathcal{E}^i(\R^n)$:
\begin{eqnarray}
\restn \circ \prn &=& \restn, \label{eqn:RestPi}\\
(dx_n\wedge )\circ \iota_{\frac{\partial}{\partial x_n}} + 
\iota_{\frac{\partial}{\partial x_n}} \circ (dx_n \wedge )
&=&\mathrm{id},
\label{eqn:Pinproj}\\
*_{\R^n}\circ \prn \circ (*_{\R^n})^{-1}&=&\mathrm{id}-\prn.\label{eqn:starPi}
 \end{eqnarray}
\end{lem}

\begin{proof}

The first identity follows immediately from \eqref{eqn:Pindx}.
A simple computation using \eqref{eqn:intn} and \eqref{eqn:d}
shows the second identity. To see  the third identity
\eqref{eqn:starPi}, we apply Lemma \ref{lem:Tsharp}.
Then
\begin{align*}
-(*_{\R^n}) \circ \Pi_{n-1} \circ (*_{\R^n})^{-1}
&=\left((-1)^{n-(i-1)} *_{\R^n} \circ 
\iota_{\frac{\partial}{\partial x_n}}
\circ (*_{\R^n})^{-1} \right) \circ
\left( (-1)^{n-i} *_{\R^n} \circ (dx_n \wedge) \circ (*_{\R^n})^{-1}\right)\\
&=(-dx_n \wedge)\circ \iota_{\frac{\partial}{\partial x_n}}.
\end{align*}
Hence we get \eqref{eqn:starPi} 
by \eqref{eqn:Pin1} and \eqref{eqn:Pinproj}.
\end{proof}

By the definition of the 
\index{B}{interior multiplication}
interior multiplication, we have the following direct sum decomposition
\index{A}{1iotan@$\iotan$, interior multiplication}
\begin{equation*}
\mathcal{E}^i(\R^n)=\mathrm{Image}(dx_n \wedge ) 
\oplus \mathrm{Ker}\left(\iota_{\frac{\partial}{\partial x_n}}\right).
\end{equation*}
Then the formul\ae{} \eqref{eqn:Pindx} and \eqref{eqn:Pinproj} show that
the operators
$\mathrm{id} -\Pi_{n-1} = (dx_n \wedge ) \circ \iota_{\frac{\partial}{\partial x_n}}$
and $\Pi_{n-1} = \iota_{\frac{\partial}{\partial x_n}} \circ (d x_n \wedge)$
are the first and second projections, respectively.

Next, we consider the conjugation by the two Hodge star operators 
$*_{\R^n}$ and $*_{\R^{n-1}}$ on $\R^n$ and $\R^{n-1}$, 
simultaneously. For this,
we observe the following basic formula.

\begin{lem}\label{lem:starRest}
We have
\begin{equation*}
*_{\R^{n-1}}\circ \mathrm{Rest}_{x_n=0}\circ (*_{\R^n})^{-1}
=(-1)^{k+1}\mathrm{Rest}_{x_n=0}\circ \iota_{\frac\partial{\partial x_n}}\quad
\emph{\text{on $\mathcal E^k(\R^n)$.}}
\end{equation*}
\end{lem}

\begin{proof}
Fix $I \in \mathcal{I}_{n,k}$ and take $f(x) \equiv f(x',x_n) \in C^\infty(\R^n)$.
We set $\omega:=f(x)dx_I$. By \eqref{eqn:stareI},
we have
\begin{equation*}
(*_{\R^n})^{-1}\omega = \eps_n(I^c)f(x)dx_{I^c},
\end{equation*}
and thus
\begin{equation*}
\mathrm{Rest}_{x_n=0}\circ(*_{\R^n})^{-1}\omega
=\begin{cases}
\; \eps_n(I^c)f(x',0)dx_{I^c} & \text{if $n \in I$,}\\
\; 0 & \text{otherwise}.
\end{cases}
\end{equation*}
In turn, we obtain from \eqref{eqn:eJn}
\begin{equation*}
*_{\R^{n-1}}\circ\mathrm{Rest}_{x_n=0}\circ(*_{\R^n})^{-1}\omega
=\begin{cases}
\; f(x',0) dx_{I\setminus \{n\}} & \text{if $n \in I$,}\\
\; 0 & \text{otherwise}.
\end{cases}
\end{equation*}
Now the proposed equality follows from the identity \eqref{eqn:intn}.
\end{proof}

We collect some useful formul\ae{} involving $*_{\R^n}$ and $*_{\R^{n-1}}$. 
All of the operators $T$ in the next lemma
decrease the degree of forms by one.

\begin{lem}\label{lem:152305}
Let
$(T, T^\flat)$
be a pair of 
linear operators
$T\colon \mathcal E^{n-i}(\R^n)\To\mathcal E^{n-i-1}(\R^n)$ 
and $T^\flat\colon \mathcal E^{i}(\R^n)\To\mathcal E^{i}(\R^n)$
such that 
\begin{enumerate}
\item[(1)] $(T, T^\flat) = (T, -\iotan T^\sharp)$ with 
$T^\sharp$ the linear operator defined in Definition \ref{defn:Tsharp}, or
\item[(2)] $T$ and $T^\flat$ are given in Table \ref{tab:Tflat}.
\end{enumerate}
Then they
satisfy the following identity:
\begin{equation}\label{eqn:Tflat}
(-1)^{n-1}*_{\R^{n-1}}
\circ\restn\circ T\circ(*_{\R^n})^{-1}=\restn\circ T^\flat.
\end{equation}
\begin{table}[H]
\begin{center}
\caption{Pairs of operators $(T,T^\flat)$ satisfying \eqref{eqn:Tflat}}
\begin{tabular}{c|c|c|c|c}
$T$ & $d_{\mathbb{R}^n}^*$ & $\iotan$ &$-d_{\mathbb{R}^n}^*\iotan d_{\mathbb{R}^n}$
&$\frac{\partial}{\partial x_n} \iotan+d_{\mathbb{R}^n}^*$\\
&&&&\\
\hline
&&&&\\
$T^\flat$ & $-\iotan d_{\mathbb{R}^n}$& $\mathrm{id}$ & 
$d_{\mathbb{R}^n}d_{\mathbb{R}^n}^*$ & 
$d_{\mathbb{R}^n}\iotan$
\end{tabular}
\label{tab:Tflat}
\end{center}
\end{table}
\end{lem}

\begin{rem}\label{rem:Tflat}
We note that $T^\flat$ is not uniquely determined by $T$. 
For instance, $(T,T^\flat) = (\iotan, \Pi_{n-1})$ also satisfies 
\eqref{eqn:Tflat}, as $-\iotan (\iotan)^\sharp = \Pi_{n-1}$.
The choices of $T^\flat$ in Table \ref{tab:Tflat} are 
intended for simple description of differential symmetry breaking operators
$\mathcal{D}^{i \to j}_{u,\delta}$, see \eqref{eqn:Dii1}-\eqref{eqn:160462}.
\end{rem}

\begin{proof}[Proof of Lemma \ref{lem:152305}]
(1) We compose the formula
\begin{equation*}
*{_{\R^{n-1}}}\circ\mathrm{Rest}_{x_n=0}\circ
\left(*{_{\R^{n}}}\right)^{-1}=(-1)^i\mathrm{Rest}_{x_n=0}\circ\iota_{\frac{\partial}{\partial x_n}}
\quad \text{on $\mathcal E^{i+1}(\R^n)$}
\end{equation*}
(see Lemma \ref{lem:starRest} (1))
with the defining relation of $T^\sharp$:
$$
*_{\R^n}\circ T\circ(*_{\R^n})^{-1}=(-1)^{n-i}
T^\sharp.
$$
Then we see that \eqref{eqn:Tflat} is equivalent to the relation
\begin{equation}\label{eqn:Tflatsharp}
\restn\circ T^\flat=-\restn\circ\iotan T^\sharp.
\end{equation}
Hence the first statement is proved.

(2) For $T= d^*_{\R^n}$, we have $T^\sharp = d_{\R^n}$ by 
the second formula of 
Lemma \ref{lem:Tsharp}, and therefore $T^\flat = -\iotan d_{\R^n}$ 
satisfies \eqref{eqn:Tflat}.

For $T=\iotan$, we have $T^\sharp = - dx_n \wedge $ by 
the third formula of 
Lemma \ref{lem:Tsharp}, and therefore $-\iotan T^\sharp = \Pi_{n-1}$
by \eqref{eqn:Pin1}. Hence \eqref{eqn:Tflat} holds by \eqref{eqn:RestPi}.

For $T=-d^*_{\R^n} \iotan d_{\R^n}$, we have 
$T^\sharp = - (dx_n \wedge) \circ d_{\R^n} d^*_{\R^n}$ by the fifth
formula of  Lemma \ref{lem:Tsharp},
and therefore $-\iotan T^\sharp = \Pi_{n-1} d_{\R^n}d^*_{\R^n}$. Hence
\eqref{eqn:Tflat} holds again by \eqref{eqn:RestPi}.

For $T= \frac{\partial}{\partial x_n} \iotan + d^*_{\R^n}$, we have
$T^\sharp = d_{\R^n}-\frac{\partial}{\partial x_n} dx_n \wedge$ by 
the fourth formula of Lemma \ref{lem:Tsharp},
and therefore 
\begin{equation*}
-\iotan T^\sharp 
= - \iotan d_{\R^n}+\Pi_{n-1}\frac{\partial}{\partial x_n} 
= d_{\R^n}\iotan + (\Pi_{n-1} - \mathrm{id})\frac{\partial}{\partial x_n}
\end{equation*}
by Lemma \ref{lem:152457} (1). Hence \eqref{eqn:Tflat} holds by 
\eqref{eqn:RestPi}.
\end{proof}

For the operator $T=d_{\R^n}d^*_{\R^n}\iotan$, we also need another expression:

\begin{lem}\label{lem:160235}
For $T=d_{\R^n}d^*_{\R^n} \iotan$,
the following equality holds as elements in 
$\mathrm{Hom}_{\C}(\mathcal{E}^i(\R^n), \mathcal{E}^i(\R^{n-1}))$
\begin{equation}\label{eqn:160235}
(-1)^{n-1}*_{\R^{n-1}}\circ \restn \circ T \circ (*_{\R^n})^{-1}
=d^*_{\R^{n-1}} d_{\R^{n-1}} \circ \restn.
\end{equation}
\end{lem}

\begin{proof}
By the sixth formula of 
Lemma \ref{lem:Tsharp}, $T^\sharp = - d^*_{\R^n} d_{\R^n} \circ (dx_n \wedge)$.
By \eqref{eqn:Tflat} and \eqref{eqn:Tflatsharp}, 
it suffices to show
\begin{equation}\label{eqn:160235a}
\restn \circ \iotan d^*_{\R^n}d_{\R^n} \circ (dx_n \wedge )
=d^*_{\R^{n-1}}d_{\R^{n-1}} \circ \restn.
\end{equation}
By Lemma \ref{lem:152457} (2), the left-hand side of 
\eqref{eqn:160235a} amounts to 
\begin{equation*}
-\restn \circ d^*_{\R^n} \iotan d_{\R^n} \circ (dx_n \wedge).
\end{equation*}
By Lemma \ref{lem:152456} (2) and by $(\iotan)^2 = 0$, 
this is equal to 
\begin{equation*}
-d^*_{\R^{n-1}} \restn \circ \iotan d_{\R^n} \circ (dx_n \wedge).
\end{equation*}
Using Lemma \ref{lem:152457} (1), and by 
the obvious identity
$\restn \circ (dx_n \wedge )=0$,
this is equal to
\begin{equation*}
d^*_{\R^{n-1}}d_{\R^{n-1}} \restn \circ \iotan \circ (dx_n \wedge).
\end{equation*}
Now the desired equation \eqref{eqn:160235a} follows from
\eqref{eqn:Pin1} and \eqref{eqn:RestPi}.
\end{proof}

%%%%%%%%%%%%%%%%%%%%%%%%%%%%%%%%%%%%%%%%%%%%%%%
\subsection{Symbol maps for differential operators acting on forms}\label{subsec:symb}

In this section, 
we relate 
matrix-valued invariant polynomials
\index{A}{H1ijk@$H_{i\to j}^{(k)}$}
\index{A}{H1wijk@$\widetilde H_{i\to j}^{(k)}
\colon\Exterior^i(\C^N)\To \Exterior^j(\C^N)\otimes \mathcal{H}^k(\C^N)$}
\begin{align*}
H^{(k)}_{i \to j} &\in \mathrm{Hom}_{O(N)}\left(\Exterior^i(\C^N), 
\Exterior^j(\C^N) \otimes \mathrm{Pol}^k[\zeta_1, \ldots, \zeta_N]\right),\\
\widetilde{H}^{(k)}_{i\to j} &\in \mathrm{Hom}_{O(N)}\left(\Exterior^i(\C^N), 
\Exterior^j(\C^N) \otimes \mathcal{H}^k(\C^N) \right)
\end{align*}
(see Section \ref{subsec:4Altpol}) with basic operators in differential geometry
via the symbol map
\begin{equation*}
\index{A}{Symb|textbf}
\index{A}{Diffconst@$\mathrm{Diff}^{\mathrm{const}}$}
\mathrm{Symb}\colon 
\mathrm{Diff}^{\mathrm {const}}(\mathcal E^i(\R^N),\mathcal E^j(\R^N))\To
\operatorname{Hom}_\C\left(\Exterior^i(\C^N),\Exterior^j(\C^N)\otimes\operatorname{Pol}[\zeta_1,\cdots,\zeta_N]\right).
\end{equation*}
The dimension $N$ will be taken to be $n-1$ in the next section and to be $n$ in 
Chapter \ref{sec:intertwiner}.

\begin{lem}\label{lem:symbol}
\begin{equation*}
\begin{matrix*}[l]
(1) & \mathrm{Symb}(d_{\R^N}) &=H_{i\to i+1}^{(1)}.\\
(2) & \mathrm{Symb}(d^*_{\R^N}) &=-H_{i\to i-1}^{(1)}.\\
(3) & \mathrm{Symb}(d_{\R^N}d^*_{\R^N}) 
&=-H_{i\to i}^{(2)}=-\widetilde{H}_{i\to i}^{(2)}
-\frac iN Q_N H_{i\to i}^{(0)}.\\
(4) & \mathrm{Symb}(d^*_{\R^N}d_{\R^N}) 
&=-Q_NH_{i\to i}^{(0)}+H_{i\to i}^{(2)}=\widetilde{H}_{i\to i}^{(2)}
+\left(\frac iN-1\right) Q_N H_{i\to i}^{(0)}.\\
(5) & \mathrm{Symb}(d_{\R^N}d^*_{\R^N}
+d^*_{\R^N}d_{\R^N}) &=-Q_NH_{i\to i}^{(0)}.\\
(6) & \mathrm{Symb}\left(\left(\frac iN-1\right)  
d_{\R^N}d^*_{\R^N}+\frac iN d^*_{\R^N}d_{\R^N}\right)
&=\widetilde{H}_{i\to i}^{(2)}.
\end{matrix*}
\end{equation*}
\end{lem}
\begin{proof}
We compare \eqref{eqn:d} with \eqref{eqn:F1p}, which yields the first identity.
Likewise, comparing \eqref{eqn:dstar} with \eqref{eqn:F1m}, we get the 
second identity.

The third and fourth statements follow from \eqref{eqn:ddstar} and \eqref{eqn:dstard}. The last two identities are now clear.
\end{proof}

As a consequence of Lemma \ref{lem:symbol}, we give a short proof of Lemma \ref{lem:Tri}
which has been postponed.

\begin{proof}[Proof of Lemma \ref{lem:Tri}.]
Since the symbol map is $O(N)$-equivariant, and since both $d_{\R^n}$ 
and $d^*_{\R^n}$ commute with $O(N)$-actions, we conclude that all the terms in the right-hand sides in Lemma \ref{lem:symbol} are $O(N)$-equivariant maps.

Therefore the bilinear maps 
\index{A}{B(k)@$B^{(k)}$, bilinear map}
$B^{(k)}$ $(k = 0, 1, 2)$ 
are $O(N)$-equivariant because $\Exterior^j(\C^N)$ is
self-dual as an $O(N)$-module.
\end{proof}

In Proposition \ref{prop:153417} we have determined 
the triple $(i,j,k)$ of nonnegative integers for which the space
$\mathrm{Hom}_{O(n-1)}\left(\Exterior^i(\C^n),\;
\Exterior^j(\C^{n-1}) \otimes \mathcal{H}^k\left(\C^{n-1}\right) \right)$
is nonzero, and found an explicit basis $h^{(k)}_{i \to j}$
in \eqref{eqn:Hi--} -- \eqref{eqn:Hi+}.
The next proposition describes
differential operators $T^{(k)}_{i \to j} \colon 
\mathcal{E}^i(\R^n) \To \mathcal{E}^j(\R^n)$ 
such that 
\index{A}{H0hijk@$h^{(k)}_{i\to j}$}
$\mathrm{Symb}\left(T^{(k)}_{i \to j}\right)
=h^{(k)}_{i \to j}$ in all the cases.

\begin{prop}\label{prop:160483}
We have
\begin{enumerate}
%%%%%%%%%%%%%%
\item[]\emph{Case} $j= i-2.$
\begin{enumerate}
\item[]\quad \emph{(1)}
$\displaystyle{
h^{(1)}_{i\to i-2}= \mathrm{Symb}
\left(-d^*_{\R^n} \circ \iotan\right)}$.
\end{enumerate}
\vskip 0.1in

%%%%%%%%%%%%%%
\item[]\emph{Case} $j=i-1.$
\index{A}{1iotan@$\iotan$, interior multiplication}
\begin{enumerate}
\item[]\quad \emph{(2)}
$\displaystyle{
h^{(0)}_{i \to i-1}=\mathrm{Symb}\left(\iotan \right)}.$
\item[]\quad \emph{(3)}
$\displaystyle{
h^{(1)}_{i \to i-1}=\mathrm{Symb}\left(-\prn \circ d^*_{\R^n} 
- \frac{\partial}{\partial x_n} \iotan\right)}.$
\item[]\quad \emph{(4)}
$\displaystyle{
h^{(2)}_{i \to i-1}=\mathrm{Symb}\left(\left(-d_{\R^n}d^*_{\R^n}
-\frac{i-1}{n-1}\Delta_{\R^{n-1}}\right)\circ \iotan\right)}.$
\end{enumerate}
\vskip 0.1in

%%%%%%%%%%%%%%
\item[]\emph{Case} $j=i.$
\index{A}{11prn@$\prn$, projection onto $\mathrm{Ker}(\iotan)$}
\begin{enumerate}
\item[]\quad \emph{(5)}
$\displaystyle{
h^{(0)}_{i \to i}=\mathrm{Symb}\left(\prn\right)}.$
\item[]\quad \emph{(6)}
$\displaystyle{
h^{(1)}_{i \to i}=\mathrm{Symb}\left(d_{\R^n}\circ \iotan\right)}.$
\item[]\quad \emph{(7)}
$\displaystyle{
h^{(2)}_{i \to i}=\mathrm{Symb}\left(\prn\circ
\left(-d_{\R^n}d^*_{\R^n}-d_{\R^n}\frac{\partial}{\partial x_n}\iotan
-\frac{i}{n-1}\Delta_{\R^{n-1}}\right)\right)}.$
\end{enumerate}
\vskip 0.1in

%%%%%%%%%%%%%%
 \index{A}{H0ii+1@$h^{(1)}_{i\to i+1}$}
\item[]\emph{Case} $j=i+1.$
\begin{enumerate}
\item[]\quad \emph{(8)}
$\displaystyle{
h^{(1)}_{i \to i+1}=\mathrm{Symb}\left(\prn \circ d_{\R^n}\right)}$.
\end{enumerate}

\end{enumerate}
\end{prop}

\begin{proof}
We shall prove the formula for $h^{(k)}_{i \to j}$ according as $k = 0,1,2$.

Case $k=0$, namely,
(2) and (5). We compare \eqref{eqn:intn} with
the formula for $h^{(0)}_{i \to i-1}$ 
in Table \ref{table:Table160487}, and get the second 
identity. Likewise, comparing \eqref{eqn:Pindx} with the formula
for $h^{(0)}_{i\to i}$ in Table \ref{table:Table160487}, we get the fifth identity.

\vskip 0.1in

Case $k=1$, namely, $(1)$, $(3)$, $(6)$, and $(8)$.

(1) By \eqref{eqn:intn} and \eqref{eqn:dstar}, we have
\begin{equation*}
d^*_{\R^n} \circ \iotan\left(f dx_I\right)
=(-1)^i \sum_{\ell \in I \setminus \{n\}} \sgn(I \setminus \{n\} ; \ell)
\frac{\partial f}{\partial x_\ell} dx_{I \setminus \{\ell, n\}}
\quad \text{for $n \in I$}.
\end{equation*}
Since $(-1)^{i-1} \sgn(I\setminus\{n\};\ell)=-\sgn(I; \ell ,n)$,
we get $\mathrm{Symb}\left(d_{\R^n} \circ \iotan\right) 
= - h^{(1)}_{i \to i-2}$ by the formula of $h^{(1)}_{i \to i-2}$ in Table \ref{table:Table160487}.

(3) We apply $\Pi_{n-1}$ to \eqref{eqn:dstar}, and get
\begin{equation}\label{eqn:Pin}
\Pi_{n-1}d^*_{\R^n}(fdx_I)=\left\{
\begin{matrix}
-\sum\limits_{\ell\in I}\sgn (I;\ell) \frac{\partial f}{\partial x_\ell}dx_{I\setminus\{\ell\}} & (n\not\in I),
\\ & \\
-\sgn (I;n)\frac{\partial f}{\partial x_n}dx_{I\setminus\{n\}} & (n\in I).
\end{matrix}
\right.
\end{equation}
In turn, by using
\eqref{eqn:dstar}, 
\eqref{eqn:intn} and \eqref{eqn:Pin}, we have
\begin{equation*}
\left(
-\prn\circ d^*_{\R^n}-\frac{\partial}{\partial x_n}\iota_{\frac{\partial}{\partial x_n}}\right)
(fdx_I)=
\left\{
\begin{matrix*}
\sum_{\ell\in I}\sgn(I;\ell)\frac{\partial f}{\partial x_\ell}dx_{I\setminus\{\ell\}} & (n\not\in I),\\
0&(n\in I).
\end{matrix*}
\right.
\end{equation*}
Comparing this with the formula for $h^{(1)}_{i \to i-1}$ in 
Table \ref{table:Table160487} again, we get
the third identity.
The proofs for $(6)$ and $(8)$ are similar,
and we omit them.

Case $k=2$, namely, $(4)$ and $(7)$.
Let us prove $(4)$. 
It follows from Lemma \ref{lem:symbol} (3) and 
Proposition \ref{prop:160483} (2) that 
\begin{align*}
\mathrm{Symb}\left(
-d_{\R^n}d^*_{\R^n} \circ \iotan -\frac{i-1}{n-1}\Delta_{\R^{n-1}}\iotan\right)
=H^{(2)}_{i-1\to i-1}\circ h^{(0)}_{i \to i-1} -
\frac{i-1}{n-1}Q_{n-1}\left(\zeta'\right) h^{(0)}_{i \to i-1}.
\end{align*}
By the definitions \eqref{eqn:H2tilde} and \eqref{eqn:Hi-},
this amounts to
\index{A}{prij@$\mathrm{pr}_{i\to j}\colon \Exterior^i(\C^n) \To \Exterior^j(\C^{n-1})$}
\begin{align*}
\widetilde{H}^{(2)}_{i-1 \to i-1} \circ \mathrm{pr}_{i \to i-1}
=h^{(2)}_{i\to i-1}.
\end{align*}
The case (7) is similar.
\end{proof}

\newpage
%%%%%%%%%%%%%%%%%%%%%%%%%%%%%%%%%%%%%%%%%%%%%%%
\section{Identities of scalar-valued differential operators $\mathcal{D}^\mu_\ell$}\label{sec:formulaD}

In this chapter, we derive identities for the (scalar-valued) differential operators
$\mathcal{D}^\mu_\ell$ (see \eqref{eqn:Dl} for the definition) systematically
from those for the Gegenbauer polynomials given in Appendix.
We note that some of the formul\ae{} here were previously known 
up to the restriction map $\restn$,
see \cite{Juhl, K14, KOSS15, KS13}.

Using these identities together with the results of Chapter \ref{sec:6},
we study matrix-valued symmetry breaking operators $\mathcal{D}^{i \to j}_{u,a}$
in details.
In particular, the condition for the vanishing of  the operators 
$\mathcal{D}^{i\to i-1}_{u,a}$ and
$\mathcal{D}^{i\to i}_{u,a}$ (Proposition \ref{prop:Dnonzero}) is proved 
in 
Section \ref{subsec:prop12}, and the identity $\eqref{eqn:Dii1}=\eqref{eqn:Di-B}$
about the two expressions of $\mathcal{D}^{i \to i-1}_{u,a}$ is proved 
in Section \ref{subsec:basischange}. Various functional identities among
$\mathcal{D}^{i\to j}_{u,a}$ are proved in Chapter \ref{sec:fi}.

%%%%%%%%%%%%%%%%%%%%%%%%%%%%%%%%%%%%%%%%%%%%%%
\subsection{Homogeneous polynomial inflation $I_a$}
\label{subsec:5symbF}

Suppose $a \in \N$. 
\index{A}{Polell[t]@$\mathrm{Pol}_\ell[t]_{\mathrm{even}}$}
For $g(t)\in \mathrm{Pol}_a[t]_{\mathrm{even}}$
(see \eqref{eqn:gs}), we
define a polynomial of two variables $x$ and $y$ ($a$-inflated polynomial of $g$) by 
\index{A}{Iell@$I_\ell$, $\ell$-inflated polynomial|textbf}
\begin{equation}\label{eqn:Iag}
I_a g(x,y)=x^{\frac a2}g\left(\frac y{\sqrt x}\right).
\end{equation}

Notice that $(I_a g)(x^2,y)$ is a homogeneous polynomial of $x$ and $y$ of degree $a$.

By definition, we have
\begin{eqnarray}\label{eqn:yIg}
I_{a+1}(tg(t))(x,y)&=&y(I_a g)(x,y),\\
(I_{a+2}g)(x,y)&=& x(I_a g)(x,y).\label{eqn:xIg}
\end{eqnarray}

We recall 
\index{A}{Qn11@$Q_{n-1}(\zeta')$}
$Q_{n-1}(\zeta') = \zeta_1^2 + \cdots + \zeta_{n-1}^2$ for 
$\zeta' = (\zeta_1, \ldots, \zeta_{n-1})$, and from \eqref{eqn:Ta} that
$(T_ag)(\zeta) = Q_{n-1}(\zeta')^{\frac{a}{2}}g\left(\frac{\zeta_n}{\sqrt{Q_{n-1}(\zeta')}}\right)$
is a homogeneous polynomial of $n$-variables $\zeta=(\zeta_1, \ldots, \zeta_{n-1}, \zeta_n)$ 
of degree $a$. By definition, we have the following identity:
\index{A}{Ta@$T_a$}
\begin{equation}\label{eqn:TaIa}
(T_ag)(\zeta)=I_ag(Q_{n-1}(\zeta'),\zeta_n).
\end{equation}
If we substitute the differential operators $\Delta_{\R^{n-1}}$ and $\frac{\partial}{\partial x_n}$
into $I_ag(x,y)$, we get a homogeneous differential operator 
$I_ag\left(\Delta_{\R^{n-1}},\frac{\partial}{\partial x_n}\right)$
of order $a$. It then follows from \eqref{eqn:TaIa} that
its symbol (see \eqref{eqn:symb}) is given by
 \begin{equation}\label{eqn:Symbg}
 \index{A}{Symb}
 \mathrm{Symb}\left( I_ag(\Delta_{\R^{n-1}},\frac{\partial}{\partial x_n})\right)=T_ag.
 \end{equation}

We recall from \eqref{eqn:Dl} that
$\mathcal D_a^\mu=(I_a\widetilde{C}_a^\mu)
\left( -\Delta_{\R^{n-1}},{\frac{\partial}{\partial x_n}}\right)$
is a homogeneous differential operator on $\R^n$ of order $a$,
where 
\index{A}{C1ell1@$\widetilde C_\ell^\mu(t)$,
renormalized Gegenbauer polynomial}
$\widetilde{C}_a^\mu(t)$ is the renormalized Gegenbauer polynomial
(see \eqref{eqn:Gegen2}). Then its symbol is given as follows:

\begin{lem}\label{lem:imaginary}
 \index{A}{Dell@$\mathcal D_\ell^\mu$}
 $\mathrm{Symb}(\mathcal D_a^\mu)
 =e^{-\frac{\pi \sqrt{-1}a}2} T_a\left(\widetilde C_a^\mu\left(
 e^{\frac{\pi\sqrt{-1}}{2}}\;\cdot\;\right)\right)$.
\end{lem}

\begin{proof}
Suppose 
$g(t)\in\mathrm{Pol}_a[t]_{\mathrm{even}}$ is of the form $g(t)=e^{-\frac{\pi \sqrt{-1}a}2}
\varphi(e^{\frac{\pi \sqrt{-1}}2}t)$
with $\varphi(s) \in \mathrm{Pol}_a[s]_{\mathrm{even}}$.
By definition we have
\begin{equation*}
I_ag(x,y)=I_a\varphi(-x,y),
\end{equation*}
and thus $I_ag(\Delta_{\R^{n-1}},\frac{\partial}{\partial x_n})=I_a\varphi(-\Delta_{\R^{n-1}},\frac{\partial}{\partial x_n})$.
In turn,
$\mathrm{Symb}\left( I_a\varphi(-\Delta_{\R^{n-1}},\frac{\partial}{\partial x_n})\right)
=T_ag$ by \eqref{eqn:Symbg}.
Hence Lemma follows.
\end{proof}

%%%%%%%%%%%%%%%%%%%%%%%%%%%%%%%%%%%%%%%%%%%%%%%
\subsection{Identities among Juhl's conformally covariant differential operators}\label{subsec:fiD}
\index{B}{Juhl's operator}

The composition $\restn \circ \mathcal{D}^\mu_a\colon C^\infty(\R^n) \To C^\infty(\R^{n-1})$
is a conformally covariant differential operator, which we refer to as
Juhl's operator.
In this section we collect identities
for the scalar-valued differential operators $\mathcal{D}^\mu_a$
that hold before taking the restriction operator $\restn$:

\begin{itemize}
\item three-term relations for general parameter $\mu$ (Proposition \ref{prop:1522102})
\vskip 0.1in
\item factorization identities for integral parameter $\mu$ 
(Proposition \ref{prop:1524114}, Lemma \ref{lem:152475}).
\end{itemize}

\begin{prop}\label{prop:1522102}
 Let $a\in\N$, $\mu\in\C$, and $\gamma(\mu,a)$ be defined as in \eqref{eqn:gamma}.
 Then we have
\begin{align}
&\mathcal D^{\mu+1}_{a-2}\Delta_{\R^{n-1}}+\gamma(\mu,a)\mathcal D^{\mu+1}_{a-1}\frac{\partial}{\partial x_n}
 =\frac a2 \mathcal D^{\mu}_a.\label{eqn:1522102}\\
 & \mathcal D^{\mu+1}_{a-2}\frac{\partial}{\partial x_n}-\gamma(\mu,a) \mathcal D^{\mu+1}_{a-1}+\gamma(\mu-\frac12,a) \mathcal D^{\mu}_{a-1}=0.\label{eqn:152617}\\
 &\mathcal D_{a-2}^{\mu+1}\Delta_{\R^n}+\left(\mu-\frac12\right)\mathcal D_a^\mu=
\left(\mu+\left[\frac a2\right]-\frac12\right)\mathcal D_a^{\mu-1}.\label{eqn:152592}\\
&\gamma(\mu,a)\mathcal D_{a-1}^{\mu+1}\Delta_{\R^n}+\left(\mu-\frac12\right)
\mathcal D_a^\mu\frac{\partial}{\partial x_n}=\frac12{(a+1)\gamma(\mu-\frac12,a)}
\mathcal D_{a+1}^{\mu-1}.\label{eqn:152580}\\
&(\mu+a)\mathcal D_a^\mu-\mathcal D_{a-2}^{\mu+1}\Delta_{\R^{n-1}}=
\left(\mu+\left[\frac{a+1}2\right]\right)\mathcal D_a^{\mu+1}.\label{eqn:152562}
\end{align}
\end{prop}

\begin{proof}
By using \eqref{eqn:yIg} and \eqref{eqn:xIg}, we see that these three-term 
relations for $\mathcal{D}^\mu_a
=\left(I_a\widetilde{C}^\mu_a\right)\left(-\Delta_{\R^{n-1}},\frac{\partial}{\partial x_n}\right)$ 
are derived from those for Gegenbauer 
polynomials $\widetilde{C}^\mu_a(z)$
that will be proved in Chapter \ref{sec:appendix} (Appendix).
The correspondence is given in the following table:
\begin{table}[h]
\begin{center}
\begin{tabular}{c|c|c|c|c|c}
Identities for $\mathcal{D}^\mu_a$ & \eqref{eqn:1522102} & 
\eqref{eqn:152617} & \eqref{eqn:152592} & \eqref{eqn:152580} & \eqref{eqn:152562}\\
\hline
\rule{0pt}{3ex}Identities for $\widetilde{C}^\mu_a$ & \eqref{eqn:1605103} & \eqref{eqn:152633}
&\eqref{eqn:152594} & \eqref{eqn:152579} & \eqref{eqn:152563} 
\end{tabular}
\end{center}
\qedhere
\end{table}

\end{proof}

The (scalar-valued) differential operator $\mathcal D_\ell^\mu$ 
for specific parameter $\mu$ and $\ell$ may
be written as the product of another operator 
$\mathcal D_{\ell'}^{\mu'}$ and
the
Laplacian $\Delta_{\R^n}$ (or $\Delta_{\R^{n-1}}$). For example,
\begin{equation*}
\mathcal D_3^\mu=\Delta_{\R^{n-1}}\mathcal D_1^\mu\;\mathrm{if}\;\mu=-2;\quad
\mathcal D_3^\mu=\Delta_{\R^{n}}\mathcal D_1^{1-\mu}\;\mathrm{if}\;\mu=-\frac12.
\end{equation*}
We collect such factorization identities as follows.

For $a\in\N$ and $\ell\in\N_+$, we recall from \eqref{eqn:Kla} that
\index{A}{Kell@$K_{\ell, a}$}
$\Kla := \prod_{k=1}^\ell\left(\left[\frac a2\right]+k\right)$
is a positive integer. For $\ell = 0$, we set $K_{\ell, a} = 1$.

\begin{prop}\label{prop:1524114}
Let $a,\ell\in\N$. Then
\begin{align}
\Kla\mathcal D_{a+2\ell}^{\frac12-\ell}&=\mathcal D_a^{\frac12+\ell}\Delta_{\R^n}^\ell.
\tag{1}\\
\Kla\mathcal D_{a+2\ell}^{-a-\ell}&=\mathcal D_a^{-a-\ell}\Delta_{\R^{n-1}}^\ell.
\tag{2}
\end{align}
\end{prop}
\begin{proof}
According to definition \eqref{eqn:Iag} for every $\ell\in\N$ we have
\begin{equation}
I_{a+2\ell}\left((z^2-1)^\ell g\right)(x,y)=(y^2-x)^\ell(I_ag)(x,y).
\end{equation}
Thus, applying $I_{a+2\ell}$ to the identity \eqref{eqn:1524113} in Proposition \ref{prop:1524113}
we get (1).

Similarly, applying $I_{a+2\ell}$ to \eqref{eqn:152471} and using \eqref{eqn:xIg} we get (2) and conclude the proof.
\end{proof}
Analogous formul\ae{} are derived from Proposition \ref{prop:1524114}, and will be used in the proof of Theorems \ref{thm:factor1} and \ref{thm:factor2}.

\begin{lem}\label{lem:152475}
Let $a\in \N$ and $\ell \in \N_+$.
\begin{eqnarray*}
&(1)& \Kla\mathcal D_{a+2\ell-2}^{-\ell+\frac32}=\left(\ell+\left[\frac a2\right]\right)
\mathcal D_a^{\ell-\frac12}\Delta_{\R^{n}}^{\ell-1}.\\
&(2)& 4 \Kla\gamma\left(\ell+\frac12,a-1\right)\gamma\left(-\ell+\frac12,a+2\ell\right)
\mathcal D_{a+2\ell-1}^{\frac32-\ell}=(a+1)(a+2\ell)\mathcal D_{a+1}^{\ell-\frac12}\Delta_{\R^{n}}^{\ell}.\\
&(3)& \Kla\mathcal D_{a+2\ell-2}^{-a-\ell+1}=\left(\ell+\left[\frac a2\right]\right)
\mathcal D_a^{-a-\ell+1}\Delta_{\R^{n-1}}^{\ell-1}.\\
&(4)& \gamma(-a,a)\Kla\mathcal D_{a+2\ell-1}^{-a-\ell+1}
=\gamma(-a-\ell,a)\mathcal D_{a-1}^{-a-\ell+1}\Delta_{\R^{n-1}}^{\ell}.
\end{eqnarray*}
\end{lem}
\begin{proof}
\begin{enumerate}
\item Apply Proposition \ref{prop:1524114} (1) with $\ell$ replaced by $\ell-1$.
\item We again apply Proposition \ref{prop:1524114} (2) with $a$ replaced by $a+1$
and $\ell$ replaced by $\ell-1$ this time. Then the assertion follows from the identity below
\begin{equation}\label{eqn:152589}
\frac{K_{\ell,a}}{K_{\ell-1,a+1}}
=\frac{(a+1)(a+2\ell)}{4\gamma\left(\ell+\frac12,a-1\right)\gamma\left(-\ell+\frac12,a+2\ell\right)}.
\end{equation}
The proof of \eqref{eqn:152589} is elementary, and we omit it.
\end{enumerate}

Identities (3) and (4) follow from Proposition \ref{prop:1524114} (2) by similar argument as
we used for cases (1) and (2) above. We also use an elementary formula 
\begin{equation*}
\frac{K_{\ell, a-1}}{K_{\ell, a}} = \frac{\gamma(-a,a)}{\gamma(-a-\ell,a)}.
\end{equation*}

\end{proof}

In the rest of this chapter,
we apply the three-term relations
given in Proposition \ref{prop:1522102}.

%%%%%%%%%%%%%%%%%%%%%%%%%%%%%%%%%%%%%%%%%%%%%%%
\subsection{Proof of Proposition \ref{prop:Dnonzero}}\label{subsec:prop12}

Given a linear operator $T:\mathcal E^i (\R^n)\to\mathcal E^j(\R^{n-1})$, we define the 
``matrix component" $T_{IJ}$ for $I\in \mathcal I_{n,i}$ and $J\in\mathcal I_{n-1,j}$ by the identity:
$$
T(fdx_I)=\sum_{J\in\mathcal I_{n-1,j}}(T_{IJ}f)dx_J.
$$
If $T$ is a differential operator, so is $T_{IJ}:C^\infty(\R^n)\to C^\infty(\R^{n-1})$.

We find the $(I,J)$-component of the  symmetry breaking operator
$\mathcal D^{i\to i-1}_{u,a}\colon \mathcal E^i(\R^n)\To\mathcal E^{i-1}(\R^{n-1})$ introduced in \eqref{eqn:Dii1}
as follows:
\begin{lem}\label{lem:1607100}
For $I\in\mathcal I_{n,i}$ and $J\in\mathcal I_{n-1,i-1}$, we consider
\begin{enumerate}
\item[Case 1.]  $n\in I, J=I\setminus\{n\}$,
\item[Case 2.] $n\in I, \vert J\setminus I\vert=1$, say $I=K\cup\{p,n\}, J=K\cup\{q\}$,
\item[Case 3.] $n\not\in I, J\subset I$, say $I=J\cup\{p\}$.
\end{enumerate}
Let $\mu:= u+ i-\frac12(n-1)$. Then the matrix component $\left( D^{i\to i-1}_{u,a}\right)_{IJ}$ is given as
\begin{enumerate}
\item[Case 1.]  $-\mathcal D_{a-2}^{\mu+1}\sum_{p\in I^c}\frac{\partial^2}{\partial x_p^2}+\frac12
(a+u+2i-n)\mathcal D_a^\mu$,
\item[Case 2.]  $(-1)^{i-1}\sgn(I;p,q)\mathcal D_{a-2}^{\mu+1}$,
\item[Case 3.]  $\sgn(I;p)\gamma(\mu,a)\mathcal D_{a-1}^{\mu+1}$,
\end{enumerate}
followed by the restriction map $\restn$. Here $I^c=\{1,2,\cdots,n\}\setminus I$ in Case 1.
Otherwise, the $(I,J)$ component $\left( D^{i\to i-1}_{u,a}\right)_{IJ}$ vanishes.
\end{lem}

\begin{proof}
We recall from \eqref{eqn:Dii1} that
$$
\mathcal D_{u,a}^{i\to i-1}=\mathrm{Rest}_{x_n=0}
\circ\left(-\mathcal D_{a-2}^{\mu+1}d_{\mathbb{R}^n}d^*_{\mathbb{R}^n}
\iota_{\frac{\partial}{\partial x_n}}
-\gamma(\mu,a) \mathcal D_{a-1}^{\mu+1}d^*_{\mathbb{R}^n}+\frac12(u+2i-n)
\mathcal D_{a}^{\mu}\iota_{\frac{\partial}{\partial x_n}}
\right)
$$
We begin by computing the $(I,J)$-components of the basis elements $\restn\circ
d_{\mathbb{R}^n}d^*_{\mathbb{R}^n}\iotan, \restn\circ d^*_{\mathbb{R}^n}$,
and $\restn\iotan$.

It follows from \eqref{eqn:intn}, \eqref{eqn:dstar}, and \eqref{eqn:ddstar} that 
$(I,J)$-components
of these operators are given as the entries in the table below, followed by the restriction map $\restn$:

\begin{table}[H]
\begin{center}
\begin{tabular}{c|ccc}
 & $(\restn\circ d_{\R^n}d^*_{\R^n}\iotan)_{IJ}$ & $(\restn d_{\mathbb{R}^n}^*)_{IJ}$& $(\restn \circ\iotan)_{IJ}$\\
&&&\\
\hline
&&\\
Case 1 &$(-1)^i\sum\limits_{p\in I\setminus\{n\}}\frac{\partial^2}{\partial x_p^2}$ &$(-1)^i\frac{\partial}{\partial x_n}$&$(-1)^{i-1}$\\
Case 2 &$(-1)^i\sgn(I;p,q)\frac{\partial^2}{\partial x_p\partial x_q}$ &$0$&$0$\\
Case 3 &$0$&$-\sgn(I;p)\frac{\partial}{\partial x_p}$&$0$.
\end{tabular}
\end{center}
\end{table}
Then Cases 2 and 3 of the lemma follow from \eqref{eqn:Dii1}. In Case 1, 
the $(I,J)$-component 
of $\mathcal D_{u,a}^{i\to i-1}$ is given by
$$
(-1)^{i-1}\restn\circ\left(\mathcal D_{a-2}^{\mu+1}\sum\limits_{p\in I\setminus\{n\}}\frac{\partial^2}{\partial x_p^2}+\gamma(\mu,a)\mathcal D_{a-1}^{\mu+1}\frac{\partial}{\partial x_n}+
\frac12(u+2i-n)\mathcal D_a^\mu\right),
$$
which amounts to
$$
(-1)^{i-1}\restn\circ\left(\frac12(a+u+2i-n)\mathcal D_a^\mu+\mathcal D_{a-2}^{\mu+1}
\left(\sum\limits_{p\in I\setminus\{n\}}\frac{\partial^2}{\partial x_p^2}-\Delta_{\R^{n-1}}\right)\right)
$$
by the three-term relation \eqref{eqn:1522102} for $\mathcal D_a^\mu$. Thus the lemma is proved.
\end{proof}

Lemma \ref{lem:1607100} will be used for the proof of
Proposition \ref{prop:Dnonzero} (1). We may deduce Proposition \ref{prop:Dnonzero} (2)
from Proposition \ref{prop:Dnonzero} (1) by the duality \eqref{eqn:dualCi}, however,
we give explicit formul{\ae} for the matrix components of 
$\mathcal D_{u,a}^{i\to i}$ for later purpose.

\begin{lem}\label{lem:153284}
For $I\in\mathcal I_{n,i}$ and $J\in\mathcal I_{n-1,i}$, we consider
\begin{enumerate}
\item[Case 1.] $n\not\in I, J=I$.
\item[Case 2.] $n \not\in I, \vert J\setminus I\vert=1$, say $I=K\cup\{p\}, J=K\cup\{q\}$.
\item[Case 3.] $n\in I, \vert J\setminus I\vert=1$, say $I=K\cup\{n\}, J= K\cup\{q\}$.
\end{enumerate}
Let $\mu:= u+i-\frac{n-1}2$. Then the matrix component
\index{A}{Dii2@$\mathcal D_{u,a}^{i\to i}$}
$\left(\mathcal D_{u,a}^{i\to i}\right)_{IJ}$ is given as
\begin{enumerate}
\item[Case 1.] $-\mathcal D_{a-2}^{\mu+1}\sum_{p\in I}\frac{\partial^2}{\partial x_p^2}
+\frac12(u+a)\mathcal D_a^\mu$,
\item[Case 2.] $-\sgn(I;p,q)\mathcal D_{a-2}^{\mu+1}\frac{\partial^2}{\partial x_p\partial x_q}$,
\item[Case 3.] $-\sgn(I;q,n)\gamma(\mu,a)\mathcal D_{a-1}^{\mu+1}\frac\partial{\partial x_q}$,
\end{enumerate}
followed by the restriction map $\restn$. Otherwise, the $(I,J)$-component
$(\mathcal{D}_{u, a}^{i \to i})_{IJ}$ is equal to zero.
\end{lem}

\begin{proof}
From the expressions \eqref{eqn:ddstar}, \eqref{eqn:intn} and \eqref{eqn:d}, we have:
\begin{table}[H]
\begin{center}
\begin{tabular}{c|cc}
 & $(\restn\circ d_{\R^n}d^*_{\R^n})_{IJ}$ & $(\restn \circ d_{\mathbb{R}^n}\iotan)_{IJ}$ \\
&&\\
\hline
&&\\
Case 1 &$-\restn\circ\sum_{p\in I}\frac{\partial^2}{\partial x_p^2}$ &$0$\\
Case 2 &$-\sgn(I;p,q)\restn\circ\frac{\partial^2}{\partial x_p\partial x_q}$ &$0$\\
Case 3 &$-\sgn(I;q,n)\restn\circ\frac{\partial^2}{\partial x_q\partial x_n}$ &
$(-1)^{i-1}\sgn(I;q)\restn\circ\frac{\partial}{\partial x_q}$.
\end{tabular}
\end{center}
\end{table}

Then Cases 1 and 2 of the lemma follow from \eqref{eqn:Dii}. In Case 3, we also use the identity
$\sgn(I;q,n)=(-1)^{i-1}\sgn(I;q)$ and the three-term relation \eqref{eqn:152617}.
\end{proof}
We are ready to complete the proof of Proposition \ref{prop:Dnonzero}.

\begin{proof}[Proof of Proposition \ref{prop:Dnonzero}]
(1). Suppose $i=n$. Then, only Case 1 in Lemma \ref{lem:1607100} occurs. In this case $I^c=\emptyset$.
Thus $\mathcal D_{u,a}^{n\to n-1}=0$ if and only if $a+u+2i-n=0$, equivalently, $u=-n-a$.

Suppose $1\leq i\leq n-1$. Then Cases 1 and 3 in  Lemma \ref{lem:1607100} occur, and Case 2 occurs if
$2\leq i\leq n-1$.

First, we see from Lemma \ref{lem:1607100} that 
$\left(\mathcal D_{u,a}^{i\to i-1}\right)_{IJ}=0$ in Case 1 if and only if $\mathcal D_{a-2}^{\mu+1}=0$
and $a+u+2i-n=0$, equivalently, $n-u-2i=a\in\{0,1\}$.

Second, $\left(\mathcal D_{u,a}^{i\to i-1}\right)_{IJ}=0$ in Case 3 if and only if $\gamma(\mu,a)\mathcal D_{a-1}^{\mu+1}=0$. This happens if and only if $a=0$ because $a\in\{0,1\}$. Hence $\mathcal D_{u,a}^{i\to i-1}=0$ implies that $(u,a)=(n-2i,0)$. The converse statement also holds because $\left(\mathcal D_{u,a}^{i\to i-1}\right)_{IJ}$ vanishes in Case 2 if $a=0$. Thus Proposition \ref{prop:Dnonzero} (1) is proved.

(2). The proof of Proposition \ref{prop:Dnonzero} (2) is similar to the one of (1) by using Lemma
\ref{lem:153284}, and we omit it.
\end{proof}

%%%%%%%%%%%%%%%%%%%%%%%%%%%%%%%%%%%%%%%%%%%%%%%
\subsection{Two expressions of $\mathcal{D}^{i\to i-1}_{u,a}$}\label{subsec:basischange}

In this section,
we prove in Proposition \ref{prop:152751} 
the identity $\eqref{eqn:Dii1} = \eqref{eqn:Di-B}$ 
for the two expressions of the differential operator
$\mathcal{D}^{i \to i-1}_{u,a}\colon \mathcal{E}^i(\R^n) \To 
\mathcal{E}^{i-1}(\R^{n-1})$ by using the three-term relations that we established
in Section \ref{subsec:fiD}. 

In order to prove the identity $\eqref{eqn:Dii1} = \eqref{eqn:Di-B}$,
we begin with the relationship between the following two triples of 
matrix-valued differential operators
\begin{equation*}
\{d_{\R^n}d^*_{\R^n}\iotan, d^*_{\R^n},\iotan\} \quad \text{and} \quad
\{-d^*_{\R^n}\iotan d_{\R^n}, \frac{\partial}{\partial x_n}\iotan+d^*_{\R^n},\iotan\}
\end{equation*}
that map $\mathcal{E}^i(\R^n)$ to $\mathcal{E}^{i-1}(\R^n)$.

\begin{lem}\label{lem:152750}
Suppose $A,B,C,P,Q$ and $R$ are scalar-valued differential operators on $\R^n$ satisfying
\begin{equation}\label{eqn:basePQR}
P=-A,\quad Q=B-A\frac{\partial}{\partial x_n},\quad 
R=-A\Delta_{\R^{n-1}}-B\frac{\partial}{\partial x_n}+C.
\end{equation}
Then
\begin{equation}\label{eqn:basethm2}
Ad_{\R^n}d^*_{\R^n}\iotan+Bd^*_{\R^n}+C\iotan=P(-d^*_{\R^n}\iotan  d_{\R^n})
+Q(\frac{\partial}{\partial x_n}\iotan+d^*_{\R^n})+R\iotan.
\end{equation}
\end{lem}

\begin{proof}
It follows from Lemma \ref{lem:152457} (1) and \eqref{eqn:Lapform}
that
$$
-d^*_{\R^n}\iotan d_{\R^n}=d^*_{\R^n}d_{\R^n}\iotan-\frac{\partial}{\partial x_n}d^*_{\R^n}
=-d_{\R^n}d^*_{\R^n}\iotan-\frac{\partial}{\partial x_n}d^*_{\R^n}-\Delta_{\R^n}\iotan.
$$
Hence the right-hand side of \eqref{eqn:basethm2} is equal to 
$$
-Pd_{\R^n}d^*_{\R^n}\iotan+(-P\frac{\partial}{\partial x_n}+Q)d^*_{\R^n}
+(-P\Delta_{\R^n}+Q\frac{\partial}{\partial x_n}+R)\iotan.
$$
Thus the equality \eqref{eqn:basethm2} holds if 
\begin{equation*}
A=-P,\quad B=-P\frac{\partial}{\partial x_n}+Q,\quad C=-P\Delta_{\R^n}+Q\frac{\partial}{\partial x_n}+R,
\end{equation*}
or equivalently if \eqref{eqn:basePQR} is satisfied.
\end{proof}

\begin{lem}\label{lem:152751}
Suppose $\mu\in\C$ and $a\in\N$. Then we have the following identity as linear operators 
from $\mathcal E^i(\R^n)$ to $\mathcal E^{i-1}(\R^n)$:
\begin{eqnarray*}
&&-\mathcal D_{a-2}^{\mu+1}d_{\R^n}d^*_{\R^n}
\iotan-\gamma(\mu,a)\mathcal D_{a-1}^{\mu+1}d^*_{\R^n}
+\frac12(\mu+i-\frac{n+1}2)\mathcal D_a^\mu\iotan\\
&=&-\mathcal D_{a-2}^{\mu+1}d^*_{\R^n}\iotan d_{\R^n}
-\gamma(\mu-\frac12,a)\mathcal D_{a-1}^\mu
(\frac{\partial}{\partial x_n}\iotan+d^*_{\R^n})
+\frac12(\mu+i-\frac{n+1}2+a)\mathcal D_a^\mu\iotan.
\end{eqnarray*}
\end{lem}

\begin{proof}
By Lemma \ref{lem:152750} with
$$
A=-\mathcal D_{a-2}^{\mu+1},\quad
B=-\gamma(\mu,a)\mathcal D_{a-1}^{\mu+1},\quad
C=\frac12(\mu+i-\frac{n+1}2)\mathcal D_a^\mu,
$$
the proof of Lemma \ref{lem:152751}
reduces to the following identities
\begin{eqnarray*}
&&\mathcal D_{a-2}^{\mu+1}\frac{\partial}{\partial x_n}-\gamma(\mu,a)
\mathcal D_{a-1}^{\mu+1}
= -\gamma(\mu-\frac12,a)\mathcal D_{a-1}^{\mu},\\
&&\mathcal D_{a-2}^{\mu+1}\Delta_{\R^{n-1}}+\gamma(\mu,a)\mathcal D_{a-1}^{\mu+1}
\frac{\partial}{\partial x_n}+\frac12(\mu+i-\frac{n+1}2)\mathcal D_a^\mu=
\frac12(\mu+i-\frac{n+1}2+a)\mathcal D_a^\mu.
\end{eqnarray*}
These are nothing but the three-term relations among the operators 
$\mathcal D_\ell^\lambda$ that we proved in \eqref{eqn:152617} and \eqref{eqn:1522102}, respectively.
\end{proof}

We are ready to prove the second expression \eqref{eqn:Di-B} of $\mathcal D_{u,a}^{i\to i-1}$.

\begin{prop}\label{prop:152751}
As operators $\mathcal{E}^i(\R^n) \To \mathcal{E}^{i-1}(\R^{n-1})$, we have 
$\eqref{eqn:Dii1} = \eqref{eqn:Di-B}$.
\end{prop}

\begin{proof}
It follows from Lemma \ref{lem:152456} (2) that
$$
\restn\circ\mathcal D_{a-1}^\mu\left(\frac{\partial}{\partial x_n}\iotan+d_{\R^n}^*\right)=
d^*_{\R^{n-1}}\circ\restn\circ\mathcal D_{a-1}^\mu.
$$
Hence the proposition
follows from Lemma \ref{lem:152751} composed by $\restn$.
\end{proof}

By the expression \eqref{eqn:Dii1}, the symmetry breaking operator
$\mathcal{D}^{i\to i-1}_{u,a}$ takes a simpler form when $i=1$:
\begin{equation*}
\mathcal{D}^{1 \to 0}_{u,a}
=\restn\left(-\gamma(u-\frac{n-3}{2},a)\mathcal{D}^{u-\frac{n-5}{2}}_{a-1}
d^*_{\R^n} + \frac{1}{2}(u+2-n)\mathcal{D}^{u-\frac{n-3}{2}}_a\iotan\right)
\end{equation*}
because 
\begin{equation*}
d_{\R^n}d^*_{\R^n}\iotan=0\quad\mathrm{on}\;\mathcal E^1(\R^n),
\end{equation*}
and so the first term of \eqref{eqn:Dii1} vanishes.
On the other hand, by the expression \eqref{eqn:Di-B},
we see that the symmetry breaking operator
$\mathcal D_{u,a}^{i\to i-1}$ takes a simpler form when $i=n$:
\begin{equation}\label{eqn:Dnn1}
\mathcal D_{u,a}^{n\to n-1}=\frac12(u+n+a)\restn\circ\mathcal D_a^{u+\frac{n+1}2}\iotan,
\end{equation}
since both the operators
\begin{equation*}
-d^*_{\R^n}\iotan d_{\R^n}
\quad\mathrm{and}\quad \restn\circ\left(\frac{\partial}{\partial x_n}\iotan+
d_{\R^n}^*\right)\; (=d^*_{\R^{n-1}}\circ\restn)
\end{equation*}
in the first and third terms of \eqref{eqn:Di-B} vanish on $\mathcal E^n (\R^n)$.
This operator is dual (via the Hodge star operator) to the symmetry breaking operator $\mathcal{D}^{0 \to 0}_{u+2i-n,a}$ 
\index{B}{Juhl's operator}
(Juhl's operator)
for functions (see Section \ref{subsec:6ii}).

\newpage
%%%%%%%%%%%%%%%%%%%%%%%%%%%%%%%%%%%%%%%%%%%%%%%
\section{Construction of differential symmetry breaking operators}\label{sec:5}

We proved in Proposition \ref{prop:153091} that 
there exist nonzero 
\index{B}{differential symmetry breaking operator}
differential symmetry breaking operators
from the $G$-representation 
\index{A}{Iilambda@$I(i,\lambda)_\alpha$, principal series of $O(n+1,1)$}
$I(i,\lambda)_\alpha$
to the $G'$-representation 
\index{A}{Jjnu@$J(j,\nu)_\beta$, principal series of $O(n,1)$}
$J(j,\nu)_\beta$
only if 
\begin{equation*}
j \in \{i-2, i-1, i, i+1\}.
\end{equation*}
In this chapter, we complete the proof of Theorem \ref{thm:Cps} which
provides explicit formul\ae{} of these symmetry breaking operators.
The formul\ae{} are given in the 
\index{B}{flat picture}
flat picture \eqref{eqn:Npic}, namely,
as differential operators $\mathcal{E}^i(\R^n) \To \mathcal{E}^j(\R^{n-1})$.

By the F-method (see Fact \ref{fact:Fmethod}), 
we have a natural bijection
(see \eqref{eqn:153091}) 
\index{A}{1sigma-lambda-alpha@$\sigma^{(i)}_{\lambda, \alpha}$, representation of $P$ on $\Exterior^i(\C^n)$}
\index{A}{1tau-nu-beta@$\tau^{(j)}_{\nu,\beta}$, representation of $P'$ on $\Exterior^j(\C^{n-1})$}
\begin{equation}\label{eqn:529again}
\mathrm{Diff}_{G'}(I(i,\lambda)_\alpha, J(j,\nu)_\beta)
\simeq
Sol(\mathfrak{n}_+; \sigma^{(i)}_{\lambda,\alpha}, \tau^{(j)}_{\nu,\beta}),
\end{equation}
where the right-hand side consists of (vector-valued) polynomial solutions
to the 
\index{B}{F-system}
F-sytem.
In the previous chapters, 
we determined explicitly these polynomial when
$j=i-1$ and $i+1$
(see Theorems \ref{thm:Fi-} and \ref{thm:Fiiplus}, respectively).
Then the proof for Theorem \ref{thm:Cps} is divided into the
following two parts:
\begin{itemize}
\item For $j=i-1$ and $i+1$, we translate these polynomial solutions
into geometric operators acting on differential forms
via the symbol map according to the F-method. 
We show that the resulting
symmetry breaking operators coincide with
$\widetilde{\C}^{i,i-1}_{\lambda,\nu}$ and $\widetilde{\C}^{i,i+1}_{\lambda,\nu}$,
respectively.

\item For $j=i-2$ and $i$, we use
\index{B}{duality theorem for symmetry breaking operators
(principal series)}
the duality theorem of symmetry breaking operators (Theorem \ref{thm:psdual}).
\end{itemize}
This completes the proof of Theorem \ref{thm:Cps}.
In the next chapter, we shall derive Theorems \ref{thm:2}-\ref{thm:2ii-2} from 
Theorem \ref{thm:Cps}.

%%%%%%%%%%%%%%%%%%%%%%%%%%%%%%%%%%%%%%%%%%%%%%%
\subsection{Proof of Theorem \ref{thm:Cps} in the case $j=i-1$}\label{subsec:pfThm}

In this section, we give a proof of Theorem \ref{thm:Cps} in the case $j=i-1$.
Suppose that we are in Case 2 of
Theorem \ref{thm:1A}, namely, 
\begin{equation*}
\text{$1\leq i \leq n$,
$a:=\nu-\lambda$ $(\in \N)$ and $\beta - \alpha \equiv a \; \mathrm{mod} \; 2$}.
\end{equation*}

Let $(g_0, g_1, g_2)$ be the triple of the nonzero polynomials
given in Theorem \ref{thm:Fi-}
so that 
\begin{equation*}
Sol(\mathfrak{n}_+; \sigma^{(i)}_{\lambda,\alpha}, \tau^{(i-1)}_{\nu,\beta})
=\C \sum_{k=0}^2(T_{a-k}g_k)h^{(k)}_{i\to i-1}.
\end{equation*}
We recall that $g_1=g_2=0$ if $i=n$ or $\lambda=\nu$.
By the isomorphism \eqref{eqn:529again},
the generator $\sum_{k=0}^2(T_{a-k}g_k)h^{(k)}_{i\to i-1}$
gives rise to a differential symmetry breaking operator, 
to be denoted by $D$.
What remains to prove is that $D$
is a nonzero scalar multiple of 
\index{A}{Clambdanu6@$\widetilde{\C}^{i,i-1}_{\lambda,\nu}$}
$\widetilde{\C}^{i,i-1}_{\lambda,\nu}=\widetilde{\mathcal{D}}^{i \to i-1}_{\lambda-i,\nu-\lambda}$
defined in \eqref{eqn:reCij} in the flat coordinates.
We set
\index{A}{11prn@$\prn$, projection onto $\mathrm{Ker}(\iotan)$}
\begin{equation}\label{eqn:Dii1prime}
P:=
\begin{cases}
\mathcal{D}^{\lambda-\frac{n-1}{2}}_a \iotan & \text{if $i=n$},\\
\iotan &\text{if $\lambda =\nu$},\\
-\mathcal{D}^{\lambda - \frac{n-3}{2}}_{a-2}d_{\R^n}d^*_{\R^n}
 \iota_{\frac{\partial}{\partial x_n}}
-\gamma\left(\lambda - \frac{n-1}{2}, a\right) \Pi_{n-1} 
\mathcal{D}^{\lambda-\frac{n-3}{2}}_{a-1} d^*_{\R^n} 
+ \frac{\lambda - n +i}{2} \mathcal{D}^{\lambda - \frac{n-1}{2}}_a
\iota_{\frac{\partial}{\partial x_n}} & \text{otherwise}.
\end{cases}
\end{equation}

We shall verify
\index{A}{H0ii-k@$h^{(k)}_{i\to i-1}$}
\begin{alignat*}{2}
\bullet\; \;&\mathrm{Symb}(P) \; &&=
\begin{cases}
(T_ag_0)h^{(0)}_{i \to i-1} &\text{$i=n$ or $\lambda =\nu$},\\
e^{-\frac{\pi\sqrt{-1}(a-2)}{2}}\sum_{k=0}^2
\left(T_{a-k}g_k\right)h^{(k)}_{i\to i-1} &\text{otherwise}.
\end{cases}\\
\bullet\; \; &\widetilde{\C}^{i,i-1}_{\lambda,\nu} &&= \mathrm{Rest}_{x_n=0}\circ P.
\end{alignat*}

By the general theory of the F-method (Fact \ref{fact:Fmethod}), 
Theorem \ref{thm:Cps} in the case $j=i-1$ follows from these two statements.
The second statement is clear from the identity 
$\restn \circ \prn = \restn$ (see \eqref{eqn:RestPi}) 
and the definition \eqref{eqn:reCij} of the renormalized operator 
$\widetilde{\mathbb{C}}^{i, i-1}_{\lambda,\nu}$.
The first statement in the case $i=n$ or $\lambda=\nu$ follows
directly from 
the formula for the symbol map given in 
Lemma \ref{lem:imaginary} and Proposition \ref{prop:160483} (2).
Thus the rest of this section
will be devoted to a proof of 
the first statement in the case $i\neq n$ and $\lambda \neq \nu$
(see Lemma \ref{lem:152436}), which requires some few computations.

Let $A,B,C\in\C$, and we set
\begin{eqnarray*}
g_2(t)&=&\widetilde C_{a-2}^\mu\left(e^{\frac{\pi\sqrt{-1}}2}t\right),\\
g_1(t)&=&e^{-\frac{\pi\sqrt{-1}}2} A \widetilde C_{a-1}^\mu\left(e^{\frac{\pi\sqrt{-1}}2}t\right),\\
g_0(t)&=&e^{-\frac{\pi\sqrt{-1}}2} Bt \widetilde C_{a-1}^\mu\left(e^{\frac{\pi\sqrt{-1}}2}t\right)+
C \widetilde C_{a-2}^\mu\left(e^{\frac{\pi\sqrt{-1}}2}t\right).
\end{eqnarray*}

\begin{lem}\label{lem:152277}
Let $a\in \N$ and $\mu \in \C$. 
We set 
\begin{eqnarray*}
D_1&:=&\left(- d_{\R^n}d^*_{\R^n}
+\left(C-\frac{i-1}{n-1}\right)\Delta_{\R^{n-1}}\right)\iota_{\frac{\partial}{\partial x_n}},\\
D_2&:=&-A\prn\circ d^*_{\R^n}
+(-A+B){\frac{\partial}{\partial x_n}}\iota_{\frac{\partial}{\partial x_n}}.
\end{eqnarray*}
Then the symbol of the differential operator
\begin{equation*}
\mathcal{D}^{\mu}_{a-2}D_1 + \mathcal{D}^\mu_{a-1}D_2\colon 
\mathcal{E}^i(\R^n) \To \mathcal{E}^{i-1}(\R^n)
\end{equation*}
is given by
\begin{equation*}
\mathrm{Symb}(\mathcal D^\mu_{a-2}D_1+\mathcal D^\mu_{a-1}D_2)=
e^{-\frac{\pi\sqrt{-1}}2(a-2)}\sum_{k=0}^2(T_{a-k}g_k)h^{(k)}_{i\to i-1}.
\end{equation*}
\end{lem}

\begin{proof}
We first claim the following equalities:
\begin{eqnarray*}
\mathrm{Symb}(\mathcal D^\mu_{a-2})&=&e^{-\frac{\pi\sqrt{-1}(a-2)}2}T_{a-2}g_2.\\
\mathrm{Symb}(A\mathcal D^\mu_{a-1})&=&e^{-\frac{\pi\sqrt{-1}(a-2)}2}T_{a-1}g_1.\\
\mathrm{Symb}\left(B\mathcal D^\mu_{a-1}\frac{\partial}{\partial x_n}+
C\mathcal D^\mu_{a-2}\Delta_{\R^{n-1}}\right)&=&e^{-\frac{\pi\sqrt{-1}(a-2)}2}T_{a}g_0.\\
\end{eqnarray*}
The first two follow from Lemma \ref{lem:imaginary}.
For the third equality we note that 
\begin{equation*}
T_ag_0 = e^{-\frac{\pi \sqrt{-1}}{2}}B\zeta_n T_{a-1}
\left(\widetilde{C}^\mu_{a-1}\left(e^{\frac{\pi \sqrt{-1}}{2}}\cdot\right)\right)
+CQ_{n-1}(\zeta')T_{a-2}\left(\widetilde{C}^{\mu}_{a-2}\left(e^{\frac{\pi\sqrt{-1}}{2}}\cdot\right)
\right)
\end{equation*}
by Lemma \ref{lem:1-6} (1) and (2).

Combining the above formul\ae{} with Proposition \ref{prop:160483} (4), (3), and (2),
respectively, we get
\begin{eqnarray*}
&&\mathrm{Symb}(-\mathcal D^\mu_{a-2}\left(d_{\R^n}d^*_{\R^n}
+\frac{i-1}{n-1}\Delta_{\R^{n-1}}\right)
\iota_{\frac{\partial}{\partial x_n}}+A\mathcal D^\mu_{a-1}(-\prn d^*_{\R^n}
-\frac{\partial}{\partial x_n}\iota_{\frac{\partial}{\partial x_n}})\\
&+&
(B\mathcal D^\mu_{a-1}\frac{\partial}{\partial x_n}+C\mathcal D^\mu_{a-2}\Delta_{\R^{n-1}})\iota_{\frac{\partial}{\partial x_n}})\\
&=&e^{-\frac{\pi\sqrt{-1}(a-2)}2}\left((T_{a-2} g_2) h^{(2)}_{i\to i-1}+(T_{a-1}g_1)h^{(1)}_{i\to i-1}
+(T_ag_0) h^{(0)}_{i\to i-1}\right).
\end{eqnarray*}
A simple computation shows that the left-hand side is equal to
\begin{equation*}
\mathrm{Symb}(\mathcal D^\mu_{a-2}D_1+\mathcal D^\mu_{a-1}D_2).
\end{equation*}
Hence Lemma \ref{lem:152277} is proved.
\end{proof}

We put
\begin{equation*}
A:=\gamma(\lambda-\frac{n-1}{2}, a), \;
B:=\left(1+\frac{\lambda-n+i}a\right)A, \;
C:=\frac{\lambda-n+i}a+\frac{i-1}{n-1}.
\end{equation*}

\begin{lem}\label{lem:152436}
Let $a:=\nu-\lambda \in \N_+$ and $i\neq n$.
Suppose $g_0(t), g_1(t)$, and $g_2(t)$ are given by the above $A,B,C$ with $\mu=\lambda-\frac{n-3}2$. Then the matrix-valued differential operator $P$ given in 
\eqref{eqn:Dii1prime}  satisfies
\begin{equation*}
\mathrm{Symb}(P)= e^{-\frac{\pi\sqrt{-1}(a-2)}2}\sum_{k=0}^2(T_{a-k}g_k)h^{(k)}_{i\to i-1}.
\end{equation*}
\end{lem}

\begin{proof}
With the above constants $A$, $B$, and $C$, 
the differential operators $D_1$ and $D_2$ in Lemma \ref{lem:152277} 
amount to
 \begin{eqnarray*}
D_1&=&
\left(- d_{\R^n}d^*_{\R^n}+\frac{\lambda-n+i}a\Delta_{\R^{n-1}}\right)\iota_{\frac{\partial}{\partial x_n}},\\
D_2&=&\gamma\left(-\prn\circ d^*_{\R^n}+\frac{\lambda-n+i}a{\frac{\partial}{\partial x_n}}\iota_{\frac{\partial}{\partial x_n}}\right).
\end{eqnarray*}
Therefore
  \begin{align*}
  \mathcal D^{\lambda-\frac{n-3}2}_{a-2}D_1+\mathcal D^{\lambda-\frac{n-3}2}_{a-1}D_2
  &= -\mathcal D^{\lambda-\frac{n-3}2}_{a-2}d_{\R^n}d^*_{\R^n}\iota_{\frac{\partial}{\partial x_n}}
  -\gamma(\lambda-\frac{n-1}{2},a) \mathcal D^{\lambda-\frac{n-3}2}_{a-1}\prn d^*_{\R^n}\\
  &\quad +\frac{\lambda-n+i}a\left(\mathcal D^{\lambda-\frac{n-3}2}_{a-1}\Delta_{\R^{n-1}}+
  \gamma(\lambda-\frac{n-1}{2},a)\mathcal D^{\lambda-\frac{n-3}2}_{a-1}\frac{\partial}{\partial x_n}\right)\iota_{\frac{\partial}{\partial x_n}}.
    \end{align*}
Applying 
the three-term relation \eqref{eqn:1522102} of the 
scalar-valued differential operators $\mathcal{D}^{\mu}_{\ell}$,
it amounts to 
\begin{equation*}
-\mathcal D^{\lambda-\frac{n-3}2}_{a-2}d_{\R^n}d^*_{\R^n}
\iota_{\frac{\partial}{\partial x_n}}-\gamma(\lambda-\frac{n-1}{2},a)\prn
\mathcal D^{\lambda-\frac{n-3}2}_{a-1}d^*_{\R^n}
+\frac{\lambda-n+i}2\mathcal D^{\lambda-\frac{n-1}2}_{a}\iota_{\frac{\partial}{\partial x_n}}.
\end{equation*}
Hence, Lemma  \ref{lem:152277} implies the statement of Lemma \ref{lem:152436}.
\end{proof}

Thus we have completed the proof of Theorem \ref{thm:Cps} in the case $j=i-1$.

%%%%%%%%%%%%%%%%%%%%%%%%%%%%%%%%%%%%%%%%%%%%%%%
\subsection{Proof of Theorem \ref{thm:Cps} in the case $j=i+1$}\label{subsec:thm2ii+1}

In this section, we give a proof of Theorem \ref{thm:Cps} in the case $j=i+1$.
Suppose we are in Cases 4 or 4$'$ in Theorem \ref{thm:Cps}, namely, 

Case 4. $1\leq i\leq n-2$,
$(\lambda,\nu) = (i,i+1)$ and $\beta \equiv \alpha +1 \; \mathrm{mod}\;2$,

Case 4$'$. $i=0$,
$\lambda \in -\N$, $\nu = 1$ and 
$\beta \equiv \alpha + \lambda +1 \;\mathrm{mod} \;2$.

Then we have from Theorem \ref{thm:Fiiplus}
 \index{A}{H0ii+1@$h^{(1)}_{i\to i+1}$}
\begin{equation*}
Sol(\mathfrak{n}_+;\sigma^{(i)}_{\lambda,\alpha}, \tau^{(i+1)}_{\nu,\beta})
=
\begin{cases}
\C\left(T_{-\lambda}\widetilde{C}^{\lambda-\frac{n-1}{2}}_{-\lambda}
\left(e^{\frac{\pi\sqrt{-1}}{2}}\cdot \right) \right)
h^{(1)}_{i\to i+1} & \text{in Case 4$'$},\\
\C h^{(1)}_{i\to i+1} & \text{in Case 4}. 
\end{cases}
\end{equation*}
We define a differential operator $Q\colon  \mathcal{E}^i(\R^n) \To \mathcal{E}^{i+1}(\R^n)$
by the formula
\begin{equation*}
Q :=
\begin{cases}
e^{-\frac{\pi \sqrt{-1}\lambda}{2}} \Pi_{n-1} \circ 
\mathcal{D}^{\lambda-i-\frac{n-1}{2}}_{-\lambda} d_{\R^n}
& \text{if $ i=0$,}\\
\Pi_{n-1} \circ d_{\R^n}, & \text{if $1\leq i \leq n-2$.}
\end{cases}
\end{equation*}
We shall verify the following claims in both Case 4 and Case 4$'$:
\begin{itemize}
\item $\mathrm{Symb}(Q)$ is a generator of 
$Sol\left(\mathfrak n_+; \sigma^{(i)}_{\lambda,\alpha}, \tau^{(i+1)}_{\nu,\beta}\right)$.
\item $\restn \circ Q\colon \mathcal{E}^i(\R^n) \To \mathcal{E}^{i+1}(\R^{n-1})$ coincides with
$\widetilde{\C}^{i,i+1}_{\lambda,\nu}$.
\end{itemize}
By the general theory of the F-method (Fact \ref{fact:Fmethod}),
Theorem \ref{thm:Cps} in the case $j=i+1$ follows from these two claims.
The first claim follows from 
the computation of the symbol in 
Lemma \ref{lem:imaginary}
and
Proposition \ref{prop:160483} (8).
Now, by use of 
the identity $\restn \circ \prn=\restn$ 
(see \eqref{eqn:RestPi}) and 
the definition \eqref{eqn:reCi+} of 
\index{A}{Dii6@$\widetilde{\mathcal{D}}^{i \to i+1}_{u,a}$}
$\widetilde{\C}^{i,i+1}_{\lambda,\nu} 
(=\widetilde{\mathcal{D}}^{i\to i+1}_{\lambda-i, \nu-\lambda})$,
we obtain $\widetilde{\C}^{i,i+1}_{\lambda,\nu}=\restn \circ Q$.
Thus we have completed the proof of Theorem \ref{thm:Cps} in the case $j=i+1$.

%%%%%%%%%%%%%%%%%%%%%%%%%%%%%%%%%%%%%%%%%%%%%%%
\subsection{Application of the duality theorem for symmetry breaking operators}
\label{subsec:Cdual}

In the following two Sections \ref{subsec:6ii} and \ref{subsec:Cpsi--}, 
we shall give a proof of Theorem \ref{thm:Cps} in the cases $j = i$ and $i-2$
by applying the duality theorem for symmetry breaking operators
(Theorem \ref{thm:psdual}), 
instead of solving the F-system. We shall see that
the cases $j=i$ and $i-2$ are derived
from the cases $\tilde{j} = \tilde{i}-1$ and $\tilde{i}+1$, for which
the proof was completed in Sections \ref{subsec:pfThm} and \ref{subsec:thm2ii+1},
respectively. 

In this section we give a set-up for the duality theorem.
We put
\begin{equation*}
\tilde{i}:=n-i, \qquad \tilde{j}:=n-1-j.
\end{equation*}

First we 
examine a geometric meaning of 
the proof of Lemma \ref{lem:psdual} 
and Theorem \ref{thm:psdual}.
Let $\chi_{--}$ be the one-dimensional representation of $G$ as defined 
in \eqref{eqn:chiab}. Then
 the proof of Lemma \ref{lem:psdual} shows that
the 
\index{B}{Hodge star operator}
Hodge star operator on $\mathcal{E}^i(\R^n)$ induces the $G$-isomorphism
$I(i,\lambda)_\alpha \simeq I(\tilde{i},\lambda)_\alpha\otimes \chi_{--}$
in the 
\index{B}{flat picture}
flat picture (see \eqref{eqn:Iiflat}) as below:
\index{A}{1iotaI@$\iota^{(i)}_\lambda$, map to flat picture}
\begin{alignat}{3}\label{eqn:psRstar}
\mathcal E^{i}(\R^n)&\stackrel{*_{\R^n}\;\;}
{\joinrel\relbar\joinrel\relbar\joinrel\relbar\joinrel\relbar\joinrel\relbar\joinrel\rightarrow }
\mathcal{E}^{\tilde{i}}(\R^n)\; &\simeq & \; \mathcal E^{\tilde{i}}(\R^n)\otimes \C\\
{\iota^{(i)}_{\lambda}} 
{\;\mathrel{\rotatebox[origin=c]{90}{$\hooklongrightarrow$}}} \hskip 15pt
&&&\hskip 43pt 
{\mathrel{\rotatebox[origin=c]{90}{$\hooklongrightarrow$}}\;} 
 {\iota_\lambda^{(\tilde{i})}} \nonumber\\
I(i,\lambda)_\alpha&
\joinrel\relbar\joinrel\relbar\joinrel\relbar
\joinrel\relbar\joinrel\relbar\joinrel\relbar
\joinrel\relbar\joinrel\relbar
\joinrel\relbar\joinrel\relbar\joinrel\relbar\joinrel\relbar&
\!\!\joinrel\rightarrow & I(\tilde{i},\lambda)_\alpha \otimes \chi_{--}.\nonumber
\end{alignat}
We recall the proof of Theorem \ref{thm:psdual} 
is based on the $G$- and 
$G'$-isomorphisms
\index{A}{1x--@$\chi_{--}$}
\begin{align*}
I(i,\lambda)_\alpha &\simeq I(\tilde{i}, \lambda)_\alpha \otimes \chi_{--},\\
J(j,\nu)_\beta &\simeq J(\tilde{j},\nu)_\beta \otimes \chi_{--}\vert_{G'},
\end{align*}
which induce the duality of symmetry breaking operators
\begin{align*}
\mathrm{Diff}_{G'}\left(I(i,\lambda)_\alpha, J(j,\nu)_\beta\right)
\simeq
\mathrm{Diff}_{G'}\left(I(\tilde{i},\lambda)_\alpha \otimes \chi_{--},
J(\tilde{j},\nu)_\beta \otimes \chi_{--}\vert_{G'} \right),
\quad T \mapsto \widetilde{T}\otimes \mathrm{id}.
\end{align*}
In the flat picture, this isomorphism is realized by \eqref{eqn:psRstar} in the 
following key diagram:
\begin{equation*}
\xymatrix{
\mathcal{E}^i(\R^n) 
&& I(i,\lambda)_\alpha\eq[d]
\ar@{_{(}->}[ll]
\ar[rr]^{T}
&&
J(j,\nu)_\beta\eq[d]
\ar@{^{(}->}[rr]
&& \mathcal{E}^j(\R^{n-1})\\
\mathcal{E}^{\tilde{i}}(\R^n)\ar[u]^{*_{\R^n}}  
&& I(\tilde{i},\lambda)_\alpha\otimes\chi_{-, -}
\ar@{_{(}->}[ll]
\ar[rr]^{\widetilde T\otimes \mathrm{id}}
&&
J(\tilde{j},\nu)_\beta\otimes\chi_{-, -}\ar@{^{(}->}[rr]
&& \mathcal{E}^{\tilde{j}}(\R^{n-1})\ar[u]_{*_{\R^{n-1}}}
}
\end{equation*}

We note that the 6-tuple $(i,j,\lambda,\nu,\alpha,\beta)$ for $j=i$ belongs to
Case 3 in Theorem \ref{thm:1A} (which we shall consider in this section) 
if and only if 
$(\tilde{i},\tilde{i}-1,\lambda,\nu,\alpha,\beta)$ belongs to Case 2 in Theorem \ref{thm:1A} (which was treated
in Section \ref{subsec:pfThm}).
Likewise $(i,j,\lambda,\nu,\alpha,\beta)$ for $j=i-2$ belongs to
Cases 1 and 1${}'$ in Theorem \ref{thm:1A} if and only if 
$(\tilde{i},\tilde{i}+1,\lambda,\nu,\alpha,\beta)$ belongs to Cases 4 and 4${}'$ in Theorem \ref{thm:1A} (which were treated
in Section \ref{subsec:thm2ii+1}).

\index{A}{Clambdanu8@$\widetilde{\C}^{i,i+1}_{\lambda,\nu}$}
\index{A}{Clambdanu7@$\widetilde{\C}^{i,i}_{\lambda,\nu}$}
In view of the above geometric interpretation of the duality theorem
(Theorem \ref{thm:psdual}),
\index{A}{Clambdanu6@$\widetilde{\C}^{i,i-1}_{\lambda,\nu}$}
Theorem \ref{thm:Cps} in the case $j=i$ and $i-2$ is deduced from the
\index{A}{Clambdanu3@$\widetilde{\C}^{i,i-2}_{\lambda,\nu}$}
following identities
\begin{eqnarray}
\label{eqn:Ci-dual}
\widetilde\C^{i,i}_{\lambda,\nu}&=&(-1)^{n-1} *_{\R^{n-1}}\circ \widetilde\C^{\tilde i,\tilde i-1}_{\lambda,\nu}
\circ \left(*_{\R^n}\right)^{-1},\\
\label{eqn:Ci+dual}
\widetilde\C^{i,i-2}_{\lambda,\nu}&=&(-1)^{n-1} *_{\R^{n-1}}\circ \widetilde\C^{\tilde i,\tilde i+1}_{\lambda,\nu}
\circ \left(*_{\R^n}\right)^{-1},
\end{eqnarray}
\noindent
in the flat picture, which will  be treated in Propositions \ref{prop:dualCi} and \ref{prop:dualCi+},
respectively, in the next two sections.

%%%%%%%%%%%%%%%%%%%%%%%%%%%%%%%%%%%%%%%%%%%%%%%
\subsection{Proof of Theorem \ref{thm:Cps} in the case $j=i$}
\label{subsec:6ii}

In this section, we prove the duality \eqref{eqn:Dii1prime} as well as the equality
$\eqref{eqn:Cijln}=\eqref{eqn:Cii2}$ for the two expressions of $\C^{i,i}_{\lambda,\nu}$
(or equivalently, $\eqref{eqn:Dii}=\eqref{eqn:DiB}$ for $\mathcal{D}^{i\to i}_{u,a}$),
and complete the proof of Theorem \ref{thm:Cps} in the case $j=i$.

\begin{prop}\label{prop:dualCi}
Let $0\leq i\leq n-1$, 
and $(\lambda,\nu)\in\C^2$ with $\nu-\lambda\in\N$.
We consider a matrix-valued differential operator
\begin{equation}\label{eqn:dualCi}
(-1)^{n-1} *_{\R^{n-1}}\circ\C^{\tilde i,\tilde i-1}_{\lambda,\nu}\circ\left(*_{\R^n}\right)^{-1}
\colon \mathcal{E}^i(\R^n) \To \mathcal{E}^i(\R^{n-1}),
\end{equation}
where $\tilde i:=n-i$. Then \eqref{eqn:dualCi} and the two expressions
\eqref{eqn:Cijln}, \eqref{eqn:Cii2} of $\C^{i,i}_{\lambda,\nu}$
are equal to each other. Moreover, we have the following identity for 
the renormalized symmetry breaking operators (see \eqref{eqn:reCij} for the definition)
\begin{equation}\label{eqn:reCiidual}
\widetilde{\C}^{i,i}_{\lambda,\nu}=(-1)^{n-1} *_{\R^{n-1}} \circ
\widetilde{\C}^{\tilde{i}, \tilde{i}-1}_{\lambda,\nu} \circ (*_{\R^n})^{-1}.
\end{equation}
\end{prop}

\begin{proof}
We recall the notation
from \eqref{eqn:Cln} that 
$\widetilde{\C}_{\lambda,\nu} = \restn \circ \mathcal{D}^{\lambda-\frac{n-1}{2}}_{\nu-\lambda}$.
By \eqref{eqn:Cijln2} or by \eqref{eqn:Dii1}, we have
\begin{eqnarray*}
&&\C_{\lambda,\nu}^{\tilde i,\tilde i-1}=\mathcal D_{\lambda-\tilde i,\nu-\lambda}^{\tilde i\to\tilde i-1}\\
&=&
\restn\circ\left(-\mathcal D_{\nu-\lambda-2}^{\lambda-\frac{n-3}2}d_{\R^n}d_{\R^n}^*\iotan
-\gamma(\lambda-\frac{n-1}2,\nu-\lambda)\mathcal D_{\nu-\lambda-1}^{\lambda-\frac{n-3}2}d_{\R^n}^*
+\frac12(\lambda-i)\mathcal D_{\nu-\lambda}^{\lambda-\frac{n-1}2}\iotan\right).
\end{eqnarray*}

\noindent
Applying 
the formul{\ae} for $*_{\R^{n-1}}\circ \restn\circ T\circ (*_{\R^n})^{-1}$
given in 
Lemma \ref{lem:152305} ($T=d^*_{\R^n}$ and $\iotan$)
and in Lemma \ref{lem:160235} ($T=d_{\R^n}d^*_{\R^n}\iotan$), 
we obtain
\begin{align*}
&(-1)^{n-1} *_{\R^{n-1}} \circ 
\widetilde{\C}^{\tilde{i},\tilde{i}-1}_{\lambda,\nu} \circ (*_{\R^n})^{-1}\\
&=-d^*_{\R^{n-1}} d_{\R^{n-1}} \restn \circ 
\mathcal{D}^{\lambda-\frac{n-3}{2}}_{\nu-\lambda-2}\\
&\quad+\restn\circ \left(\gamma(\lambda-\frac{n-1}{2}, \nu-\lambda)
\mathcal{D}^{\lambda-\frac{n-3}{2}}_{\nu-\lambda-1} \iotan d_{\R^n} 
+\frac{1}{2}(\lambda-i) \mathcal{D}^{\lambda-\frac{n-1}{2}}_{\nu-\lambda}\right)\\
&=-d^*_{\R^{n-1}}d_{\R^{n-1}}\widetilde{\C}_{\lambda+1, \nu-1} 
+\gamma(\lambda-\frac{n-1}{2}, \nu-\lambda) \widetilde{\C}_{\lambda+1, \nu}
\iotan d_{\R^n} + \frac{\lambda-i}{2}\widetilde{\C}_{\lambda,\nu}.
\end{align*}
Hence we have proved the equality $\eqref{eqn:dualCi}=\eqref{eqn:Cii2}$.

On the other hand, we have proved in Proposition \ref{prop:152751} that
$\C^{\tilde{i},\tilde{i}-1}_{\lambda,\nu}$ is equal to 
\begin{eqnarray*}
&&\restn\circ\\
&&\left(-\mathcal D_{\nu-\lambda-2}^{\lambda-\frac{n-3}2}d_{\R^n}^*\iotan d_{\R^n}
+\frac12(\nu-i)\mathcal D_{\nu-\lambda}^{\lambda-\frac{n-1}2}\iotan
-\gamma(\lambda-\frac{n}2,\nu-\lambda)\mathcal D_{\nu-\lambda-1}^{\lambda-\frac{n-1}2}
\left(d_{\R^n}^*+\frac\partial{\partial x_n}\iotan\right)
\right).
\end{eqnarray*}

Applying Lemma \ref{lem:152305} 
to $T=-d^*_{\R^n}\iotan d_{\R^n}, \iotan$, 
and $d^*_{\R^n}+\frac{\partial}{\partial x_n} \iotan$,
we have
\begin{align*}
&(-1)^{n-1} *_{\R^{n-1}} \circ 
\C^{\tilde{i},\tilde{i}-1}_{\lambda,\nu} \circ (*_{\R^n})^{-1}\\
&= \restn \circ \left(\mathcal{D}^{\lambda-\frac{n-3}{2}}_{\nu-\lambda-2}d_{\R^n}d^*_{\R^n}
+\frac{1}{2}(\nu-i)\mathcal{D}^{\lambda-\frac{n-1}{2}}_{\nu-\lambda} 
- \gamma(\lambda-\frac{n}{2}, \nu-\lambda)\mathcal{D}^{\lambda-\frac{n-1}{2}}_{\nu-\lambda-1}
d_{\R^n}\iotan \right)\\
&=\widetilde{\C}_{\lambda+1, \nu-1} d_{\R^n}d^*_{\R^n} + \frac{1}{2}(\nu-i)
\widetilde{\C}_{\lambda,\nu}-\gamma(\lambda-\frac{n}{2},\nu-\lambda) 
\widetilde{\C}_{\lambda,\nu-1},
\end{align*}
which is equal to the formula \eqref{eqn:Cijln}. Thus we have shown the equalities:
$\eqref{eqn:Cijln}=\eqref{eqn:Cii2}=\eqref{eqn:Ci-dual}$.

Finally, let us prove the identity \eqref{eqn:reCiidual}.
We have already shown \eqref{eqn:reCiidual} when $\lambda \neq \nu$
and $i\neq 0$ (\emph{i.e.}\ $\tilde{i} \neq n$) because 
$\widetilde{\C}^{i,i}_{\lambda,\nu} = \C^{i,i}_{\lambda,\nu}$ and
$\widetilde{\C}^{\tilde{i}, \tilde{i}-1}_{\lambda,\nu} = \C^{\tilde{i},\tilde{i}-1}_{\lambda,\nu}$
in this case. For $\lambda=\nu$ or $i=0$ (\emph{i.e.}\ $\tilde{i}=n$), the identity 
\eqref{eqn:reCiidual} is an immediate consequence of the definition
\eqref{eqn:reCij} and Lemma \ref{lem:152305} with $T= \iotan$.
Thus the proof of the proposition is completed.
\end{proof}

%%%%%%%%%%%%%%%%%%%%%%%%%%%%%%%%%%%%%%%%%%%%%%
\subsection{Proof of Theorem \ref{thm:Cps} in the case $j=i-2$}
\label{subsec:Cpsi--}

In this section, we prove Theorem \ref{thm:Cps} in the remaining case,
namely, $j=i-2$.
We keep the notation $(\tilde{i}, \tilde{j}) = (n-i,n-1-j)$,
and assume $j=i-2$ in this section. Then
Cases 1 (resp. 1$'$) and 4 (resp. 4$'$) in Theorem \ref{thm:1A} are 
dual to each other, namely,
\begin{enumerate}
\item[] Case 4: 
$\tilde{j}=\tilde{i}+1$,
$1\leq \tilde{i} \leq n-2$, 
$(\lambda,\nu) = (\tilde{i}, \tilde{i}+1)$, $\beta \equiv \alpha+1 \; \mathrm{mod}\; 2$,
\vskip 0.1in
\item[] Case 4$'$: 
$(\tilde{i},\tilde{j}) = (0,1)$, $\lambda\in -\N$, $\nu = 1$, 
$\beta \equiv \alpha + \lambda + 1 \; \mathrm{mod} \; 2$,
\end{enumerate}
are equivalent to
\begin{enumerate}
\item[] Case 1: 
$j=i-2$,
$2\leq i \leq n-1$, 
$(\lambda,\nu) = (n-i, n-i+1)$, $\beta \equiv \alpha+1\;\mathrm{mod}\;2$,
\vskip 0.1in
\item[] Case 1$'$: 
$(i,j) = (n,n-2)$, $\lambda \in -\N$, $\nu=1$, $\beta \equiv \alpha+\lambda+1\;\mathrm{mod}\;2$,
\end{enumerate}
respectively. 

By the duality (Theorem \ref{thm:psdual}) and the proof of Theorem \ref{thm:Cps}
in the case $\tilde{j}=\tilde{i}+1$, Theorem \ref{thm:Cps} in the case 
$j=i-2$ is deduced from the following proposition.

\begin{prop}\label{prop:dualCi+}
We have the identity \eqref{eqn:Ci+dual}, namely,
\index{A}{Clambdanu4@$\widetilde{\C}^{i,i-2}_{n-i,n-i+1}$}
\index{A}{Clambdanu5@$\widetilde{\C}^{n,n-2}_{\lambda,1}$}
\begin{alignat*}{3}
&\widetilde{\C}^{i, i-2}_{n-i, n-i+1} &&=
(-1)^{n-1} *_{\R^{n-1}}\circ \widetilde{\C}^{\tilde{i},\tilde{i}+1}_{\tilde{i}, \tilde{i}+1}
\circ (*_{\R^n})^{-1} \quad &&\emph{in Case 1},\\
&\widetilde{\C}^{n,n-2}_{\lambda,1} &&=
(-1)^{n-1} *_{\R^{n-1}} \circ \widetilde{\C}^{0,1}_{\lambda,1} \circ (*_{\R^n})^{-1}
\quad &&\emph{in Case 1$'$.}
\end{alignat*}
\end{prop}

\begin{proof}
We recall from \eqref{eqn:reCi+} that 
$\widetilde{\C}^{\tilde{i},\tilde{i}+1}_{\lambda,\nu} 
=\restn \circ \mathcal{D}^{\lambda-\tilde{i}-\frac{n-1}{2}}_{\tilde{i}-\lambda} d_{\R^n}$.
It follows from Lemma \ref{lem:152305} (1) and from Lemma \ref{lem:Tsharp} with
$T=d_{\R^n}$ that
\begin{equation*}
(-1)^{n-1} *_{\R^{n-1}} \circ \restn \circ d_{\R^n} \circ (*_{\R^n})^{-1}
=\restn \circ \iotan d^*_{\R^n}.
\end{equation*}
Hence we have
\begin{align}\label{eqn:dCi+}
(-1)^{n-1} *_{\R^{n-1}} \circ \widetilde{\C}^{\tilde{i}, \tilde{i}+1}_{\lambda,\nu}
\circ (*_{\R^n})^{-1}
=\restn \circ \mathcal{D}^{\lambda - \tilde{i} - \frac{n-1}{2}}_{\tilde{i}-\lambda} \iotan d^*_{\R^n}.
\end{align}

In Case 1, $\lambda = \tilde{i}$ and therefore 
\eqref{eqn:dCi+} amounts to 
$\restn \iotan d^*_{\R^n} = \widetilde{\C}^{i,i-2}_{n-i,n-i+1}$.

In Case 1$'$, $i=n$, $\tilde{i}=0$, and therefore
\eqref{eqn:dCi+} amounts to
$\restn \circ \mathcal{D}^{\lambda-\frac{n-1}{2}}_{-\lambda} \iotan d^*_{\R^n}$
which is equal to $\widetilde{\C}^{n,n-2}_{\lambda,1}$. 
\end{proof}

Hence the proof of Theorem \ref{thm:Cps} is completed.

\newpage
%%%%%%%%%%%%%%%%%%%%%%%%%%%%%%%%%%%%%%%%%%%%%%
\section{Solutions to Problems \ref{prob:1} and \ref{prob:2} for $(S^n, S^{n-1})$}
\label{sec:solAB}

In this chapter, we complete the proof of Theorem \ref{thm:1} and 
Theorems \ref{thm:2}--\ref{thm:2ii-2}, which solve
Problems \ref{prob:1} and \ref{prob:2} of 
conformal geometry
for the model space $(X,Y) = (S^n, S^{n-1})$, respectively.

%%%%%%%%%%%%%%%%%%%%%%%%%%%%%%%%%%%%%%%%%%%%%%%
\subsection{Problems \ref{prob:1} and \ref{prob:2} for conformal transformation group 
$\mathrm{Conf}(X;Y)$}\label{subsec:pbAB}

We begin with the general setting where $(X, g_X)$ is a 
pseudo-Riemannian manifold of dimension $n$, and
$Y$ is a submanifold of dimension $m$
such that the metric tensor $g_X$ is
nondegenerate when restricted to $Y$.
We define $\mathrm{Conf}(X;Y)$ as a subgroup of the 
full conformal group 
\index{A}{C1X@$\mathrm{Conf}(X)$|textbf}
$\mathrm{Conf}(X):=\{\varphi \colon X \To X \;
\text{is a conformal diffeomorphism}\}$ by
\index{A}{C1XY@$\mathrm{Conf}(X;Y)$|textbf}
\begin{equation}\label{eqn:pbAB}
\mathrm{Conf}(X;Y):=\{\varphi\in \mathrm{Conf}(X) : \varphi(Y) = Y\}.
\end{equation}
Then $\mathcal{E}^i(X)_{u,\delta}$ is a $\mathrm{Conf}(X)$-module
for $0\leq i \leq n$, $u \in \C$, $\delta \in \Z/2\Z$, and
$\mathcal{E}^j(Y)_{v,\eps}$ is a $\mathrm{Conf}(X;Y)$-module
for $0\leq j \leq m$, $v \in \C$, $\eps \in \Z/2\Z$. 
The group $\mathrm{Conf}(X;Y)$ is the largest effective group for 
Problems \ref{prob:1} and \ref{prob:2} on differential symmetry breaking
operators from $\mathcal{E}^i(X)_{u,\delta}$ to $\mathcal{E}^j(Y)_{v,\eps}$.

The first reduction is the 
\index{B}{duality theorem for symmetry breaking operators (conformal geometry)}
duality theorem for symmetry breaking operators. 
We recall from Proposition \ref{prop:Hodgeconf} 
that the Hodge star operator $*_X$ carrying $i$-forms to $(n-i)$-forms is a conformally equivariant operator for $X$, and $*_Y$ carrying $j$-forms to $(m-j)$-forms
is a conformally equivariant for $Y$.
Then, a solution to Problem \ref{prob:1} (or Problem \ref{prob:2})
 for $i$-forms on $X$ and $j$-forms on the submanifold $Y$, 
 to be denoted by the $(i,j)$ case, leads
us to solutions for $(i,m-j)$,
$(n-i,j)$, and $(n-i,m-j)$ cases via the following natural bijections:
\begin{equation}\label{eqn:4stars}
\xymatrix{
\mathrm{Diff}_{\mathrm{Conf}(X;Y)}\left(\mathcal E^i(X)_{u,0},\mathcal E^j(Y)_{v,0}\right)\qquad
\eq[d]\eq[r]&
\qquad\mathrm{Diff}_{\mathrm{Conf}(X;Y)}\left(\mathcal E^i(X)_{u,0},\mathcal E^{m-j}(Y)_{v-m+2j,1}\right)
\eq[d]\\
\mathrm{Diff}_{\mathrm{Conf}(X;Y)}\left(\mathcal E^{n-i}(X)_{u-n+2i,1},\mathcal E^j(Y)_{v,0}\right)\eq[r]&
\mathrm{Diff}_{\mathrm{Conf}(X;Y)}\left(\mathcal E^{n-i}(X)_{u-n+2i,1},\mathcal E^{m-j}(Y)_{v-m+2j,1}\right),
}
\end{equation}
given by
\begin{equation}\label{eqn:4SBO}
\xymatrix{
D \ar@{|->}[r] \ar@{|->}[d] &  {*}_Y\circ D\ar@{|->}[d]\\
D\circ {*}_X\ar@{|->}[r] &{*}_Y\circ D\circ {*}_X.
}
\end{equation}
In other words, a solution to Problem \ref{prob:1} (or Problem \ref{prob:2})
for a fixed $(\delta,\eps) \in (\Z/2\Z)^2$ yields solutions to Problem \ref{prob:1}
(or Problem \ref{prob:2}, respectively) for the other three cases of 
$(\delta, \eps) \in (\Z/2\Z)^2$.

%%%%%%%%%%%%%%%%%%%%%%%%%%%%%%%%%%%%%%%%%%%%%%%
\subsection{Model space $(X,Y)=(S^n,S^{n-1})$}\label{subsec:ABSn}

From now we consider the model space $(X,Y) = (S^n, S^{n-1})$.
We shall see that Problems \ref{prob:1} and \ref{prob:2} are
 deduced from the problems on symmetry breaking operators
between principal series representations of $G=O(n+1,1)$ and 
$G'=O(n,1)$ which were proved in Theorems \ref{thm:1A} and \ref{thm:Cps},
respectively. For this, we first clarify small differences such as disconnected
components and coverings between the groups $G$ and $\mathrm{Conf}(X)$,
and also between $G'$ and $\mathrm{Conf}(X;Y)$.

We recall from Section \ref{subsec:ps} that the natural action of 
$G=O(n+1,1)$ on the light cone 
\index{A}{11OXi@$\Xi$, light cone}
$\Xi \subset \R^{n+1,1}$ induces
a conformal action on the standard Riemann sphere $S^n$ via the isomorphism
$S^n \simeq \Xi/\R^\times$. Conversely, it is well-known that any conformal
transformation of the standard sphere $X=S^n$ is obtained in this manner if $n\geq 2$,
and thus we have a natural isomorphism:
\begin{equation}\label{eqn:ConfSn}
\mathrm{Conf}(X) \simeq O(n+1,1)/\{\pm I_{n+2}\}.
\end{equation}

Let us compute $\mathrm{Conf}(X;Y)$ for $(X,Y) = (S^n, S^{n-1})$.
We realize $Y = S^{n-1}$ as a submanifold 
$\{(x_0, \ldots, x_{n-1}, x_n) \in S^n : x_n = 0\}$ of $X=S^n$ as before.

\begin{lem}\label{lem:20160319}
Via the isomorphism
\eqref{eqn:ConfSn},
we have 
\begin{equation*}
\mathrm{Conf}(X;Y) \simeq \left(O(n,1) \times O(1)\right)/\{\pm I_{n+2}\}.
\end{equation*}
\end{lem}

\begin{proof}
Suppose $g=(g_{ij})_{0\leq i,j\leq n+1} \in O(n+1,1)$ leaves
$Y=S^{n-1}$ invariant. This means that $\sum_{j=0}^{n+1}g_{nj}\xi_{j}=0$
for all $\xi=(\xi_0,\ldots, \xi_{n+1}) \in \Xi$ with $\xi_n = 0$,
which implies $g_{nj}=0$ for all $j \neq n$.
In turn, $g_{in}=0$ for all $i \neq n$ and $g_{nn} = \pm 1$ 
because $g\in O(n+1,1)$. Hence we have shown $g \in O(n,1) \times O(1)$.
Conversely, any element of $O(n,1) \times O(1)$ clearly leaves $S^{n-1}$ invariant.
Thus the lemma is proved.
\end{proof}

The above lemma says that 
the group $\mathrm{Conf}(X;Y)$ is the quotient of the direct product 
group of $G' = O(n,1)$ and $O(1)$, however, we do not have to consider
the second factor $O(1)$ in solving Problems \ref{prob:1} and \ref{prob:2}.
In order to state this claim precisely, we write
\begin{equation*}
\delta \cdot i :=
\begin{cases}
i & \text{if $\delta \equiv 0$}\\
n-i & \text{if $\delta \equiv 1$},
\end{cases}
\qquad
\eps \cdot j :=
\begin{cases}
j & \text{if $\eps \equiv 0$}\\
n-1-j & \text{if $\eps \equiv 1$},
\end{cases}
\end{equation*}
for $\delta, \eps \in \Z/2\Z$ and $0\leq i \leq n, 0\leq j \leq n-1$.
We recall that 
\index{A}{Iilambda@$I(i,\lambda)_\alpha$, principal series of $O(n+1,1)$}
$I(i,\lambda)_\alpha$ is a principal series representation 
with parameter $\lambda\in \C$ and $\alpha \in \Z/2\Z$
of $G=O(n+1,1)$, and $J(j,\nu)_\beta$ is that of $G' = O(n,1)$.
Then we have

\begin{lem}
For $(X,Y) = (S^n, S^{n-1})$, we have a natural isomorphism:
\index{A}{C1XY@$\mathrm{Conf}(X;Y)$}
\begin{equation}\label{eqn:XYSBO}
\mathrm{Hom}_{\mathrm{Conf}(X;Y)}\left(
\mathcal{E}^i(X)_{u,\delta}, \mathcal{E}^j(Y)_{v,\eps}
\right)
\simeq
\mathrm{Hom}_{O(n,1)}\left(
I(\delta\cdot i, u+i)_{\delta\cdot i},
J(\eps\cdot j,v+j)_{\eps \cdot j}\right).
\end{equation}
Here the subscripts $\delta \cdot i$ and $\eps \cdot j$ are 
regarded as elements in $\Z/2\Z$.
\end{lem}

\begin{proof}
For $\alpha \in \Z/2\Z$, we write $(-1)^\alpha$ for the one-dimensional 
representation of $O(1)$ as before, namely,
\begin{equation*}
(-1)^\alpha =
\begin{cases}
\one\quad \hspace{0.35cm} \text{(trivial representation)} & \text{if $\alpha \equiv 0$},\\
\mathrm{sgn} \quad \text{(signature representation)} & \text{if $\alpha \equiv 1$}.
\end{cases}
\end{equation*}

Since the central element $-I_{n+2}$ of $G$ acts on the principal series 
representation $I(i,\lambda)_\alpha$ as the scalar $(-1)^{i+\alpha}$, and
since $-I_{n+1}$ acts on $J(j,\nu)_\beta$ as the scalar $(-1)^{j+\beta}$,
we have
\begin{align*}
&\mathrm{Hom}_{O(n,1) \times O(1)} \left(I(i,\lambda)_\alpha,
J(j,\nu)_\beta \boxtimes (-1)^\gamma\right)\\
&\simeq
\begin{cases}
\mathrm{Hom}_{O(n,1)}\left(I(i,\lambda)_\alpha, J(j,\nu)_\beta \right)
& \text{if $\gamma \equiv i+j+\alpha +\beta \; \mathrm{mod} \; 2$,}\\
\{0\} & \text{otherwise}.
\end{cases}
\end{align*}

On the other hand, since the second factor $O(1)$ acts trivially
on the submanifold $Y=S^{n-1}$, 
Proposition \ref{prop:identific} implies that
the representation $\varpi^{(j)}_{v,\eps}$ of $\mathrm{Conf}(X;Y)$
on $\mathcal{E}^j(S^{n-1})$ is given by the outer tensor product
representation of $O(n,1) \times O(1)$ as below:
\begin{equation*}
\varpi^{(j)}_{v,\eps} \simeq 
J(\eps\cdot j , v+j)_{\eps\cdot j} \boxtimes \one.
\end{equation*}
Again by Proposition \ref{prop:identific}, we have an isomorphism 
$\varpi^{(i)}_{u,\delta} \simeq I(\delta \cdot i,u+i)_{\delta \cdot i}$
as representations of $G=O(n+1,1)$. Thus we conclude 
\begin{align*}
\mathrm{Hom}_{\mathrm{Conf}(X;Y)}
(\varpi^{(i)}_{u,\delta}, \varpi^{(j)}_{v,\eps})
&\simeq \mathrm{Hom}_{O(n,1) \times O(1)}
\left(I(\delta\cdot i, u+i)_{\delta \cdot i}, J(\eps \cdot j, v+j)_{\eps \cdot j} 
\boxtimes \one\right)\\
&\simeq \mathrm{Hom}_{O(n,1)}\left(I(\delta\cdot i, u+i)_{\delta \cdot i},
J(\eps \cdot j, v+j)_{\eps \cdot j}\right).
\end{align*}
Hence the lemma is proved.
\end{proof}

%%%%%%%%%%%%%%%%%%%%%%%%%%%%%%%%%%%%%%%%%%%%%%%
\subsection{Proof of Theorem \ref{thm:1}}\label{subsec:pfA}

In this section we complete the proof of Theorem \ref{thm:1}.
We shall see that Theorem \ref{thm:1} (conformal geometry) is derived from 
Theorem \ref{thm:1A} (representation theory).
Actually,
we only need principal series representations 
$I(i',\lambda)_{\alpha}$ and $J(j',\nu)_{\beta}$
with $\alpha \equiv i'$ and $\beta \equiv j' \;\mathrm{mod}\;2$
in order to classify
 $\mathrm{Diff}_{G'}\left(\mathcal{E}^i(S^n)_{u,\delta},
\mathcal{E}^j(S^{n-1})_{v,\eps} \right)$,
see Remark \ref{rem:identific}.

Suppose that a symmetry breaking operator 
$D\colon  I(i', \lambda)_{\alpha} \To J(j', \nu)_{\beta}$
with $\alpha \equiv i'$ and $\beta \equiv j' \; \mathrm{mod} \; 2$
is given.
We set
\begin{equation*}
\tilde{i}' := n-i', \quad \tilde{j}':= n-1-j', \quad
b:=j'-i', \quad \tilde{b}:=-b-1.
\end{equation*}

\noindent
We note $\tilde{b}=\tilde{j}'-\tilde{i}'$ and that $b \mapsto \tilde{b}$ defines a permutation of 
the finite set $\{-2,-1,0,1\}$. Then the diagram \eqref{eqn:4SBO}
of the double dualities induces four symmetry breaking operators 
$T\colon \mathcal{E}^{i}(S^n)_{u,\delta} \To \mathcal{E}^{j}(S^{n-1})_{v,\eps}$
with $(T, i,j,u,v,\delta,\eps)$ listed in Table \ref{table:pstogeom} below:

\begin{table}[H]
\caption{Conditions for $(T, i, j, u, v, \delta, \eps)$ in the double dualities}
\begin{center}
\begin{tabular}{c|c|c|c|c|c|c}
$T$ & $\; i \; $ & $j$ & $u$ & $v$ & $\; \delta\;$ & $\;\eps\;$\\[3pt]
\hline
$D$ & $i'$ & $j'=i'+b$ & $\lambda-i'$ & $\nu-j' = \nu -i'-b$ & $0$ & $0$\\[3pt]
$* \circ D \circ *$ & $\tilde{i}'$ & $\tilde{j}' = \tilde{i}' + \tilde{b}$ & $\lambda - \tilde{i}'$
& $\nu - \tilde{j}' = \nu - \tilde{i}' - \tilde{b}$ & $1$ & $1$\\[3pt]
$* \circ D$ & $i'$ & $\tilde{j}' = n-i' + \tilde{b}$ & $\lambda - i'$ & 
$\nu-\tilde{j}' = \nu + i' - n - \tilde{b}$ & $0$ & $1$\\[3pt]
$D \circ *$ & $\tilde{i}'$ & $j' = n - \tilde{i}' + b$ &  $\lambda - \tilde{i}'$ & 
$\nu - j' = \nu + \tilde{i}' - n - b$ & $1$ & $0$
\end{tabular}
\end{center}
\label{default}\label{table:pstogeom}
\end{table}

In the columns in Table \ref{table:pstogeom},
we give formul\ae{} for $v$ in two ways for later purpose.
We note that $b$ or $\tilde{b}$ gives a relationship
between $i$ and $j$.

Let us translate Theorem \ref{thm:1A} on symmetry breaking operators
for principal series representations into those for 
conformal geometry via the isomorphism \eqref{eqn:XYSBO}
by using the dictionary in Table \ref{table:pstogeom}. 
The resulting list is given in Table \ref{table:6tuple}.
We note that among
the six cases in Theorem \ref{thm:1A} (iii),
Case 1 does not contribute 
to Problem \ref{prob:1} for $(X,Y) = (S^n, S^{n-1})$
because 
Proposition \ref{prop:153091} (1) requires 
$\nu-\lambda \equiv \beta -\alpha\; \mathrm{mod}\;2$
for $\mathrm{Diff}_{G'}(I(i',\lambda)_\alpha, J(j',\nu)_\beta)$
not to be zero, whereas
$\nu -\lambda = (n-i'+1)-(n-i')=1$ in Case 1 does not have the same parity with
$\beta - \alpha$ if we take $\alpha\equiv i'$ 
and $\beta \equiv j' (=i'-2) \; \mathrm{mod} \;2$.
Then the remaining five cases in Theorem \ref{thm:1A} (iii) yield
$5 \times 4 = 20$ cases according to the choice of 
$(\alpha, \beta) \in (\Z/2\Z)^2$, which are listed in Table \ref{table:6tuple}.

Let us explain Table \ref{table:6tuple} in more details.
We fix a case among the five cases 1$'$, 2,  3, 4, or 4$'$ in Theorem \ref{thm:1A} (iii),
choose $(\alpha,\beta) \in (\Z/2\Z)^2$, and take a nonzero 
$D\in\mathrm{Hom}_{G'}(I(i',\lambda)_\alpha, J(j',\nu)_\beta)$
which is unique up to scalar multiplication. Here we assume $\alpha \equiv i'$ 
and $\beta \equiv j' \; \mathrm{mod} \;2$, 
which was not necessary in Theorem \ref{thm:1A} (iii).
The operators
$T = D$, $* \circ D$, $D \circ *$, or $* \circ D \circ *$ 
(see \eqref{eqn:4SBO}) are listed
in Table \ref{table:6tuple} according to the choice of 
$(\delta, \eps) \in (\Z/2\Z)^2$, and the operator $T$ gives a symmetry breaking
operator $\mathcal{E}^i(S^n)_{u,\delta} \To \mathcal{E}^j(S^{n-1})_{v,\eps}$
where $(i,j,u,v)$ is determined by the formul\ae{}
$(i', j', \lambda, \nu) \mapsto 
(i,j,u,v)$ given by Table \ref{table:pstogeom} for each fixed
$\delta, \eps \in \Z/2\Z$.
This procedure transforms the classification data
given in Theorem \ref{thm:1A} (iii) with the additional parity condition 
$\alpha \equiv i'$ and $\beta \equiv j'$ into Table \ref{table:6tuple}.

For instance, 
$* \circ (4)  \circ *$ in Table \ref{table:6tuple} means the following:
we begin with parameter $(i', j', \lambda, \nu, \alpha, \beta)$ belonging to Case 4 in 
Theorem \ref{thm:1A} (iii), namely, $j' = i' +1$, $1\leq i'\leq n-2$, 
$(\lambda,\nu) = (i', i'+1)$,
take $D \in \mathrm{Diff}_{O(n,1)}\left(I(i',\lambda)_{\alpha}, J(j',\nu)_{\beta}\right)$
with $\alpha \equiv i'$ and $\beta \equiv j' \; \mathrm{mod} \; 2$,
and then obtain
$* \circ D \circ * \in \mathrm{Diff}_{O(n,1)}\left(
\mathcal{E}^i(S^n)_{u,\delta}, \mathcal{E}^j(S^{n-1})_{v,\eps}\right)$
where $(i,u,\delta,j,v,\eps)$ is determined by
\begin{equation*}
\delta=\eps \equiv 1 \; \mathrm{mod} \; 2, \quad
i=\tilde{i}' (=n-i'), \quad
j = \tilde{j}'(=n-1-j'), \quad
u=\lambda-\tilde{i}',\quad
\text{and} \quad
v=\nu-\tilde{j}'.
\end{equation*}
A short computation shows that
\begin{equation*}
j=i-2, \quad 2\leq i \leq n-1, \quad 
\text{and} \quad
(u,v)=(n-2i, n-2i+3),
\end{equation*}
giving the first row of Table \ref{table:6tuple}.

The order of differential symmetry breaking operators of $D$ is given by
$a:=\nu-\lambda$.
Since the Hodge star operator is of order zero as a differential operator,
the operators
$* \circ D \circ *$, $* \circ D$, and $D \circ *$ 
have the same order $a$.
We listed also the data for $a$ in Table \ref{table:6tuple}.
Collecting these data according to the values of $j$ and $i$,
we get the classification of the $6$-tuples $(i,j,u,v,\delta,\eps)$
for the nonvanishing of $\mathrm{Diff}_{O(n,1)}\left(\mathcal{E}^i(S^n)_{u,\delta},
\mathcal{E}^j(S^{n-1})_{v,\eps}\right)$ as listed in Table \ref{table:6tuple}, or 
exactly the condition in Theorem \ref{thm:1} (iii). Thus
Theorem \ref{thm:1} is proved.

%%%%%%%%%%%%%%%%%%%%%%%%%%%%%%%%%%%%%%%%%%%%%%%
\subsection{Proof of Theorems \ref{thm:2}--\ref{thm:2ii-2}}\label{subsec:pfB}

Theorems \ref{thm:2}, \ref{thm:2ii}, \ref{thm:2ii+1}, and \ref{thm:2ii-2}
are derived from Theorem \ref{thm:Cps} by using Table \ref{table:6tuple} and 
by the formula $\widetilde{\C}^{i,j}_{\lambda,\nu} = \widetilde{\mathcal{D}}^{i\to j}_{u,a}$
with $a = \nu -\lambda$ and $u=\lambda-i$ (see \eqref{eqn:CDdict})
and the duality results (Propositions \ref{prop:dualCi} and \ref{prop:dualCi+}).

We give a proof of Theorem \ref{thm:2} below.
The other three theorems are similarly shown.

\begin{proof}[Proof of Theorem \ref{thm:2}]
There are two rows in Table \ref{table:6tuple} that
deal with the case $j=i-1$. The symmetry breaking operator $T$ in this case
is given as
\begin{equation*}
T=
\begin{cases}
(2) &\text{for $(\delta,\eps) = (0,0)$},\\
*\circ (3) \circ * &\text{for $(\delta, \eps) = (1,1)$},
\end{cases}
\end{equation*}
where $T=(2)$ means that $T$ is proportional to $\widetilde{\C}^{i,i-1}_{\lambda,\nu}$
in the flat picture corresponding to Case 2 of Theorem
\ref{thm:Cps}, and $T=* \circ (3) \circ *$ means that $T$ is proportional to 
$*_{\R^{n-1}} \circ \widetilde{\C}^{n-i,n-i}_{\lambda,\nu} \circ *_{\R^n}$
corresponding to Case 3 of Theorem \ref{thm:Cps}.
It follows from Proposition \ref{prop:dualCi} that the latter equals 
$\pm \widetilde{\C}^{i,i-1}_{\lambda,\nu}$. In both cases, $T$ is proportional to 
$\widetilde{\C}^{i,i-1}_{\lambda,\nu} 
= \widetilde{\mathcal{D}}^{i\to i-1}_{\lambda -i,\nu-\lambda}$
(see \eqref{eqn:Cln}). Thus Theorem \ref{thm:2} is proved.
\end{proof}

%%%%%%%%%%%%%%%%%%%%%%%%%%%%%%%%%%%%%%%%%%%%%%%
{\footnotesize{
\begin{table}[p]
\caption{Relation between Theorem \ref{thm:1} for 
$\displaystyle{\mathrm{Diff}_{O(n,1)}\left(\mathcal{E}^i(S^n)_{u,\delta}, \;
\mathcal{E}^j(S^{n-1})_{v,\eps} \right)}$
and operators in Theorem \ref{thm:Cps} in cases $(1)$-$(4)'$}
\begin{center}
\def\arraystretch{0.8}
\begin{tabular}{c|c|c|c|c|c|c|c}
%%%%%%%%%%%%%%%%%%%%%%%%%%%%%%%%%%%%%%%%%%%%%%%
$\quad j\quad $&$i$&$u$&$v$&$\quad \delta \quad $&$\quad \eps\quad $
&Operators&$a$\\
\hline
%%%%%%%%%%%%%%%%%%%%%%%%%%%%%%%%%%%%%%%%%%%%%%%
& $2 \leq i \leq n-1$ & $n-2i$ & $n-2i+3$ & $1$ & $1$ & $*\circ (4) \circ*$ 
&$1$ \rule[0mm]{0mm}{5mm} \\[5pt]
\cline{2-8} 
$i-2$& & $u \in -n-1-2\N$& & $0$ & $0$ & $(1)'$ & 
\rule[0mm]{0mm}{5mm} \\[5pt]
\cline{3-3} \cline{5-7} 
& \raisebox{11pt}{$i=n$} & 
$u \in -n - 2\N$ & 
\raisebox{11pt}{$3-n$} & 
$1$ & $1$ & $*\circ (4)' \circ *$ 
&\raisebox{11pt}{$1-u-n$} \rule[0mm]{0mm}{5mm} \\[5pt]
\hline
%%%%%%%%%%%%%%%%%%%%%%%%%%%%%%%%%%%%%%%%%%%%%%%
&  & \multicolumn{2}{c|}{$v-u \in 2\N + 2$} & $0$ & $0$
& $(2)$ &\rule[0mm]{0mm}{5mm} \\[5pt]
\cline{3-7} 
\raisebox{11pt}{$i-1$}
& \raisebox{11pt}{$1\leq i \leq n$} 
& \multicolumn{2}{c|}{$v-u \in 2\N + 1$} & $1$ & $1$ & $*\circ (3) \circ *$ 
& \raisebox{11pt}{$v-u-1$}
\rule[0mm]{0mm}{5mm} \\[5pt]
\hline
%%%%%%%%%%%%%%%%%%%%%%%%%%%%%%%%%%%%%%%%%%%%%%%
&  & \multicolumn{2}{c|}{$v-u \in 2\N $} & $0$ & $0$ & $(3)$ &\rule[0mm]{0mm}{5mm} \\[5pt]
\cline{3-7} 
\raisebox{11pt}{$i$}
& \raisebox{11pt}{$0\leq i \leq n-1$} 
& \multicolumn{2}{c|}{$v-u \in 2\N + 1$} & $1$ & $1$ & $*\circ (2) \circ *$ 
& \raisebox{11pt}{$v-u$}
\rule[0mm]{0mm}{5mm} \\[5pt]
\hline
%%%%%%%%%%%%%%%%%%%%%%%%%%%%%%%%%%%%%%%%%%%%%%%
& $1 \leq i \leq n-2$& $0$& & $0$ & $0$ & $(4)$ & $1$\rule[0mm]{0mm}{5mm} \\[5pt]
\cline{2-3} \cline{5-8} 
$i+1$& &$u\in-2\N$ & $0$
 & $0$ & $0$ & $(4)'$ & 
\rule[0mm]{0mm}{5mm} \\[5pt]
\cline{3-3}\cline{5-7} 
& \raisebox{11pt}{$i=0$} & 
$u \in -1-2\N$ & & 
$1$ & $1$ & $*\circ (1)' \circ*$ 
&\raisebox{11pt}{$1-u$}\rule[0mm]{0mm}{5mm} \\[5pt]
\hline
\hline
%%%%%%%%%%%%%%%%%%%%%%%%%%%%%%%%%%%%%%%%%%%%%%%
%%%%%%%%%%%%%%%%%%%%%%%%%%%%%%%%%%%%%%%%%%%%%%%
& $1 \leq i \leq n-2$ & 
$0$ & $2i-n+3$ & $0$ & $1$ & $*\circ (4)$ &\rule[0mm]{0mm}{5mm} \\[5pt]
\cline{2-7} 
$n-i-2$& & $u \in -2\N$ & & $0$ & $1$ & $*\circ (4)'$ & 
\rule[0mm]{0mm}{5mm} \\[5pt]
\cline{3-3}\cline{5-7} 
& \raisebox{11pt}{$i=0$} & 
$u \in -1-2\N$ &\raisebox{11pt}{$3-n$} & 
$1$ & $0$ & $(1)' \circ*$ 
&\raisebox{11pt}{$1-u$}\rule[0mm]{0mm}{5mm} \\[5pt]
\hline
%%%%%%%%%%%%%%%%%%%%%%%%%%%%%%%%%%%%%%%%%%%%%%%
&  & \multicolumn{2}{c|}{$\quad v-u \in (2i-n+1) +2\N\quad$} & $0$ & $1$ & 
$*\circ (3)$ &\rule[0mm]{0mm}{5mm} \\[5pt]
\cline{3-7} 
\raisebox{11pt}{$n-i-1$}
& \raisebox{11pt}{$0\leq i \leq n-1$} 
& \multicolumn{2}{c|}{$v-u \in (2i-n+2) + 2\N$} & $1$ & $0$ & $(2)\circ*$ 
& \raisebox{11pt}{$v-u+n-2i-1$}
\rule[0mm]{0mm}{5mm} \\[5pt]
\hline
%%%%%%%%%%%%%%%%%%%%%%%%%%%%%%%%%%%%%%%%%%%%%%%
&  & \multicolumn{2}{c|}{$v-u \in (2i-n+1) +  2\N $} & $0$ & $1$ & $*\circ (2)$ &\rule[0mm]{0mm}{5mm} \\[5pt]
\cline{3-7} 
\raisebox{11pt}{$n-i$}
& \raisebox{11pt}{$1\leq i \leq n$} 
& \multicolumn{2}{c|}{$v-u \in (2i-n) + 2\N$} & $1$ & $0$ & $(3) \circ *$ 
& \raisebox{11pt}{$v-u+n-2i$}
\rule[0mm]{0mm}{5mm} \\[5pt]
\hline
%%%%%%%%%%%%%%%%%%%%%%%%%%%%%%%%%%%%%%%%%%%%%%%
& $2 \leq i \leq n-1$ & 
$n-2i$ & & $1$ & $0$ & $(4) \circ *$ 
&$1$\rule[0mm]{0mm}{5mm} \\[5pt]
\cline{2-3} \cline{5-8} 
$n-i+1$& &$u\in -n-1-2\N$ & $0$
 & $0$ & $1$ & $* \circ (1)'$ & \rule[0mm]{0mm}{5mm} \\[5pt]
\cline{3-3}\cline{5-7} 
& \raisebox{11pt}{$i=n$} & 
$u \in -n-2\N$ &  & 
$1$ & $0$ & $(4)' \circ *$ 
&\raisebox{11pt}{$1-u-n$}\rule[0mm]{0mm}{5mm} \\[5pt]
\hline
\end{tabular}
\end{center}
\label{table:6tuple}
\end{table}
}}

%%%%%%%%%%%%%%%%%%%%%%%%%%%%%%%%%%%%%%%%%%%%%%%
\subsection{Change of coordinates in symmetry breaking operators}\label{subsec:RtoS}

So far we have discussed explicit formul\ae{} of symmetry breaking operators in the 
flat coordinates. This section explains how to compute explicit symmetry breaking
operators in the coordinates of $(X,Y)=(S^n,S^{n-1})$ from the formul\ae{} that
we found in the flat coordinates of $(\R^n,\R^{n-1})$.

We recall from \eqref{eqn:stereo} and \eqref{eqn:RnSn} that the 
\index{B}{stereographic projection}
stereographic projection
and its inverse are given, respectively by
\index{A}{1iota@$\iota$, conformal compactification}
\index{A}{p@$p$, stereographic projection}
\index{A}{1ks1@$[\xi^-]$, north pole in $S^n$}
\begin{alignat*}{2}
&p\colon S^n \setminus \{[\xi^-]\} \To \R^n,
\quad &&\omega=\trans (\omega_0,\ldots, \omega_n)\mapsto 
\frac{1}{1+\omega_0} \trans (\omega_1, \ldots, \omega_n),\\
&\iota\colon \R^n \To S^n, 
\quad &&x=\trans (x_1, \ldots, x_n) \mapsto 
\frac{1}{1+Q_n(x)}\trans (1-Q_n(x), 2x_1, \ldots, 2x_n),
\end{alignat*}
where $\xi^- = \trans (-1, 0,\ldots, 0)$.
As is well-known, $p$ and $\iota$ are conformal maps with the following conformal factors
(see \cite[Lem.\ 3.3]{KO1} for example):
\index{A}{Qn@$Q_n(x)$}
\begin{alignat*}{2}
\iota^*g_{S^n, \iota(x)}&=\left(\frac{2}{1+Q_n(x)}\right)^2g_{\R^n,x} \quad
&&\text{for $x\in\R^n$},\\
p^*g_{\R^n,p(\omega)}&=\left(\frac{1}{1+\omega_0}\right)^2 g_{S^n, \omega} \quad 
&&\text{for $\omega \in S^n \setminus \{[\xi^-]\}$}.
\end{alignat*}
In turn, the 
\index{B}{twisted pull-back}
twisted pull-back of differential forms defined in \eqref{eqn:twistpb} amounts to
\begin{alignat*}{2}
&\left(p^{(i)}_{u,\delta}\right)^*\colon 
\mathcal{E}^i(\R^n) \To \mathcal{E}^i(S^n \setminus \{[\xi^-]\}),
\quad &&\alpha \mapsto (1+\omega_0)^{-u} p^*\alpha,\\
&\left(\iota^{(i)}_{u,\delta}\right)^*\colon  \mathcal{E}^i(S^n)\To \mathcal{E}^i(\R^n),
\quad &&\beta \mapsto \left(\frac{1+Q_n(x)}{2} \right)^{-u} \iota^*\beta,
\end{alignat*}
for $u\in \C$, $\delta \in \Z/2\Z$, and $0\leq i \leq n$.

Then the following proposition gives a change of coordinates in
differential symmetry breaking operators.

\begin{prop}\label{prop:20160410}
Suppose a 6-tuple $(i,j,u,v,\delta, \eps)$ belongs to 
Cases $\mathrm{(I)}$-$(\mathrm{IV}')$ or Cases $(*\mathrm{I})$-$(*\mathrm{IV}')$
in Theorem \ref{thm:1}, and 
$D=\widetilde{\mathcal{D}}^{i \to j}_{u,v+j-u-i}$ 
(or $*_{\R^{n-1}} \circ \widetilde{\mathcal{D}}^{i\to n-1-j}_{u, v+j-u-i}$, respectively)
is a differential operator $\mathcal{E}^i(\R^n) \To \mathcal{E}^j(\R^{n-1})$
defined as in \eqref{eqn:reDii1}-\eqref{eqn:160462}.
Then the compositions 
\begin{align*}
\left(p^{(j)}_{v,\eps}\right)^* \circ D \circ \left(\iota^{(i)}_{u,\delta}\right)^*\colon &
\mathcal{E}^i(S^n) \To \mathcal{E}^j(S^{n-1}),\\
*_{S^{n-1}} \circ \left(p^{(n-j-1)}_{v-n+2j+1, \eps + 1}\right)^*
\circ D \circ \left(\iota^{(i)}_{u,\delta}\right)^*\colon &
\mathcal{E}^i(S^n)\To \mathcal{E}^j(S^{n-1}),
\end{align*}
respectively, are differential symmetry breaking operators
from $\left(\varpi^{(i)}_{u,\delta}, \mathcal{E}^i(S^n)\right)$ to 
$\left(\varpi^{(j)}_{v,\eps}, \mathcal{E}^j(S^{n-1})\right)$ in the coordinates of 
$(S^n, S^{n-1})$.
\end{prop}

In Proposition \ref{prop:20160410},
$\iota\colon \R^n \To S^n$ denotes the conformal compactification in the $n$-dimensional
setting as before,
 but $p\colon  S^{n-1} \setminus \{[\xi^-]\} \To \R^{n-1}$
is the stereographic projection in the $(n-1)$-dimensional setting.

The proof of Proposition \ref{prop:20160410} in Cases $\mathrm{(I)}$-$(\mathrm{IV}')$
is clear. For Cases $(*\mathrm{I})$-$(*\mathrm{IV}')$, 
we use Lemma \ref{lem:20160317}:
\begin{equation}\label{eqn:starSR}
*_{S^{n-1}} \circ \left(p^{(n-j-1)}_{v-n+2j+1, \eps+1}\right)^*
=\left(p^{(j)}_{v,\eps}\right)^* \circ *_{\R^{n-1}} 
\quad \text{on $\mathcal{E}^{n-j-1}(\R^{n-1})$}.
\end{equation}

We end this section by giving 
some few examples of $\left(p^{(j)}_{v,\eps}\right)^*\circ D \circ 
\left(\iota^{(i)}_{u,\delta}\right)^*$ from Lemmas \ref{lem:dstarconf},
\ref{lem:1604116}, and \ref{lem:1531111}.
The last one is related to the factorization identity, which 
we see in Theorem \ref{thm:160422} (4).

\begin{table}[H]
\begin{center}
\begin{tabular}{c|c|c|c|c|l|c}
$j$ & $u$ & $\delta$ & $v$ & $\eps$ & 
\hspace{55pt}$D$ & 
$\left(p^{(j)}_{v,\eps}\right)^* \circ D \circ \left(\iota^{(i)}_{u,\delta}\right)^*$\\
\hline
$i-1$ & $u$ & $1$ & $u+1$ & $1$ & 
\rule{0pt}{3ex}
$\widetilde{\mathcal{D}}^{i\to i-1}_{u,0} = \restn\circ \iotan$&
$\mathrm{Rest}_{S^{n-1}}\circ \iota_{N_{S^{n-1}}(S^{n})}$\\
$i$ & $u$ & $0$ & $u$ & $0$ & 
\rule{0pt}{3ex}
\hspace{7pt}
$\widetilde{\mathcal{D}}^{i\to i}_{u,0} = \restn$ &
$\mathrm{Rest}_{S^{n-1}}$\\
$i+1$ & $0$ & $0$ & $0$ & $0$ &
\rule{0pt}{3ex}$
\hspace{4pt}
\widetilde{\mathcal{D}}^{i\to i+1}_{0,0} = \restn\circ d_{\R^n}$ &
$\mathrm{Rest}_{S^{n-1}}\circ d_{S^n}$\\
$i-1$ & $n-2i$ & $0$ & $n-2i+2$ & $0$ &
\rule{0pt}{3ex}
\hspace{-3pt}
$\widetilde{\mathcal{D}}^{i\to i-1}_{n-2i,1} = -\restn\circ d^*_{\R^n}$ &
$-\mathrm{Rest}_{S^{n-1}}\circ d^*_{S^n}$\\
$i-2$ & $n-2i$ & $1$ & $n-2i+3$ & $1$ &
\rule{0pt}{3ex}
\hspace{-3pt}
$\widetilde{\mathcal{D}}^{i\to i-2}_{n-2i,1} = \restn\circ \iotan \circ d^*_{\R^n}$ &
$\mathrm{Rest}_{S^{n-1}}\circ \iota_{N_{S^{n-1}}(S^n)}\circ d^*_{S^n}$\\
\end{tabular}
\end{center}
\end{table}

\newpage
%%%%%%%%%%%%%%%%%%%%%%%%%%%%%%%%%%%%%%%%%%%%%%%
\section{Intertwining operators}\label{sec:intertwiner}

In this chapter we determine all conformally covariant differential operators
between the spaces of differential forms on the standard Riemannian sphere $S^n$,
and thus solve Problems \ref{prob:1} and \ref{prob:2} in the case
where $X=Y=S^n$. We note that the case $X=Y$ (and $G=G'$) is much 
easier than the case $X \supsetneqq Y$ which we have discussed in 
Chapters \ref{sec:7}-\ref{sec:solAB}.

We have seen in Proposition \ref{prop:dstarconf} that
the differential $d\colon \mathcal{E}^i(X)\To \mathcal{E}^{i+1}(X)$
(the codifferential $d^*\colon \mathcal{E}^{i+1}(X)\To \mathcal{E}^i(X)$,
respectively) intertwines two representations $\varpi^{(i)}_{u,\delta}$ and
$\varpi^{(i+1)}_{v,\eps}$ (see \eqref{eqn:varpi}) of the conformal group of
any oriented pseudo-Riemannian manifold $X$ for appropriate twisting
parameters $(u,\delta)$ and $(v,\eps)$, respectively.
Conversely, our classification (Theorem \ref{thm:GGC}) shows
that $d$ is the unique differential operator from $\mathcal E^i(S^n)$
to $\mathcal E^{i+1}(S^n)$ (up to scalar multiplication) 
that commutes with conformal diffeomorphisms of $S^n$.
Similarly, we shall prove that the codifferential $d^*$ is characterized as 
the unique differential operator (up to scalar multiplication) 
$\mathcal E^{i+1}(S^n)\To\mathcal E^{i}(S^n)$ that intertwines 
twisted representations of the conformal group of $S^n$.
On the other hand, we find countably many bases of 
conformally covariant differential operators of higher order
that map $\mathcal E^i(S^n)$ into $\mathcal{E}^j(S^n)$ when $j=i$ 
(see Theorem \ref{thm:GGC}). 

One could give a proof of those results by combining the algebraic results
on the classification of homomorphisms between generalized Verma 
modules by Boe--Collingwood \cite{BC86} with the geometric construction of differential
operators by Branson \cite{Branson}, although the existing literature treats
only connected groups and one needs some extra work to discuss 
disconnected groups. Alternatively, we shall give a self-contained proof
of these results from scratch by the matrix-valued F-method.
We know we could shorten a significant part of the proof 
(\emph{e.g.}\ the relationship between $\lambda$ and $\mu$) if we used some 
elementary results on Verma modules. Instead we provide an alternative
approach, as this baby example might be illustrative about the use of the F-method
in a more general matrix-valued setting.

%%%%%%%%%%%%%%%%%%%%%%%%%%%%%%%%%%%%%%%%%%%%%%%
\subsection{Classification of differential intertwining operators between forms on $S^n$}
\label{subsec:Branson}

Let $0\leq i\leq n$. For $\ell\in\N_+$, define a differential operator
(\index{B}{Branson's operator|textbf}\emph{Branson's operator})
\begin{equation*}
\mathcal T_{2\ell}^{(i)}\colon \mathcal E^i(\R^n)\To\mathcal E^i(\R^n)
\end{equation*}
by
\index{A}{T2ell@$\mathcal T_{2\ell}^{(i)}$, Branson's operator|textbf}
\begin{eqnarray}\label{eqn:T2li}
\mathcal T_{2\ell}^{(i)}&:=&\left(
\left(\frac n2-i-\ell\right)
d_{\R^n}d^*_{\R^n}+
\left(\frac n2-i+\ell\right)
d^*_{\R^n}d_{\R^n}\right))\Delta^{\ell-1}_{\R^n}\\
&=&\left(-2\ell\,d_{\R^n}d^*_{\R^n}-\left(\frac12n+\ell-i\right)\Delta_{\R^n}\right)\Delta_{\R^n}^{\ell-1}\nonumber.
\end{eqnarray}
Then the following theorem is the main result of this chapter.

\begin{thm}\label{thm:GGC}
Let $n\geq 2$.
\begin{enumerate}
\item Let $0 \leq i \leq n$ and $\ell \in \N_+$. 
We set 
\begin{equation*}
u:=\frac n2-i-\ell,\quad v:=\frac n2-i+\ell.
\end{equation*}
Then the differential operator $\mathcal{T}^{(i)}_{2\ell}$ 
extends to the conformal compactification $S^n$ of $\R^n$, and
induces a nonzero
$O(n+1,1)$-homomorphism 
\index{A}{Ei@$\mathcal{E}^i(S^n)_{u,\delta}$}
$\mathcal{E}^i(S^n)_{u,\delta} \To \mathcal{E}^i(S^n)_{v,\delta}$
for $\delta \in \Z/2\Z$, to be denoted simply by the same letter $\mathcal{T}^{(i)}_{2\ell}$.

\item Let $0\leq i,j\leq n, (u,v)\in\C^2$ and $(\delta,\eps)\in\left(\Z/2\Z\right)^2$. 
Then the space
of conformally covariant differential operators,
$\mathrm{Diff}_{O(n+1,1)}(\mathcal E^i(S^n)_{u,\delta}, \mathcal E^j(S^n)_{v,\eps})$,
is at most one-dimensional.
More precisely, this space is nonzero in the eight cases listed below. The corresponding
generators are given as follows:

\begin{itemize}

\item[]\emph{Case a}. $0\leq i \leq n$, $u \in \C$, $\delta \in \Z/2\Z$.
\begin{equation*}
\mathrm{id}\colon  \mathcal{E}^i(S^n)_{u,\delta} \To \mathcal{E}^i(S^n)_{u,\delta}.
\end{equation*}

\item[]\emph{Case b}. $0\leq i \leq n-1$, $\delta \in \Z/2\Z$.
\begin{equation*}
d\colon  \mathcal{E}^i(S^n)_{0,\delta} \To \mathcal{E}^{i+1}(S^n)_{0,\delta}.
\end{equation*}

\item[]\emph{Case c}. $1\leq i \leq n$, $\delta \in \Z/2\Z$.
\begin{equation*}
d^*\colon  \mathcal{E}^i(S^n)_{n-2i,\delta} \To \mathcal{E}^{i-1}(S^n)_{n-2i+2,\delta}.
\end{equation*}

\item[] \emph{Case d}. $0 \leq i\leq n$, $\ell \in \N_+$, $\delta \in \Z/2\Z$.
\begin{equation*}
\mathcal{T}^{(i)}_{2\ell} \colon \mathcal{E}^i(S^n)_{\frac{n}{2}-\ell - i, \delta}
\To \mathcal{E}^i(S^n)_{\frac{n}{2}+\ell-i,\delta}.
\end{equation*}

\item[] \emph{Case $*$a}. $0\leq i \leq n$, $u \in \C$ and $\delta \in \Z/2\Z$.
\begin{equation*}
*\colon  \mathcal{E}^i(S^n)_{u,\delta} \To \mathcal{E}^{n-i}(S^n)_{u-n+2i,\delta+1}.
\end{equation*}

\item[]\emph{Case $*$b}. $0\leq i \leq n-1$, $\delta \in \Z/2\Z$.
\begin{equation*}
*\circ d\colon  \mathcal{E}^i(S^n)_{0,\delta} \To \mathcal{E}^{n-i-1}(S^n)_{2i+2-n,\delta+1}.
\end{equation*}

\item[]\emph{Case $*$c}. $1\leq i \leq n$, $\delta \in \Z/2\Z$.
\begin{equation*}
d\circ*\colon \mathcal{E}^i(S^n)_{n-2i,\delta} \To \mathcal{E}^{n-i+1}(S^n)_{0,\delta+1}.
\end{equation*}

\item[] \emph{Case $*$d}. $0\leq i \leq n$, $\ell \in \N_+$, $\delta \in \Z/2\Z$.
\begin{equation*}
*\circ \mathcal{T}^{(i)}_{2\ell}\colon \mathcal{E}^i(S^n)_{\frac{n}{2}-\ell-i,\delta}
\To \mathcal{E}^{n-i}(S^n)_{-\frac{n}{2}+\ell+i,\delta+1}.
\end{equation*}
\end{itemize}
\end{enumerate}
\end{thm}

We shall give a proof of Theorem \ref{thm:GGC} in Section \ref{subsec:ThmGGC}.

%%%%%%%%%%%%%%%%%%%%%%%%%%%%%%%%%%%%%%%%%%%%%%
\subsection{Differential symmetry breaking operators between principal series representations}
\label{subsec:psBr}

We reformulate the problem in terms of representation theory. 
Let $I(i,\lambda)_\alpha$
be the principal series representation
of the Lorentz group $G=O(n+1,1)$.
We determine differential symmetry breaking operators between 
$I(i,\lambda)_\alpha$s as follows:

\begin{thm}\label{thm:psGGB}
Let $n\geq 2$, $0 \leq i, j\leq n$, $(\lambda, \nu) \in \C^2$ and $(\alpha, \beta) \in (\Z/2\Z)^2$.

\begin{enumerate}
\item The following three conditions on the 6-tuple $(i,j,\lambda,\nu, \alpha, \beta)$
are equivalent:
\index{A}{Iilambda@$I(i,\lambda)_\alpha$, principal series of $O(n+1,1)$}
\begin{itemize}
\item[(i)] $\mathrm{Diff}_{O(n+1, 1)}(I(i,\lambda)_\alpha, I(j, \nu)_\beta) \neq \{0\}$.
\item[(ii)] $\mathrm{dim}_{\C}\mathrm{Diff}_{O(n+1, 1)}(I(i,\lambda)_\alpha, I(j, \nu)_\beta)=1$.
\item[(iii)] The 6-tuple belongs to one of the following:
\begin{enumerate}
\item[] \emph{Case 1}. $j = i+1$, $(\lambda,\nu) = (i,i+1)$, 
and $\alpha \equiv \beta + 1\; \mathrm{mod}\, 2$;
\item[] \emph{Case 2}. $j = i-1$, 
$(\lambda, \nu) = (n-i,n-i+1)$, and $\alpha \equiv \beta + 1\; \mathrm{mod}\, 2$;
\item[] \emph{Case 3}. $j=i$, $\lambda + \nu = n$, 
$\nu-\lambda \in 2\N_+$, and $\alpha \equiv \beta\; \mathrm{mod}\, 2$;
\item[] \emph{Case 4}.
$j=i$, $\lambda=\nu$, and $\alpha \equiv \beta\; \mathrm{mod}\, 2$.
\end{enumerate}
\end{itemize}

\item Any differential $G$-intertwining operators from $I(i,\lambda)_\alpha$ to 
$I(j,\nu)_\beta$ are proportional to the following differential operators
$\mathcal{E}^i(\R^n)\To \mathcal{E}^j(\R^n)$ in the flat picture:
\begin{itemize}
\item[] \emph{Case 1.} $d$;
\item[] \emph{Case 2.} $d^*$;
\item[] \emph{Case 3.} $\mathcal{T}^{(i)}_{\nu-\lambda}=
\left(\frac{1}{2}(n-2i-\nu+\lambda)d_{\R^n}d^*_{\R^n} + 
\frac{1}{2}(n-2i+\nu-\lambda)d^*_{\R^n}d_{\R^n}\right)
\Delta_{\R^n}^{\frac{1}{2}(\nu-\lambda)-1}$;
\item[] \emph{Case 4.} $\mathrm{id}$.
\end{itemize}
\end{enumerate}
\end{thm}

For the proof we apply the F-method in the special case where
$G=G'=O(n+1,1)$.
For $0 \leq i \leq n$, $\alpha \in \Z/2\Z$, and $\lambda \in \C$,
we denote by 
\index{A}{1sigma-lambda-alpha@$\sigma^{(i)}_{\lambda, \alpha}$, representation of $P$ on $\Exterior^i(\C^n)$}
$\sigma^{(i)}_{\lambda, \alpha}$ the outer tensor product
representation $\Exterior^i(\C^n) \boxtimes (-1)^\alpha \boxtimes \C_\lambda$
of the Levi subgroup $L = MA \simeq O(n) \times O(1) \times \R$ on the $i$-th
exterior tensor space $\Exterior^i(\C^n)$. We recall that the principal 
series representation $I(i, \lambda)_\alpha$ of $G=O(n+1,1)$ is the 
unnormalized induction from the representation $\sigma^{(i)}_{\lambda, \alpha}$
of $P$ with trivial action by $N_+$.

By Fact \ref{fact:Fmethod} we have a bijection: 
\begin{equation}\label{eqn:GGFmethod}
{\qquad\qquad}\mathrm{Diff}_{O(n+1,1)}
(I(i,\lambda)_\alpha, I(j,\nu)_\beta)
\stackrel{\sim}{\to}
{Sol}\left(\mathfrak n_+; \;
\sigma^{(i)}_{\lambda,\alpha},\; \sigma^{(j)}_{\nu,\beta}\right),
\end{equation}
where the right-hand side is given by Lemma \ref{lem:C0enough} as
\begin{equation*}
\left\{\psi\in \mathrm{Hom}_L\left(\sigma^{(i)}_{\lambda,\alpha},
\; \sigma^{(j)}_{\nu,\beta} \otimes \mathrm{Pol}[\zeta_1,\ldots, \zeta_n]\right)
:\widehat{d\pi_{(i,\lambda)^*}}(N_1^+)\psi=0\right\}.
\end{equation*}
We recall from 
\eqref{eqn:F0}--\eqref{eqn:F1p}
\index{A}{H1ijk@$H_{i\to j}^{(k)}$}
\index{A}{H1wijk@$\widetilde H_{i\to j}^{(k)}
\colon\Exterior^i(\C^N)\To \Exterior^j(\C^N)\otimes \mathcal{H}^k(\C^N)$}
and \eqref{eqn:H2tilde} that
 $H^{(k)}_{i\to j}$ and $\widetilde{H}^{(2)}_{i\to i}$ are
$\mathrm{Hom}_{G}(\Exterior^i(\C^n),\Exterior^j(\C^n))$-valued
harmonic polynomials.
Then the following proposition holds:

\begin{prop}\label{prop:SolGG}
Suppose $n \geq 2$.
Let the 6-tuple $(i,j,\lambda,\nu,\alpha,\beta)$ be as in 
Cases 1-4 of Theorem \ref{thm:psGGB}. 
Then,
\begin{eqnarray*}
&&Sol\left(\mathfrak n_+;
\sigma^{(i)}_{\lambda,\alpha}, \sigma^{(j)}_{\nu,\beta}\right)\\     
&=&
\left\{
\begin{matrix*}[l]
\C H_{i \to i+1}^{(1)} & \mathrm{Case\; 1},\\
\C H_{i \to i-1}^{(1)} & \mathrm{Case\; 2},\\
\C\left(-\frac12(n+\nu-\lambda)\left(1-\frac{2i}n\right) Q_n^{\frac{\nu-\lambda}2}
H_{i \to i}^{(0)}+(\nu-\lambda)  Q_n^{\frac{\nu-\lambda}2-1}\widetilde H_{i \to i}^{(2)}\right) & \mathrm{Case\; 3},\\
\C H^{(0)}_{i\to i} & \mathrm{Case\; 4},\\
0&\mathrm{otherwise}.
\end{matrix*}
\right.
\end{eqnarray*}
\end{prop}

We note that $\widetilde{H}^{(2)}_{i\to i}=0$ for $i=0$, $n$.

We shall give a proof of Proposition \ref{prop:SolGG} in Sections \ref{subsec:GGalg} to \ref{subsec:GGB-}. Admitting Proposition \ref{prop:SolGG}, we first complete the proof
of Theorem \ref{thm:psGGB}.

\begin{proof}[Proof of Theorem \ref{thm:psGGB}.]
The first statement is a direct consequence of Proposition \ref{prop:SolGG} and the 
bijection \eqref{eqn:GGFmethod}. To see the second statement, we 
recall from Fact \ref{fact:Fmethod} that the bijection \eqref{eqn:GGFmethod} is given by the symbol map if we use the flat coordinates.
Since 
\index{A}{Symb}
$\mathrm{Symb}\left(d_{\R^n}\right) =H^{(1)}_{i\to i+1}$, 
$\mathrm{Symb}\left(d^*_{\R^n}\right)=H^{(1)}_{i\to i-1}$,
and
$\mathrm{Symb}\left(\mathrm{id}\right) = H^{(0)}_{i\to i}$ by
Lemma \ref{lem:symbol},
the second statement in Cases 1, 2, and 4 is verified.

In Case 3, we need a supplementary computation. Indeed,
we apply Lemma \ref{lem:symbol} (3) and (4) to get the formula
\begin{equation*}
\mathrm{Symb}\left(\left(-A+\left(\frac in-1\right)B\right)  
d_{\R^n}d^*_{\R^n}+\left(-A+\frac inB\right) d^*_{\R^n}d_{\R^n}\right)
=AQ_n H^{(0)}_{i \to i}+B\widetilde H^{(2)}_{i \to i}.
\end{equation*}
By putting $A=-\frac{1}{2}(n+\nu-\lambda)\left(1-\frac{2i}{n}\right)$ and
$B=\nu-\lambda$, we have
\begin{equation*}
\mathrm{Symb}\left(u d_{\R^n}d^*_{\R^n}+v d^*_{\R^n}d_{\R^n}\right)
=-\frac12(n+\nu-\lambda)\left(1-\frac{2i}n\right)Q_nH^{(0)}_{i \to i}+(\nu-\lambda)\widetilde H^{(2)}_{i \to i},
\end{equation*}
where $u=\frac{1}{2}(n-2i-\nu+\lambda)$ and $v = \frac{1}{2}(n-2i+\nu-\lambda)$.
Thus the second statement in Case 3 is also verified.
\end{proof}

%%%%%%%%%%%%%%%%%%%%%%%%%%%%%%%%%%%%%%%%%%%%%%%
\subsection{Description of 
$\mathrm{Hom}_L(V,W\otimes\mathrm{Pol}(\mathfrak n_+))$}\label{subsec:GGalg}
In order to prove Proposition \ref{prop:SolGG}, we begin with an elementary algebraic lemma.

\begin{lem}\label{lem:103}

\begin{eqnarray*}
&&
\mathrm{Hom}_L\left(
\sigma^{(i)}_{\lambda,\alpha}, \; \sigma^{(j)}_{\nu,\beta} 
\otimes \mathrm{Pol}[\zeta_1, \ldots, \zeta_n]
\right)
\\
&&=
\left\{
\begin{array}{llll}
\C Q_n^{\frac{\nu-\lambda-1}2}H_{i \to i+1}^{(1)} &\mathrm{if}& j=i+1,
\; \nu-\lambda\in 2\N+1,
&\beta \equiv \alpha+1\,\mathrm{mod}\,2,\\
\C Q_n^{\frac{\nu-\lambda-1}2}H_{i \to i-1}^{(1)} &\mathrm{if}& j=i-1,
\; \nu-\lambda\in 2\N+1,
&\beta \equiv \alpha+1\,\mathrm{mod}\,2,\\
\C Q_n^{\frac{\nu-\lambda}2}H_{i \to i}^{(0)}+\C Q_n^{\frac{\nu-\lambda}2-1}
\widetilde{H}_{i \to i}^{(2)} &\mathrm{if}& j=i \in \{1,\ldots, n-1\},
\; \nu-\lambda\in 2\N_+,
&\beta \equiv \alpha\,\mathrm{mod}\,2,\\
\C Q_n^{\frac{\nu-\lambda}2}H_{0 \to 0}^{(0)} 
&\mathrm{if}& j=i\in\{0,n\},
\; \nu-\lambda\in 2\N_+,
&\beta \equiv \alpha\,\mathrm{mod}\,2,\\
\C H_{i \to i}^{(0)} &\mathrm{if}& j=i,
\; \nu=\lambda,
&\beta \equiv \alpha\,\mathrm{mod}\,2,\\
0&&\mathrm{otherwise}.&
\end{array}
\right.
\end{eqnarray*}
\end{lem}

\begin{proof} 
We may restrict ourselves to homogeneous polynomials because $L$ preserves
the degree of homogeneity in $\mathrm{Pol}(\mathfrak{n}_+)$.
We consider the action of the second and third factors of 
$L \simeq O(n)\times O(1) \times \R$. 
\index{A}{H1zero@$H_0$, generator of $\mathfrak{a}$}
Since $e^{tH_0} \in A$
and $-1 \in O(1)$ act on $\mathfrak{n}_+ \simeq \C^n$
as the scalars $e^t$ and $-1$, respectively, we conclude
\begin{equation*}
\mathrm{Hom}_{O(1) \times A}\left(\sigma^{(i)}_{\lambda,\alpha},\;
\sigma^{(j)}_{\nu,\beta} \otimes \mathrm{Pol}^a(\mathfrak{n}_+)\right)
\neq \{0\}
\end{equation*}
if and only if
\begin{equation*}
\nu=\lambda + a \quad \mathrm{and} \quad \beta \equiv \alpha + a\; \mathrm{mod} \;2.
\end{equation*}
In this case, we have
\begin{align*}
\mathrm{Hom}_L\left(
\sigma^{(i)}_{\lambda,\alpha}, \sigma^{(j)}_{\nu,\beta} \otimes
\mathrm{Pol}^a[\zeta_1,\ldots, \zeta_n]\right)
&\simeq 
\mathrm{Hom}_{O(n)}\left(\Exterior^i(\C^n),
\Exterior^j(\C^n)\otimes \mathrm{Pol}^a[\zeta_1,\ldots, \zeta_n]\right)\\
&\simeq \bigoplus_{\substack{0\leq k \leq a\\k \equiv a \;\mathrm{mod}\;2}}
\mathrm{Hom}_{O(n)}\left(\Exterior^i(\C^n) ,\Exterior^j(\C^n)\otimes 
\mathcal{H}^k(\C^n)\right)
\end{align*}
because we have an $O(n)$-isomorphism:
\index{A}{HkCN@$\mathcal H^k(\C^N)$, harmonic polynomials}
\begin{equation*}
\mathrm{Pol}[\zeta_1,\cdots,\zeta_n]\simeq \mathrm{Pol}[Q_n(\zeta)]\otimes\left(\bigoplus_{k=0}^\infty
\mathcal H^k(\C^n)\right).
\end{equation*}
Now Lemma follows from
Lemma \ref{lem:Altharmonic} and Proposition \ref{prop:Fij}.
\end{proof}

In order to prove Proposition \ref{prop:SolGG}, it is sufficient to find $\psi\; (\neq0)$ that belongs to the right-hand side of the identity in Lemma \ref{lem:103} satisfying
\index{A}{dpi6ihat@$\widehat{d\pi_{(i,\lambda)^*}}$}
$\widehat{d\pi_{(i,\lambda)^*}}(N_1^+)\psi=0$.
We shall carry out this computation in the next sections.

%%%%%%%%%%%%%%%%%%%%%%%%%%%%%%%%%%%%%%%%%%%%%%%
\subsection{Solving the F-system when $j=i+1$}\label{subsec:GGd}
\index{B}{F-system}

This section treats the case $j=i+1$. We shall use $I,I'$ to denote elements in $\mathcal I_{n,i}$ and $\widetilde{I}$ for those
in $\mathcal I_{n,i+1}$. This is slightly different from the convention
for index sets adopted in the previous chapters.

According to Lemma \ref{lem:103}, 
we may assume $\nu-\lambda=2\ell+1$ for some $\ell\in\N$ and 
$\beta \equiv \alpha + 1$ mod 2. We set $\psi=Q_n^\ell H^{(1)}_{i\to i+1}$. 
With respect to the standard basis $\{e_I:I\in\mathcal I_{n,i}\}$
of $\Exterior^i(\C^n)$ and $\{e_{\widetilde{I}}:\widetilde{I}\in\mathcal I_{n,i+1}\}$ of $\Exterior^{i+1}(\C^n)$, 
we set, as in Section \ref{subsec:MIJF},
\begin{eqnarray*}
\psi_{I\widetilde{I}}(\zeta)&:=&Q_n(\zeta)^\ell\left( H^{(1)}_{i\to i+1}\right)_{I\widetilde{I}}(\zeta),\\
M_{I\widetilde{I}}&:=&\left\langle\widehat{d\pi_{(i,\lambda)^*}}(N_1^+)\psi(e_I),e_{\widetilde{I}}^\vee\right\rangle.
\end{eqnarray*} 
Then the proof of Proposition \ref{prop:SolGG} for $j = i+1$ reduces to the following lemma:

\begin{lem}\label{lem:160111}
Suppose $0\leq i \leq n-1$. Then the following three conditions are equivalent:
\begin{enumerate}
\item[(i)] $\widehat{d\pi_{(i,\lambda)^*}}(N_1^+)\psi =0$.
\item[(ii)] $M_{I\widetilde{I}}=0$ for all $I \in \mathcal{I}_{n,i}$ and
 $\widetilde{I} \in \mathcal{I}_{n,i+1}$.
\item[(iii)] $(\lambda,\nu) = (i,i+1)$ and $\ell = 0$.
\end{enumerate}
\end{lem}

Let us verify this lemma. 
According to the decomposition of $\widehat{d\pi_{(i,\lambda)^*}}(N_1^+)$
\index{B}{scalar-part@scalar part of $d\pi_{(\sigma, \lambda)^*}$}
\index{B}{vector-part@vector part of $d\pi_{(\sigma, \lambda)^*}$}
into the scalar and vector parts,
we decompose $M_{I\widetilde{I}}$ 
its matrix components
\index{A}{MIJ@$M_{IJ}$, matrix component of 
$\widehat{d\pi_{(i,\lambda)^*}}(N_1^+)\psi$}
\index{A}{Mscalar@$M_{IJ}^{\mathrm{scalar}}$}
\index{A}{Mvect@$M_{IJ}^{\mathrm{vect}}$}
$M_{I\widetilde{I}} = M_{I\widetilde{I}}^{\mathrm{scalar}}+M_{I\widetilde{I}}^{\mathrm{vect}}$
as in Proposition \ref{prop:MIJ},
where
\index{A}{AII'@$A_{II'}$, matrix component of $A_\sigma$}
\begin{eqnarray*}
M_{I\widetilde{I}}^{\mathrm{scalar}}&=&\widehat{d\pi_{\lambda^*}}(N_1^+)\psi_{I\widetilde{I}}=\left(\lambda\frac{\partial}{\partial\zeta_1}+ E_\zeta\frac{\partial}{\partial\zeta_1}-\frac12
\zeta_1\Delta_{\C^n}\right)\psi_{I\widetilde{I}},\\
M_{I\widetilde{I}}^{\mathrm{vect}}&=&\sum_{I'\in\mathcal I_{n,i}}A_{II'}\psi_{I'\widetilde{I}}.
\end{eqnarray*}

\begin{lem}
For $I \in \mathcal{I}_{n,i}$ and $\tilde{I} \in \mathcal{I}_{n,i+1}$, we have
\begin{equation*}
\psi_{I\widetilde{I}}=
\left\{
\begin{matrix*}[l]
Q_n^\ell(\zeta)\sgn(I;p)\zeta_p & \mathrm{if} & \widetilde{I}=I\cup\{p\},\\
0& \mathrm{if} & I\not\subset \widetilde{I}.
\end{matrix*}
\right.
\end{equation*}
\end{lem}

\begin{proof}
Clear from the definition of $H^{(1)}_{i \to i+1}$ given in \eqref{eqn:F1p}.
\end{proof}

\begin{lem}\label{lem:104}
For $I\in\mathcal I_{n,i}$ and $\widetilde{I}\in\mathcal I_{n,i+1}$, we have
$$
M_{I\widetilde{I}}^{\mathrm{scalar}}=
\left\{
\begin{matrix*}[l]
0 &\mathrm{if} & I\not\subset \widetilde{I},\\
\ell(2\lambda+2\ell-n)\zeta_1^2 Q_n^{\ell-1}(\zeta)+(\lambda+2\ell)Q_n^{\ell}(\zeta)&\mathrm{if} & \widetilde{I}\setminus I=\{1\},\\
\ell(2\lambda+2\ell-n)\sgn(I;p)\zeta_1\zeta_pQ_n^{\ell-1}(\zeta)&\mathrm{if} & \widetilde{I}\setminus I=\{p\}, p\neq1.
\end{matrix*}
\right.
$$
\end{lem}

\begin{proof}
Since $\Delta_{\C^n}Q_n^\ell(\zeta) = 2\ell (2\ell - 2 +n)Q_n^{\ell-1}(\zeta)$, 
we have the identity:
\begin{equation}\label{eqn:NQl}
\widehat{d\pi_{\lambda^*}}(N_1^+)Q_n^\ell(\zeta)
=\ell\,(2\lambda+2\ell-n)\zeta_1Q_n^{\ell-1}(\zeta).
\end{equation}
Then the lemma follows from Lemma \ref{lem:Poincare}.
\end{proof}

By the formula of $A_{II'}$ (see Lemma \ref{lem:AII}), we have
\begin{equation}\label{eqn:MIJvect}
M_{I\widetilde{I}}^{\mathrm{vect}}=\left\{
\begin{matrix*}[l]
\sum\limits_{q\in I}\sgn(I;q)\frac{\partial}{\partial\zeta_q}\psi_{I\setminus\{q\}\cup\{1\},\widetilde{I}}&\mathrm{if}& 1\not\in I,\\
\sum\limits_{q\not\in I}\sgn(I;q)\frac{\partial}{\partial\zeta_q}\psi_{I\setminus\{1\}\cup\{q\},\widetilde{I}}&\mathrm{if}& 1\in I.
\end{matrix*}
\right.
\end{equation}

Combining Lemma \ref{lem:104} with an easy computation of $M_{I\widetilde{I}}^{\mathrm{vect}}$, we get the following.
\begin{lem}\label{lem:1521107}
Let $n\geq 2$, $0\leq i \leq n-1$, $I\in \mathcal{I}_{n,i}$ and $\tilde{I} \in \mathcal{I}_{n,i+1}$.
\begin{enumerate}
\item Assume $1\in I$ and $\widetilde{I}=I\cup\{p\}$ for some $p\not\in I$. Then
$$
M_{I\widetilde{I}}=\ell(2\lambda+2\ell-n+2)\sgn(I;p)\zeta_1\zeta_pQ_n^{\ell-1}(\zeta).
$$
\item Assume $1\not\in I$ and $\widetilde{I}=I\cup\{1\}$. Then
\index{A}{Qw1I@$Q_I(\zeta)$, quadratic form for index set $I$}
$$
M_{I\widetilde{I}}=(\lambda-i+2\ell)Q_n^\ell(\zeta)
+(2\lambda+2\ell-n)\ell\zeta_1^2Q_n^{\ell-1}(\zeta)-2\ell Q_I(\zeta) Q_n^{\ell-1}(\zeta).
$$
\item Assume $\widetilde{I}=K\cup\{1,q\}, I=K\cup\{p\}$ with $1\neq p\neq q\neq1$. Then,
$$
M_{I\widetilde{I}}=-2\ell\,\sgn(K;p,q)\zeta_p\zeta_q Q_n^{\ell-1}(\zeta).
$$
\item Assume $1\not\in \widetilde{I}$ and $\widetilde{I}=I\cup\{p\}$. Then 
$$
M_{I\widetilde{I}}=\ell(2\lambda+2\ell-2)\sgn(I;p)\zeta_1\zeta_p Q_n^{\ell-1}(\zeta).
$$
\item Otherwise, $M_{I\widetilde{I}}=0$.
\end{enumerate}
\end{lem}

We are ready to give a proof of Lemma \ref{lem:160111}, 
and consequently Proposition \ref{prop:SolGG} for $j = i  + 1$.

\begin{proof}[Proof of Lemma \ref{lem:160111}]
Suppose $M_{I\widetilde{I}} = 0$ for all $I \in \mathcal{I}_{n,i}$
and $\widetilde{I} \in \mathcal{I}_{n,i+1}$. 
Then $\ell$ must be zero, as is seen from 
Lemma \ref{lem:1521107} (3) which works
for $i\neq 0$ or from Lemma \ref{lem:1521107} (4)
which works for $i=0$ and $n\geq 2$.
In turn, Lemma \ref{lem:1521107} (2) implies $\lambda = i$,
and thus $\nu = i+ 1$.

Conversely, if $\ell = 0$ and $(\lambda,\nu) = (i,i+1)$, then
clearly $M_{I\widetilde{I}} = 0$ for all $I$ and $\widetilde{I}$ by Lemma \ref{lem:1521107}.
Thus Lemma \ref{lem:160111} is proved.
\end{proof}

%%%%%%%%%%%%%%%%%%%%%%%%%%%%%%%%%%%%%%%%%%%%%%%
\subsection{Solving the F-system when $j=i$}\label{subsec:GGB}

In this section we treat the case $j=i$ and $\beta-\alpha \equiv 0 \; \mathrm{mod}\;2$.
According to Lemma \ref{lem:103}, we need to consider the case $\nu-\lambda\in2\N$. 
Since the case $\nu=\lambda$ is easy, let us assume $\nu-\lambda\in 2\N_+$. We set
\begin{equation*}
\ell:=\frac12(\nu-\lambda)-1\in\N.
\end{equation*}

First consider that $j=i=0$. Then any element in 
$\mathrm{Hom}_L\left(\sigma^{(i)}_{\lambda,\alpha}, \sigma^{(j)}_{\nu,\beta}
\otimes \mathrm{Pol}[\zeta_1, \ldots, \zeta_n] \right)$
is proportional to $Q_n^{\ell+1}H_{0\to 0}^{(0)}$ by Lemma \ref{lem:103}.
We set $\psi:=Q_n^{\ell+1}H^{(0)}_{0\to 0}$. Then we have the following.

\begin{lem}
Suppose $i=0$ or $n$ and $\nu-\lambda = 2\ell+2$ with $\ell \in \N$.
Then the following two conditions are equivalent:
\begin{enumerate}
\item[(i)] $\widehat{d\pi_{(i,\lambda)^*}}(N_1^+)\psi=0$.
\item[(ii)] $(\lambda,\nu) = \left(\frac{n}{2}-\ell-1, \frac{n}{2}+\ell + 1\right)$.
\end{enumerate}
\end{lem}

\begin{proof}
Since $i=0$ or $n$, 
\index{B}{vector-part@vector part of $d\pi_{(\sigma, \lambda)^*}$}
there is no ``vector part" of $\widehat{d\pi_{(i,\lambda)^*}}(N_1^+)
=\widehat{d\pi_{\lambda^*}}(N_1^+)$,
and thus
\begin{equation*}
\widehat{d\pi_{(i,\lambda)^*}}(N_1^+)\psi(\zeta)
=(\ell+1) (2\lambda+2\ell -n+2)\zeta_1 Q_n^\ell(\zeta)H^{(0)}_{0\to 0}
\end{equation*}
by \eqref{eqn:NQl}. Now the lemma is clear.
\end{proof}

From now, we assume $i\neq 0$ and $n\geq 2$.
For $A,B\in \C$ we set
\begin{align}
\psi\quad &:= \quad AQ_n^{\ell+1}H^{(0)}_{i \to i}+B Q_n^\ell \widetilde H^{(2)}_{i \to i}
\in\operatorname{Hom}_L\left(\sigma^{(i)}_{\lambda,\alpha},\;
\sigma^{(j)}_{\nu,\beta}\otimes\mathrm{Pol}[\zeta_1, \ldots, \zeta_n]\right),
\label{eqn:psiAB}\\
M_{II'}\quad &:=\quad \langle \widehat{d\pi_{(i,\lambda)^*}}(N_1)^+\psi(e_I),e_{I'}^\vee
\rangle \quad \mathrm{for} \quad I,I' \in \mathcal{I}_{n,i}. \nonumber
\end{align}

\begin{lem}\label{lem:20160412b}
Let $1 \leq i \leq n-1$, and 
$\nu-\lambda = 2\ell + 2$ with $\ell \in \N$.
Suppose $(A,B)\neq(0,0)$. 
Then the following three conditions are equivalent:
\begin{itemize}
\item[(i)] $\widehat{d\pi_{(i,\lambda)^*}}(N_1^+)\psi=0.$
\item[(ii)] $M_{II'} = 0$ \emph{for all} $I, I' \in \mathcal{I}_{n,i}$.
\item[(iii)] $(\lambda,\nu)=\left(\frac{n}{2}-\ell-1,\frac{n}{2}+\ell+1\right)$
and $(A,B)$ is proportional to $\left(-\frac{1}{2}(n+\nu-\lambda)(1-\frac{2i}{n}),
\nu-\lambda\right)$.
\end{itemize}
\end{lem}

The equivalence (i) $\Leftrightarrow$ (ii) is obvious, and we shall prove the 
equivalence (ii) $\Leftrightarrow$ (iii) after Lemma \ref{lem:20160412a}
where we compute $M_{II'}$ explicitly. For this we use a couple of lemmas as follows.

\begin{lem}\label{lem:20160411a}
Let $\psi$ be given as in \eqref{eqn:psiAB}. 
Then, for $I, I' \in \mathcal{I}_{n,i}$, we have
\index{A}{QwI@$\widetilde{Q}_I(\zeta)$}
\begin{align*}
\psi_{II'}(\zeta)
=\begin{cases}
AQ_n^{\ell+1}(\zeta)+B\widetilde{Q}_I(\zeta)Q_n^\ell(\zeta)
&\text{if $I=I'$,}\\
B\sgn(K;p,q)Q_n^\ell(\zeta)\zeta_p\zeta_q
& \text{$\mathrm{if}$ $I=K\cup \{p\}$ $\mathrm{and}$ $I' = K \cup \{q\}$},\\
0 & \text{$\mathrm{otherwise}$,}
\end{cases}
\end{align*}
where we recall $\widetilde{Q}_I(\zeta) = 
\sum_{\ell \in I} \zeta^2_\ell - \frac{i}{n} Q_n(\zeta)$ from \eqref{eqn:QItilde}.
\end{lem}

\begin{proof}
Clear from the definitions of $H^{(0)}_{i \to i}$ and $\widetilde{H}^{(2)}_{i \to i}$
given in \eqref{eqn:F0} and \eqref{eqn:H2tilde}.
\end{proof}

\begin{lem}\label{lem:20160411b}
For $I, I' \in \mathcal{I}_{n,i}$, the scalar part $M_{II'}^{\mathrm{scalar}}$
is given as follows. 
\begin{enumerate}
\item $I=K\cup \{p\}$, $I'=K\cup\{q\}:$
\begin{align*}
M_{II'}^{\mathrm{scalar}}=
\begin{cases}
B\sgn(K;p)\big( \ell(2\lambda + 2\ell - n)\zeta_1^2\zeta_pQ_n^{\ell-1}(\zeta)
+(\lambda+2\ell+1)\zeta_pQ_n^\ell(\zeta) \big) & \text{ $\mathrm{if}$ $q=1$},\\
B\sgn(K;p,q)(\ell(2\lambda+2\ell-n)\zeta_1\zeta_p\zeta_qQ_n^{\ell-1}(\zeta)
& \text{ $\mathrm{if}$ $q\neq 1$}.
\end{cases}
\end{align*}
\item $I=I':$ Suppose that $I=K\cup \{p\}$. Then,
\begin{align*}
&M_{II'}^{\mathrm{scalar}}=\\
&\begin{cases}
\zeta_1Q_n^{\ell-1}(\zeta)\left(B\ell(a-2)Q_I(\zeta)
+\left\{a A(\ell+1)
-\frac{B}{n}\left((n+a)(n-i)+\ell(2n-ia)\right)
\right\}
Q_n(\zeta)\right) & \text{ $\mathrm{if}$ $p=1$},\\
\zeta_1Q_n^{\ell-1}(\zeta)\left(B\ell(a-2)Q_I(\zeta)
+\left\{a A(\ell+1)-\frac{i B}{n}(a(\ell+1)+n)\right\}
Q_n(\zeta)\right) & \text{ $\mathrm{if}$ $p\neq1$}.
\end{cases}
\end{align*}
\end{enumerate}
\end{lem}

\begin{proof}
Direct computation by using Lemma \ref{lem:20160411a}, 
\eqref{eqn:NQl} and Lemma \ref{lem:Poincare}.
\end{proof}

We set
\begin{align*}
a\quad &:=\quad 2\lambda+2\ell-n+2,\\
b\quad &:=\quad 2A(\ell+1)+B\left(\frac n2+\ell+1\right)\left(1-\frac{2i}n\right).
\end{align*}

\begin{lem}\label{lem:20160412a}
Let $\psi$ be as in \eqref{eqn:psiAB} and $I, I' \in \mathcal{I}_{n,i}$.
\begin{enumerate}

\item Assume $1\not\in I=I'$. Then
$$
M_{II'}=a \zeta_1 Q_n^{\ell-1}(\zeta)
\left((\ell+1)\left(A-\frac{iB}{n}\right)Q_n(\zeta) + B\ell Q_I(\zeta) \right).
$$

\item Assume $1\in I=I'$. Then
$$
M_{II'}=a
\zeta_1 Q_n^{\ell-1}\left(\left(A(\ell+1)+B\left(1-\frac{i(\ell+1)}n\right)\right)Q_n+
B\ell Q_I\right).
$$

\item Assume $I=K\cup\{1\},I'=K\cup\{p\}$ with $p\neq1$. Then
$$
M_{II'}=\sgn(K;p)\zeta_p Q_n^{\ell-1}
\left(a\ell B \zeta_1^2 +\left(\frac{1}{2} aB-b\right)Q_n
\right).
$$

\item Assume $I=K\cup\{p\},I'=K\cup\{1\}$ with $p\neq1$. Then
$$
M_{II'}=\sgn(K;p)\zeta_p Q_n^{\ell-1}
\left(a\ell B \zeta_1^2+\left(\frac{1}{2} a B + b\right) Q_n \right).
$$

\item Assume $I=K\cup\{p\},I'=K\cup\{q\}$ with $p\neq q$, $1\not\in I$, and $1\not\in I'$. Then
$$
M_{II'}=a B\ell\,\sgn(K;p,q)\zeta_1\zeta_p\zeta_q\, Q_n^{\ell-1}.
$$

\item Otherwise, 
$M_{II'}=0$.
\end{enumerate}
\end{lem}

\begin{proof}
We give a proof for (1). Suppose $1 \notin I = I'$. It follows from Lemmas \ref{lem:AII} 
and \ref{lem:20160411a} that
\begin{align*}
M^{\mathrm{vect}}_{II'}
&=\sum_{k \in I} \left(\sgn(I;k) \frac{\partial}{\partial \zeta_k}\right)
\left(B\sgn(I;1,k)Q_n^\ell(\zeta)\zeta_1 \zeta_k\right)\\
&=B\zeta_1\sum_{k \in I}\frac{\partial}{\partial \zeta_k}\left(Q^\ell_n(\zeta) \zeta_k\right)\\
&=B\zeta_1\left(2\ell Q_n^{\ell-1}(\zeta)Q_I(\zeta) + i Q^\ell_n(\zeta)\right).
\end{align*}
Together with the formula of $M^{\mathrm{scalar}}_{II'}$ given in 
Lemma \ref{lem:20160411b}, we have
\begin{align*}
M_{II'}&=M^{\mathrm{scalar}}_{II'} + M^{\mathrm{vector}}_{II'}\\
&=a \zeta_1 Q_n^{\ell-1}(\zeta)\left(
(\ell+1)\left(A-\frac{iB}{n}\right)Q_n(\zeta) + B\ell Q_I(\zeta)\right).
\end{align*}
Thus the first assertion is proved. The other cases are similar and omitted.
\end{proof}

We are ready to give a proof of Lemma \ref{lem:20160412b}, and 
consequently Proposition \ref{prop:SolGG} for $j=i$.

\begin{proof}[Proof of Lemma \ref{lem:20160412b}]
Since $1\leq i \leq n-1$, the cases (3) and (4) occur in 
Lemma \ref{lem:20160412a}. Hence $M_{II'}=0$ implies that
\begin{equation*}
a \ell B = \frac{1}{2}aB \pm b = 0,
\end{equation*}
or equivalently $a=b=0$ or $b=B=0$.
Since $(A,B) \neq (0,0)$, the case $b = B=0$ does not occur.
The condition $a =0$ implies $\lambda + \nu = n$, and the condition
$b= 0$ gives the ratio of $(A,B)$ as stated in (iii).
Therefore the implication (ii) $\Rightarrow$ (iii) is proved.
The converse statement (iii) $\Rightarrow$ (ii) is clear from 
Lemma \ref{lem:20160412a}.
\end{proof}

%%%%%%%%%%%%%%%%%%%%%%%%%%%%%%%%%%%%%%%%%%%%%%
\subsection{Solving the F-system when $j=i-1$}\label{subsec:GGB-}

The case $j=i-1$ is similar to the case $j=i+1$. 

According to Lemma \ref{lem:103}, we may assume $\nu-\lambda = 2\ell+1$ for some 
$\ell \in \N$ and $\beta \equiv \alpha+1 \; \mathrm{mod} \; 2$. We set
$\psi:=Q_n^\ell H^{(1)}_{i\to i-1}$. Then we have:

\begin{lem}\label{lem:20160410b}
Suppose $1\leq i \leq n$. Then the following two conditions are equivalent:
\begin{enumerate}
\item[(i)] $\widehat{d\pi_{(i,\lambda)^*}}(N_1^+)\psi=0$.
\item[(ii)] $(\lambda,\nu) = (n-i,n-i+1)$ and $\ell = 0$.
\end{enumerate}
\end{lem}

The proof of Lemma \ref{lem:20160410b} goes similarly to that of Lemma \ref{lem:160111}
in the case $j=i+1$, and so we omit it. Alternatively, Lemma \ref{lem:20160410b} 
follows from Lemma \ref{lem:160111}. 
In fact, the following 
\index{B}{duality theorem (principal series)}
duality
as in Theorem \ref{thm:psdual}
\begin{equation*}
\mathrm{Diff}_G(I(i,\lambda)_\alpha, I(j,\nu)_\beta)
\simeq
\mathrm{Diff}_G(I(n-i,\lambda)_\alpha, I(n-j,\nu)_\beta)
\end{equation*}
implies a bijection between the space of solutions for the 
\index{B}{F-system}
F-systems
\begin{equation*}
Sol(\mathfrak{n}_+; \sigma^{(i)}_{\lambda,\alpha}, \sigma^{(j)}_{\nu,\beta})
\simeq
Sol(\mathfrak{n}_+; \sigma^{(n-i)}_{\lambda,\alpha}, \sigma^{(n-j)}_{\nu,\beta}).
\end{equation*}
Thus we have completed the proof of Proposition \ref{prop:SolGG}, 
whence Theorem \ref{thm:psGGB}.

%%%%%%%%%%%%%%%%%%%%%%%%%%%%%%%%%%%%%%%%%%%%%%
\subsection{Proof of Theorem \ref{thm:GGC}}\label{subsec:ThmGGC}

In this section, we deduce Theorem \ref{thm:GGC} (conformal representations) from 
Theorem \ref{thm:psGGB} (principal series representations), 
as we did in Chapter \ref{sec:solAB}
for symmetry breaking operators. Needless to say, the case $(G=G')$ in this section
is much simpler than the case $(G\neq G')$ in the previous chapter.
We recall from Proposition \ref{prop:identific} that there are the natural isomorphisms 
as $G$-modules:

\begin{equation}\label{eqn:160165}
I(i,\lambda)_i \simeq \varpi^{(i)}_{\lambda-i,0} \simeq \varpi^{(n-i)}_{\lambda-n+i,1}.
\end{equation}

We note that there are two geometric models
of the same principal series representations 
$I(i,\lambda)_i$. We translate the four cases in 
Theorem \ref{thm:psGGB} in terms of \eqref{eqn:160165}.
\vskip 0.1in

\noindent
Case 1. $j = i+1$, $(\lambda, \nu) = (i,i+1)$, $\alpha \equiv \beta + 1$.
\smallskip

We take $\alpha \equiv i$, and $\beta \equiv i+1$ mod 2. Then we have
\begin{align*}
I(i,i)_i &\simeq \varpi^{(i)}_{0,0} \simeq \varpi^{(n-i)}_{2i-n,1},\\
I(i+1, i+1)_{i+1} &\simeq \varpi^{(i+1)}_{0,0} \simeq \varpi^{(n-i-1)}_{2i+2-n,1}.
\end{align*}
\vskip 0.1in

\noindent
Case 2. $j = i-1$, $(\lambda,\nu) = (n-i,n-i+1)$.
\smallskip

We take $\alpha \equiv i$ and $\beta \equiv i-1$ mod 2. Then we have
\begin{align*}
I(i,n-i)_i &\simeq \varpi^{(i)}_{n-2i,0} \simeq \varpi^{(n-i)}_{0,1},\\
I(i-1, n-i+1)_{i-1} &\simeq \varpi^{(i-1)}_{n-2i+2,0} \simeq \varpi^{(n-i+1)}_{0,1}.
\end{align*}
\vskip 0.1in

\noindent
Then the intertwining operators assured in Theorem \ref{thm:psGGB} in Cases 1 and 2
are given in the following four arrows (after switching $i$ and $n-i$ in Case 2):
\begin{equation*}
\xymatrix{
\mathcal{E}^i(S^n)_{0,\delta}\ar@{-->}[r]_{*\hspace{18pt}}\ar[d]
\ar[dr]&\mathcal{E}^{n-i}(S^n)_{2i-n,\delta+1}\ar[d]\ar[dl]\\
\mathcal E^{i+1}(S^n)_{0,\delta}\ar@{-->}[r]_{*\hspace{25pt}}
&\mathcal E^{n-i-1}(S^n)_{2i+2-n,\delta+1}
}
\end{equation*}

The vertical arrows are scalar multiples of $d$ and $d^*$ 
(see Proposition \ref{prop:SolGG}, and the horizontal dotted arrows
are given by Hodge star operators.
This explains the four Cases b, $*$b, c, and $*$c
in Theorem \ref{thm:GGC}.
\vskip 0.1in

\noindent
Case 3. $j=i$, $\lambda+\nu = n$, $\nu - \lambda \in 2\N_+$.
\smallskip

We set $\ell:=\frac{1}{2}(\nu-\lambda)$ so that
$\lambda = \frac{n}{2}-\ell$ and $\nu = \frac{n}{2}+\ell$.
Then we have
isomorphisms as $G$-modules:
\begin{align*}
I\left(i,\frac{n}{2}-\ell\right)_i &\simeq \varpi^{(i)}_{\frac{n}{2}-\ell-i,0}
\simeq \varpi^{(n-i)}_{-\frac{n}{2}-\ell+i,1},\\
I\left(i,\frac{n}{2}+\ell\right)_i &\simeq \varpi^{(i)}_{\frac{n}{2}+\ell-i,0}
\simeq \varpi^{(n-i)}_{-\frac{n}{2}+\ell+i,1}.
\end{align*}
\vskip 0.1in

\noindent
This yields Cases d and $*$d in Theorem \ref{thm:GGC}.
Case 4 yields Cases a and $*$a. 
Thus we have listed all possible cases,
and the proof of Theorem \ref{thm:GGC} is completed.

%%%%%%%%%%%%%%%%%%%%%%%%%%%%%%%%%%%%%%%%%%%%%%
\subsection{Hodge star operator and 
Branson's operator $\mathcal T_{2\ell}^{(i)}$}\label{subsec:THodge}
 
By the multiplicity-freeness result,
\begin{equation*}
\dim \mathrm{Diff}_{O(n+1,1)}
\left(\mathcal{E}^i(S^n)_{u,\delta}, 
\mathcal{E}^j(S^n)_{v,\eps}\right) \leq 1
\end{equation*}
in Theorem \ref{thm:GGC}, we know \emph{a priori} that 
the composition 
$* \circ \mathcal{T}^{(i)}_{2\ell} \circ *^{-1}$ of conformally covariant operators is
proportional to $\mathcal{T}^{(n-i)}_{2\ell}$. To be explicit, we have the following proposition.

\begin{prop}\label{prop:152520}
Let $0\leq i\leq n$ and $\ell\in\N_+$. We put $u:=\frac n2-i-\ell$, $v:=\frac n2-i+\ell$. Then the following diagram commutes for any $\delta\in\Z/2\Z$.
$$
\xymatrix{
\mathcal E^i(S^n)_{u,\delta}\ar[rr]^{\mathcal T_{2\ell}^{(i)}}
\ar[d]_{*} &&\mathcal E^i(S^n)_{v,\delta}\ar[d]^{*} \\
\mathcal E^{n-i}(S^n)_{-v,\delta+1}\ar[rr]^{-\mathcal T_{2\ell}^{(n-i)}}
&&\mathcal E^{n-i}(S^n)_{-v+2\ell,\delta+1}.
}
$$
\end{prop}
\begin{proof}
The vertical isomorphisms are given by Proposition \ref{prop:Hodgeconf} because $u-n+2i=-v$. It then follows from Lemma \ref{lem:dd} and \eqref{eqn:T2li} that
$$
\mathcal T_{2\ell}^{(n-i)}=-*_{\R^n}\circ \mathcal T_{2\ell}^{(i)}\circ(*_{\R^n})^{-1}
$$
in the flat coordinates. 
Hence the proposition follows from Lemma \ref{lem:20160317} (see also \eqref{eqn:starSR}).
\end{proof}

\newpage
%%%%%%%%%%%%%%%%%%%%%%%%%%%%%%%%%%%%%%%%%%%%%%%
\section{Matrix-valued factorization identities}\label{sec:fi}

Differential symmetry breaking operators
may be expressed as a composition of two equivariant differential operators
for some special values of the parameters.
Such formul\ae{} in the \emph{scalar case} are called 
\index{B}{factorization identity|textbf}
``factorization identities" in \cite{Juhl}
or ``functional identities for symmetry breaking operators" in \cite{KS13}.

In this chapter we establish ``factorization identities" for conformally covariant differential operators on differential forms.
A number of results have been known in the scalar case
\cite{Juhl, K14, KOSS15, KS13}, 
however, what we treat here is
\index{B}{matrix-valued functional identities}
\emph{matrix-valued} factorization identities.
The setting we consider is $(X,Y) = (S^n, S^{n-1})$
and $(G,G') = (O(n+1,1), O(n,1))$ as before, and $T_X$ 
(respectively, $T_Y$) is a $G$- (respectively, $G'$-) intertwining operator in 
the following diagrams:
\begin{eqnarray*}
\xymatrix{
\mathcal E^i(X)_{u,\delta}\ar[d]_{T_X}
\ar[drr]^{D^{i\to j}_{u,a}}
&&\\
\mathcal E^{i'} (X)_{u',\delta'}\ar[rr]_{D^{i'\to j}_{u',b}}
&&\mathcal E^j(Y)_{v,\eps},
}
&&
\xymatrix{
\mathcal E^i(X)_{u,\delta}\ar[rr]^{D^{i\to j}_{u,a}}
\ar[drr]_{D^{i\to j'}_{u,c}}
&&\mathcal E^j (Y)_{v,\eps}\ar[d]^{T_Y}\\
&&\mathcal E^{j'}(Y)_{v',\eps'}
}
\end{eqnarray*}
In Theorem \ref{thm:1},
we have classified all the parameters for which there exist nonzero 
differential symmetry breaking operators, 
and have proved a multiplicity-one theorem. This guarantees the following
``factorization identities" for some $p,q \in \C$:
\begin{align*}
D^{i'\to j}_{u',b} \circ T_X &=pD^{i\to j}_{u,a},\\
T_Y \circ D^{i\to j}_{u,a} &= q D^{i\to j'}_{u,c}.
\end{align*}
Explicit generators for symmetry breaking operators $D^{i\to j}_{u,a}$ are given
in Theorems \ref{thm:2}-\ref{thm:2ii-2}, whereas nontrivial $G$- or 
$G'$-intertwining operators $T_X$ or $T_Y$ are classified in Theorem \ref{thm:GGC}.
In this chapter, we shall consider all possible combinations
of these operators under the parity condition
\begin{equation*}
\delta\equiv \delta' \equiv \eps \; \mathrm{mod}\; 2
\quad \text{or} \quad
\delta \equiv \eps \equiv \eps' \; \mathrm{mod} \; 2,
\end{equation*}
and determine factorization identities. The explicit formul\ae{} are given in
Theorem \ref{thm:factor1} for $T_X = \mathcal{T}_{2\ell}^{(i)}$ 
(Branson's operator), in Theorem \ref{thm:factor2} for 
$T_Y = \mathcal{T}'^{(j)}_{2\ell}$, in Theorem \ref{thm:152673} for 
$T_X = d$ or $d^*$, and in Theorem \ref{thm:152557} for $T_Y = d$ or $d^*$.
Factorization identities for the other parity case are derived easily 
from the same parity case by using the Hodge star operators.

In Section \ref{subsec:factor}, we summarize these factorization identities
in terms of the unnormalized symmetry breaking operators 
$\mathcal{D}^{i\to j}_{u,a}$ rather than the renormalized ones
$\widetilde{\mathcal{D}}^{i \to j}_{u,a}$, because the formul\ae{} 
take a simpler form. 
Factorization identities for the renormalized operators will be 
discussed in Section \ref{subsec:RFI} with focus on the exact
condition when the composition of two nonzero operators vanishes.

%%%%%%%%%%%%%%%%%%%%%%%%%%%%%%%%%%%%%%%%%%%%%%%
\subsection{Matrix-valued factorization identities}\label{subsec:factor}

This section summarizes the factorization identities for unnormalized operators
for $(X,Y) = (S^n, S^{n-1})$ with the same parity of $\delta, \delta',\eps$, and $\eps'$.

We begin with the case $T_X=\mathcal{T}^{(i)}_{2\ell}$ or 
$T_Y = \mathcal{T}'^{(j)}_{2\ell}$, where
we recall from Theorem \ref{thm:GGC} that for $\ell\in\N_+$ the $G$- and $G'$-intertwining operators 
\index{B}{Branson's operator}
(Branson's operators)
\index{A}{T2ell@$\mathcal T_{2\ell}^{(i)}$, Branson's operator}
\index{A}{T2ell'@$\mathcal {T'}_{2\ell}^{(j)}$}
\begin{eqnarray*}
\mathcal T_{2\ell}^{(i)}&\colon &\mathcal E^i(S^n)_{\frac n2-i-\ell}\To \mathcal E^i(S^n)_{\frac n2-i+\ell},\\
\mathcal {T'}_{2\ell}^{(j)}&\colon &\mathcal E^j(S^{n-1})_{\frac{n-1}2-j-\ell}\To \mathcal E^j(S^{n-1})_{\frac{n-1}2-j+\ell}.
\end{eqnarray*}

Let $\delta \equiv a + i + j \; \mathrm{mod}\;2$.
Here are basic diagrams:

\begin{eqnarray*}
\xymatrix{
\mathcal E^i(S^n)_{\frac n2-i-\ell,\delta}\ar[d]_{\mathcal T_{2\ell}^{(i)}}
\ar[drr]^{{\mathcal D}^{i\to j}_{\frac n2-i-\ell,a+2\ell}}
&&\\
\mathcal E^i (S^n)_{\frac n2-i+\ell,\delta}\ar[rr]_{{\mathcal D}^{i\to j}_{\frac n2-i+\ell,a}}&&\mathcal E^j(S^{n-1})_{\frac{n}2-j+a+\ell,\delta},
}
&&
\xymatrix{
\mathcal E^i(S^n)_{\frac {n-1}2-i-a-\ell,\delta}\ar[rr]^{\mathcal D^{i\to j}_{\frac{n-1}2-i-a-\ell,a}}
\ar[drr]_{{\mathcal D}^{i\to j}_{\frac{n-1}{2}-i-a-\ell,a+2\ell}}
&&\mathcal E^j (S^{n-1})_{\frac {n-1}2-j-\ell,\delta}\ar[d]^{{\mathcal T'}^{(j)}_{2\ell}}\\
&&\mathcal E^j(S^{n-1})_{\frac{n-1}2-j+\ell,\delta}.
}
\end{eqnarray*}

As in \eqref{eqn:Kla}, we define a positive integer $K_{\ell,a}$ by
$\Kla=\prod_{k=1}^\ell\left(\left[\frac{a}{2}\right] + k\right)$  
for $\ell \in \N_+$ and $a \in \N$. Then we have:

\begin{thm}\label{thm:factor1} 
Suppose $0\leq i\leq n, a\in\N$ and $\ell\in\N_+$.
We set $u: = \frac{n}{2}-i-\ell$.
Then
\index{A}{Dii1@$\mathcal D_{u,a}^{i\to i-1}$}
\index{A}{Dii2@$\mathcal D_{u,a}^{i\to i}$}
\begin{eqnarray*}
&(1)&\mathcal {D}_{u+2\ell,a}^{i\to i-1}\circ 
\mathcal {T}_{2\ell}^{(i)}=-\left(\frac n2-i-\ell\right)
\Kla\mathcal {D}_{u,a+2\ell}^{i\to i-1} \quad\mathrm{if}\; i\neq 0.\\
&(2)&\mathcal {D}_{u+2\ell,a}^{i\to i}
\circ 
\mathcal {T}_{2\ell}^{(i)}=-\left(\frac n2-i+\ell\right)
\Kla\mathcal {D}_{u,a+2\ell}^{i\to i} \quad\mathrm{if}\; i\neq n.
\end{eqnarray*}
\end{thm}

\begin{thm}\label{thm:factor2} 
Suppose $0\leq i\leq n$, $a\in\N$ and $\ell\in\N_+$. We set $u:=\frac{n-1}2-i-\ell-a$. 
Then 
\index{A}{Kell@$K_{\ell, a}$}
\begin{alignat*}{3}
&(1)\quad \mathcal {T'}_{2\ell}^{(i-1)}\circ 
\mathcal {D}_{u,a}^{i\to i-1}&&=
-\left(\frac{n+1}2-i+\ell\right)
\Kla\mathcal {D}_{u,a+2\ell}^{i\to i-1}&&\quad\mathrm{if}\,i\neq0.\\
&(2)\quad  \mathcal{T'}_{2\ell}^{(i)}\circ\mathcal {D}_{u,a}^{i\to i}&&=
-\left(\frac{n-1}2-i-\ell\right)\Kla
\mathcal {D}_{u,a+2\ell}^{i\to i}&&\quad\mathrm{if}\, i\neq n.
\end{alignat*}
\end{thm}

We shall prove Theorems \ref{thm:factor1} and \ref{thm:factor2} 
in Sections \ref{subsec:pffactor2} and \ref{subsec:pffactor2b}, respectively.

Next, we consider the case where
$T_X=d$ or $d^*$. Here are basic diagrams:

\begin{eqnarray*}
\xymatrix{
\mathcal E^i(S^n)_{0,\delta}\ar[d]_{d}
\ar[drr]^{{\mathcal D}^{i\to j}_{0,a+1}}
&&\\
\mathcal E^{i+1}(S^n)_{0,\delta}
\ar[rr]_{{\mathcal D}^{i+1\to j}_{0,a} \quad}
&&\mathcal E^j(S^{n-1})_{a+1+i-j,\delta}
}
&&
\xymatrix{
\mathcal E^i(S^n)_{n-2i,\delta}\ar[d]_{d^*}
\ar[drr]^{{\mathcal D}^{i\to j}_{n-2i,a+1}}
&&\\
\mathcal E^{i-1}(S^n)_{n-2i+2,\delta}\ar[rr]_{{\mathcal D}^{i-1\to j}_{n-2i+2,a} \quad}
&&\mathcal E^{j}(S^{n-1})_{n-i-j+a+1,\delta}.
}
\\
(j=i, i+1) \qquad \quad \qquad \qquad &\quad & \qquad \qquad(j=i-1, i-2)
\end{eqnarray*}
In this setting, the condition on $\delta \in \Z/2\Z$ from Theorem \ref{thm:1}
is listed in the theorem below.

\begin{thm}\label{thm:152673}
For any $a\in\N$, 
we have
\begin{alignat*}{5}
&(1)&& 
\quad \mathcal D_{0,a}^{i+1\to i}\circ d 
&&=\gamma(i+1-\frac n2,a)\mathcal D_{0,a+1}^{i \to i},
\quad &&0\leq i\leq n-1, \quad&&\delta\equiv a+1 \; \mathrm{mod} \;2.\\
&(2)&& 
\quad\mathcal D_{0,a}^{i+1\to i+1}\circ d &&=0,
\quad &&0\leq i \leq n-1, \quad&&\delta\equiv 0 \; \mathrm{mod} \;2.\\
&(3)&& 
\quad\mathcal D_{n-2i+2,a}^{i-1\to i-1}\circ d^*
&&=-\gamma(-i+1+\frac n2,a)\mathcal D_{n-2i,a+1}^{i \to i-1},
\quad &&1\leq i \leq n, \quad&&\delta\equiv a \; \mathrm{mod} \;2.\\
&(4)&& 
\quad\mathcal D_{n-2i+2,a}^{i-1\to i-2}\circ d^*&&=0,
\quad &&2\leq i \leq n, \quad&&\delta\equiv 1 \; \mathrm{mod} \;2.
\end{alignat*}
\end{thm}
Here we recall from \eqref{eqn:gamma} that 
\index{A}{1Cgamma@$\gamma(\mu,a)$}
$\gamma(\mu,a) = 1$ if 
$a$ is odd; $=\mu+\frac{a}{2}$ if $a$ is even.

The proof of Theorem \ref{thm:152673} will be given in Section \ref{subsec:idi}.

Finally, we consider the case where $T_Y=d$ or $d^*$. Here are basic diagrams:

\begin{eqnarray*}
\xymatrix{
\mathcal E^i(S^n)_{-a-i+j,\delta}\ar[rr]^{\mathcal D^{i\to j}_{-a-i+j,a}}
\ar[drr]_{{\mathcal D}^{i\to j+1}_{-a-i+j,a+1}}
&&\mathcal E^{j} (S^{n-1})_{0,\delta}\ar[d]^{d}\\
&&\mathcal E^{j+1}(S^{n-1})_{0,\delta},
}
&&
\xymatrix{
\mathcal E^i(S^n)_{n-i-j-a-1,\delta}\ar[rr]^{\mathcal D^{i\to j}_{n-i-j-a-1,a}}
\ar[drr]_{{\mathcal D}^{i\to j-1}_{n-i-j-a-1,a+1}}
&&\mathcal E^{j} (S^{n-1})_{n-2j-1,\delta}\ar[d]^{d^*}\\
&&\mathcal E^{j-1}(S^{n-1})_{n-2j+1,\delta}.
}
\\
(j=i-1, i)  \qquad \qquad &\quad & \qquad \quad \qquad \qquad \quad (j=i-1,i)
\end{eqnarray*}

In these two diagrams, $\delta \equiv a+i+j\; \mathrm{mod}\;2$.

\begin{thm}\label{thm:152557}
For any $a\in \N$, we have
\begin{alignat*}{4}
&(1)&& \quad d\circ \mathcal D^{i\to i-1}_{-a-1,a}
&&=-\gamma\left(-a+i-\frac{n+1}2,a\right)\,\mathcal D^{i\to i}_{-a-1,a+1}
\quad &&\emph{for $1\leq i \leq n-1$}.
\\
&(2)&& \quad d\circ \mathcal{D}^{i\to i}_{-a,a} &&= 0
\quad &&\emph{for $0\leq i \leq n-2$}.\\
&(3)&& \quad d^* \circ \mathcal D^{i\to i}_{n-2i-a-1,a}
&&=-\gamma\left(-a-i+\frac{n-1}2,a\right)
\mathcal D^{i\to i-1}_{n-2i-a-1,a+1}
\quad &&\emph{for $1\leq i \leq n$}.\\
&(4)&& \quad d^*\circ \mathcal{D}^{i \to i-1}_{n-2i-a,a} 
&&= 0
\quad &&\emph{for $2\leq i \leq n$}.
\end{alignat*}
\end{thm}
Here we consider only the case where
$a=0$ if $1\leq i\leq n-1$ in (2), or if $2\leq i \leq n-1$ in (4).
The proof of Theorem \ref{thm:152557} will be given in Section \ref{subsec:dii}.

\begin{rem}
Theorem \ref{thm:152673} (2) for $a\neq 0$ reflects 
Theorem \ref{thm:1} which asserts that there is no nonzero 
conformally equivariant symmetry breaking operators 
$\mathcal E^i(S^n)\To \mathcal E^{i+1}(S^{n-1})$ other than
the obvious one $\mathrm{Rest}_{S^n} \circ d$ (i.e.\ $a=0$ case) 
up to scalar 
if $i\neq 0$. On the other hand, we recall from Proposition \ref{prop:Dnonzero} that
the differential symmetry breaking operators $\mathcal{D}^{i \to j}_{u,a}$ may vanish
for specific parameters, for instance,
\begin{alignat*}{3}
& && {\mathcal D}^{i\to i}_{-a,a}=0\quad &&\mathrm{if}\quad i=0\quad\mathrm{or}\quad a=0.
\end{alignat*}

In such cases, the fomul\ae{} like Theorem \ref{thm:152673} (2) 
for $a=0$ or 
Theorem \ref{thm:152557} (2) for $i=0$ or $a=0$ 
are trivial, however, 
by using the renormalized operators
$\widetilde{\mathcal{D}}^{i\to i}_{u,a}$ and $\widetilde{\mathcal{D}}^{i\to i+1}_{u,a}$
given in \eqref{eqn:reDii} and \eqref{eqn:Ditua}, we have
the following nontrivial factorization identities:
\begin{align*}
\widetilde{\mathcal{D}}^{i+1, i+1}_{0,0} \circ d
&=\widetilde{\mathcal{D}}^{i \to i+1}_{0,0},\\
d \circ \widetilde{\mathcal{D}}^{0\to0}_{-a,a} &=
\widetilde{\mathcal{D}}^{0\to1}_{-a, a+1} 
\quad \text{for all $a \in \N$,}\\
d\circ \widetilde{\mathcal{D}}^{i\to i}_{0,0} &=
\widetilde{\mathcal{D}}^{i\to i}_{0,1}.
\end{align*}

We shall discuss in Section \ref{subsec:RFI}
the factorization identities for renormalized symmetry breaking operators 
of this kind corresponding to Theorem \ref{thm:152673} (2) and (4), 
and Theorem \ref{thm:152557} (2) and (4) among others.
We also determine
the vanishing condition of the composition of two nonzero operators in detail.

\end{rem}

As an immediate corollary of Theorem \ref{thm:152557}, we can tell exactly when the image of the symmetry breaking operator $\mathcal D^{i\to i-1}_{-a-1,a}$ consists of closed forms on the submanifold $S^{n-1}$:

\begin{cor}\label{cor:1525103}
${}$
\begin{enumerate}
\item Assume that $n$ is odd and  $a$ is even with $0\leq a\leq n -1$. We set
\begin{equation}\label{eqn:ina}
i:=\frac12(a+n+1).
\end{equation}
Then $\frac{n+1}{2} \leq i\leq n$ and $\mathcal D^{i\to i-1}_{-a-1,a}\omega$ is a closed $i$-form on $S^{n-1}$ for any $i$-form $\omega$ on $S^n$.
\item Conversely, 
suppose $1\leq i \leq n-1$. 
If $n$ is even or if $a\neq 2i-n-1$, 
then there exists $\omega\in\mathcal E^i(S^n)$ such that $\mathcal D^{i\to i-1}_{-a-1,a}\omega$ is not a closed form on $S^{n-1}$.
\end{enumerate}
\end{cor}

\begin{proof}
For $i=n$, the $(n-1)$-form $\mathcal{D}^{n \to n-1}_{u,a} \omega$
is automatically closed for any $\omega \in \mathcal{E}^n(S^n)$.
Suppose $1\leq i \leq n-1$ and $a\in \N$. By Theorem \ref{thm:152557} (1),
$d\circ \mathcal{D}^{i\to i-1}_{-a-1,a} = 0$ if and only if 
$\gamma\left(-a-i+\frac{n+1}2,a\right)=0$.
By the definition \eqref{eqn:gamma} of $\gamma(\mu,a)$, this happens
exactly when $a$ is even and $i=\frac12(a+n+1)$.
This forces $n$ to be even. We also note that $0\leq a$ is equivalent 
to $\frac{1}{2} (n+1) \leq i$, and $a\leq n-1$ is equivalent to $i \leq n$.
Hence the corollary follows.
\end{proof}

%%%%%%%%%%%%%%%%%%%%%%%%%%%%%%%%%%%%%%%%%%%%%%%
\subsection{Proof of Theorem \ref{thm:factor1} (1)}\label{subsec:factor1}

We shall work with the flat coordinates.
Theorem \ref{thm:factor1} (1) will be shown by the following proposition.

\begin{prop}\label{prop:152599}
There exist
scalar-valued differential operators $P,Q,R,P',Q'$ and $R'$ on $\R^n$
satisfying the following three conditions:
\begin{align}
&\mathcal D_{\frac n2-i-\ell,a+2\ell}^{i\to i-1}=\mathrm{Rest}_{x_n=0}\circ
\left( P d_{\R^n}d^*_{\R^n}\iota_{\frac{\partial}{\partial x_n}}+Q d^*_{\R^n}+R\iota_{\frac{\partial}{\partial x_n}}\right),
\nonumber\\
&
\mathcal D_{\frac n2-i+\ell,a}^{i\to i-1}\circ\mathcal T_{2\ell}^{(i)}=
\mathrm{Rest}_{x_n=0}\circ
\left( P' d_{\R^n}d^*_{\R^n}\iota_{\frac{\partial}{\partial x_n}}+Q' d^*_{\R^n}+R'\iota_{\frac{\partial}{\partial x_n}}\right),
\label{eqn:DTPQR}\\
&\text{$P=c_i(\ell) P'$, $Q= c_i(\ell) Q'$, $R= c_i(\ell) R'$},\nonumber
\end{align}
where $c_i(\ell)$ is defined by
\begin{equation*}
c_i(\ell):=
-\left(\frac{n}{2} -i-\ell\right)K_{\ell,a}
\left(=-\left(\frac n2-i-\ell\right) \prod_{k=1}^\ell\left(\left[\frac a2\right]+k\right)\right).
\end{equation*}
\end{prop}
The rest of this section is devoted to the proof of Proposition \ref{prop:152599}.

We take 
\begin{eqnarray*}
P&=&-\mathcal D_{a+2\ell-2}^{-\ell+\frac 32},\\
Q&=&-\gamma(-\ell+\frac12,a+2\ell)\mathcal D_{a+2\ell-1}^{-\ell+\frac32},\\
R&=&-\frac12\left(\frac n2-i+\ell\right)\mathcal D_{a+2\ell}^{-\ell+\frac12},
\end{eqnarray*}
according to the formula \eqref{eqn:Dii1} of $\mathcal{D}^{i\to i-1}_{u,a}$.
In order to find $P', Q'$, and $R'$, we use:

\begin{lem}\label{lem:1524105}
Suppose $A,B,C,p$ and $q$ are (scalar-valued) differential operators on $\R^n$ with constant coefficients.
We set
$$
P':=-Ap\Delta_{\R^n}+Cp+Aq,\quad
Q':=-Bp\Delta_{\R^n}+Cp\frac{\partial}{\partial x_n}+Bq,\quad R':= Cq.
$$
Then,
$$
\left(
A\ d_{\R^n}d^*_{\R^n}
\iota_{\frac{\partial}{\partial x_n}}+Bd^*_{\R^n}
+C\iota_{\frac{\partial}{\partial x_n}}\right)\circ
(pd_{\R^n}d^*_{\R^n}+q)=
P'd_{\R^n}d^*_{\R^n}\iota_{\frac{\partial}{\partial x_n}}
+Q'd^*_{\R^n}+R'\iota_{\frac{\partial}{\partial x_n}}.
$$
\end{lem}

\begin{proof}
This is an easy consequence of 
the commutation relations among the operators $d_{\R^n}$, $d^*_{\R^n}$, and
$\iotan$ given in 
Lemma \ref{lem:152457} together with 
$(d^*_{\R^n})^2=0$ and $d_{\R^n}d^*_{\R^n}
+d^*_{\R^n}d_{\R^n}=-\Delta_{\R^n}$.
\end{proof}

\begin{lem}\label{lem:1525101}
The equation \eqref{eqn:DTPQR} holds if we take
\begin{eqnarray*}
P'&=&\left(\frac n2-i-\ell\right)\left(\ell+\left[\frac a2\right]\right)\mathcal D_a^{\ell-\frac12}\Delta_{\R^n}^{\ell-1},\\
Q'&=&\frac{\left(\frac n2-i-\ell\right)(a+1)(a+2\ell)}{4\gamma\left(\ell+\frac12,a-1\right)}
\mathcal D_{a+1}^{\ell-\frac12}\Delta_{\R^n}^{\ell-1},\\
R'&=&\frac12\left(\frac n2-i-\ell\right)\left(\frac n2-i+\ell\right)\mathcal D_a^{\ell+\frac12}
\Delta_{\R^n}^\ell.
\end{eqnarray*}
\end{lem}

\begin{proof}
By Theorems \ref{thm:2} and \ref{thm:GGC},
\begin{align*}
\mathcal{D}^{i\to i-1}_{\frac{n}{2}-i+\ell, a} &=\restn \circ
\left(A d_{\R^n}d^*_{\R^n}+Bd^*_{\R^n} +C\iotan\right),\\
\mathcal{T}^{(i)}_{2\ell}&=pd_{\R^n}d^*_{\R^n}+q,
\end{align*}
if we set
\begin{equation*}
A=-\mathcal D_{a-2}^{\ell+\frac32},\; B=-\gamma(\ell+\frac12,a)\mathcal D_{a-1}^{\ell+\frac32},\; C=-\frac12\left(\frac n2-i-\ell\right)\mathcal D_a^{\ell+\frac12},\;
p=-2\ell,\; q=-\left(\frac n2-i+\ell\right).
\end{equation*}
Applying Lemma \ref{lem:1524105}, we get \eqref{eqn:DTPQR} if we set
\begin{eqnarray*}
P'&=&\left(\frac n2-i-\ell\right)\left(\mathcal D_{a-2}^{\ell+\frac32}\Delta_{\R^n}+\ell
\mathcal D_a^{\ell+\frac12}\right)\Delta_{\R^n}^{\ell-1},\\
Q'&=&\left(\frac n2-i-\ell\right)\left(\gamma(\ell+\frac12,a)
\mathcal D_{a-1}^{\ell+\frac32}\Delta_{\R^n}+\ell
\mathcal D_{a}^{\ell+\frac12}\frac{\partial}{\partial x_n}\right)\Delta_{\R^n}^{\ell-1},\\
R'&=&\frac12\left(\frac n2-i-\ell\right)\left(\frac n2-i+\ell\right)\mathcal D_a^{\ell+\frac12}\Delta_{\R^n}^\ell.
\end{eqnarray*}
Then the lemma follows from the three-term relations \eqref{eqn:152592} for $P'$,
\eqref{eqn:152580} for $Q'$ with $\nu=\ell+\frac12$.
\end{proof}

We are ready to give a proof of Proposition \ref{prop:152599}.

\begin{proof}[Proof  of Proposition \ref{prop:152599}] 
The assertions $P=c_i(\ell)P'$, $Q=c_i(\ell)Q'$, and $R=c_i(\ell)R'$ are now 
reduced to the factorization identities for scalar-valued differential operators 
\index{B}{Juhl's operator}
(Juhl's operators) which were
proved in Lemma \ref{lem:152475} (1), (2), and Proposition \ref{prop:1524114} (1), respectively.
\end{proof}

Thus the proof of Theorem \ref{thm:factor1} (1) is completed.

%%%%%%%%%%%%%%%%%%%%%%%%%%%%%%%%%%%%%%%%%%%%%%%
\subsection{Proof of Theorem \ref{thm:factor1} (2)} \label{subset:factor1dual}

We 
deduce the second statement of Theorem \ref{thm:factor1} from the first one by using
\index{B}{Hodge star operator}
the Hodge star operator. By Theorem \ref{thm:factor1} (1) with $\tilde i:=n-i$,
we have
\begin{equation}\label{eqn:DTtilde}
\mathcal D_{\frac n2-\tilde i+\ell,a}^{\tilde i\to \tilde i-1}\circ\mathcal T_{2\ell}^{(\tilde i)}=
-\left(\frac n2-\tilde i-\ell\right)K_{\ell,a}
\mathcal D_{\frac n2-\tilde i-\ell,a+2\ell}^{\tilde i\to \tilde i-1}.
\end{equation}
We recall from Proposition \ref{prop:152520}
$$
*_{\R^n}\circ \mathcal T_{2\ell}^{(\tilde i)}\circ(*_{\R^n})^{-1}=-\mathcal T_{2\ell}^{(\tilde i)},
$$
and from \eqref{eqn:Diistar} with $\tilde i:=n-i$
\begin{eqnarray*}
\mathcal D_{\frac n2-i+\ell,a}^{i\to i}&=& (-1)^{n-1}*_{\R^{n-1}}\circ
\mathcal D_{\frac n2-\tilde i+\ell,a}^{\tilde i\to \tilde i-1}\circ (*_{\R^n})^{-1},\\
\mathcal D_{\frac n2-i+\ell,a+2\ell}^{i\to i}&=& (-1)^{n-1}*_{\R^{n-1}}\circ
\mathcal D_{\frac n2-\tilde i+\ell,a+2\ell}^{\tilde i\to \tilde i-1}\circ (*_{\R^n})^{-1}.
\end{eqnarray*}
Then Theorem \ref{thm:factor1} (2) follows from \eqref{eqn:DTtilde} by applying $*_{\R^{n-1}}$ from the left and $(*_{\R^n})^{-1}$ from the right.

\subsection{Proof of Theorem \ref{thm:factor2} (1)}\label{subsec:pffactor2}

This section gives a proof of Theorem \ref{thm:factor2} (1).
As in the proof of Theorem \ref{thm:factor1} (2), we shall reduce the proof 
to an analogous identity for the scalar-valued case
(Proposition \ref{prop:1524114} (2)). 
For this, we need some computation for matrix-valued differential operators
that are stated in the lemmas below.

\begin{lem}\label{lem:152481}
Suppose that $p,q,r,A,B $  and $C$ are (scalar-valued) differential operators with constant coefficients on $\R^n$. We set
$$
R:=-pA\Delta_{\R^{n-1}} - pA\frac{\partial^2}{\partial x_n^2}+ pC+rA-qB+qA\frac{\partial}{\partial x_n}.
$$
Then the following identity
holds:
$$
(pd_{\R^n}d^*_{\R^n}+qd_{\R^n}\iota_{\frac{\partial}{\partial x_n}}+r)
\circ(Ad_{\R^n}d^*_{\R^n}\iota_{\frac{\partial}{\partial x_n}}+
Bd^*_{\R^n}+C\iota_{\frac{\partial}{\partial x_n}})
=Rd_{\R^n}d^*_{\R^n}\iota_{\frac{\partial}{\partial x_n}}
+rBd^*_{\R^n}+rC\iota_{\frac{\partial}{\partial x_n}}.
$$
\end{lem}

\begin{proof}
Direct computation by Lemma \ref{lem:152457}, as in the proof of 
Lemma \ref{lem:1524105}.
\end{proof}

\begin{lem}\label{lem:152480a}
Let $p,q$, and $r$ be (scalar-valued) differential operators of 
0th, first, and second order, respectively, given by
\begin{equation*}
p=-2\ell,\quad q=-2\ell\frac{\partial}{\partial x_n},\quad r=-\left(\frac{n+1}2+\ell-i\right)\Delta_{\R^{n-1}}.
\end{equation*}
Then,
\begin{equation*}
\mathcal {T'}_{2\ell}^{(i-1)}\circ \mathrm{Rest}_{x_n=0}
=\mathrm{Rest}_{x_n=0}\circ(pd_{\R^n}d^*_{\R^n}+qd_{\R^n}
\iota_{\frac{\partial}{\partial x_n}}+r).
\end{equation*}
\end{lem}

\begin{proof}
The identity follows from the commutation relations among 
$\restn$, $d_{\R^n}d^*_{\R^n}$, and $d^*_{\R^n}d_{\R^n}$
given in Lemma \ref{lem:152456} (3) and (4).
\end{proof}

Let $u:=\frac{n-1}2-i-a-\ell$. Then by the formula \eqref{eqn:Dii1} of 
$\mathcal{D}^{i\to i-1}_{u,a}$, we have
$$
\mathcal D_{u,a}^{i\to i-1}=\mathrm{Rest}_{x_n=0}
\circ(Ad_{\R^n}d^*_{\R^n}\iota_{\frac{\partial}{\partial x_n}}
+Bd^*_{\R^n}+C\iota_{\frac{\partial}{\partial x_n}}),
$$
where
\begin{equation}\label{eqn:ABC2}
A:=-\mathcal D_{a-2}^{-a-\ell+1},\quad B:=-\gamma(-a-\ell,a)\mathcal D_{a-1}^{-a-\ell+1},\quad
C:=-\frac12\left(\frac{n+1}2-i+a+\ell\right)\mathcal D_a^{-a-\ell}.
\end{equation}
Then the composition
$$
\mathcal{T'}_{2\ell}^{(i-1)}\circ\mathcal D_{u,a}^{i\to i-1}=
\left( \mathcal{T'}_{2\ell}^{(i-1)}\circ \mathrm{Rest}_{x_n=0}\right)\circ 
(Ad_{\R^n}d^*_{\R^n}\iota_{\frac{\partial}{\partial x_n}}
+Bd^*_{\R^n}+C\iota_{\frac{\partial}{\partial x_n}})
$$
can be computed explicitly as follows.

\begin{lem}\label{lem:1214}
\begin{align*}
\mathcal{T'}_{2\ell}^{(i-1)}\circ\mathcal D_{u,a}^{i\to i-1}
&=\left(\frac{n+1}2+\ell-i\right)\mathrm{Rest}_{x_n=0}\circ
(\left(\ell+\left[\frac a2\right]\right)\mathcal D_a^{-a-\ell+1}\Delta_{\R^{n-1}}^{\ell-1}
d_{\R^n}d^*_{\R^n}\iota_{\frac{\partial}{\partial x_n}}\\
&\quad 
+\gamma(-a-\ell,a)\mathcal D_a^{-a-\ell+1}\Delta_{\R^{n-1}}^{\ell}d^*_{\R^n}
+\frac12\left(\frac{n+1}2-i+a+\ell\right)\mathcal D_a^{-a-\ell}
\Delta_{\R^{n-1}}^{\ell}\iota_{\frac{\partial}{\partial x_n}}).
\end{align*}
\end{lem}

\begin{proof}
The first two terms are derived directly from Lemmas \ref{lem:152481} 
and \ref{lem:152480a}. 
For the third term, we use three-term relations among $\mathcal{D}^\lambda_\ell$s
that were studied in Chapter \ref{sec:formulaD}. What we actually need is the claim below.
\end{proof}

\begin{claim}\label{claim:1215}
Suppose $p,q$ and $r$ are given as in Lemma \ref{lem:152480a} and $A,B$ and $C$ by
\eqref{eqn:ABC2}. 
Then the differential operator $R$ in Lemma \ref{lem:152481} amounts to
\begin{equation*}
\left(\frac{n+1}2+\ell-i\right)\left(\ell+\left[\frac a2\right]\right)\mathcal D_a^{-a-\ell+1}.
\end{equation*}
\end{claim}

\begin{proof}[Proof of Claim \ref{claim:1215}]
A direct computation shows 
\begin{eqnarray*}
R&=&\left(\frac{n+1}2-\ell-i\right)\mathcal D_{a-2}^{-a-\ell+1}\Delta_{\R^{n-1}}+
\left(\frac{n+1}2+a+\ell-i\right)\ell\,\mathcal D_{a}^{-a-\ell}\\
&-&2\ell\,\gamma(-a-\ell,a)\mathcal D_{a-1}^{-a-\ell+1}\frac{\partial}{\partial x_n}.
\end{eqnarray*}
Applying 
the three-term relation
\eqref{eqn:1522102} to 
$\mathcal{D}^{-a-\ell+1}_{a-1}\frac{\partial}{\partial x_n}$ in
the last term, we see
$$
R=\left(\frac{n+1}2+\ell-i\right)\left(\mathcal D_{a-2}^{-a-\ell+1}\Delta_{\R^{n-1}}+
\ell\,\mathcal D_{a}^{-a-\ell}\right).
$$
Finally, by the three-term relation \eqref{eqn:152562} with $\mu=-a-\ell$, 
we get the claim.
\end{proof}

We are ready to complete the proof of Theorem \ref{thm:factor2} (1). By the definition 
\eqref{eqn:Dii1} of $\mathcal{D}^{i\to i-1}_{u,b}$ again,
\begin{eqnarray*}
&&\mathcal D_{u,a+2\ell}^{i\to i-1}\\&=&\mathrm{Rest}_{x_n=0}\circ
\left(-\mathcal D_{a+2\ell-2}^{-a-\ell+1}d_{\R^n}d^*_{\R^n}
\iota_{\frac{\partial}{\partial x_n}}-\gamma(-a,a)
\mathcal D_{a+2\ell-1}^{-a-\ell+1}-\frac12\left(\frac{n+1}2-i+a+\ell\right)\mathcal D_{a+2\ell}^{-a-\ell}\iota_{\frac{\partial}{\partial x_n}}\right).
\end{eqnarray*}
Substituting the formul\ae{} for the (scalar) differential operators $\mathcal D_\ell^\lambda$ from Lemma \ref{lem:152475}, we have
\begin{eqnarray*}
\mathcal D_{u,a+2\ell}^{i\to i-1}
&=&-K_{\ell,a}^{-1}\mathrm{Rest}_{x_n=0}\circ
\left(\left(\ell+\left[\frac a2\right]\right)\right.
\mathcal D_a^{-a-\ell+1}\Delta_{\R^{n-1}}^{\ell-1}d_{\R^n}d^*_{\R^n}
\iota_{\frac{\partial}{\partial x_n}}
\\
&&
\left.
+\gamma(-a-\ell,a)\mathcal D_{a-1}^{-a-\ell+1}\Delta_{\R^{n-1}}^{\ell}d^*_{\R^n}
+\frac12\left(\frac{n+1}2-i+a+\ell\right)\mathcal D_{a}^{-a-\ell}\Delta_{\R^{n-1}}^{\ell}
\iota_{\frac{\partial}{\partial x_n}}\right).
\end{eqnarray*}
Comparing this with Lemma \ref{lem:1214}, we have completed the proof of Theorem \ref{thm:factor2} (1).

%%%%%%%%%%%%%%%%%%%%%%%%%%%%%%%%%%%%%%%%%%%%%%
\subsection{Proof of Theorem \ref{thm:factor2} (2)}\label{subsec:pffactor2b}

We deduce the second statement of Theorem \ref{thm:factor2} from the 
first one by using the Hodge operator $*$. 
By Theorem \ref{thm:factor2} (1) with $\tilde i:=n-i$ and
$\tilde u:=\frac{n-1}2-\tilde i-a-\ell$, we have
\begin{equation}\label{eqn:TDtilde}
\mathcal {T'}_{2\ell}^{(\tilde i-1)}\circ 
\mathcal {D}_{\tilde u,a}^{\tilde i\to \tilde i-1}=
-\left(\frac{n+1}2-\tilde i+\ell\right)
K_{\ell,a}
\mathcal {D}_{\tilde u,a+2\ell}^{\tilde i\to \tilde i-1}.
\end{equation}
We recall from Proposition \ref{prop:152520}
\begin{eqnarray*}
*_{\R^{n-1}}\circ \mathcal {T'}_{2\ell}^{(\tilde i-1)}\circ (*_{\R^{n-1}})^{-1}=-\mathcal {T'}_{2\ell}^{(\tilde i)},
\end{eqnarray*}
and from\eqref{eqn:Dii}
\begin{eqnarray*}
(-1)^{n-1}*_{\R^{n-1}}\circ \mathcal {D}_{\tilde u,a}^{\tilde i\to \tilde i-1}\circ (*_{\R^{n}})^{-1}&=&
\mathcal {D}_{u,a}^{i\to i},\\
(-1)^{n-1}*_{\R^{n-1}}\circ \mathcal {D}_{\tilde u,a+2\ell}^{\tilde i\to \tilde i-1}\circ (*_{\R^{n}})^{-1}&=&
\mathcal {D}_{u,a+2\ell}^{i\to i}.
\end{eqnarray*}
Then Theorem \ref{thm:factor2} (2) follows by applying $*_{\R^{n-1}}$ from the left and
$(*_{\R^{n}})^{-1}$ from the right to \eqref{eqn:TDtilde}.

%%%%%%%%%%%%%%%%%%%%%%%%%%%%%%%%%%%%%%%%%%%%%%
\subsection{Proof of Theorem \ref{thm:152673}}\label {subsec:idi}
In this section, we complete the proof of Theorem \ref{thm:152673}.

(1) We first show the following lemma: 
\begin{lem}\label{lem:160525}
Let $0\leq i \leq n-1$ and $a\in\N$.
We set $\mu:=i-\frac{n-3}2$ and define the following scalar-valued differential operators:
\begin{eqnarray*}
P'&:=&-\mathcal D_{a-2}^{\mu+1}\frac{\partial}{\partial x_n}+\gamma(\mu,a)\mathcal D_{a-1}^{\mu+1},\\
Q'&:=&-\mathcal D_{a-2}^{\mu+1}\Delta_{\R^n}-\left(i+1-\frac n2\right)\mathcal D_a^\mu,\\
R'&:=&\gamma(\mu,a)\mathcal D_{a-1}^{\mu+1}\Delta_{\R^n}+\left(i+1-\frac n2\right)\mathcal D_a^\mu \frac{\partial}{\partial x_n}.
\end{eqnarray*}
Then 
$$
\mathcal D_{0,a}^{i+1\to i}\circ d_{\R^n}=\mathrm{Rest}_{x_n=0}\circ
\left(P' d_{\R^n}d^*_{\R^n}+Q'd_{\R^n}\iotan +R'\right).
$$
\end{lem}

\begin{proof}
The lemma follows from the formula \eqref{eqn:Dii1} for $\mathcal D_{0,a}^{i+1\to i}$
by \eqref{eqn:Lapform} and Lemma \ref{lem:152457} (1).
\end{proof}

%%Kobayashi, 5/25/2016, Leave the three equations below for a while.
%
%\begin{eqnarray*}
%d_{\R^n}d^*_{\R^n}\iotan d_{\R^n}
%&=&\frac{\partial}{\partial x_n}d_{\R^n}d^*_{\R^n}+\Delta_{\R^n}d_{\R^n}\iotan,\\
%d^*_{\R^n}d_{\R^n}&=&-d_{\R^n}d^*_{\R^n}-\Delta_{\R^n},\\
%\iotan d_{\R^n}&=&\frac{\partial}{\partial x_n}-d_{\R^n}\iotan,
%\end{eqnarray*}

We observe that $i+1-\frac n2=\mu-\frac12$. Then applying
the three-term relations
\eqref{eqn:152617}, \eqref{eqn:152592}, and \eqref{eqn:152580} 
for $\mathcal{D}^\mu_a$s, we see that the scalar-valued operators
$P'$, $Q'$, and $R'$ in Lemma \ref{lem:160525} amounts to 

\begin{eqnarray*}
P'&=&\gamma(\mu-\frac12,a) \mathcal D_{a-1}^\mu,\\
Q'&=&-\left(\mu+\left[\frac a2\right]-\frac12\right)\mathcal D_a^{\mu-1},\\
R'&=&\frac12(a+1)\gamma(\mu-\frac12,a)\mathcal D_{a+1}^{\mu-1},
\end{eqnarray*}
respectively. In view of the formula \eqref{eqn:Dii} of $\mathcal{D}^{i\to i}_{0,a+1}$ 
and of the following elementary identity of 
\index{A}{1Cgamma@$\gamma(\mu,a)$}
$\gamma(\mu,a)$ (see \eqref{eqn:gamma})
\begin{equation*}
\gamma(\mu-\frac32,a+1)\gamma(\mu-\frac12,a)=\mu+\left[\frac a2\right]-\frac12,
\end{equation*}
we obtain the first statement of Theorem \ref{thm:152673}.

(2) By the formula \eqref{eqn:DiB} for 
${\mathcal D}^{i\to i}_{0,a}$, we get
$ {\mathcal D}^{i\to i}_{0,a}\circ d_{\R^n}=0$ from 
Lemma \ref{lem:152456} (1) and 
$d^2_{\R^n}=0$.

(3) By Theorem \ref{thm:152673} (1) with $i$ replaced by $n-i$, we have
$$
\mathcal D_{0,a}^{n-i+1\to n-i}\circ d_{\R^n}
=\gamma (-i+1+\frac n2,a)\mathcal D_{0,a+1}^{n-i\to n-i}.
$$
Then we apply $*_{\R^{n-1}}$ from the left and 
$(*_{\R^n})^{-1}$ from the right to the above identity, and use the following identities
\begin{alignat*}{2}
*_{\R^{n-1}}\circ\mathcal D_{0,a}^{n-i+1\to n-i}\circ(*_{\R^n})^{-1}
\;&=\; (-1)^{n-1}
\mathcal D_{n-2i,a}^{i-1\to i-1}&&,\\
*_{\R^{n}}\circ d_{\R^n}\circ(*_{\R^n})^{-1}
\;&=\;  (-1)^{n-i+1}d^*_{\R^n}\quad&&\mathrm{on}\quad\mathcal E^i(\R^n),\\
*_{\R^{n-1}}\circ\mathcal D_{0,a+1}^{n-i\to n-i}\circ(*_{\R^n})^{-1}
\;&= (-1)^{i+1}
\mathcal D_{n-2i,a}^{i\to i-1}\quad&&\mathrm{on}\quad\mathcal E^i(\R^n),
\end{alignat*}
as in the proof of Theorem \ref{thm:factor2} (2).
Now the third statement of Theorem \ref{thm:152673} follows.

(4) By the formula \eqref{eqn:Dii1} for $ {\mathcal D}^{i-1\to i-2}_{n-2i+2,0}$, we get
$ {\mathcal D}^{i-1\to i-2}_{n-2i+2,0}\circ d_{\R^n}^*=0$ from 
Lemma \ref{lem:152457} (2) and $\left(d^*_{\R^n}\right)^2=0$.

Thus Theorem \ref{thm:152673} has been proved.
%%%%%%%%%%%%%%%%%%%%%%%%%%%%%%%%%%%%%%%%%%%%%%
\subsection{Proof of Theorem \ref{thm:152557}}\label{subsec:dii}

In this section we complete the proof of Theorem \ref{thm:152557}. 

\begin{proof}[Proof of Theorem \ref{thm:152557} (1)] We set
\begin{equation*}
u:=-a-1,\quad\mu:=u+i-\frac{n-1}2.
\end{equation*}
By the first expression of \eqref{eqn:Dii1},
we have
\begin{equation*}
d_{\R^{n-1}}\circ\mathcal D_{-a-1,a}^{i\to i-1}
=\mathrm{Rest}_{x_n=0}\circ\left(-\gamma(\mu,a)
\mathcal D_{a-1}^{\mu+1} d_{\R^n}d^*_{\R^n}
+\frac12(u+2i-n)\mathcal D_a^\mu d_{\R^n}\,\iota_{\frac{\partial}{\partial x_n}}\right).
\end{equation*}
On the other hand, by the formula \eqref{eqn:Dii}
of $\mathcal{D}^{i\to i}_{u,a}$, we have
\begin{equation*}
\mathcal D_{-a-1,a+1}^{i\to i}
=\mathrm{Rest}_{x_n=0}\circ\left(\mathcal D_{a-1}^{\mu+1} d_{\R^n}d^*_{\R^n}
-\gamma(\mu-\frac{1}{2},a+1)\mathcal D_a^{\mu} d_{\R^n}\,\iota_{\frac{\partial}{\partial x_n}}\right).
\end{equation*}
It follows from these formul\ae{} that
we get
\begin{eqnarray}
&\quad& \quad d_{\R^{n-1}}\circ\mathcal D_{-a-1,a}^{i\to i-1}+\gamma(\mu,a)\mathcal D_{-a-1,a+1}^{i\to i}=\label{eqn:dii}\\
&&\mathrm{Rest}_{x_n=0}
\circ\left(\frac12(u+2i-n)
-\gamma(\mu,a)\gamma(\mu-\frac{1}{2},a+1)\right)
\mathcal{D}^{\mu}_a d_{\R^n}\iotan.
\nonumber
\end{eqnarray}

By using the following elementary identity
\index{A}{1Cgamma@$\gamma(\mu,a)$}
\begin{equation*}
\gamma(\mu,a)\gamma(\mu-\frac{1}{2},a+1) = \mu + \frac{a}{2},
\end{equation*}
we see that the right-hand side of \eqref{eqn:dii} 
vanishes.
\end{proof}

\begin{proof}[Proof of Theorem \ref{thm:152557} (2)]
It follows from the formula \eqref{eqn:Dii} of $\mathcal{D}^{i \to i}_{u,a}$ and 
Lemma \ref{lem:152456} (1) that we have
\begin{equation*}
d_{\R^{n-1}}\circ \mathcal D^{i\to i}_{u,a}=\frac12(u+a)\mathrm{Rest}_{x_n=0}\circ
\mathcal D^{\mu}_{a} \circ d_{\R^{n}}.
\end{equation*}
Hence $d_{\R^{n-1}} \circ \mathcal{D}^{i\to i}_{-a,a} = 0$.
\end{proof}

\begin{proof}[Proof of Theorem \ref{thm:152557} (3)]
By Theorem \ref{thm:152557} (1) with $i$ replaced by $n-i$, we have
\begin{equation}\label{eqn:dDdual}
d_{\R^{n-1}}\circ \mathcal D_{-a-1,a}^{n-i\to n-i-1}
=-\gamma( -a-i+\frac{n-1}2, a)
\mathcal D_{-a-1,a+1}^{n-i\to n-i}.
\end{equation}
We recall from \eqref{eqn:dstardef} 
\begin{equation*}
*_{\R^{n-1}}\circ d_{\R^{n-1}}\circ(*_{\R^{n-1}})^{-1}=(-1)^{n-i}d^*_{\R^{n-1}}
\quad \text{on $\mathcal E^i(\R^{n-1})$},
\end{equation*}
and from \eqref{eqn:Diistar}
\begin{eqnarray*}
*_{\R^{n-1}}\circ \mathcal D_{-a-1,a}^{n-i\to n-i-1}\circ(*_{\R^{n}})^{-1}&=&(-1)^{n-1}
\mathcal D_{n-2i-a-1,a}^{i\to i},\\
(*_{\R^{n-1}})^{-1}\circ \mathcal D_{-a-1,a+1}^{n-i\to n-i}\circ *_{\R^{n}}&=&(-1)^{n-1}
\mathcal D_{n-2i-a-1,a+1}^{i\to i-1}.
\end{eqnarray*}
Since $(*_{\R^{n}})^2=(-1)^{i(n-i)}\mathrm{id}$ on $\mathcal E^i(\R^{n})$ and 
$(*_{\R^{n-1}})^2=(-1)^{(n-i)(i-1)}\mathrm{id}$ on $\mathcal E^{n-i}(\R^{n-1})$, the last formula yields
$$
*_{\R^{n-1}}\circ \mathcal D_{-a-1,a+1}^{n-i\to n-i}\circ (*_{\R^{n}})^{-1}=(-1)^{i+1}
\mathcal D_{n-2i-a-1,a+1}^{i\to i-1}.
$$
Applying $*_{\R^{n-1}}$ from the left and $(*_{\R^{n}})^{-1}$ from the right to the identity \eqref{eqn:dDdual}, we get
$$
(-1)^{i+1}d^*_{\R^{n-1}}\mathcal D_{n-2i-a-1,a}^{i\to i}
=-(-1)^{i+1}\gamma( -a-i+\frac{n-1}2, a)
\mathcal D_{n-2i-a-1,a+1}^{i\to i-1}.
$$
Thus Theorem \ref{thm:152557} (3) is proved.
\end{proof}

\begin{proof}[Proof of Theorem \ref{thm:152557} (4)]
By the formula \eqref{eqn:Di-B} of $\mathcal{D}^{i\to i-1}_{u,a}$, we have
\begin{equation*}
d^*_{\R^{n-1}} \circ \mathcal{D}^{i\to i-1}_{n-2i-a,a} =
-d^*_{\R^{n-1}}\circ \restn \circ \mathcal{D}^{\mu+1}_{a-2} d^*_{\R^n}
\iotan d_{\R^n} 
\end{equation*}
because $(d^*_{\R^{n-1}})^2 = 0$. By Lemma \ref{lem:152457} (2) 
and Lemma \ref{lem:152456} (2), the right-hand side vanishes because
$(d^*_{\R^n})^2 = 0$ and $(\iotan)^2 = 0$.
\end{proof}

Alternatively, Theorem \ref{thm:152557} (3) and (4) can be deduced from
Theorem \ref{thm:152557} (1) and (2), respectively, by using the Hodge star operators.

%%%%%%%%%%%%%%%%%%%%%%%%%%%%%%%%%%%%%%%%%%%%%%
\subsection{Renormalized factorization identities}
\label{subsec:RFI}

In Section \ref{subsec:factor}, 
we have shown various
factorization identities for
(unnormalized)
symmetry breaking operators.
Now observe from Proposition \ref{prop:Dnonzero} that
$\mathcal{D}^{i \to i-1}_{u,a}$ and $\mathcal{D}^{i \to i}_{u,a}$
may vanish for specific parameters $(u,a,i)$. 
In this section, we discuss factorization identities for 
renormalized symmetry breaking operators 
$\widetilde{\mathcal{D}}^{i \to i-1}_{u,a}$ and $\widetilde{\mathcal{D}}^{i \to i}_{u,a}$
defined in \eqref{eqn:reDii1} and \eqref{eqn:reDii}, respectively.
We do not pursue finding all the constants explicitly for factorization 
identities in this section, as they are directly computed from the theorems
for the unnormalized case in Section \ref{subsec:factor} and 
from \eqref{eqn:DD1} and \eqref{eqn:DD2} below. 
Instead we shall formulate results for renormalized operators in a way that
we can benefit the following two advantages:
\begin{itemize}
\item to find some further nontrivial factorization identities that were stated 
as $T_Y \circ 0 = 0$ or $0 \circ T_X=0$ for unnormalized operators with 
specific parameters;
\item to determine exactly when the composition of two nonzero
intertwining operators vanish.
\end{itemize}

The latter view point plays an important role in the branching problem
for subquotients of principal series representations, see \cite{KS13}.
For this purpose, we use the notations $I(i,\lambda)_\alpha$ and 
$J(j,\nu)_\beta$ for principal series representations of 
$G=O(n+1,1)$ and $G'=O(n,1)$, respectively, instead of 
$\mathcal{E}^i(S^n)_{u,\delta}$ and $\mathcal{E}^j(S^{n-1})_{v,\eps}$
for the notation in conformal geometry.

By definitions \eqref{eqn:reDii1} and \eqref{eqn:Dii}
on one hand, and \eqref{eqn:reDii} and \eqref{eqn:Dii1}
on the other hand,
\index{A}{Dii5@$\widetilde {\mathcal D}_{u,a}^{i\to i }$}
we have
\index{A}{Dii4@$\widetilde {\mathcal D}_{u,a}^{i\to i-1}$}
\begin{align}
\mathcal{D}^{i\to i}_{u,a}
&=
\begin{cases}
\frac{1}{2}(a+u)\widetilde{\mathcal{D}}^{i\to i}_{u,a} & \text{if $i=0$ or $a=0$,}\\
\widetilde{\mathcal{D}}^{i\to i}_{u,a} & \text{if $i\neq 0$ and $a\neq 0$,}
\end{cases}\label{eqn:DD1}\\
&\nonumber \\
\mathcal{D}^{i\to i-1}_{u,a}
&=
\begin{cases}
\frac{1}{2}(a+u+2i-n)\widetilde{\mathcal{D}}^{i\to i-1}_{u,a} & \text{if $i=n$ or $a=0$,}\\
\widetilde{\mathcal{D}}^{i \to i-1}_{u,a} & \text{if $i\neq n$ and $a\neq 0$.}
\end{cases}\label{eqn:DD2}
\end{align}

Consider the following diagrams for $j=i$ and $j=i-1$:
\begin{eqnarray*}
\xymatrix{
I\left(i,\frac n2-\ell\right)_\alpha\ar[d]_{\mathcal T_{2\ell}^{(i)}}
\ar[drr]^{
\widetilde{\mathcal{D}}^{i\to j}_{\frac{n}{2}-i-\ell, a+2\ell}}
&& \\
I\left(i,\frac n2+\ell\right)_\alpha\ar[rr]_{
\widetilde{\mathcal{D}}^{i\to j}_{\frac{n}{2}-i+\ell,a}}&&
J\left(j,\frac n2+a+\ell\right)_\beta,
}
&&
\xymatrix{
I\left(i,\frac {n-1}2-a-\ell\right)_\alpha\ar[rr]^{
\widetilde{\mathcal{D}}^{i\to j}_{\frac{n-1}{2}-i-\ell -a,a}}
\ar[drr]_{
\widetilde{\mathcal{D}}^{i\to j}_{\frac{n-1}{2}-i-\ell-a, a+2\ell}}
&&
J\left(j,\frac{n-1}2-\ell\right)_\beta\ar[d]^{{\mathcal T'}^{(j)}_{2\ell}}\\
&&J\left(j,\frac{n-1}2+\ell\right)_\beta,
}
\end{eqnarray*}
where parameters $\delta$ and $\eps$ are chosen according to Theorem \ref{thm:1A} (iii).
In what follows, we put 
\begin{alignat*}{2}
p_\pm &\equiv p_\pm(i,\ell,a)&&:=
\begin{cases}
i\pm\ell-\frac n2& \; \quad \mathrm{if}\; a\neq0,\\
\pm2& \;\quad \mathrm{if}\; a=0,
\end{cases}\\
q&\equiv q(i,\ell, a)&&:=
\begin{cases}
i+\ell-\frac{n-1}2&\mathrm{if}\;i\neq0, a\neq0,\\
-2&\mathrm{if}\; i\neq0, a=0,\\
-\left(\ell+\frac{n-1}2\right)&\mathrm{if}\; i=0,
\end{cases}\\
r&\equiv r(i,\ell,a)&&:=
\begin{cases}
i-\ell-\frac{n+1}2&\mathrm{if}\;i\neq n, a\neq0,\\
2&\mathrm{if}\; i\neq n, a=0,\\
-\left(\ell+\frac{n+1}2\right)&\mathrm{if}\; i=n,
\end{cases}
\end{alignat*}
and recall 
\index{A}{Kell@$K_{\ell, a}$}
$\Kla =\prod_{k=1}^\ell\left(\left[\frac{a}{2}\right]+k\right)$.
Then the factorization identities for renormalized differential symmetry breaking operators 
$\widetilde{\mathcal{D}}^{i\to j}_{u,a}$
for $j\in\{i-1,i\}$ and
Branson's conformally covariant operators 
$\mathcal T_{2\ell}^{(i)}$ or $\mathcal{T'}_{2\ell}^{(j)}$ are given as follows.

\begin{thm}\label{thm:factor1'} 
Suppose $0\leq i\leq n-1, a\in\N$ and $\ell\in\N_+$. 
Then
\begin{eqnarray*}
&(1)&
\widetilde{\mathcal{D}}^{i \to i}_{\frac{n}{2}-i+\ell,a}
\circ 
\mathcal {T}_{2\ell}^{(i)}
=p_-
\Kla
\widetilde{\mathcal{D}}^{i\to i}_{\frac{n}{2}-i-\ell, a+2\ell}.\\
&(2)& 
\mathcal{T'}_{2\ell}^{(i)}\circ
\widetilde{\mathcal{D}}^{i\to i}_{\frac{n-1}{2}-i-\ell-a,a}
=
q\Kla
\widetilde{\mathcal{D}}^{i \to i}_{\frac{n-1}{2}-i-\ell-a,a+2\ell}.
\end{eqnarray*}
\end{thm}

\begin{thm}\label{thm:factor2'} 
Suppose $1\leq i\leq n$, $a\in\N$ and $\ell\in\N_+$.  Then
\begin{eqnarray*}
&(1)&
\widetilde{\mathcal{D}}^{i \to i-1}_{\frac{n}{2}-i+\ell,a}
\circ 
\mathcal {T}_{2\ell}^{(i)}=p_+
\Kla
\widetilde{\mathcal{D}}^{i\to i-1}_{\frac{n}{2}-i-\ell, a+2\ell}.\\
&(2)&\mathcal {T'}_{2\ell}^{(i-1)}\circ 
\widetilde{\mathcal{D}}^{i \to i-1}_{\frac{n-1}{2}-i-\ell-a,a}
=
r\Kla
\widetilde{\mathcal{D}}^{i\to i-1}_{\frac{n-1}{2}-i-\ell-a,a+2\ell}.\\
\end{eqnarray*}
\end{thm}

Since $K_{\ell, a}> 0$, by 
examining the condition on $(i,\ell,a)$ for 
the constants $p_\pm, q, r$ vanish, we see that
Theorems \ref{thm:factor1'} and \ref{thm:factor2'}
determine exactly when the composition of the two nonzero
operators vanish:

\begin{cor}
Suppose we are in the setting of Theorems \ref{thm:factor1'} or \ref{thm:factor2'}.
\begin{enumerate}

\item $\widetilde{\mathcal{D}}^{i \to i}_{\frac{n}{2}-i+\ell,a}
\circ \mathcal {T}_{2\ell}^{(i)}\neq 0$ except for the following case:
\begin{equation*}
\widetilde{\mathcal{D}}^{i \to i}_{0,a} \circ \mathcal{T}^{(i)}_{2i-n} =0
\quad
\emph{for $n$  even, 
$\frac{n+2}{2} \leq i \leq n-1$,
and $a \in \N_+$.}
\end{equation*}

\vskip 0.1in

\item $\widetilde{\mathcal{D}}^{i \to i-1}_{\frac{n}{2}-i+\ell,a}
\circ \mathcal {T}_{2\ell}^{(i)} \neq 0$ except for the following case:
\begin{equation*}
\widetilde{\mathcal{D}}^{i \to i-1}_{n-2i,a} \circ \mathcal{T}^{(i)}_{n-2i} =0
\quad 
\emph{for $n$ even, $1\leq i \leq \frac{n-2}{2}$, and $a \in \N_+$.}
\end{equation*}

\vskip 0.1in

\item
$\mathcal{T'}_{2\ell}^{(i-1)}\circ 
\widetilde{\mathcal{D}}^{i \to i- 1}_{\frac{n-1}{2}-i-\ell-a,a}\neq 0$
except for the following case:
\begin{equation*}
\mathcal{T'}^{(i-1)}_{2i-n-1} \circ \widetilde{\mathcal{D}}^{i \to i-1}_{n-2i-a,a} =0
\quad 
\emph{for
$n$ odd,
$\frac{n+3}{2}\leq i \leq n-1$, and
$a \in \N_+$.}
\end{equation*}

\vskip 0.1in

\item $\mathcal{T'}_{2\ell}^{(i)}\circ
\widetilde{\mathcal{D}}^{i\to i}_{\frac{n-1}{2}-i-\ell-a,a}\neq 0$ 
except for the following case:
\begin{equation*}
\mathcal{T'}^{(i)}_{n-2i-1} \circ \widetilde{\mathcal{D}}^{i \to i}_{-a,a} =0
\quad
\emph{for
$n$ odd, 
$1\leq i \leq \frac{n-3}{2}$, and
$a \in \N_+$.}
\end{equation*}

\end{enumerate}
\end{cor}

Next we consider the factorization identities 
corresponding to Theorems \ref{thm:152673} and \ref{thm:152557}.
By focusing on the ``two advantages" as we mentioned
at the beginning of this section, we omit giving the explicit
constants and formulate the results as follows:

\begin{thm}\label{thm:160422}
Let $n\geq 2$.

\begin{enumerate}
\item 
Let $a \in \N$ and $0\leq i \leq n-1$.
For 
\begin{equation*}
I(i,i)_i \stackrel{d}{\To} I(i+1, i+1)_{i+1} 
\stackrel{\widetilde{\mathcal{D}}^{i+1 \to i}_{0,a}}
{\relbar\!\relbar\!\relbar\!\To}
J(i, i+a+1)_i,
\end{equation*}
the following two conditions on $(i,a,n)$ are equivalent:
\begin{enumerate}
\item[(i)] $\widetilde{\mathcal{D}}^{i+1 \to i}_{0,a}\circ d= 0$;
\item[(ii)] 
$n$ is even,
$0\leq i < \frac{n-2}{2}$, and
$a = n-2i-2$.
\end{enumerate}
We note that, for $n$ even,
\begin{equation*}
\widetilde{\mathcal{D}}^{\frac{n}{2} \to \frac{n-2}{2}}_{0,0} \circ d 
= \widetilde{\mathcal{D}}^{\frac{n-2}{2} \to \frac{n-2}{2}}_{0,1}. 
\end{equation*}
\vskip 0.1in

\item
Let $a \in \N$ and $0\leq i \leq n-2$.
For 
\begin{equation*}
I(i,i)_i \stackrel{d}{\To} I(i+1, i+1)_{i+1} 
\stackrel{\widetilde{\mathcal{D}}^{i+1 \to i+1}_{0,a}}
{\relbar\!\relbar\!\relbar\!\relbar\!\To}
J(i+1, i+a+1)_{i+1},
\end{equation*}
the following two conditions on $(i,a,n)$ are equivalent:
\begin{enumerate}
\item[(i)] $\widetilde{\mathcal{D}}^{i+1 \to i+1}_{0,a}\circ d= 0$;
\item[(ii)] $a \in \N_+$.
\end{enumerate}
We note that
\begin{equation*}
\widetilde{\mathcal{D}}^{i+1 \to i+1}_{0,0} \circ d 
= \widetilde{\mathcal{D}}^{i \to i+1}_{0, 1}
\quad \text{for $0\leq i \leq n-2$}.
\end{equation*}
\vskip 0.1in

\item 
Let $a \in \N$ and $1 \leq i \leq n$.
For 
\begin{equation*}
I(i,n-i)_i \stackrel{d^*}{\To} I(i-1, n-i+1)_{i-1} 
\stackrel{\widetilde{\mathcal{D}}^{i-1 \to i-1}_{n-2i+2, a}}
{\relbar\!\relbar\!\relbar\!\relbar\!\To}
J(i-1, n+i+a+1)_{i-1},
\end{equation*}
the following two conditions on $(i,a,n)$ are equivalent:
\begin{enumerate}
\item[(i)] $\widetilde{\mathcal{D}}^{i-1 \to i-1}_{n-2i+2, a} \circ d^*=0$;
\item[(ii)]
$n$ is even, $\frac{n+2}{2} < i \leq n$,
and $a = 2i-n-2$.
\end{enumerate}
We note that, for $n$ even,
\begin{equation*}
\widetilde{\mathcal{D}}^{\frac{n}{2} \to \frac{n}{2}}_{0, 0} \circ d^*
=
-\widetilde{\mathcal{D}}^{\frac{n+2}{2}\to \frac{n}{2}}_{-2,1}.
\end{equation*}
\vskip 0.1in

\item
Let $a \in \N$ and $2\leq i \leq n$.
For
\begin{equation*}
I(i,n-i)_i \stackrel{d^*}{\To} 
I(i-1, n-i+1)_{i-1} \stackrel{\widetilde{\mathcal{D}}^{i-1 \to i-2}_{n-2i+2,a}}
{\relbar\!\relbar\!\relbar\!\relbar\!\To}
J(i-2, n-i+a+1)_{i-2},
\end{equation*}
the following two conditions on $(i,a,n)$ are equivalent:
\begin{enumerate}
\item[(i)] $\widetilde{\mathcal{D}}^{i-1 \to i-2}_{n-2i+2,a} \circ d^* =0$;
\item[(ii)] $a \in \N_+$.
\end{enumerate}
We note that 
\begin{equation*}
\widetilde{\mathcal{D}}^{i-1 \to i-2}_{n-2i+2,0} \circ d^*
=
\widetilde{\mathcal{D}}^{i \to i-2}_{n-2i, 1}
\quad \text{for $2\leq i \leq n$}.
\end{equation*}
\vskip 0.1in

\item Let $a \in \N$ and $1 \leq i \leq n-1$. For 
\begin{equation*}
I(i,i-a-1)_i \stackrel{\widetilde{\mathcal{D}}^{i\to i-1}_{-a-1,a}}
{\relbar\!\relbar\!\relbar\!\relbar\!\To}
J(i-1,i-1)_{i-1} \stackrel{d}{\To} 
J(i,i)_i,
\end{equation*}
the following two conditions on $(i,a,n)$ are equivalent:
\begin{enumerate}
\item[(i)] $d \circ \widetilde{\mathcal{D}}^{i\to i-1}_{-a-1,a} = 0$;
\item[(ii)] $n$ is odd, $\frac{n+1}{2} < i \leq n-1$, $a = 2i-n-1$.
\end{enumerate}
We note that, for $n$ odd,
\begin{equation*}
d \circ \widetilde{\mathcal{D}}^{\frac{n+1}{2}\to \frac{n-1}{2}}_{-1,0} 
=
-\widetilde{\mathcal{D}}^{\frac{n+1}{2}\to \frac{n+1}{2}}_{-1,1} .
\end{equation*}
\vskip 0.1in

\item
Let $a \in \N$ and $0 \leq i \leq n-2$. For
\begin{equation*}
I(i,i-a)_i \stackrel{\widetilde{\mathcal{D}}^{i \to i}_{-a,a}}
{\relbar\!\relbar\!\relbar\!\To}
J(i,i)_i \stackrel{d}{\To} J(i+1,i+1)_{i+1},
\end{equation*}
the following two conditions on $(i,a,n)$ are equivalent:
\begin{enumerate}
\item $d \circ \widetilde{\mathcal{D}}^{i \to i}_{-a,a}=0$;
\item $a \in \N_+$ and $1 \leq i \leq n-2$.
\end{enumerate}
We note that 
\begin{equation*}
d \circ \widetilde{\mathcal{D}}^{i \to i}_{-a,a}
=
\widetilde{\mathcal{D}}^{i \to i+1}_{-a, a+1}
\quad
\emph{if $a=0$ or $i=0$.}
\end{equation*}
\vskip 0.1in

\item
Let $a \in \N$ and $0\leq i \leq n-1$.
For 
\begin{equation*}
I(i,n-i-a-1)_i 
\stackrel{\widetilde{\mathcal{D}}^{i \to i}_{n-2i-a-1,a}}
{\relbar\!\relbar\!\relbar\!\relbar\!\relbar\!\To}
I(i,n-i-1)_i
\stackrel{d^*}{\To}
J(i-1, n-i)_{i-1},
\end{equation*}
the following two conditions on $(i,a,n)$ are equivalent:
\begin{enumerate}
\item[(i)] $d^* \circ \widetilde{\mathcal{D}}^{i \to i}_{n-2i-a-1,a} = 0$\
\item[(ii)] $n$ is odd, $0\leq i < \frac{n-1}{2}$, and $a = n-2i-1$.
\end{enumerate}
We note that, for $n$ odd,
\begin{equation*}
d^* \circ \widetilde{\mathcal{D}}^{\frac{n-1}{2} \to \frac{n-1}{2}}_{0,0} 
=
-\widetilde{\mathcal{D}}^{\frac{n-1}{2} \to \frac{n-3}{2}}_{0, 1}.
\end{equation*}
\vskip 0.1in

\item
Let $a \in \N$ and $2\leq i \leq n$.
For
\begin{equation*}
I(i, n-i-a)_i 
\stackrel{\widetilde{\mathcal{D}}^{i\to i-1}_{n-2i-a,a}}
{\relbar\!\relbar\!\relbar\!\relbar\!\To}
J(i-1,n-i)_{i-1}
\stackrel{d^*}{\To}
J(i-2, n-i+1)_{i-2},
\end{equation*}
the following two conditions on $(i,a,n)$ are equivalent:
\begin{enumerate}
\item $d^* \circ \widetilde{\mathcal{D}}^{i\to i-1}_{n-2i-a,a}=0$;
\item $a\in \N_+$ and $2\leq i \leq n-1$.
\end{enumerate}
We note that 
\begin{equation*}
d^* \circ \widetilde{\mathcal{D}}^{i \to i-1}_{n-2i-a,a}
=
-\widetilde{\mathcal{D}}^{i\to i-2}_{n-2i-a,a+1} 
\quad
\emph{if $a=0$ or $i=n$.}\\
\end{equation*}

\end{enumerate}
\end{thm}

\begin{proof}
Each equivalence in (1)-(8) is shown by the corresponding factorization
identities for unnormalized operators given in Theorems \ref{thm:152673}
or \ref{thm:152557}, and by \eqref{eqn:DD1} and \eqref{eqn:DD2}.

The factorization identities for specific parameters can be verified directly
from the definition \eqref{eqn:reDii1}-\eqref{eqn:160462} of the renormalized
operators $\widetilde{\mathcal{D}}^{i\to j}_{u,a}$. In fact, these operators for the 
specific values in Theorem \ref{thm:160422} are ``degenerate" and of much simple
forms.
\end{proof}

\newpage
%%%%%%%%%%%%%%%%%%%%%%%%%%%%%%%%%%%%%%%%%%%%%%
\section{Appendix: Gegenbauer polynomials}\label{sec:appendix}

This chapter collects some properties of the Gegenbauer polynomials that we use 
throughout the paper, in particular,
in the proof of the explicit formul{\ae} for symmetry breaking differential operators 
(Theorems \ref{thm:2}, \ref{thm:2ii}, \ref{thm:2ii+1}, and \ref{thm:2ii-2})  
and the factorization identities for special parameters 
(Theorems \ref{thm:factor1}, \ref{thm:factor2}, and \ref{thm:152673}).
In Section \ref{subsec:gis}, we give a proof of Theorem \ref{thm:gis} that 
determines solutions to the F-system for symmetry breaking operators from
$I(i,\lambda)_\alpha$ to $J(i-1,\nu)_\beta$ $(2\leq i \leq n)$.

%%%%%%%%%%%%%%%%%%%%%%%%%%%%%%%%%%%%%%%%%%%%%%
\subsection{Normalized Gegenbauer polynomials}\label{subsec:Apx1}

For $\lambda\in\C$ and $\ell\in\N$, 
\index{B}{Gegenbauer polynomial|textbf}
the Gegenbauer (or ultraspherical) polynomial $C_\ell^\lambda(z)$ is
given by the following formula 
(\cite[6.4]{AAR99}, \cite[3.15 (2)]{EMOT53}):
\begin{eqnarray*}
\index{A}{C1ell@$C_\ell^\mu(t)$, Gegenbauer polynomial|textbf}
C_\ell^ \lambda(z)
&:=&
\sum_{k=0}^{\left[\frac \ell2\right]}(-1)^k\frac{\Gamma(\ell-k+ \lambda)}
{\Gamma(\lambda)k!(\ell-2k)!}(2z)^{\ell-2k}\\
&=&
\frac{\Gamma(\ell+2 \lambda)}{\Gamma(2 \lambda)\Gamma(\ell+1)}
{}_2F_1\left(-\ell,\ell+2 \lambda; \lambda +\frac12;\frac{1-z}2\right).
\end{eqnarray*}
The generating function for $C^\lambda_\ell(z)$ is 
\begin{equation}\label{eqn:Cgen}
(1-2zr+r^2)^{-\lambda} 
=
\sum_{k=0}^\infty \frac{(\lambda)_k}{k!}\frac{(2zr)^k}{(1+r^2)^{k+\lambda}}
=
 \sum_{\ell =0}^\infty C^\lambda_\ell(z)r^\ell,
\end{equation}
and $C^\lambda_\ell(z)$ solves
the Gegenbauer differential equation
\begin{equation*}
G^\lambda_\ell f(z)=0,
\end{equation*}
where $G^\lambda_\ell$ is 
the \index{B}{Gegenbauer differential operator|textbf}
Gegenbauer differential operator given by
\index{A}{Gell@$G^\lambda_\ell$, Gegenbauer differential operator|textbf}
\begin{equation}\label{eqn:Gegenop}
G^\lambda_\ell:=(1-z^2)\frac{d^2}{dz^2}-(2\lambda+1)z\frac d{dz}
+\ell(\ell+2\lambda).
\end{equation}

We note that $C_\ell^ \lambda(z)\equiv0$ if $\ell\geq1$ and $\lambda =0,-1,-2,\cdots,-\left[\frac{\ell-1}2\right].$
As in \cite{KP2}, we renormalize the Gegenbauer polynomial by
\index{A}{C1ell1@$\widetilde C_\ell^\mu(t)$,
renormalized Gegenbauer polynomial|textbf}
\begin{equation}\label{eqn:Gegen2}
\widetilde C_\ell^ \lambda(z):=\frac{\Gamma(\lambda)}{\Gamma\left(\lambda + \left[\frac{\ell+1}2\right]\right)}C_\ell^ \lambda(z)
=\frac{1}{\Gamma\left(\lambda+\left[\frac{\ell+1}{2}\right]\right)}
\sum_{k=0}^{\left[\frac{\ell}{2}\right]}(-1)^k
\frac{\Gamma(\ell-k+\lambda)}{k!(\ell-2k)!}(2z)^{\ell-2k}.
\end{equation}
Then $\widetilde C_\ell^ \lambda(z)$ is a nonzero polynomial of degree $\ell$
for all $\lambda\in\C$ and $\ell\in\N$. 
Here are the first five renormalized
Gegenbauer polynomials.
\begin{itemize}
\item $\widetilde C_0^ \lambda(z)=1.$
\item $\widetilde C_1^ \lambda(z)=2 z.$
\item $\widetilde C_2^ \lambda(z)=2(\lambda +1)z^2-1$.
\item $\widetilde C_3^ \lambda(z)=\frac43(\lambda +2)z^3-2z$.
\item $\widetilde C_4^ \lambda(z)=
 \frac23 (\lambda +2) (\lambda +3) z^4-2(\lambda +2)z^2+\frac12$.
\end{itemize}
Then the 
\index{B}{inflated polynomial}
$\ell$-inflated polynomial (see \eqref{eqn:Iag}) of 
$\widetilde{C}^\lambda_\ell(z)$ is given by 
\index{A}{Iell@$I_\ell$, $\ell$-inflated polynomial}
\begin{eqnarray} \nonumber
(I_\ell\widetilde C_\ell^ \lambda)(x,y)&=&x^{\frac\ell2} \widetilde C_\ell^ \lambda\left(\frac y{\sqrt x}\right)\\
&=&\label{eqn:Cxy}
\sum_{k=0}^{\left[\frac \ell2\right]}(-1)^k\frac{\Gamma(\ell-k+ \lambda)}
{\Gamma\left(\lambda +\left[\frac{\ell+1}2\right]\right)\Gamma(k+1)\Gamma(\ell-2k+1)}(2y)^{\ell-2k}x^{k}.
\end{eqnarray}
For instance, $(I_0\widetilde C_0^ \lambda)(x,y)=1$, $(I_1\widetilde C_1^ \lambda)(x,y)=2 y$,
$(I_2\widetilde C_2^ \lambda)(x,y)=2(\lambda +1)y^2-x$, etc.

From \eqref{eqn:Gegen2}, the coefficient of $z^\ell$ in 
$\widetilde C_\ell^ \lambda(z)$ is found to be
\begin{equation}\label{eqn:topC}
\frac{\Gamma(\lambda+\ell)2^\ell}{\Gamma\left(\lambda + \left[\frac{\ell+1}2\right]\right)\ell!}.
\end{equation}

The dimension of the space of polynomial solutions to the
Gegenbauer differential equation $G^\lambda_\ell f(z)=0$
is generically one, however, it jumps to two
when $\lambda - \frac{1}{2}\in\Z$
 and $1-2\ell\leq2\lambda\leq-\ell$, for which we have found an interpretation
 in the representation theory of $SL(2,\R)$ \cite{KP2}.
 The renormalized Gegenbauer polynomial 
$\widetilde C_\ell^ \lambda(z)$ is characterized among polynomial solutions by the following 
(\cite[Thm.\ 11.4]{KP2}):

\begin{fact}\label{fact:KPGegen} 
For all $\lambda\in\C$ and $ \ell\in\N$,
\index{A}{Polell[t]@$\mathrm{Pol}_\ell[t]_{\mathrm{even}}$}
\begin{equation*}
\left\{f(z)\in\mathrm{Pol}_\ell[z]_{\mathrm{even}} : G^\lambda_\ell f(z)=0\right\}=\C
\widetilde C_\ell^ \lambda(z).
\end{equation*}
See \eqref{eqn:gs} for the definition of $\mathrm{Pol}_\ell[z]_{\mathrm{even}}$.
\end{fact}

In \eqref{eqn:Rla128}, we introduced the 
\index{B}{imaginary Gegenbauer differential equation|textbf}
\emph{imaginary} Gegenbauer differential operator
\begin{equation*}
R_\ell^\lambda=-\frac{1}{2}\left((1+t^2)\frac{d^2}{dt^2} 
+ (1+2\lambda)t\frac{d}{dt}-\ell(\ell+2\lambda)\right),
\end{equation*}
which is related with the Gegenbauer
differential operator $G_\ell^\nu$ defined in \eqref{eqn:Gegenop} as follows:
\begin{lem}\label{lem:RGop}
Let $f(z)$ be a polynomial in $z$, and $g(t)=f(z)$ with $z:=e^{\frac{\pi\sqrt{-1}}2}t$. 
Then
\index{A}{Rell@$R_\ell^\lambda$, imaginary Gegenbauer differential
operator}
\index{A}{Gell@$G^\lambda_\ell$, Gegenbauer differential operator}
\begin{equation}\label{eqn:RG}
2\left(
R_\ell^\lambda g\right)(t)=\left( G_\ell^\lambda f\right) (z).
\end{equation}
\end{lem}
\begin{proof}
Direct from $\frac{d}{dt}=e^{\frac{\pi\sqrt{-1}}2}\frac d{dz}$.
\end{proof}

Therefore, Fact \ref{fact:KPGegen} implies the following:

\begin{lem}\label{lem:Gesol} 
For any $\lambda\in\C$ and $\ell\in\N$, the 
\index{A}{Polell[t]@$\mathrm{Pol}_\ell[t]_{\mathrm{even}}$}
$\mathrm{Pol}_\ell[t]_{\mathrm{even}}$-solution space of the ordinary 
differential equation $R_\ell^\lambda g(t)=0$ is one-dimensional. Moreover, it is spanned by
$\widetilde C_\ell^\lambda\left(e^{\frac{\pi\sqrt{-1}}2}t\right)$.
\end{lem}

%%%%%%%%%%%%%%%%%%%%%%%%%%%%%%%%%%%%%%%%%%%%%%
\subsection{Derivatives of Gegenbauer polynomials}\label{subsec:diffC}

For $\mu\in\C$ and $\ell\in N$, we recall from \eqref{eqn:gamma}
\index{A}{1Cgamma@$\gamma(\mu,a)$}
\begin{equation*}
\gamma(\mu,\ell)=
\frac{\Gamma\left(\mu+1+\left[\frac{\ell}{2}\right]\right)}
{\Gamma\left(\mu+\left[\frac{\ell+1}{2}\right]\right)}
=
\left\{
\begin{matrix*}[l]
1 &\mathrm{if}\,\ell\,\mathrm{is\, odd},\\
\mu+\frac{\ell}{2} &\mathrm{if}\,\ell\,\mathrm{is\, even}.
\end{matrix*}
\right.
\end{equation*}

We collect two formul\ae{} about the first derivative of the 
renormalized Gegenbauer polynomial $\widetilde{C}^\mu_\ell(z)$.

\begin{lem}\label{lem:1524102}
Let $\mu \in \C$ and $\ell \in \N$.
\begin{align}
\frac{d}{dz}\widetilde C_\ell^\mu(z)
&=2\gamma(\mu,\ell)\widetilde C_{\ell-1}^{\mu+1}(z), \label{eqn:dzC}\\
\left(z\frac{d}{dz}-\ell\right)
\widetilde C_\ell^\mu(z)&=2\widetilde C_{\ell-2}^{\mu+1}(z).\label{eqn:1524102}
\end{align}
\end{lem}
\begin{proof}
The first identity \eqref{eqn:dzC} 
follows from $\frac{d}{dz}C^\mu_\ell(z)=2\mu C^{\mu+1}_{\ell-1}(z)$
(see \cite[(6.4.15)]{AAR99}, \cite[3.15.2 (30)]{EMOT53} for example).
To see the second identity, 
\index{A}{1theta-z@$\vartheta_z=z\frac{d}{dz}$}
let $\vartheta_z:=z\frac{\partial}{\partial z}$ and $\vartheta_r=r\frac{\partial}{\partial r}$.
Applying $\vartheta_z-\vartheta_r$ to \eqref{eqn:Cgen}, we get
\begin{equation*}
2\lambda r^2 \sum^{\infty}_{\ell = 0} C^{\lambda+1}_\ell(z)r^\ell
=\sum_{\ell=0}^\infty (\vartheta_z -\ell)C^\lambda_\ell(z)r^\ell,
\end{equation*}
whence $(\vartheta_z-\ell)C^\lambda_\ell(z)=2\lambda C^{\lambda+1}_{\ell-2}(z)$.
By  \eqref{eqn:Gegen2}, we get $(\vartheta_z-\ell) \widetilde{C}^\lambda_\ell(z)
=2\widetilde{C}^{\lambda+1}_{\ell-2}(z)$.
\end{proof}

%%%%%%%%%%%%%%%%%%%%%%%%%%%%%%%%%%%%%%%%%%%%%%
\subsection{Three-term relations among renormalized Gegenbauer polynomials}
\label{subset:3C}

In this section we collect three-term relations for renormalized Gegenbauer 
polynomials $\widetilde{C}_\ell^\mu$ for $\mu \in \C$. Further identities 
for special values $\mu$ will be treated in the next section.

We begin with useful identities for Gegenbauer differential operators 
$G^\mu_\ell$ (see \eqref{eqn:Gegenop}):

\begin{lem}\label{lem:Gop}
Let $\mu \in \C$ and $\ell \in \N$.
\begin{align}
G_\ell^{\mu+1}- G_\ell^\mu&=-2\left( z\frac{d}{dz}-\ell\right).\label{eqn:Gdiff1}\\
G_\ell^{\mu+1}- G_{\ell-2}^{\mu+1}&=4(\mu+\ell).\label{eqn:Gdiff2}\\
G_\ell^\mu z-z G_{\ell-1}^{\mu+1}&=2\frac{d}{dz}. \label{eqn:Gaz}\\
G^{\mu-1}_\ell - G^\mu_{\ell-2} &=
2(\vartheta_z+\ell + 2\mu-2).\label{eqn:Gdiff4}\\
G_\ell^{\mu-1}(z^2-1)-(z^2-1)G_{\ell-2}^{\mu+1}&=-2(2\mu-1). \label{eqn:Gz2}\\
(z^2-1)^\ell G_\ell^{\frac12+\ell}(z^2-1)^{-\ell}&=G_{\ell+2\ell}^{\frac12-\ell}.
\label{eqn:1524111}
\end{align}
\end{lem}

\begin{proof}
The first four formul\ae{} are easily obtained by the definition \eqref{eqn:Gegenop}.
For instance, the third one is obtained by the following commutation relations
\begin{equation*}
\frac{d}{dz}z-z\frac{d}{dz}=1,\quad \frac{d^2}{dz^2}z-z\frac{d^2}{dz^2}=2\frac{d}{dz}
\end{equation*}
in the 
\index{B}{Weyl algebra} 
Weyl algebra $\mathcal{D}(\mathbb{C})=\mathbb{C}\left[z, \frac{d}{dz}\right]$.
To see the sixth one, we apply the following identities:
\begin{align*}
(z^2-1)^{\ell+1}\frac{d}{dz}(z^2-1)^{-\ell}&=
(z^2-1)\frac{d}{dz}-2\ell z,\\
(z^2-1)^{\ell+2}\frac{d^2}{dz^2}(z^2-1)^{-\ell}&=
(z^2-1)^2\frac{d^2}{dz^2}-4\ell z(z^2-1)\frac{d}{dz}+2\ell((2\ell+1)z^2+1).
\end{align*}
Now \eqref{eqn:1524111} follows from the definition \eqref{eqn:Gegenop}
of the operator $G^\lambda_\ell$.
\end{proof}

\begin{lem}\label{lem:152563}
Let $\ell\in\N$ and $\mu\in\C$. Then
\begin{equation}\label{eqn:152563}
(\mu+\ell)\widetilde C_\ell^\mu(z)+\widetilde C_{\ell-2}^{\mu+1}(z)
=\left(\mu+\left[\frac{\ell+1}2\right]\right)\widetilde C_\ell^{\mu+1}(z).
\end{equation}
\end{lem}

\begin{proof}
By the relations 
\eqref{eqn:Gdiff1} and \eqref{eqn:Gdiff2}
for Gegenbauer differential operators,
we have
\begin{equation*}
G_\ell^{\mu+1}((\mu+\ell)\widetilde C_\ell^\mu(z)
+\widetilde C_{\ell-2}^{\mu+1}(z))
=-2(\mu+\ell)\left(
\left( z\frac{d}{dz}-\ell\right)\widetilde C_\ell^\mu(z)
-2\widetilde C_{\ell-2}^{\mu+1}(z)\right)=0.
\end{equation*}
The second equality follows from \eqref{eqn:1524102}. 
Since $(\mu+\ell)\widetilde C_\ell^\mu(z)+\widetilde C_{\ell-2}^{\mu+1}(z)
\in\mathrm{Pol}_\ell[z]_{\mathrm{even}}$, 
according to Fact \ref{fact:KPGegen} there exists $A\in \C$ such that
\begin{equation*}
(\mu+\ell)\widetilde C_\ell^\mu(z)+\widetilde C_{\ell-2}^{\mu+1}(z)
=A\widetilde C_\ell^{\mu+1}(z).
\end{equation*}
Comparing the coefficients of the leading term $z^a$ in the both sides by using \eqref{eqn:topC}, we get $A=\mu+\left[\frac{a+1}2\right]$.
\end{proof}

\begin{lem}\label{lem:152579}
Let $\ell\in\N$ and $\mu\in\C$. Then we have
\begin{equation}\label{eqn:152579}
\gamma(\mu,\ell)(z^2-1)\widetilde C_{\ell-1}^{\mu+1}(z)+
(\mu-\frac12)z\widetilde C_\ell^\mu(z)=
\frac12(\ell+1)\gamma(\mu-\frac12,\ell)\widetilde C^{\mu-1}_{\ell+1}(z).
\end{equation}
\end{lem}
\begin{proof}
We apply \eqref{eqn:dzC} to the left-hand side of the following formula 
(see \cite[3.15 (10)]{EMOT53}):
\begin{equation*}
\frac{d}{dz}\left((1-z^2)^{\mu-\frac12} C_\ell^\mu(z)\right)=
\frac{(\ell+1)(\ell+2\mu-1)}{2(1-\mu)}(1-z^2)^{\mu-\frac32} C_{\ell+1}^{\mu-1}(z).
\end{equation*}
By using the identity $2\gamma(\mu,\ell-1)\gamma(\mu-\frac12,\ell)=\ell+2\mu-1$, 
we see that the renormalization \eqref{eqn:Gegen2} gives the formula \eqref{eqn:152579}.
\end{proof}

\begin{lem}\label{lem:152594}
Let $\ell\in\N$ and $\mu\in\C$. Then
\begin{equation}\label{eqn:152594}
(z^2-1)\widetilde C_{\ell-2}^{\mu+1}(z)+(\mu-\frac12)\widetilde C_\ell^\mu(z)=
(\mu+\left[\frac \ell2\right]-\frac12)\widetilde C_\ell^{\mu-1}(z).
\end{equation}
\end{lem}

\begin{proof}
It follows from the identities \eqref{eqn:Gz2} and \eqref{eqn:Gdiff1} in the Weyl algebra
that
\begin{eqnarray*}
&&G_\ell^{\mu-1}\left((z^2-1)\widetilde C_{\ell-2}^{\mu+1}(z)
+(\mu-\frac12)\widetilde C_\ell^\mu(z)\right)\\
&=&(2\mu-1)\left(\left(z\frac{d}{dz}-\ell\right)\widetilde C_\ell^\mu(z)
-2\widetilde C_{\ell-2}^{\mu+1}(z)\right),
\end{eqnarray*}
which vanishes by \eqref{eqn:1524102}. By the uniqueness of the solutions to 
$G_\ell^{\mu-1} f(z)=0$ for $f\in\mathrm{Pol}_\ell[z]_{\mathrm{even}}$ (see Fact \ref{fact:KPGegen}), there exists $c\in\C$, such that
\begin{equation*}
(z^2-1)\widetilde C_{\ell-2}^{\mu+1}(z)+(\mu-\frac12)
\widetilde C_\ell^\mu(z)=c\widetilde C_\ell^{\mu-1}(z).
\end{equation*}
Comparing the coefficients of the leading term $z^\ell$ by \eqref{eqn:topC}, and using the identity
\begin{equation}\label{eqn:152599}
4(\mu+\left[\frac \ell2\right]-\frac12)(\mu+\left[\frac {\ell-1}2\right])=
\ell^2-\ell+2(2\mu-1)(\mu+\ell-1),
\end{equation}
we conclude that $c=\mu+\left[\frac \ell2\right]-\frac12$.
\end{proof}

\begin{lem}\label{lem:1605103}
Let $\ell \in \N$ and $\mu \in \C$. Then we have
\begin{equation}
\widetilde{C}^{\mu+1}_{\ell-2}(z) = 
\gamma(\mu,\ell) z \widetilde{C}^{\mu+1}_{\ell-1}(z)
-\frac{\ell}{2}\widetilde{C}^\mu_\ell(z). \label{eqn:1605103}
\end{equation}
\end{lem}

\begin{proof}
The formula is a direct consequence of 
\eqref{eqn:dzC} and \eqref{eqn:1524102}.
Alternatively, the lemma is derived from 
the following three-term relation
(see \cite[3.15 (27)]{EMOT53}:
\begin{equation}\label{eqn:Er31527}
\ell C^\mu_{\ell}(z)=-2\mu\left(z C^{\mu+1}_{\ell-1}(z)-C^{\mu+1}_{\ell-2}(z)\right).
\end{equation}
\end{proof}

\begin{lem}\label{lem:152633}
For $\ell\in\N$ and $\mu\in\C$,
\begin{equation}\label{eqn:152633}
z\widetilde C_{\ell-1}^{\mu+1}(z)-\gamma(\mu,\ell+1)\widetilde C_{\ell}^{\mu+1}(z)
+\gamma(\mu-\frac12,\ell+1)\widetilde C_{\ell}^{\mu}(z)=0.
\end{equation}
\end{lem}

\begin{proof}
By \eqref{eqn:Gdiff2} and \eqref{eqn:Gaz}, we get
\begin{equation*}
G_\ell^\mu\left(z\widetilde C_{\ell-1}^{\mu+1}(z)-\gamma(\mu,\ell+1)\widetilde C_{\ell}^{\mu+1}(z)\right)=2\frac{d}{dz}\widetilde C_{\ell-1}^{\mu+1}(z)-2\gamma(\mu,\ell+1)
\left(z\frac{d}{dz}-\ell\right)\widetilde C_{\ell}^{\mu+1}(z).
\end{equation*}
By \eqref{eqn:dzC} and \eqref{eqn:1524102}, this amounts to
$$
4\gamma(\mu+1,\ell-1)\widetilde C_{\ell-2}^{\mu+2}(z)-4\gamma(\mu,\ell+1)\widetilde C_{\ell-2}^{\mu+2}(z)=0,
$$
because $\gamma(\mu+1,\ell-1)=\gamma(\mu,\ell+1)$.

Since $z\widetilde C_{\ell-1}^{\mu+1}(z)-\gamma(\mu,\ell+1)\widetilde C_{\ell}^{\mu+1}(z)
\in\mathrm{Pol}_\ell[z]_{\mathrm{even}}$, there exists $A\in\C$ by Fact \ref{fact:KPGegen}
such that
$$
z\widetilde C_{\ell-1}^{\mu+1}(z)-\gamma(\mu,\ell+1)\widetilde C_{\ell}^{\mu+1}(z)=A
\widetilde C_{\ell}^{\mu}(z).
$$
Comparing the coefficients of the leading terms $z^\ell$ on both sides, we have
\begin{equation*}
A=-\frac{\gamma(\mu,\ell+1)(\ell+2\mu)}{2\left(\mu+\left[\frac{\ell+1}{2}\right]\right)}
=-\gamma(\mu-\frac12,\ell+1).
\end{equation*}

Alternatively, the lemma follows directly from the three-term relation
\cite[3.15 (28)]{EMOT53} for the corresponding (unnormalized) Gegenbauer polynomials.
\end{proof}

%%%%%%%%%%%%%%%%%%%%%%%%%%%%%%%%%%%%%%%%%%%%%%
\subsection{Duality of Gegenbauer polynomials for special values}\label{subsist:dualC}
\index{B}{duality of Gegenbauer polynomials|textbf}

We recall from \eqref{eqn:Kla} that 
$\Kla = \prod_{j=1}^\ell
\left(\left[\frac a2\right]+j\right)$ is a 
positive integer for any $\ell, a \in \N$.

\begin{prop}\label{prop:1524113}
Let $a,\ell\in\N$. Then
\begin{eqnarray}
\widetilde C_a^{-a-\ell}(z)&=& (-1)^\ell
\Kla\widetilde C_{a+2\ell}^{-a-\ell}(z),\label{eqn:152471}\\
(z^2-1)^\ell\widetilde C_a^{\frac12+\ell}(z)&=&
\Kla
\widetilde C_{a+2\ell}^{\frac12-\ell}(z).\label{eqn:1524113}
\end{eqnarray}
\end{prop}

\begin{proof}
The first equality \eqref{eqn:152471} was proved in \cite[Lem.\ 4.12]{KOSS15}.
We thus give a proof of the second equality \eqref{eqn:1524113}. 
Since $G_a^{\frac12+\ell}\widetilde C_a^{\frac12+\ell}(z)=0$,
we get from \eqref{eqn:1524111}
$$
G^{\frac12-\ell}_{a+2\ell}\left((z^2-1)^\ell \widetilde C_a^{\frac12+\ell}(z)\right)=0.
$$
Since $(z^2-1)^\ell \widetilde C_a^{\frac12+\ell}(z)\in
\mathrm{Pol}_{a+2\ell}[z]_{\mathrm{even}}$, there exists $A\in\C$ such that
$$
(z^2-1)^\ell \widetilde C_a^{\frac12+\ell}(z)=A\widetilde C_{a+2\ell}^{\frac12-\ell}(z)
$$
by Fact \ref{fact:KPGegen}.
 Comparing the coefficients of the leading term $z^{a+2\ell}$ by \eqref{eqn:topC}, 
 we have
\begin{equation*}
A=\frac{\Gamma\left(\frac12+\left[\frac{a+1}{2}\right]\right)(a+2\ell)!}
{2^{2\ell}\Gamma\left(\frac12+\ell+\left[\frac{a+1}{2}\right]\right)a!}=
\prod_{j=1}^\ell\left(\left[\frac a2\right]+j\right)
=\Kla. 
\end{equation*}
Hence \eqref{eqn:1524113} is proved.
\end{proof}

%%%%%%%%%%%%%%%%%%%%%%%%%%%%%%%%%%%%%%%%%%%%%%
\subsection{Proof of  Theorem \ref{thm:gis}}\label{subsec:gis}

As an application of the three-term relations of (renormalized) 
Gegenbauer polynomials developed in Section \ref{subset:3C},
we give a proof of Theorem \ref{thm:gis} 
\index{B}{F-system}
(solving the F-system)
in this section. 

Let $a, i \in \N$ and $\mu \in \C$. 
We define a linear isomorphism
$\Psi \equiv \Psi(a,i,\mu,n)$ by
\begin{equation}\label{eqn:gf}
\Psi\colon
\bigoplus_{k=0}^{2}\mathrm{Pol}_{a-k}[t]_{\mathrm{even}}
\stackrel{\sim}{\To}
\bigoplus_{k=0}^{2}\mathrm{Pol}_{a-k}[z]_{\mathrm{even}},
\quad (g_0,g_1,g_2)\mapsto (f_0,f_1,f_2)
\end{equation}
with the following relations: $z=e^{\frac{\pi\sqrt{-1}}2}t$ and
\begin{eqnarray*}
g_2(t)&=&f_2(z),\\
g_1(t)&=&e^{-\frac{\pi\sqrt{-1}}2}f_1(z),\\
g_0(t)&=& 
\begin{cases}
f_0 & \text{ if $a = 0$,}\\
f_0(z)-\frac{1}{a}\left(a+\mu-\frac{n+3}2+i\right) zf_1(z)
&\text{ if $a \in \mathbb{N}_+$.}
\end{cases}
\end{eqnarray*}
Via the isomorphism \eqref{eqn:gf}, the convention 
\eqref{eqn:gjvan} for $(g_0(t), g_1(t), g_2(t))$ is translated into the 
following one for $(f_0(z), f_1(z), f_2(z))$:
\begin{equation}\label{eqn:fjvan}
f_1 = f_2 = 0 \quad \text{for $a=0$}; \quad
f_2=0 \quad \text{for $a=1$}; \quad
f_2=0 \quad \text{for $i=1$}; \quad
f_1 = f_2 = 0 \quad \text{for $i=n$}.
\end{equation}
In connection to the F-system for symmetry breaking operator
from $I(i,\lambda)_\alpha$ to $J(i-1,\nu)_\beta$, 
the parameter $\lambda \in \C$ in the principal series representation
$I(i,\lambda)_\alpha$ will be related as
\begin{equation*}
\mu=\lambda-\frac{n-3}2.
\end{equation*}
If $(g_0(t), g_1(t), g_2(t)) = (\eqref{eqn:g0},\eqref{eqn:g1}, \eqref{eqn:g2})$
in Theorem \ref{thm:Fi-}, then
\begin{equation}\label{eqn:gtFz}
\Psi(g_0,g_1,g_2)=\left(C \widetilde{C}^\mu_{a-2}(z), 
A\widetilde{C}^\mu_{a-1}(z), 
\widetilde{C}^\mu_{a-2}(z) \right)
\end{equation}
with 
\begin{equation}\label{eqn:AC}
A:=\gamma(\mu-1,a),\, C:=\frac{\lambda-n+i}a+\frac{i-1}{n-1}=
\frac1a(\mu-\frac{n+3}2+i)+\frac{i-1}{n-1}.
\end{equation}
In what follows, 
we denote by 
\index{A}{L1Lj@$(Lj)$|textbf}
$(Lj)$ the differential equation 
\index{A}{L1Ljg@$L_j(g_0, g_1, g_2)$}
$L_j(g_0, g_1, g_2) = 0$ 
(see \eqref{eqn:Fe}-\eqref{eqn:Fabc})
for simplicity.
Then, via the transformation $\Psi$,
we observe that the differential equations $(Lj)$
for $(g_0,g_1,g_2)$ in Section \ref{subsec:71}
are transferred to differential equations for $(f_0,f_1,f_2)$ as follows:

\begin{lem}\label{lem:20151001}
Via the isomorphism \eqref{eqn:gf},
the triple
$(g_0(t),g_1(t),g_2(t))$ satisfies 
$(Lj)$ if and only if $(f_0(z),f_1(z),f_2(z))$ 
satisfies the corresponding differential equation 
\index{A}{G1Gj@$(Gj)$|textbf}
$(Gj)$ for each $j=1,2,\cdots,9$,
where we set:

\begin{eqnarray*}
&(G1)& G_{a-2}^\mu f_2(z)=0,\\
&(G2)& G_{a-1}^\mu f_1(z)=0,\\
&(G3)& (\vartheta_z-a+1)f_1(z)-\frac{d f_2}{dz}(z)=0,\\
&(G4)& (\vartheta_z+2\mu+a-2)f_2(z)-\frac{d f_1}{dz}(z)=0,\\
&(G5)&\frac{d f_0}{dz}(z)
-\frac{1}a\left(a+\mu-\frac{n+3}2+i\right) (\vartheta_z-a+1)f_1(z)+\frac{n-i}{n-1}
\frac{d f_2}{dz}(z)=0,\\
&(G6)&
a f_0(z)=\left(\mu-\frac{n+3}2+i+\frac{a(i-1)}{n-1}\right)f_2(z)\\
&& \quad \qquad +
e^{-\frac{\pi\sqrt{-1}}2}z\left(
\frac{d f_0}{dz}(z)
-\frac{1}{a}\left(a+\mu-\frac{n+3}2+i\right) (\vartheta_z-a+1)f_1(z)+\frac{n-i}{n-1}
\frac{d f_2}{dz}(z)
\right),
 \\
&(G7)& \frac12 G_{a}^{\mu-1}\left(f_0(z)-\frac{1}{a}\left(a+\mu-\frac{n+3}2+i\right)zf_1(z)\right)+\frac{n-i}{n-1}\frac{df_1}{dz}(z)=0,\\
&(G8)& a f_0(z)=\left(\mu-\frac{n+3}2+i+\frac{a(i-1)}{n-1}\right)f_2(z),\\
&(G9)& \frac{df_0}{dz}(z)
-\left(\frac{i-1}{n-1} + \frac{1}{a} \left(\mu-\frac{n+3}{2}+i\right)\right) (\vartheta_z-a+1) 
f_1(z) = 0.
\end{eqnarray*}
\end{lem}

With the constants $A$, $C$ as in \eqref{eqn:AC}, we define 
polynomials $F_k(z)$ $(k=0,1,2)$ by
\begin{alignat*}{2}
(1) & \quad 
\text{$i=1$, $a\geq 1:$}
&&
\quad
(F_0(z), F_1(z), F_2(z)) 
:= \left(C\widetilde{C}^\mu_{a-2}(z), A\widetilde{C}^\mu_{a-1}(z), 0\right);\\
(2) & \quad
\text{$2\leq i\leq n-1$, $a\geq 1:$}
&&
\quad
(F_0(z), F_1(z), F_2(z)):= 
\left(C\widetilde{C}^\mu_{a-2}(z), A\widetilde{C}^\mu_{a-1}(z), 
\widetilde{C}^\mu_{a-2}(z)\right);\\
(3)& \quad
\text{$i=n$, $a\geq 1:$}
&& 
\quad
(F_0(z), F_1(z), F_2(z)) 
:= \left(\widetilde{C}^\mu_{a-2}(z), 0, 0\right);\\
(4) & \quad
\text{$1 \leq i\leq n$, $a=0:$}
&&
\quad
(F_0(z), F_1(z), F_2(z)) :=(1,0,0).
\end{alignat*}
Note that $F_0=C\widetilde C_{a-2}^\mu(z)\in\mathrm{Pol}_{a-2}[z]_{\mathrm{even}}
\subset \mathrm{Pol}_a[z]_{\mathrm{even}}$.
Then, Theorem \ref{thm:gis} is equivalent to the following assertion via the 
transformation $\Psi$.

\begin{prop}\label{prop:Fgis}
Let $n\geq 3$ and $1 \leq i \leq n$. 
Suppose
 $f_k(z) \in \mathrm{Pol}_{a-k}[z]_{\mathrm{even}}$ $(k=0,1,2)$
with the convention \eqref{eqn:fjvan}. Then, up to scalar multiple,
the solution $(f_0,f_1,f_2)$ to 
\begin{alignat*}{2}
&(G2),\; (G7), \; (G9) && i=1,\\
&(Gr)\; \; (r=1,2,\ldots, 7)\quad && 2\leq i \leq n-1,\\
&(G1)\; &&i=n, 
\end{alignat*}
is given by $(F_0,F_1,F_2)$.
\end{prop}

Proposition \ref{prop:Fgis} in the case $a=0$ is trivial.
Since $f_1 = f_2 = 0$ for $i=n$, 
Proposition \ref{prop:Fgis} in the case $i=n$
is clear from Fact \ref{fact:KPGegen} 
because $(G7)$ is reduced to the 
\index{B}{Gegenbauer differential equation}
Gegenbauer differential equation $G^{\mu-1}_a f_0 =0$.

For $2\leq i \leq n-1$,  
the proof of Proposition \ref{prop:Fgis} is divided into
Lemmas \ref{lem:G1} and \ref{lem:G2} bellow.

\begin{lem}\label{lem:G1} 
Suppose that $a \geq 1$ and $\mu\in\C$. Then we have
\begin{equation*}
\dim_\C\left\{(f_0, f_1, f_2)\in\bigoplus_{k=0}^{2}
\mathrm{Pol}_{a-k}[z]_{\mathrm{even}}:
(f_0,f_1,f_2)\,\mathrm{solves}\, (Gj),\,j=1,2,3,4,8\right\}\leq 1.
\end{equation*}
\end{lem}

The following lemma shows that the left-hand side is equal to one.

\begin{lem}\label{lem:G2}
Suppose $2\leq i \leq n-1$. 
Then, for any $a \in \mathbb{N}_+$ and $\mu\in\C$, the triple $(F_0,F_1,F_2)$
solves $(Gj)$ for all $j=1, \ldots, 9$.
\end{lem}

\begin{proof}[Proof of Lemma \ref{lem:G1}]
We shall prove that 
$(f_0,f_1,f_2)\in
\displaystyle{\bigoplus_{k=0}^{2}
\mathrm{Pol}_{a-k}[z]_{\mathrm{even}}}$ satisfies $(Gj)$ 
for $j=1,2,3,4,8$ only if
$(f_0,f_1,f_2)=p(F_0,F_1,F_2)$ for some $p\in\C$.
We consider the cases $a=1$ and $a\geq 2$, separately.

1) $a=1$:
If $a =1$, then $f_2=0$ by \eqref{eqn:fjvan}. In turn,
$f_0 = 0$ by $(G8)$. Since $\mathrm{Pol}_{a-k}[t]_{\mathrm{even}} =\C \cdot 1$
for $a=k=1$, $f_1$ is a constant function. 
Thus $(f_0,f_1, f_2) \in \C(0, 1, 0)=\C( F_0, F_1, F_2)$.

2) $a\geq 2$:
First we apply Fact \ref{fact:KPGegen} to see that the polynomial solutions to $(G1)$
$(G2)$ are of the form
$f_2(z)=p\widetilde C_{a-2}^\mu(z)(=pF_2(z))$, 
$f_1(z)=q\widetilde C_{a-1}^\mu
(z)$ for some $p,q\in\C$. It then follows from \eqref{eqn:dzC} and \eqref{eqn:1524102}
that $(G3)$ is equivalent to
\begin{equation}
{\tag*{(G3)${}^{\prime}$}}  2(q-p\gamma(\mu,a-2))\widetilde C^{\mu+1}_{a-3}(z)=0,
\end{equation}
whence we get for $a\geq 3$
\begin{equation}\label{eqn:pqG}
q=p\gamma(\mu,a-2)=p\gamma(\mu-1,a).
\end{equation}
Similarly it follows from \eqref{eqn:dzC} and \eqref{eqn:1524102}, and Lemma
\ref{lem:152563} that $(G4)$ is equivalent to
\begin{equation}
\tag*{(G4)${}^{\prime}$} 2\gamma(\mu,a-1)\left( p\gamma(\mu,a-2)-q\right)\widetilde C^{\mu+1}_{a-2}(z)=0,
\end{equation}
where we have used the identity
\begin{equation}\label{eqn:gamma2}
\gamma(\mu,a-1)\gamma(\mu,a-2)=\mu+\left[\frac{a-1}2\right].
\end{equation}
Hence \eqref{eqn:pqG} holds 
if $\gamma(\mu,a-1)\neq 0$, in particular, if $a=2$.
Thus $f_1 = q \widetilde{C}^\mu_{a-1} 
= p \gamma(\mu-1,a)\widetilde{C}^\mu_{a-1}
= pF_1$ for any $a\geq 2$.

Finally, $(G8)$ is equivalent to 
$f_0 = C f_2$ with $C$ in \eqref{eqn:AC}.
Since $f_2 = p F_2$ and $F_0 = C F_2$, this implies
$f_0(z)=CpF_2(z)=pF_0(z)$. 
Hence $(f_0, f_1, f_2) = p (F_0,F_1, F_2)$, 
and the proof is completed.
\end{proof}

\begin{proof}[Proof of Lemma \ref{lem:G2}]
We consider the cases that $a=1$ and $a\geq 2$, separately.

1) $a =1$: 
Obviously, $(F_0,F_1,F_2) = (0,1,0)$ satisfies the equations $(G1)$-$(G9)$.
\vskip 0.1in

2) $a\geq 2$:
$(F_0,F_1,F_2)$ satisfies $(Gr)$ ($r=1,2,8$) by
the definition of $(F_0,F_1,F_2)$ and Fact \ref{fact:KPGegen},
and $(Gr)$ ($r=3,4$) as is in
the proof of Lemma \ref{lem:G1}.
Thus
it remains to prove that
$(F_0,F_1,F_2)$ solves (G5) and (G7). 
(We recall that (G6) and (G9) are linear combinations
of the others.)

For $(f_0,f_1,f_2)=(F_0,F_1,F_2)$, the equation $(G5)$ amounts to
\begin{equation*}
2\left( C+\frac{n-i}{n-1}\right)\gamma(\mu,a-2)\widetilde C^{\mu+1}_{a-3}(z)-
\frac{2}{a}(a+\mu-\frac{n+3}2+i)A\widetilde C^{\mu+1}_{a-3}(z)=0
\end{equation*}
by \eqref{eqn:dzC} and \eqref{eqn:1524102}. This identity obviously holds by the definition \eqref{eqn:AC} of the constants $A$ and $C$.

Finally, let us verify that the triple $(F_0,F_1,F_2)$ satisfies the equation (G7). 
For this we use
\eqref{eqn:Gaz} and \eqref{eqn:Gdiff4}.
Since $G_{a-2}^\mu F_0=G_{a-1}^\mu F_1=0$, the left-hand side of $(G7)$ applied to
$(f_0,f_1,f_2)=(F_0,F_1,F_2)$ amounts to
\begin{eqnarray*}
&&(\vartheta_z+a+2\mu-2)F_0(z)-\frac{1}a(a+\mu-\frac{n+3}2+i)\frac{dF_1}{dz}(z)+
\frac{n-i}{n-1}\frac{dF_1}{dz}(z)\\
&=&(\vartheta_z-a+2)F_0(z)+2(a+\mu-2)F_0(z)-C\frac{dF_1}{dz}(z)\\
&=&
C\left((\vartheta_z-a+2)\widetilde{C}^\mu_{a-2}(z)
+2(a+\mu-2)\widetilde{C}^\mu_{a-2}(z)
-\gamma(\mu-1,a)\frac{d}{dz}\widetilde{C}^\mu_{a-1}(z)\right).
\end{eqnarray*}
By \eqref{eqn:dzC}, \eqref{eqn:1524102} and \eqref{eqn:gamma2} again, this equals
\begin{equation*}
2C\left(\widetilde C_{a-4}^{\mu+1}(z)+(a+\mu-2)\widetilde C_{a-2}^\mu(z)
-(\mu+\left[\frac{a-1}2\right])\widetilde C_{a-2}^{\mu+1}(z) \right),
\end{equation*}
which vanishes by the three-term relation given in \eqref{eqn:152563}.
Hence the proof of Lemma \ref{lem:G2} is complete. 
\end{proof}

Thus Proposition \ref{prop:Fgis} for $2\leq i \leq n-1$ is proved.

Finally, let us consider Proposition \ref{prop:Fgis} in 
the case $i=1$. It is sufficient to show:

\begin{lem}\label{lem:G1'} 
For any $a \in \mathbb{N}_+$ and $\mu\in\C$, we have
\begin{equation*}
\left\{(f_0, f_1)\in\bigoplus_{k=0}^{1}\mathrm{Pol}_{a-k}[z]_{\mathrm{even}}:
(f_0,f_1,0)\,\mathrm{solves}\, (G2),\, (G7), \, (G9)\right\} =\C(F_0,F_1).
\end{equation*}
\end{lem}

\begin{proof}[Proof of Lemma \ref{lem:G1'}]
Since $(F_0,F_1, F_2)$ solves $(Gr)$ for all $r=1,\ldots, 9$, and 
since $(Gr)$ $(r=2, 7, 9)$ does not involve $g_2$, 
we conclude that $(F_0,F_1,0)$ also solves $(Gr)$ $(r=2, 7, 9)$.

Conversely, let us show $(f_0,f_1) \in \mathbb{C}(F_0,F_1)$ 
if $(f_0,f_1,0)$ satisfies $(G2)$, $(G7)$, and $(G9)$.
We observe that for $i=1$, $(G7)$ and $(G9)$ amount to
\begin{eqnarray*}
&(G7)& \frac{1}{2} G_{a}^{\mu-1}
\left(f_0(z)-\frac{1}{a}\left(a+\mu-\frac{n+1}2\right) zf_1(z)\right)+\frac{df_1}{dz}(z)=0,\\
&(G9)&\frac{d f_0}{dz}(z)-\frac{1}{a}
\left(\mu-\frac{n+1}2\right)(\vartheta_z-a+1)f_1(z)=0,
\end{eqnarray*}
respectively. 
By $(G2)$, we have $f_1(z) = p \widetilde C_{a-1}^\mu(z)$ for some constant $p$.
It follows from \eqref{eqn:1524102} that $(\vartheta_z-a+1)f_1(z) = 
2p\widetilde{C}_{a-3}^{\mu+1}(z)$.
Thus $(G9)$ amounts to 
\begin{equation}\label{eqn:f0f1}
\frac{df_0}{dt}(z)=\frac{2}{a}p\left(\mu-\frac{n+1}2\right)\widetilde{C}_{a-3}^{\mu+1}(z).
\end{equation}
By \eqref{eqn:dzC}, $f_0(z)$ is then of the form 
$f_0(z) = q_1 \widetilde C_{a-2}^\mu(z)+q_2$, where $q_1$ and $q_2$ 
are some constants satisfying 
\begin{equation}\label{eqn:pq}
q_1\gamma(\mu-1,a) =\frac{1}{a}p\left(\mu-\frac{n+1}{2}\right).
\end{equation}
Thus,
for $f_1(z) = p \widetilde C_{a-1}^\mu(z)$ and 
$f_0(z) = q_1 \widetilde C_{a-2}^\mu(z)+q_2$ with \eqref{eqn:pq},
by using 
the identities \eqref{eqn:Gaz} and
\eqref{eqn:Gdiff4} of the Gegenbauer differential operator
$G^\mu_\ell$s and the three-term relation
\eqref{eqn:152563},
we see that $(G7)$ implies
\begin{align*}
0=\frac{1}{2}G^{\mu-1}_a\left(
f_0(z) - \frac{1}{a}\left(a + \mu-\frac{n+1}{2}\right)zf_1(z)\right) + \frac{d f_1}{dz}
=\frac{1}{2}G_a^{\mu-1}q_2.
\end{align*}
Therefore $q_2=0$ and so $f_0(z) = q_1 \widetilde C_{a-2}^\mu(z)$.
It is then clear from \eqref{eqn:pq} that 
$(f_0,f_1)$ is proportional to $(F_0,F_1)$
when $\gamma(\mu-1,a)\neq0$. 

Now suppose that $\gamma(\mu-1,a)=0$. 
Since in this case we have $(F_0, F_1) \in \C (\widetilde{C}_{a-2}^\mu , 0)$,
it suffices to show $p=0$ for $f_1(z) = p \widetilde{C}_{a-1}^\mu(z)$.
It follows from \eqref{eqn:gamma} that
if $\gamma(\mu-1,a)=0$, then $\mu-1+\frac{a}{2} = 0$.
If $\mu -\frac{n+1}{2}= 0$, then, as $n$ is assumed to be $n\geq 3$, 
we would have $a \leq -2$.
Thus $\mu-\frac{n+1}{2} \neq 0$ and so, by \eqref{eqn:pq}, $p=0$.
This proves the lemma. 
\end{proof}

Hence Proposition \ref{prop:Fgis} is proved, and
therefore the proof of Theorem \ref{thm:gis} is completed.

\newpage
%%%%%%%%%%%%%%%%%%%%%%%%%%%%%%%%%%%%%%%%%%%%%%
\textheight=220mm

\vskip7pt

\printindex{A}{List of Symbols}
\printindex{B}{Index}

\end{document}